\documentclass[a4paper,reqno]{amsart} 

\usepackage{amssymb, mathtools, mathrsfs, xcolor, environ, hyperref}
\usepackage[alphabetic,short-journals,initials,nobysame]{amsrefs}

\calclayout 
\mathtoolsset{showonlyrefs}
\numberwithin{equation}{section}\allowdisplaybreaks

\theoremstyle{plain}
\newtheorem{theorem}{Theorem}[section]
\newtheorem{lemma}[theorem]{Lemma}
\newtheorem{proposition}[theorem]{Proposition}
\newtheorem{corollary}[theorem]{Corollary}
\theoremstyle{remark}

\newtheorem*{example*}{Example}
\newtheorem{remark}[theorem]{Remark}
\newtheorem*{remark*}{Remark}
\newtheorem{note}{Note}
	
\newcommand{\1}{\mathbf 1}
\newcommand{\R}{\mathbb R}
\newcommand{\eq}{\refstepcounter{equation}(\theequation)}
\newcounter{constant}\newcommand{\C}{\refstepcounter{constant}C_{\theconstant}}\newcommand{\Cr}[1]{C_{\ref*{#1}}}
\newcounter{time}\newcommand{\T}{\refstepcounter{time}T_{\thetime}}\newcommand{\Tr}[1]{T_{\ref*{#1}}}
\DeclarePairedDelimiter{\RBA}{(}{)} \newcommand{\RB}{\RBA*}
\DeclarePairedDelimiter{\SBA}{[}{]} \newcommand{\SB}{\SBA*}
\DeclarePairedDelimiter{\CBA}{\{}{\}} \newcommand{\CB}{\CBA*}
\DeclarePairedDelimiter{\ABA}{\langle}{\rangle} \newcommand{\AB}{\ABA*}
\DeclarePairedDelimiter{\ABSA}{|}{|} \newcommand{\ABS}{\ABSA*}
\DeclarePairedDelimiter{\NORMA}{\|}{\|} \newcommand{\NORM}{\NORMA*}

\newenvironment{extra}{\color{gray}}{\ignorespacesafterend}
\RenewEnviron{extra}{}
\begin{document}

\title[Wright-Fisher Equations with irregular drifts]{ Wright-Fisher stochastic heat equations with irregular drifts}

\author[C. Barnes]{Clayton Barnes}
%\address[C. Barnes]{The Faculty of Data and Decision Sciences \\ Technion --- Israel Institute of Technology \\ Haifa 3200003\\ Israel}
%\email{clayleroy2@gmail.com}

\author[L. Mytnik]{Leonid Mytnik}
%\address[L. Mytnik]{The Faculty of Data and Decision Sciences \\ Technion --- Israel Institute of Technology \\ Haifa 3200003\\ Israel}
%\email{leonid@ie.technion.ac.il}

\author[Z. Sun]{Zhenyao Sun}
%\address[Z. Sun]{School of Mathematics and Statistics \\ Beijing Institute of Technology \\ Beijing 100081 \\ China} 
%\email{zhenyao.sun@gmail.com}

%NEW
\address{The Faculty of Data and Decision Sciences \\ Technion --- Israel Institute of Technology \\ Haifa 3200003\\ Israel}
\email[C. Barnes]{clayleroy2@gmail.com}
\email[L. Mytnik]{leonid@ie.technion.ac.il}
\address{School of Mathematics and Statistics \\ Beijing Institute of Technology \\ Beijing 100081 \\ China} 
\email[Z. Sun]{zhenyao.sun@gmail.com}
%END NEW

\begin{abstract}
	Consider the $[0,1]$-valued continuous random field solution $(u_t(x))_{t\geq 0, x\in \mathbb R}$ to the one-dimensional stochastic heat equation \[ \partial_t u_t = \frac{1}{2}\Delta u_t + b(u_t) + \sqrt{u_t(1-u_t)} \dot W, \] where $b(1)\leq 0\leq b(0)$ and $\dot W$ is space-time white noise. In this paper, we establish the weak existence and uniqueness of the above equation for a class of drifts $b(u)$ that may be irregular at the points where the noise coefficient is non-Lipschitz and degenerate, specifically at $u=0$ or $u=1$. This class of drifts includes non-Lipschitz drifts like $b(u) = u^q(1-u)$ for every $q\in (0,1)$, and some discontinuous drifts like $b(u) = \mathbf 1_{(0,1]}(u)-u$. This demonstrates a regularization effect of the multiplicative space-time white noise without the standard assumption that the noise coefficient is Lipschitz and non-degenerate. 
	
	The method we apply is a further development of a moment duality technique that uses branching-coalescing Brownian motions as the dual particle system. To handle an irregular drift in the above equation, particles in the dual system  are allowed to have a number of offspring with infinite expectation, and even an infinite number of offspring with positive probability. We show that, even though the branching mechanism with an infinite number of offspring causes explosions in finite time, immediately after each explosion, the total population comes down from infinity due to the coalescing mechanism. Our results on this dual particle system are of independent interest.
\end{abstract}
\maketitle

\section{Introduction}
\subsection{Motivation}
	In this paper, we consider the $[0,1]$-valued continuous random field solution $(u_t(x))_{t\geq 0,x\in \mathbb R}$ to the stochastic partial differential equation (SPDE)
\begin{equation} \label{eq:SPDE}
\begin{cases} 
	\partial_t u_t(x) 
	= \frac{1}{2} \Delta u_{t}(x) + b(u_t(x)) + \sigma(u_t(x)) \dot W_{t,x}, 
	&\quad t>0, x\in \mathbb R,
	\\ u_0(x) = f(x), 
	&\quad x\in \mathbb R,
\end{cases}
\end{equation}
	with Wright-Fisher noise coefficient  $\sigma(z) := \sqrt{z(1-z)}$ 
and drift
\begin{equation} \label{eq:drift}
	b(z) 
	:=
	\sum_{k\in \bar{\mathbb N}} b_k z^k = \sum_{k=0}^\infty b_k z^k + b_\infty \mathbf 1_{\{1\}}(z), \quad z\in [0,1],
\end{equation}
	satisfying $b(0)\geq 0 \geq b(1)>-\infty$.
	Here, $\dot W$ is a space-time white noise; $f$ is a $[0,1]$-valued continuous function on $\mathbb R$; $\bar {\mathbb N} = \mathbb N \cup \{\infty\}$; $(b_k)_{k \in \bar {\mathbb N}}$ is a family of real numbers; and $z^\infty := \mathbf 1_{\{1\}}(z)$ for every $z\in [0,1]$.
	(Note that $b(\cdot)$, without the subscript, is a map from $[0,1]$ to $\mathbb R$, while $(b_k)_{k\in \bar{\mathbb N}}$ is a family of real numbers.)

	We are going to  show that there exists a unique in law solution to the above equation \eqref{eq:SPDE} provided there exists an $R\geq 1$ such that 
\begin{equation} \label{eq:CLZsCondition}
	b_1 
	\leq - \sum_{k \in \bar{\mathbb N}\setminus \{1\}} |b_k| R^{k-1}.
\end{equation}
	We will see later that under this assumption the drift $b$ can be non-Lipschitz, and sometimes even discontinuous.  
	In particular, we can handle one-dimensional equations like
	\begin{equation}
	\partial_t u_t
	= \frac{1}{2} \Delta u_{t} + u_t^q(1-u_t) + \sqrt{u_t(1-u_t)} \dot W
	\end{equation} 
	for every $q\in(0,1)$, and 
		\begin{equation}
		\partial_t u_t
		= \frac{1}{2} \Delta u_{t} +  \mathbf 1_{(0,1]}(u_t)-u_t + \sqrt{u_t(1-u_t)} \dot W.
	\end{equation} 
	In the former case, this gives uniqueness in law for the stochastic reaction-diffusion equation with Wright-Fisher noise and H\"older drift near zero (example \ref{rmk:example} below). In the latter case, the drift coefficient $b(\cdot)$ is discontinuous at $u = 0$, exactly the place where the noise coefficient $\sigma(u) = \sqrt{u(1 - u)}$ is non-Lipschitz and degenerate---the well-posedness results for this type of stochastic partial differential equations are rare.
	In fact, as we will discuss later with more details, the corresponding stochastic differential equation (SDE)
\begin{equation}
	\mathrm d X_t =  \RB{\mathbf 1_{(0,1]}(X_t)-X_t}  \mathrm dt+ \sqrt{X_t(1-X_t)} \mathrm dB_t
\end{equation}
	is ill-posed. 

Our interest in investigating the equation \eqref{eq:SPDE} arises from an enormous body of literature studying the SPDE~\eqref{eq:SPDE} with $b(\cdot)$ belonging to a class of certain smooth functions. 
Such equation is sometimes called the
heat equation with Wright-Fisher noise and it  arises as the scaling limits of the stepping stone model in 
population genetics (see~\cite{MR948717}) and other important particle systems (see e.g.~\cites{MR1346264, MR3582808, MR4278798}).
In the above papers the drift belongs to a particular class of smooth functions, however, more general models may give rise to more general drifts (see ~\cite{koskela2024bernoulli} for an interesting discussion on non-spatial models), whose corresponding equations have
been also studied~(see~\cite{barnes2023effect} and \cite{MR4259374}).   
This justifies studying well-posedness for~\eqref{eq:SPDE} in presence of various types of drifts. 
The weak existence of \eqref{eq:SPDE}  is standard for some continuous drift $b(\cdot)$ (see \cite{MR1271224} and \cite{MR4259374}). 
As for uniqueness of solutions, let us mention that the pathwise and strong uniqueness is still not resolved for~\eqref{eq:SPDE} even in the case of zero drift.
Thus, as we have mentioned above,  we will concentrate on deriving weak existence/uniqueness of~\eqref{eq:SPDE} for a class of irregular drifts.

First note, the weak uniqueness for~\eqref{eq:SPDE} with a smooth drift of the form
\begin{equation}
b(z)=c_1(1-z)-c_2z+c_3z(1-z), \quad z\in[0,1], 
\end{equation}
(for  $c_1,c_2 \geq 0, c_3\in \R$)  has been derived in~\cite{MR948717} via a duality argument where the dual process is a system of coalescing Brownian motions with binary branching.
Then, great progress was made in~\cite{MR1813840} where the weak uniqueness has been verified  
for a class of Lipschitz drifts that can be expressed in terms of power series whose coefficients satisfy certain assumptions. 
Note that noise coefficients more general than $\sigma(u)=\sqrt{u(1-u)}$  are allowed in~\cite{MR1813840}, and duality with 
	self-catalytic branching
 Brownian motions (including the branching-coalescing Brownian motions as one of the special cases) is again used for the proof of weak uniqueness. 
 However, one should keep in mind, that in~\cite{MR1813840} only branching with finite mean is allowed in the dual model, which imposes certain restrictions on the drift coefficient
$b$, such as the Lipschitz assumption among others.  

Then an immediate question arises: is it possible to show well-posedness for~\eqref{eq:SPDE} with drifts that are not-necessarily Lipschitz? 
Here we should mention another technique that is often used for resolving the weak uniqueness for stochastic equations. 
Namely, the Girsanov theorem (sometimes its version applied to SPDEs is called  Dawson-Girsanov theorem). 
The Girsanov theorem was used in~\cite{MR4259374} to derive weak uniqeness for \eqref{eq:SPDE} with the drift bounded as follows:
\begin{equation}
|b(z)|\leq K\sqrt{z(1-z)},\quad z\in [0,1],
\end{equation}  
which, at points $z=0$ and $z=1$, are
H\"older continuous with exponent  $1/2$. 
This leads us to a further question:  \emph{Does weak-uniqueness hold for ~\eqref{eq:SPDE} with a drift whose H\"older exponent is less than 1/2 at the points where the noise coefficient degenerates?}

As we have mentioned above, this paper gives an affirmative answer to this question.
To prove it we use a modification of the duality method used in~\cite{MR1813840}.
This modification is by no means trivial. It requires construction of the dual branching-coalescing Brownian motions with branching mechanism not-necessarily having a finite first moment. 
Even more than that, to treat some irregular drifts, the particles are asked to possibly have an {\it infinite number} of children at its branching time! 
In what follows we call this ``infinite" branching.
We establish a set of novel results for the branching-coalescing particle system with infinite branching, including its construction, which we believe are of independent interest. 
As we will show below, the total number of alive particles in this system is ``reflecting from infinity''; the expected number of alive particles at any {\it fixed positive} time is almost surely finite; and the expectation of the number of births in the process over any finite time interval is also finite (see Theorem~\ref{prop:Key} below). 
	The similar phenomenon of ``reflecting from infinity'' is also observed in other branching-coalescing type models, see \cite{MR3729616} and \cite{MR3940763}. In order to handle infinite branching, we used techniques from our previous work \cite{barnes2022coming}, where we prove a ``coming down from infinity'' result for coalescing Brownian motions, where we give necessary and sufficient conditions for an initially infinite collection of coalescing Brownian motions to collapse down to a finite number.

Another motivation for this work comes from the so-called \emph{regularization by noise} area (see~\cite{bib:flandoli15}) which is flourishing nowadays. 
In regularization by noise, one addresses the following question: does adding noise 
transform an ill-posed deterministic differential equations into a well-posed equation?
In this context, the following ordinary SDE has been extensively studied in the literature:  
\begin{equation}
\label{eq:sde1}
\mathrm dX_t=b(X_t)\mathrm dt + \mathrm dB_t, \quad X_0=x_0\in \mathbb R^d, 
\end{equation}   
where $B$ is a $d$-dimensional  Brownian motion and  $b: \R^d\mapsto \R^d$ is a possibly irregular drift, see for example~\cites{bib:zvonkin74, bib:veret80, bib:krylov_rockner05}.
Also  in many cases, well-posedness has been established for ~\eqref{eq:sde1} for the drift $b$ being a generalized function.  
For example, in~\cites{bib:HS81,  bib:legall84, bib:bass_chen01, bib:HLM17} strong existence and uniqueness of solutions to~\eqref{eq:sde1}  has been established for different types of distributional drifts. Weak existence and uniqueness  of solutions to~\eqref{eq:sde1}  has been also established in a number of papers, see for example~\cite{bib:zz17}. Let us note that regularization by additive noise for ordinary  differential equations has been studied for other noises as well,  (L\'evy processes, fractional Brownian processes), see for example~\cites{bib:tanake_tsu74, bib:priola12, bib:CG16, bib:ABM20, bib:Le20, bib:Kremp_Perk22, bib:BLM23,butkovsky2024weak}

As for regularization by additive noise for partial differential equations, the prominent example here  is the
 stochastic heat equation driven by the additive space-time white noise, that is, one takes $\sigma\equiv 1$ in 
\eqref{eq:SPDE}. One of the first results on strong existence and uniqueness for such SPDEs
 with irregular function-valued drifts $b$,  have been obtained in~\cite{bib:gyon_p93a} and \cite{bib:gyon_p93b}. 
 Recently, there have been results on strong existence and uniqueness for such SPDE with $b$ being a generalized function in a certain class (see e.g.~\cite{bib:ABLM22} and \cite{butkovsky2024weak}). 

The well-posedness for  equations with irregular drift driven by muliplicative noise has been studied mainly in the case of non-degenerate and Lipschitz noise coefficients: see, for example, 
\cite{bib:veret80} and \cite{bib:zz17} in the SDE setting and~~\cite{bib:gyon98} in the SPDE setting. As for the well-posedness of equations with degenerate noise coefficients and  irregular drifts see, for example,~\cite{bib:chery_engel05}  for the SDE setting. 
For SPDEs driven by noises with degenerate non-Lipschitz coefficients, the results are not that rich. 
Strong well-posedness has been proved  in~\cite{MR2773025}  for an SPDE in the form~\eqref{eq:SPDE} with  $\sigma$ belonging to H\"older continuous functions with exponent greater than  $3/4$ and Lipschitz drift $b$. 
Weak well-posedness has been recently given in~\cite{han2022exponential} with some non-degenerate $\sigma$ belonging to H\"older continuous functions with exponent greater than  $3/4$ and H\"older drift $b$. 
As for the SPDE~\eqref{eq:SPDE} with  $\sigma$ being a H\"older function with exponent less than or equal to $3/4$, only weak uniqueness for very particular noise coefficients $\sigma$ (such as 
$\sigma(u)=u^\gamma$ with $\gamma\geq 1/2$, or $\sigma(u)=\sqrt{u(1-u)}$) and ``nice" drifts $b$ is known. By ``nice" drift we mean either it satisfies conditions of Girsanov theorem or it is suitable for the duality technique (see~\cite{MR4259374} and \cite{MR1813840} which are already mentioned above). 

Thus, the goal of this paper is to extend the class of drifts in the stochastic heat equation driven by the Wright-Fisher space-time white noise for which the weak well-posedness holds.

\subsection{Main Results}   \label{sec:Main}
Before we introduce our main results, let us first discuss the rigorous definition  of the solution to~\eqref{eq:SPDE}.
	Denote by $\mathcal C(\mathbb R, [0,1])$ the collection of $[0,1]$-valued continuous functions on $\mathbb R$, equipped with the topology of uniform convergence on compact sets.
	If there exists a filtered probability space $(\Omega, \mathscr F, (\mathscr F_t)_{t\geq 0}, \mathbb P_f)$, and on this space, an adapted  $\mathcal C(\mathbb R,[0,1])$-valued continuous process   $(u_{t})_{t\geq 0}$, and an adapted space-time white noise $\dot W$, satisfying $u_0 = f$, and that for every $(t,x)\in (0,\infty)\times \mathbb R$ almost surely
\begin{align} \label{eq:mild}
	u_{t}(x)
	& = \int p_{t}(x-y)f(y)\mathrm dy + \iint_0^t p_{t-s}(x-y) b(u_s(y))\mathrm ds \mathrm dy + {}
	\\&\quad \iint_0^t p_{t-s}(x-y) \sigma(u_s(y)) W(\mathrm ds\mathrm dy),
\end{align}
	then we say $(u_t)_{t\geq 0}$ is a solution to the SPDE \eqref{eq:SPDE}.
	Here, 
\begin{equation} \label{eq:HK}
	p_{t}(x):= e^{-x^2/(2t)}/\sqrt{2\pi t}, \quad (t,x)\in (0,\infty)\times \mathbb R
\end{equation} 
	is the heat kernel, and the third term on the right hand side of \eqref{eq:mild}
		is
	Walsh's stochastic integral  driven by the space-time white noise \cite{MR876085}. 
	Equation \eqref{eq:mild} is also known as the mild form of the SPDE \eqref{eq:SPDE}.

		Let us be more precise about the existence of the solutions.
		In this paper, we will be considering the weak existence. 
		By that, we mean the existence of a filtered probability space $(\Omega, \mathscr F, (\mathscr F_t)_{t\geq 0}, \mathbb P_f)$, a random field $(u_{t}(x))_{t\geq 0,x\in \mathbb R}$, as well as a space-time white noise $W$, satisfying all the requirements above.   
		If $b_\infty = 0$, then the drift coefficient $b(\cdot)$ is continuous; in this case, the weak existence of SPDE \eqref{eq:SPDE} is standard (see \cite[Theorem 2.6]{MR1271224} and \cite[Section 2.1]{MR4259374}). 
		However, if $b_\infty \neq 0$, then the drift coefficient $b(\cdot)$ is discontinuous; in this case, the weak existence of SPDE \eqref{eq:SPDE} is not trivial, and will be part of our main result. 
		
	Let us also be more precise about the uniqueness of the solutions.
Recall two uniqueness concepts---the pathwise  uniqueness and the weak uniqueness.  
	We say that pathwise uniqueness holds for the SPDE \eqref{eq:SPDE} if any two solutions on the same probability space driven by the same white noise are indistinguishable, i.e.\ they are equal for all time, almost surely.
	We say weak uniqueness holds for the SPDE \eqref{eq:SPDE} if any two solutions sharing the same initial value, not necessarily living in the same probability space nor driven by the same white noise, induce the same law in the path space $\mathcal C([0,\infty), \mathcal C(\mathbb R, [0,1]))$.

	As it has been mentioned in the first subsection, the pathwise uniqueness for the SPDE \eqref{eq:SPDE} is still open even in the case of zero drift. 
	The weak uniqueness results are established in \cites{MR948717, MR1813840} and \cite{MR4259374} for a class of ``nice" drifts. 
	In what follows we will say that weak well-posedness holds for the SPDE~\eqref{eq:SPDE} if both weak existence and weak uniqueness hold for it.
	Let us now present our main result.
	\begin{extra}
		\begin{note} \label{sec:ATC}
			The condition for Athreya and Tribe's result \cite{MR1813840}*{Theorem 1} is quite technical.
			Let us explain here that \eqref{eq:AandTsCondition} is exactly what they needed when the noise term $\sigma(z) := \sqrt{ z(1-z)}$ for $z \in [0,1]$. 
			Using the terminology in \cite{MR1813840}*{Theorem 1}, the solution $(u_t(x))_{t\geq 0, x\in \mathbb R}$ is bounded by $R_0 = 1$.
			Notice that \eqref{eq:AandTsCondition} implies that the convergence radius of the power series $b(z) = \sum_{k=0}^\infty b_k u^k$ is strictly larger than $R_0 = 1$.
			If one define $\tilde b(z) = \sum_{k\in \mathbb Z_+\setminus\{1\}} |b_k| z^{k-1}$, then \eqref{eq:AandTsCondition} can be written as $b_1 < - \tilde b(R)$ for some $R>R_0$.
			Therefore, the drift term satisfies (H2) of \cite{MR1813840}*{Theorem 1}. 
			Let us also mention that the Wright-Fisher noise term $\sigma(z) = \sqrt{z(1-z)}$ also satisfies Athreya and Tribe's condition.
			Actually, using Athreya and Tribe's language, we can write the Wright-Fisher coefficient into
			\[
			u(1-u) = \sigma(u)^2 = \sum_{k=0}^\infty \sigma_k u^k
			\]
			with
			\[
			\sigma_k = 
			\begin{cases}
				1, &\quad k = 1,\\
				-1, &\quad k = 2,\\
				0, &\quad \text{otherwise}.
			\end{cases}
			\]
			The convergence radius of this power series is $\infty$.
			Define $\tilde \sigma (z) := \sum_{k\neq 2} |\sigma_k| z^{k-2} = z^{-1}.$
			Now it is clear that for any $R>R_0 = 1$ we have $-1=\sigma_2 < -\tilde \sigma(R)= -R^{-1}$ hold.
			In other word, the noise term also satisfies (H2) of \cite{MR1813840}*{Theorem 1}.
		\end{note}
	\end{extra}
\begin{theorem}
	\label{thm:Main}
	Let $f \in \mathcal C(\mathbb R, [0,1])$ be arbitrary.
	Let $(b_k)_{k\in \bar{\mathbb N}}$ be a family of real numbers. 
	Let the function $b: z\mapsto b(z)$ be given as in \eqref{eq:drift} satisfying $b(0)\geq 0\geq b(1)>-\infty$.
	If there exists an $R\geq 1$ such that \eqref{eq:CLZsCondition} holds,
	then weak well-posedness holds for the SPDE~\eqref{eq:SPDE}.
\end{theorem}

\subsection{Examples}
	As we have mentioned above, Theorem \ref{thm:Main} provides weak existence and weak uniqueness of SPDE \eqref{eq:SPDE} for a set of more singular drifts $b$ than is obtained in \cite{MR1813840} and \cite{MR4259374}.
	Let us discuss some examples.

\subsubsection{Examples with non-Lipschitz drifts} \label{rmk:example}
	From the generalized binomial theorem, we have
	\begin{align}
		(1+x)^q
		= 1 + q x + \frac{q (q - 1)}{2!} x^2 + \frac{q (q - 1)(q - 2)}{3!} x^3 + \dots 
		= \sum_{k=0}^\infty \binom{q}{k} x^k
	\end{align}
	for any $q > 0$ and $|x|\leq 1$,
		where
\[
	\quad \binom{q}{0} := 1; 
	\quad \text{and} \quad \binom{q}{n} := \frac{q(q - 1)\dots (q - n + 1)}{n!}, 
	\quad n\geq 1.
\]Therefore, for $q \in (0, 1)$, setting
\[
	b_0 
	:= 0; 
	\quad b_k 
	:= (-1)^k \binom{q}{k-1}, 
	\quad k \in \mathbb N; 
	\quad \text{and} \quad 
	b_\infty 
	:= 0,
\]
we have $b(z) = - (1-z)^q z$ for $z\in [0,1]$. 
One can also verify that \eqref{eq:CLZsCondition} holds with $R=1$.
\begin{extra}
	(We will verify this in Note \ref{sec:example}.)
\end{extra}
	Therefore,  by Theorem~\ref{thm:Main}, there exists a unique in law solution $u$ to the SPDE \eqref{eq:SPDE} with this drift and arbitrary $f\in  C(\mathbb R, [0,1])$.
	Now $w = 1-u$ will be the unique in law solution to the SPDE
\begin{equation} \label{eq:SPDE2}
\begin{cases} 
	\partial_t w_t
	= \frac{1}{2} \Delta w_{t} + w_t^q (1-w_t) + \sqrt{w_t(1-w_t)} \dot W, 
	&\quad t>0, x\in \mathbb R,
	\\ w_0(x) = 1-f(x), 
	&\quad x\in \mathbb R.
\end{cases}
\end{equation}
	The weak uniqueness of the SPDE \eqref{eq:SPDE2}, which confirms a conjecture we made in \cite{barnes2023effect}, does not follow from the result in \cite{MR1813840}, because the drift term is not Lipschitz.
	It also extends the weak well-posedness result covered in \cite{MR4259374} to allow H\"older drift exponents $q \in (0,1/2)$.

\begin{extra}
	\begin{note}\label{sec:example}
		Let us recall the generalized binomial series which says that for any $q, x \in \mathbb C$ with  $\Re(q) > 0$ and $|x|\leq 1$,
		\begin{align}
			(1+x)^q
			= 1 + q x + \frac{q (q - 1)}{2!} x^2 + \frac{q (q - 1)(q - 2)}{3!} x^3 + \dots 
			= \sum_{k=0}^\infty \binom{q}{k} x^k.
		\end{align}
		Therefore, when $q \in (0,1)$ the function $b(z) := - (1-z)^q z$ has decomposition
		\begin{align}
			b(z) 
			= -z \sum_{k=0}^\infty \binom{q}{k} (-z)^k 
			= \sum_{k=1}^\infty (-1)^k \binom{q}{k-1} z^{k} 
			= \sum_{k=0}^\infty b_k z^k, \quad z\in [0,1],
		\end{align}
		as claimed in Remark \ref{rmk:example}. 
		Let us now verify that the coefficient $(b_k)_{k=0}^\infty$ satisfies \eqref{eq:CLZsCondition}. 
		Notice that 
		\begin{align}
			b_0 = 0, 
			b_1 = -1, 
			b_2 = q, 
			b_3 = - \frac{q(q - 1)}{2!}, 
			\dots
		\end{align}
		In particular, $b_1<0$ and $b_k\geq 0$ for any $k \in \mathbb Z_+ \setminus\{1\}$. 
		Therefore
		\begin{align}
			\sum_{k\in \mathbb Z_+: k\neq 1} |b_k| = \sum_{k\in \mathbb Z_+} b_k - b_1 = b(1) - (-1) = 1.
		\end{align}
		Now, it is clear that \eqref{eq:CLZsCondition} holds with $R = 1$.
	\end{note}
\end{extra}

\subsubsection{Examples with discontinuous drifts}
	Assume that $b_k = 0$ for every finite $k\neq 1$, $b_1 = -1$, and $b_\infty\in[-1,1]$. 
	Let $f\in C(\mathbb R, [0,1])$ be arbitrary.
	Then, by Theorem~\ref{thm:Main}, there exists a unique in law solution $(u_t(x))_{t\geq 0, x\in \mathbb R}$ to the SPDE
	\begin{equation}
	\begin{cases}
		\partial_t u_t = \frac{1}{2}\Delta u_t - u_t + b_\infty \mathbf 1_{\{1\}}(u_t) + \sqrt{u_t(1-u_t)} \dot W,
		\\ u_0 =  1-f \in  C(\mathbb R, [0,1]).
	\end{cases}
	\end{equation}
	By defining $w_t=1-u_t$, we obtain a unique in law solution $(w_t(x))_{t\geq 0, x \in \mathbb R}$ to the SPDE
	\begin{equation}
\label{SPDE_3_1}
	\begin{cases}
		\partial_t w_t = \frac{1}{2}\Delta w_t + 1 -w_t -b_\infty\mathbf 1_{\{0\}}(w_t)   + \sqrt{w_t(1-w_t)} \dot W, 
		\\ w_0 =  f.
	\end{cases}
\end{equation}
We found~\eqref{SPDE_3_1} of a particular interest, since it shows the well-posedness of Wright-Fisher SPDEs with drifts that differ only by their values at the point $w=0$. Also, 
one can 
check that  solutions corresponding to different $b_\infty$ have different distributions: this result is stated in the following lemma, whose proof is delayed to Section \ref{sec:Last1}. 
In what follows we say $f\not\equiv 1$ if  there exists $x\in \mathbb R$ such that $f(x)\not=1$. 
\begin{lemma}\label{lem:discont1}
Let $f\in  C(\mathbb R, [0,1]) $ and let $f\not\equiv 1$. Fix arbitrary $b^{(1)}_\infty, b^{(2)}_\infty\in[-1,1]$ with $b^{(1)}_\infty\not=b^{(2)}_\infty$. 
For $i=1,2$, let $w^{(i)}$ be the unique in law solution to \eqref{SPDE_3_1} with $b_\infty=b^{(i)}_\infty$. 
Then $w^{(1)}$ and $w^{(2)}$ induce different laws on the path space $\mathcal C([0,\infty), \mathcal C(\mathbb R, [0,1]))$.
\end{lemma}

Now,  let us consider the SDE analogue of~\eqref{SPDE_3_1} with  $b_\infty\in[-1,1]$:   	
\begin{equation}
\label{SDE_4_1}
	\mathrm dX_t =\RB{\RB{1 -X_t}- b_\infty \mathbf 1_{\{0\}}(X_t) }\mathrm dt+\sqrt{X_t(1-X_t)} \mathrm dB_t, \quad X_0 = x\in [0,1],
\end{equation}
where $B$ is a Brownian motion. 
In the next lemma we will show that the situation for~\eqref{SDE_4_1} differs drastically from its  SPDE counterpart. 
\begin{lemma}\label{lem:discont2}
$~$
\begin{itemize}
\item[(i)]
Let $b_\infty=1$. Then weak uniqueness does not hold for~\eqref{SDE_4_1}.   
\item[(ii)] For any $b_\infty\in [-1,1)$, there is a pathwise unique solution to~\eqref{SDE_4_1} which solves the equation
\begin{equation}
\label{SDE_4_2}
	\mathrm dX_t = (1 -X_t) \mathrm dt  +   \sqrt{X_t(1-X_t)} \mathrm dB_t, \quad  X_0 = x\in [0,1]. 
\end{equation}
\end{itemize}
\end{lemma}
The proof of the above lemma is simple and is delayed to Section \ref{sec:Last1}. 
As we see from the above lemma, the Wright-Fisher noise has a very different regularizing effect in the SPDE setting compare to the SDE setting. In the case of $b_\infty=1$, the well-posedness holds for the
SPDE but not for the corresponding SDE. As for the case of $b_\infty\in[-1,1)$, in the SPDE setting there is a whole family of unique in law solutions corresponding to different $b_\infty$, while in the SDE setting all the solutions are the same and the value of $b_\infty$ does not play any role.  

\subsection{The dual particle system} \label{sec:DD}
To prove Theorem \ref{thm:Main}, we establish the moment duality between the SPDE \eqref{eq:SPDE} and a branching-coalescing Brownian particle system complimenting the previous results \cite{MR948717}*{Theorem 5.2} and \cite{MR1813840}*{Theorem 1}. 
This particle system has three parameters: 
\begin{itemize}
	\item 
	the branching rate $\mu > 0$;
	\item 
	the offspring distribution $(p_k)_{k\in \bar{\mathbb N}}$, which is a probability measure on $\bar {\mathbb N}$; and
	\item 
	the initial configuration $\mathbf{x}_0 = (x_i)_{i = 1}^n$, which is a (possibly infinite) list of real numbers. Here $n\in \mathbb N \cup \{\infty\}$.
	If $n<\infty$, then $(x_i)_{i = 1}^n$ is a finite list; and if $n=\infty$, then $(x_i)_{i = 1}^\infty$ is an infinite sequence. 
	By our convention we denote $\mathrm{supp}(\mathbf{x}_0) = \{x_i\}_{i = 1}^n \subset \R$ as the set of unique values in $\mathbf{x}_0$, which is the set of different initial locations.
\end{itemize}

Let us give an informal description of the branching-coalescing Brownian particle system, with the above parameters, through \eqref{eq:InitialParticles}--\eqref{eq:Coalescing} below.
\begin{itemize}
	\item[\eq\label{eq:InitialParticles}]
	At time $0$, there are $n$ many initial particles. 
	For each finite integer $i \leq n$, the $i$-th initial particle is located at position $x_i$. 
	\item[\eq\label{eq:Movement}]
	The particles in the system move as independent one-dimensional Brownian motions unless one of the events in the following steps occur. 
	\item[\eq\label{eq:Branching}] 
	 Each particle in the system induces a branching event according to an independent rate $\mu$ exponential clock. 
	At each branching event, the corresponding particle (referred to as the parent) will be killed and replaced by a random number of new particles (referred to as the children) at the location where the parent is killed. 
	The number of the children is independently sampled according to the offspring distribution  $(p_k)_{k\in \bar {\mathbb N}}$.
	\item[\eq\label{eq:Coalescing}]
	Given the pairwise intersection local times of the particles in the system, 
	 each (unordered) pair of particles induces a coalescing event according to an independent rate $1/2$ exponential clock with respect to their intersection local time;  
	and at this coalescing event, one of the particle in that pair will be killed. 
\end{itemize}

We want to mention that \eqref{eq:InitialParticles}--\eqref{eq:Coalescing} does not give a rigorous definition of a particle system yet, due to a problem that, if the total population reaches $\infty$ at some finite time, then it is not clear how the pairwise dynamic \eqref{eq:Coalescing} will work afterwards. 
Notice that this explosion of the total population would occur in finite time if either there are infinitely many initial particles already, or if one assumes that $p_\infty > 0$.  
And as it will be made clear later, to handle the discontinuous drifts $b(\cdot)$ with $b_\infty \neq 0$, we must handle the case of infinite branching. 

This is why we will give a more rigorous construction of the branching-coalescing Brownian particle system in Section \ref{eq:DualParticle}.
In that detailed definition, the trajectory of each particle is constructed using an inductive procedure which allows us to be precise about the meaning of the pairwise dynamic \eqref{eq:Coalescing} even after the total population explodes. 
A similar construction for the coalescing Brownian particle system (without branching) appeared  in \cite{MR1339735}, and was employed already in our recent work \cite{barnes2022coming} where the number of the initial particles is allowed to be infinity. 

For the sake of discussing some of our main results for the dual particle system, let us introduce some notation.
We will use $X^\alpha_t$, an $\mathbb R \cup \{\dagger\}$-valued random variable, to represent the location of a particle labeled by $\alpha \in \mathcal U$ at time $t\geq 0$.
Here the cemetery state $\dagger$ is an element not contained in $\mathbb R$, and 
\[
\mathcal U 
:= \bigcup_{k=1}^\infty \mathbb N^k
\]
is the space of the Ulam-Harris labels.
We will label the particles in the system using the Ulam-Harris labels in a way that suggests their lineages: 
the initial particles are labeled with integers $\{i\in \mathbb N : i \leq n\}$, and if a particle has the label $\alpha = (\alpha_1, \cdots, \alpha_{m-1}, \alpha_m)$, then it is the $\alpha_m$-th child created in the branching event induced by the particle $\overleftarrow{\alpha} := (\alpha_1, \cdots, \alpha_{m-1})$. 
  
For every $t\geq 0$, let us also denote by
$
	I_t 
	:= \{\alpha \in \mathcal U: X^\alpha_t \in \mathbb R\}
$
the collection of labels of the particles alive at time $t$; and by $J_t$ the collection of labels of the particles who induced a branching event up to time $t$.
The cardinality of a given set $I$ will be denoted by $|I|$. 
For some technical reason, we always assume that the set of the initial
	locations $\mathrm{supp}(\mathbf{x}_0)$ has finite cardinality. 
That is, the number of initial particles may be infinite but the set of their locations is finite.
The probability space for the dual particle system will be denoted by $(\tilde \Omega, \tilde {\mathcal F}, \tilde{\mathbb P})$, which is not necessarily the same probability space corresponding to the SPDE \eqref{eq:SPDE}.  

Our main result on the branching-coalescing Brownian particle system, which is of the  independent interest,  
is given in the next theorem.  Note that  $\{(X_t^\alpha)_{t\geq 0} : \alpha \in \mathcal U\}$ in this theorem denotes 
a branching-coalescing Brownian particle system which will be formally constructed in Section~\ref{eq:DualParticle}.

	\begin{theorem} \label{prop:Key} 
	Suppose that $\{(X_t^\alpha)_{t\geq 0} : \alpha \in \mathcal U\}$ is a branching-coalescing Brownian particle system with arbitrary branching rate $\mu>0$, offspring distribution $(p_k)_{k \in \bar{\mathbb N}}$, and initial configuration $\mathbf{x}_0$.
	Suppose that the number of initial locations is finite, that is $|\mathrm{supp}(\mathbf{x}_0)|<\infty$.  
	Then, the following statements hold.
	\begin{enumerate}
		\item For every $t > 0$, 
		$\tilde{\mathbb E}\SB{\ABS{I_t}} < \infty$. 
		\item
		For every $t\geq 0$, 	
		\[
		\tilde{\mathbb E}\SB{|J_t| }
		= \mu \tilde {\mathbb E} \SB{\int_0^t |I_s| \mathrm ds }
		< \infty.
		\] 
	\end{enumerate}
\end{theorem}

\begin{remark}\label{ren:NE}
	Recall that we allow infinite offspring with positive probability, i.e.\ it is possible that $p_\infty > 0.$
	We want to mention some immediate corollaries of Theorem \ref{prop:Key}. 
	Note that by this theorem, there are $|J_t|< \infty$ many branching events up to the finite time $t$. 
	Since they are induced by independent exponential clocks, those branching events happens at different times.
	Let us denote the times of those branching events by 
	\[0<\tau_1 < \tau_2 < \dots < \tau_{|J_t|} < t\]
	and define $\tau_0 = 0$ for convention.
	For each positive integer $k \leq |J_t|$, note that $t\mapsto |I_t|$ is non-increasing in the interval $(\tau_{k-1},\tau_k)$ due to the coalescing of the particles; also note that Theorem \ref{prop:Key} (2) implies that almost surely $\int_{\tau_k}^{\tau_{k+1}} |I_s| \mathrm ds< \infty$. 
	So it must be the case that, almost surely, $|I_s|<\infty$ for every $s \in (\tau_{k}, \tau_{k+1})$. 
	In other word, the total population is reflecting from infinity, and can only reach infinity  at the times $\{\tau_k: k = 0, 1,2,\dots\}$. 
\end{remark}

	Our duality formula between the SPDE \eqref{eq:SPDE} and its dual branching-coalescing Brownian particle system is a natural generalization of \cite{MR1813840}*{Theorem 1}, and will be presented later in Section \ref{eq:DualParticle}. 
	The weak uniqueness of \eqref{eq:SPDE} is a standard corollary of this duality formula. 
	In the proof of the weak existence of \eqref{eq:SPDE}, the duality formula will also play a crucial role when the drift is discontinuous.

\subsection{Organization of the paper}
	In Section \ref{eq:DualParticle}, we construct the dual particle system, state the crucial duality formula in Proposition \ref{prop:Duality}, and give the proof of the weak uniqueness part of Theorem \ref{thm:Main} using the duality.
	In Section \ref{sec:key}, we prove several key properties for the dual particle system, in particular, Theorem \ref{prop:Key} .
	In Section \ref{sec:Duality}, we give the proof of the duality formula Proposition \ref{prop:Duality} using the results we proved in Section \ref{sec:key}. 
	In Section \ref{sec:Last}, we give the proof of the weak existence part of Theorem \ref{thm:Main}. 
	In the Appendix, we collect the proofs of several technical lemmas. 
	
\subsection*{Acknowledgement}
	We want to thank Julien Berestycki, Pascal Maillard,
	Michel Pain, and Xicheng Zhang for helpful conversations. 
		Clayton Barnes was supported as a Zuckerman Postdoctoral Fellow at the Technion during the time the research was completed.
		Leonid Mytnik is supported by ISF grant 1985/22.
		Zhenyao Sun is supported by National Key R\&D Program of China no.~2023YFA1010100 and National Natural Science Foundation of China no.~12301173. 
		Zhenyao Sun is the corresponding author. 

\section{Duality} \label{eq:DualParticle}
\subsection{The construction of the branching-coalescing Brownian particle system}
		In this subsection, we give the formal construction of the branching-coalescing Brownian particle system. 
		Recall from Subsection \ref{sec:DD} that this model has three parameters: the branching rate $\mu$, the offspring distribution $(p_k)_{k\in \bar{\mathbb N}}$ and the initial configuration $(x_i)_{i=1}^n$. 
		Again, the initial number of the particles $n \in \mathbb N \cup\{\infty\}$ is allowed to be infinite; and for every finite integer $i \leq n$, $x_i$ is the location of the $i$-th initial particles. 
		Also recall that $\mathcal U$ is the collection of the Ulam-Harris labels. 
	
	The Ulam-Harris labeling system is commonly used in the study of the branching particle systems, see \cite{MR0030716} for one of its early appearances. 
	A different labeling system, using the prime factorization, is proposed in \cite{athreya1998probability} for self-catalytic branching Brownian motions. It was mentioned in both \cite{athreya1998probability} and \cite{MR1813840}  that the particular labeling convention is not crucial to the duality method.
	However, in this paper, the Ulam-Harris labeling will help us be precise about the pairwise dynamic given by \eqref{eq:Coalescing}. 
	In fact, as it will be made more clear later, when two particles coalesce, we will always remove the one with the larger label according to a total order $\prec$ of the space $\mathcal{U}$. 
	This order is defined such that for any $\alpha \in \mathcal{U}$ and $\beta \in \mathcal{U}$, $\alpha \prec \beta$ if and only if one of the following three statements holds:
	\begin{itemize}
		\item[(i)]
		$\|\alpha\|_\infty < \|\beta\|_\infty$;
		\item[(ii)] 
		$\|\alpha\|_\infty= \|\beta\|_\infty$ and $|\alpha| < |\beta|$;
		\item[(iii)] 
		$\|\alpha\|_\infty = \|\beta\|_\infty$, $|\alpha| = |\beta|$, and there exists an integer $m>1$ such that $\alpha_m < \beta_m$ and $\alpha_k = \beta_k$ for every $k < m$.
	\end{itemize}
	Here, $|\alpha| := k$ and $\|\alpha\|_\infty:= \max\{\alpha_1,\alpha_2, \dots, \alpha_k\}$  for every $\alpha = (\alpha_1, \cdots, \alpha_k) \in \mathcal{U}$. 
	We want to mention that different orders for the labeling space $\mathcal{U}$ are possible for the construction of the branching-coalescing Brownian motions. 
	However, this order $\prec$ is particularly designed so that some technical lemmas (Lemmas \ref{lem:TC}, \ref{lem:DT} and \ref{lem:KS} below) hold.

	Let us be precise about some building blocks for the construction. 
	Let $\{(B^\alpha_t)_{t\geq 0}: \alpha \in \mathcal U\}$ be a family of one-dimensional independent standard Brownian motions initiated at position $0$.
	\begin{itemize}
	\item[\eq\label{eq:hatN}]Let $\mathfrak N$ be a Poisson random measure on the space $(0,\infty) \times \mathcal U \times \bar {\mathbb N}$  with intensity $\hat {\mathfrak N}$ given so that
	$\hat {\mathfrak N} ((0,t] \times \{\alpha\} \times\{k\}) = \mu t p_k$ for $t>0, \alpha \in \mathcal U, k \in \bar{\mathbb N}.$
\end{itemize}
	We assume that both $\{(B^\alpha_t)_{t\geq 0}: \alpha \in \mathcal U\}$  and $\mathfrak N$ are defined on the same complete probability space, and are independent from each other.  
	As mentioned in Subsection \ref{sec:DD}, this probability space will be denoted by $(\tilde \Omega, \tilde {\mathscr F}, \tilde{\mathbb P})$. 

	Inductively for each $\beta \in \mathcal U$, we construct random elements 
\begin{equation} \label{eq:TheRandomElement}
	\mathcal X_\beta:= \RB{\xi_\beta, (\tilde X^\beta_t)_{t\geq 0}, (L^{\alpha,\beta}_t)_{\alpha \prec \beta, t\geq 0}, (\mathfrak M(\cdot \times \CB{(\alpha,\beta)}))_{\alpha \prec \beta} ,  (\zeta_{\alpha,\beta})_{\alpha \preceq \beta}, \zeta_\beta, (X^\beta_t)_{t\geq 0}, Z_\beta}
\end{equation}
	according to the following rules \eqref{eq:InitialParticle}--\eqref{eq:TheLastRule} assuming that the random elements
\begin{equation} \label{eq:ConditionalConstruction} 
	\CB{\mathcal X_\alpha : \alpha \in \mathcal U, \alpha \prec \beta }
\end{equation}
	are already constructed.
\begin{itemize}
	\item[\eq\label{eq:InitialParticle}]
 	Define a $\mathbb R_+$-valued random variable $\xi_\beta$ and Brownian motion $(\tilde X^\beta_t)_{t\geq 0}$ so that
\begin{itemize}
\item[(i)]
	if $|\beta| = 1$ and $\beta \leq n$, then $\xi_\beta := 0$ and $\tilde X^\beta_t := x_\beta + B^\beta_{t}$ for every $t\geq 0$;
\item[(ii)]
if $|\beta| = 1$ and $\beta > n$, then $\xi_\beta := 0$ and $\tilde X^\beta_t :=B^\beta_{t}$ for every $t\geq 0$;
\item[(iii)]
	if $|\beta| > 1$, then $\xi_\beta := \zeta_{\overleftarrow{\beta}, \overleftarrow{\beta}}$ and
\[
	\tilde X^\beta_t  := 
\begin{cases}
	\tilde X^{\overleftarrow{\beta}}_{t}, 
	&\quad t\in [0, \xi_\beta),
	\\ \tilde X^{\overleftarrow{\beta}}_{\xi_\beta}+ B^\beta_{t} - B^\beta_{\xi_\beta} , 
	&\quad t \in [\xi_\beta, \infty).
\end{cases}
\]
\end{itemize}
\item[\eq\label{eq:LocalTime}]
	For each $\alpha \in \mathcal U$ with $\alpha \prec \beta$, $t\geq 0$ and $z\in \mathbb R$, define $L_{t,z}^{\alpha, \beta}$ to be the local time of the process $\tilde X^\alpha_\cdot - \tilde X^\beta_\cdot$ at position $z$ up to time $t$, i.e. 
\[
	L_{t,z}^{\alpha, \beta} :=  \lim_{h\downarrow 0}\frac{1}{h} \int_0^t \mathbf 1_{\{\tilde X^\alpha_s- \tilde X^\beta_s \in [z,z+h)\}} \mathrm d\AB{ \tilde X^\alpha_\cdot - \tilde X^\beta_\cdot }_s;
\]
	and write $L_{t}^{\alpha, \beta} := L_{t,0}^{\alpha, \beta}$.
	Here, $\langle \tilde X^\alpha_\cdot - \tilde X^\beta_\cdot \rangle$ is the quadratic variation of the process  $\tilde X^\alpha_\cdot - \tilde X^\beta_\cdot$.
	Without loss of generality, we assume that almost surely, $(t,z)\mapsto L_{t,z}^{\alpha, \beta} $ is continuous, c.f. \cite[Corollary 1.8 Chapter VI]{MR1725357}.
\item[\eq\label{eq:ComM}]
	Conditioned on $\{(B^\alpha_t)_{t\geq 0}: \alpha \in \mathcal U\}$, $\mathfrak N$, and the random elements in \eqref{eq:ConditionalConstruction}, for every $\alpha \in \mathcal U$ with $\alpha \prec \beta$, let $\mathfrak M(\cdot \times \{(\alpha,\beta)\})$ be a Poisson random measure on $(0,\infty)$ with intensity $\hat {\mathfrak M}(\cdot \times \{(\alpha,\beta)\})$ given such that
	\begin{equation} 
		\hat {\mathfrak M} ((0,t] \times \{(\alpha,\beta)\}) 
		= \frac{1}{2}L_t^{\alpha,\beta},  
		\quad t>0.
	\end{equation}
\item[\eq\label{eq:xiab}]
	Define
	\[
	\zeta_{\beta, \beta} := \inf\{t>\xi_\beta: \mathfrak N(\{t\}\times \{\beta\} \times \overline {\mathbb N}) = 1\},
	\]
	and, for each $\alpha \in \mathcal U$ with $\alpha \prec \beta$, 
	\[
	\zeta_{\alpha,\beta} := \inf \{t> \xi_\beta \vee \xi_\alpha: \mathfrak M (\{t\} \times \{(\alpha,\beta)\}) = 1\}.
	\]
\item[\eq\label{eq:Lifetime}]
	Define a $\mathbb R_+$-valued random variable $\zeta_\beta$ so that
	\begin{itemize}
		\item[(i)] 
		if $|\beta|=1$ and $\beta \leq n$, then \[\zeta_\beta := \inf \RB{ \CB{\zeta_{\beta,\beta}} \cup \CB{\zeta_{\alpha,\beta}: \alpha \in \mathcal U, \alpha \prec \beta, \zeta_{\alpha,\beta} \leq \zeta_\alpha}};\]
		\item[(ii)] 
		if $|\beta| > 1$, $\zeta_{\overleftarrow{\beta}} = \zeta_{\overleftarrow{\beta},\overleftarrow{\beta}}$ and $\beta_{|\beta|} \leq Z_{\overleftarrow{\beta}}$, then  \[\zeta_\beta := \inf \RB{ \CB{\zeta_{\beta,\beta}} \cup \CB{\zeta_{\alpha,\beta}: \alpha \in \mathcal U, \alpha \prec \beta, \zeta_{\alpha,\beta} \leq \zeta_\alpha}};\]
		\item[(iii)] if neither of the conditions in (i) nor (ii) hold, then $\zeta_\beta := \xi_\beta$.
	\end{itemize}
\item[\eq]
	Define the $\mathbb R\cup \{\dagger\}$-valued process 
\begin{equation}
	X^\beta_t:= 
	\begin{cases}
	\dagger, &\quad t\in [0, \xi_\beta),
	\\ \tilde X^\beta_t, &\quad t\in [\xi_\beta, \zeta_\beta),
	\\ \dagger, & \quad t\in [\zeta_\beta, \infty).
	\end{cases}
\end{equation}
\item[\eq\label{eq:TheLastRule}]
	Define a $\bar{\mathbb N}$-valued random variable $Z_\beta$ to be the unique $z \in \bar{\mathbb N}$ such that $\mathfrak N$ has an atom at $(\zeta_{\beta,\beta}, \beta, z)$.
\end{itemize}
	We will refer to the family of processes $\{(X_t^\alpha)_{t\geq 0}: \alpha \in \mathcal U\}$, constructed through  \eqref{eq:InitialParticle}--\eqref{eq:TheLastRule}, as a branching-coalescing Brownian particle system with branching rate $\mu$, offspring distribution $(p_k)_{k\in \bar{\mathbb N}}$, and initial configuration $(x_i)_{i=1}^n$. 
	
	Let us give some further comments on the above construction.	
	For each $\beta \in \mathcal U$, we call the random variables $\xi_\beta$, and $\zeta_\beta$, the birth-time, and the death-time, of the particle $\beta \in \mathcal U$, respectively.
	For each $\beta \in \mathcal U$, if $\zeta_\beta = \zeta_{\beta, \beta}$ holds, then we say particle $\beta$ induced a branching event at the time $\zeta_{\beta,\beta}$;
	and if there exists an $\alpha \in \mathcal U$ with $\alpha \prec \beta$ such that $\zeta_\beta = \zeta_{\alpha, \beta}$, then we say the particle pair $(\alpha,\beta)$ induced a coalescing event at time $\zeta_{\alpha,\beta}$.
	Note that $\mathfrak M$ is a  random measure on $(0,\infty)\times \mathcal R$ where 
\[
	\mathcal R
	:= \{(\alpha, \beta) \in \mathcal U^2: \alpha \prec \beta\}.
\] 
	Intuitively speaking, the branching, and the coalescing, events are governed by the random measure $\mathfrak{N}$, and $\mathfrak M$, respectively. 
	As have already been mentioned in Subsection \ref{sec:DD}, for every $t\geq 0$, we denote by
\begin{equation} \label{eq:It}
	I_t 
	:= \{\alpha \in \mathcal U: X^\alpha_t \in \mathbb R\} = \{\alpha \in \mathcal U: \xi_\alpha \leq t< \zeta_\alpha \}
\end{equation}
	the collection of labels of the particles alive at time $t$; and by 
\begin{equation} \label{eq:Jt}
	J_t 
	:= \{\alpha \in \mathcal U: \zeta_{\alpha,\alpha} = \zeta_\alpha  \leq t \}
\end{equation}
	the collection of labels of the particles who induced a branching event up to time $t$.
	
	We say a branching-coalescing Brownian particle system is a \emph{coalescing Brownian particle system} if its offspring distribution satisfies $p_1 = 1$; and say it is a \emph{killing-coalescing Brownian particle system} if its offspring distribution satisfies $p_0=1$.
	In \cite{barnes2022coming}, we give a necessary and sufficient condition for the total population of a coalescing Brownian particle system to 
come down from infinity.
	We also identified all the coming down rates for different initial configurations. 
	The proof of Theorem \ref{prop:Key} heavily relies on those results in  \cite{barnes2022coming}. 

\subsection{The duality formula}	
	For the rest of this paper, let us assume without loss of generality that 
\begin{equation} \label{eq:NonDegenerate}
	\sum_{k=2}^\infty |b_k| + |b_\infty|  
	> 0.
\end{equation}
	(Otherwise, $b(z)= b_0 + b_1 z$ for every $z\in [0,1]$; and the weakly well-posedness of \eqref{eq:SPDE} in this case is already given by \cite{MR948717} and \cite{MR1813840}.)
	\begin{extra}
		(We will explain this a little bit in Note \ref{note:021}.)
	\end{extra}
	To build a connection between the branching-coalescing Brownian particle system and the SPDE \eqref{eq:SPDE}, we will be working with a specific branching rate $\mu$ and a specific offspring distribution $(p_k)_{k \in \bar {\mathbb N}}$ given by 
\begin{equation} \label{eq:TheBranchingRate}
	\mu 
	= |b_0| + \sum_{k=2}^\infty |b_k| + |b_\infty| > 0
\end{equation}
	and
\begin{equation} \label{eq:offspring}
	p_k 
	= \mu^{-1} |b_k| \mathbf 1_{\{k\neq 1\}}, 
	\quad k\in \bar{\mathbb N}.
\end{equation}

\begin{extra}
	\begin{note} \label{note:021}
		If $b(z)= b_0 + b_1 z$, then from the condition $b(0)\geq 0 \geq b(1)$, we have  $b_0 \geq 0$ and $b_0+b_1 \leq 0$. 
		This allows us to write $b(z) = c_1 (1-z)- c_2 z$ where $c_1 = b_0 \geq 0$ and $c_1=-b_1-b_0\geq 0$. 
		Now it is clear that the weak-uniqueness of \eqref{eq:SPDE} in this case is already given by \cite{MR948717}.
	\end{note}
\end{extra}

	We first give the finiteness of the expectation of a certain functional of the particle system that will be used in the presentation of the duality formula. 

\begin{proposition} \label{prop:ExponentialTerm}
	Let $(b_k)_{k\in \bar{\mathbb N}}$ be a family of real numbers satisfying \eqref{eq:NonDegenerate} and \eqref{eq:CLZsCondition} for some $R\geq 1$.
	Let $\{(X_t^\alpha)_{t\geq 0} : \alpha \in \mathcal U\}$ be a branching-coalescing Brownian particle system with branching rate $\mu$ given as in \eqref{eq:TheBranchingRate}, offspring distribution $(p_k)_{k \in \bar{\mathbb N}}$ given as in \eqref{eq:offspring}, and an initial configuration $(x_i)_{i=1}^n$ with finite $n$.
	Then, it holds that
\begin{equation} \label{eq:exp}
	\tilde{\mathbb E}\SB{e^{K_t}} < \infty,
	\quad t\geq 0
\end{equation}
	where
\begin{equation} \label{eq:Kt}
	K_t
	:= (\mu+b_1) \int_0^t |I_s| \mathrm ds, 
	\quad t\geq 0.
\end{equation}
\end{proposition}

	The proof of Proposition \ref{prop:ExponentialTerm} is postponed to Subsection \ref{sec:ET}. 
	
	Now we are ready to present the duality formula between the SPDE \eqref{eq:SPDE} and our branching-coalescing Brownian particle system. 

\begin{proposition} \label{prop:Duality}
	Let $f\in \mathcal C(\mathbb R, [0,1])$, $n\in \mathbb N$, and $(x_i)_{i=1}^n$ be a finite list of real numbers. 
	Let $(b_k)_{k\in \bar{\mathbb N}}$ be a family of real numbers satisfying \eqref{eq:NonDegenerate} and \eqref{eq:CLZsCondition} for some $R\geq 1$.
	Suppose that the real-valued function $(b(z))_{z\in [0,1]}$, given as in \eqref{eq:drift}, satisfies $b(0)\geq 0 \geq b(1)>-\infty$. 
	Suppose that the $\mathcal C(\mathbb R, [0,1])$-valued process $(u_t)_{t\geq 0}$, on a filtered probability space $(\Omega, \mathscr F, (\mathscr F_t)_{t\geq 0},\mathbb P_f)$, is a solution to the SPDE \eqref{eq:SPDE} with initial value $u_0 = f$.
	Also, suppose that $\{(X_{t}^\alpha)_{t\geq 0}: \alpha \in \mathcal U\}$ is a branching-coalescing Brownian particle system, on a probability space $(\tilde \Omega, \tilde {\mathscr F}, \tilde{\mathbb P})$, with initial configuration $(x_i)_{i=1}^n$, branching rate $\mu>0$ given as in \eqref{eq:Branching}, and offspring distribution $(p_k)_{k\in \bar{\mathbb N}}$ given as in \eqref{eq:offspring}.
	Then it holds for every $T \geq 0$ that
\begin{equation} \label{eq:Duality}
	\mathbb E_f\SB{\prod_{i=1}^n u_T(x_i)} 
	= \tilde{\mathbb E}\SB{  (-1)^{\ABS{\tilde J_T}} e^{K_T} \prod_{\alpha \in I_T} f(X_T^\alpha) }.
\end{equation}
	Here, $(K_t)_{t\geq 0}$ is given as in \eqref{eq:Kt} and
\begin{equation} \label{eq:DefTJT}
	\tilde J_t
	:= \{\alpha \in \mathcal U: \zeta_{\alpha,\alpha}=\zeta_\alpha \leq t, b_{Z_\alpha} < 0\}, \quad t\geq 0.
\end{equation}
\end{proposition}
	The proof of Proposition \ref{prop:Duality} is postponed to Section \ref{sec:Duality}.
As we have mentioned, Proposition \ref{prop:Duality} is a generalization of Theorem 1 of \cite{MR1813840}. 
	One of the main assumptions in \cite{MR1813840} is the finiteness of the expected number of offspring in the branching mechanism that guarantees the non-explosion of the system.
	We do not make such an assumption, but still are capable (due to the coalescent mechanism) of showing the finiteness of the total population at almost every time, as well as the above duality formula.
	
	The weak uniqueness part of Theorem \ref{thm:Main} is a direct corollary of the duality formula Proposition \ref{prop:Duality}. 

\begin{proof}[Proof of the weak uniqueness part of Theorem \ref{thm:Main}]
	From Theorem~\ref{prop:Key} and \ref{prop:ExponentialTerm} we know that the right hand side of \eqref{eq:Duality} is well-defined and finite. 
	The desired result now follows from Proposition \ref{prop:Duality} and \cite[Lemma 1]{MR1813840}.
\end{proof}

\section{Analysis for the dual particle system} \label{sec:key} 
	In this section, we  prove several properties for the branching-coalescing Brownian particle system. 
	In particular, we will prove Theorem~\ref{prop:Key}, which implies that the total population is finite at almost every time. 
	We will also give the proof of Proposition \ref{prop:ExponentialTerm}, and several other integrability results, which will be used in the proof of Proposition \ref{prop:Duality} in Section \ref{sec:Duality}. 
\subsection{The truncated particle system} \label{sec:Truncated}
	In this subsection, let us take a branching-coalescing Brownian particle system $\{(X^\alpha_t)_{t\geq 0}: \alpha \in \mathcal U\}$ with an arbitrary branching rate $\mu >0$, initial configuration $(x_i)_{i=1}^n$, and offspring distribution $(p_k)_{k\in \bar{\mathbb N}}$, constructed through  \eqref{eq:InitialParticle}--\eqref{eq:TheLastRule}.
	There is a natural coupling between this branching-coalescing Brownian particle system $\{(X^\alpha_t)_{t\geq 0}: \alpha \in \mathcal U\}$ and a branching Brownian  particle system $\{(\bar X^\alpha_t)_{t\geq 0}: \alpha \in \mathcal U\}$, sharing the same initial configuration, spatial movement of the particles, and the number of offspring in each of their shared branching events.
	The difference is that in the latter system we remove the coalescing mechanism, so the particles will branch but no longer coalesce with each other. 
	As a consequence, this new system dominates the original one, in the sense that, almost surely, for every $\alpha \in \mathcal U$ and $t\geq 0$, $X^\alpha_t\in \mathbb R$ (that is, $X^\alpha_t$ is not in the cemetary state) implies that $\bar X^\alpha_t = X^\alpha_t$. 
	More precisely, this coupling is realized through \eqref{eq:BBM1} and \eqref{eq:BBM2} below.
\begin{itemize}
\item[\eq\label{eq:BBM1}]
	For each $\beta \in \mathcal U$, define a $\mathbb R_+$-valued random variable $\bar \zeta_\beta$ inductively so that
\begin{itemize}
\item[(i)] 
	if $|\beta|=1$ and $\beta \leq n$, then $\bar \zeta_\beta := \zeta_{\beta,\beta}$. 
\item[(ii)] 
	if $|\beta| > 1$, $\bar \zeta_{\overleftarrow{\beta}} = \zeta_{\overleftarrow{\beta},\overleftarrow{\beta}}$ and $\beta_{|\beta|} \leq Z_{\overleftarrow{\beta}}$, then $\bar \zeta_\beta := \zeta_{\beta,\beta};$
\item[(iii)] 
	if neither of the conditions in (i) nor (ii) hold, then $\bar \zeta_\beta := \xi_\beta$.
\end{itemize}
\item[\eq\label{eq:BBM2}]
	For each $\beta \in \mathcal U$, define a $\mathbb R\cup \{\dagger\}$-valued process 
\begin{equation}
	\bar X^\beta_t:= 
\begin{cases}
	\dagger, &\quad t\in [0, \xi_\beta),
	\\ \tilde X^\beta_t, &\quad t\in [\xi_\beta, \bar \zeta_\beta),
	\\ \dagger, & \quad t\in [\bar \zeta_\beta, \infty).
\end{cases}
\end{equation}
\end{itemize}

	Define 
\begin{equation} \label{eq:barI}
	\bar I_t 
	= \{\alpha \in \mathcal U: \bar X^\alpha_t \in \mathbb R\},
\end{equation}
	and
\[
	\bar J_t 
	= \{\alpha \in \mathcal U: \zeta_{\alpha,\alpha}=\bar \zeta_\alpha \leq t\}
\]
	to be the labels of all living particles at time $t\geq 0$, and the labels of all particles who induced a branching event before time $t\geq 0$, respectively, for this branching Brownian particle system.
	The following lemma allows us to control the branching-coalescing Brownian particle systems using this coupling. 
	It says that the set of labels of living particles in the system with coalescing is contained in the set of labels in the system without coalescing, and similarly for the set of labels of branching events. 

	We omit its proof, since it is elementary.

\begin{lemma}\label{lem:Coupling}
	$I_t \subset \bar I_t$ and $J_t \subset \bar J_t$ for every $t\geq 0$ almost surely.
\end{lemma}
\begin{extra}
	(We add a proof here for the sake of completeness.)
\begin{proof}
	Let us first claim that $\zeta_\alpha \leq \bar \zeta_\alpha$ for every $\alpha \in \mathcal U$ by induction.
	To show this claim, let us fix an arbitrary $\beta \in \mathcal U$, and assume that the above claim holds for every $\alpha \prec \beta$ already. 
	Then we argue in the following three cases:
\begin{itemize}
\item[(i)] 
	$|\beta|=1$ and $\beta \leq n$. In this case, $\bar \zeta_\beta = \zeta_{\beta,\beta} \geq \zeta_\beta$. 
\item[(ii)] 
	$|\beta| > 1$, $\zeta_{\overleftarrow{\beta}} = \zeta_{\overleftarrow{\beta},\overleftarrow{\beta}}$ and $\beta_{|\beta|} \leq Z_{\overleftarrow{\beta}}$. 
	In this case, from the fact that $ \zeta_{\overleftarrow{\beta},\overleftarrow{\beta}} = \zeta_{\overleftarrow{\beta}} \leq \bar{\zeta}_{\overleftarrow{\beta}} \leq \zeta_{\overleftarrow{\beta},\overleftarrow{\beta}}$, we know that $\bar{\zeta}_{\overleftarrow{\beta}} = \zeta_{\overleftarrow{\beta},\overleftarrow{\beta}}$. 
	Now we have $\bar \zeta_\beta = \zeta_{\beta,\beta} \geq \zeta_\beta$.
\item[(iii)] 
	If neither of the conditions in (i) nor (ii) holds, then $ \zeta_\beta = \xi_\beta \leq \bar \zeta_\beta$.
\end{itemize}
	Now the desired claim holds. 
	
	To verify $I_t \subset \bar I_t$, we observe that $\alpha \in I_t \iff t \in [\xi_\alpha, \zeta_\alpha) \implies t \in [\xi_\alpha, \bar \zeta_\alpha) \iff \alpha \in \bar I_t$.
	To verify $J_t \subset \bar J_t$, we observe that $\alpha \in J_t \iff \zeta_{\alpha,\alpha} = \zeta_\alpha \leq t \implies \zeta_{\alpha,\alpha} = \bar \zeta_\alpha \leq t \iff \alpha \in \bar J_t$. 
\end{proof}
\end{extra}

	We say the offspring distribution $(p_k)_{k\in \bar{\mathbb N}}$ is bounded, if there exists an $m  \in \mathbb N$ such that $p_k = 0$ for every $k>m$ including $k = \infty$. 
	Some elementary results for the branching-coalescing Brownian particle system with bounded offspring distribution  and finite many initial particles  are easy to obtain from the above lemma.
	For example, if the offspring distribution is bounded and the initial number of particles $n$ is finite, we know that almost surely for every $t\geq 0$, $| I_t| \leq |\bar I_t|< \infty$ and $|J_t|\leq |\bar J_t| < \infty$;
	in this case, the $\mathcal N$-valued c\`adl\`ag Markov process
	\begin{equation} 
		\bar {\mathbb X}_t 
		:= \sum_{\alpha \in \bar I_t} \delta_{\bar X^\alpha_t}, \quad t\geq 0
	\end{equation}
	is called the branching Brownian motion; and we can verify that the process of the counting measures
\begin{equation} \label{eq:MeasureValuedProcess}
	\mathbb X_t 
	:= \sum_{\alpha \in I_t} \delta_{X^\alpha_t}, \quad t\geq 0
\end{equation}
	is also an $\mathcal N$-valued c\`adl\`ag Markov process, where $\mathcal N$ is the space of finite $\mathbb Z_+$-valued measures on $\mathbb R$ equipped with the weak topology. 

	However, if there is no assumption made about $(p_k)_{k\in \bar{\mathbb N}}$ and $n$, then it is not a priori clear whether $|I_t|$ and $|J_t|$ are finite.
	To proof that this is indeed the case for almost every $t$, our strategy is to approximate them from below using a family of truncated branching-coalescing Brownian particle systems
\[
	\CB{(X^{(l,m),\alpha}_t)_{t\geq 0}: \alpha \in \mathcal U}, \quad l,m\in \mathbb N \cup \{\infty\}.
\]
	Here, $l$ is the truncation number for the initial particles, $m$ is the truncation number for the branching mechanism, and $\{(X^{(l,m),\alpha}_t)_{t\geq 0}: \alpha \in \mathcal U\}$ will be constructed as a branching-coalescing Brownian particle system with initial configuration $(x_i)_{i=1}^{n\wedge l}$ and some offspring distribution bounded by $m$.
	In fact, $\{(X^{(l,m),\alpha}_t)_{t\geq 0}: \alpha \in \mathcal U\}$ will be constructed in the same 
	probability space as the non-truncated particle system $\{(X^\alpha_t)_{t\geq 0}: \alpha \in \mathcal U\}$ 
	in a way that the truncated particle system is dominated by the original particle system. 
	That is, $X^{(l,m),\alpha}_t \in \mathbb R$ (i.e. $X^{(l,m), \alpha}_t \neq \dagger$) implies that $X^{\alpha}_t = X^{(l,m),\alpha}_t$, for every $\alpha \in \mathcal U$ and $t\geq 0$ almost surely.
	More precisely, this truncated particle system is defined through \eqref{eq:T}--\eqref{eq:TTT} below.
\begin{itemize}
\item[\eq\label{eq:T}]
	For each $\alpha \in \mathcal U$ and $m\in \mathbb N\cup \{\infty\}$, define $Z^{(m)}_{\alpha} :=Z_\alpha \wedge m$. 
	(Recall \eqref{eq:TheLastRule}.)
\item[\eq\label{eq:TT}]
	For each $l,m\in \mathbb N\cup\{\infty\}$, define a family of $\mathbb R_+$-valued random variables $\{\zeta^{( l, m)}_\beta: \beta \in \mathcal U\}$ inductively so that for each $\beta \in \mathcal U$, 
\begin{itemize}
\item[(i)] 
	if $|\beta|=1$ and $\beta \leq n\wedge l$, then
\[
	\zeta^{(l, m)}_\beta 
	:= \inf \RB{ \CB{\zeta_{\beta,\beta}} \cup \CB{\zeta_{\alpha,\beta}: \alpha \in \mathcal U, \alpha \prec \beta, \zeta_{\alpha,\beta} \leq \zeta^{(l,m)}_\alpha}};
\] 
\item[(ii)] 
	if $|\beta| > 1$, $\zeta^{(l,m)}_{\overleftarrow{\beta}} = \zeta_{\overleftarrow{\beta},\overleftarrow{\beta}}$ and $\beta_{|\beta|} \leq Z^{(m)}_{\overleftarrow{\beta}}$, then 
\[
	\zeta^{(l,m)}_\beta 
	:= \inf \RB{ \CB{\zeta_{\beta,\beta}} \cup \CB{\zeta_{\alpha,\beta}: \alpha \in \mathcal U, \alpha \prec \beta, \zeta_{\alpha,\beta} \leq \zeta^{(l,m)}_\alpha}};
\] 
\item[(iii)] 
	if neither of the conditions in (i) nor (ii) hold, then $\zeta^{(l,m)}_\beta := \xi_\beta$.
\end{itemize}
\item[\eq\label{eq:TTT}]
	For each $l,m\in \mathbb N\cup\{\infty\}$ and $\alpha \in \mathcal U$, define $\mathbb R\cup \{\dagger\}$-valued process 
\begin{equation}
	X^{(l,m),\alpha}_t:= 
\begin{cases}
	\dagger, &\quad t\in [0, \xi_\alpha),
	\\ \tilde X^\alpha_t, &\quad t\in [\xi_\alpha, \zeta^{(l,m)}_\alpha),
	\\ \dagger, & \quad t\in [ \zeta^{(l,m)}_\alpha, \infty).
\end{cases}
\end{equation}
\end{itemize}

It is not hard to verify that $\{(X^{(l,m),\alpha}_t)_{t\geq 0}: \alpha \in \mathcal U\}$ is a branching-coalescing Brownian particle system with branching rate $\mu$, initial configuration $(x_i)_{i=1}^{n\wedge l}$, and offspring distribution $(p^{(m)}_k)_{k\in \bar{\mathbb N}}$ such that for every $k\in \bar {\mathbb N}$,
\begin{equation} \label{eq:BOD}
p^{(m)}_k := 
\begin{cases}
	p_k, &\quad k < m,
	\\ \sum_{j \in \bar{\mathbb N}, j\geq m} p_j, &\quad k=m,
	\\ 0, &\quad k > m.
\end{cases}
\end{equation}
We call $\{(X^{(l,m),\alpha}_t)_{t\geq 0}: \alpha \in \mathcal U\}$ the $(l,m)$-truncated version of $\{(X^\alpha_t)_{t\geq 0}: \alpha \in \mathcal U\}$.
	Also note that $\{(X^\alpha_t)_{t\geq 0}: \alpha \in \mathcal U\}$ is the $(\infty,\infty)$-truncated version of itself.

	For each $l,m\in \mathbb N\cup\{\infty\}$ and $t\geq 0$, define 
		\begin{equation} \label{eq:Imt}
		I_t^{(l,m)}
		:=\{\alpha\in \mathcal U: X_t^{(l,m),\alpha} \in \mathbb R\} = \{\alpha\in \mathcal U: \xi_\alpha \leq t<\zeta^{(l,m)}_\alpha\} 
	\end{equation}
	as the collection of labels of alive particles at time $t$ in the $(l,m)$-truncated particle system;
	and 
	\begin{equation}\label{eq:Jmt}
		J_t^{(l,m)} 
		:=\{\alpha \in \mathcal U: \zeta_{\alpha,\alpha}=\zeta^{(l,m)}_\alpha\leq t\}
	\end{equation}
	as the collection of labels of particles who induced a branching event up to time $t$ in the $(l,m)$-truncated particle system. 
	Note that, if $l< \infty$ and $m< \infty$, then the explosion won't happen for the $(l,m)$-truncated version of the branching-coalescing Brownian particle system, since its initial number of particles is bounded by $l$ and its offspring distribution is bounded by $m$; in other words, almost surely for every $t\geq 0$, $| I_t^{(l,m)}|$ and $|J_t^{(l,m)}|$ are finite. 
	
	We often truncate the initial number of the particles and the offspring distribution using the same number, that is, $l = m$. 
	To simplify notations in this case, we write $\zeta^{(m)}_\alpha := \zeta^{(m,m)}_\alpha$, $X^{(m), \alpha}_t:= X^{(m,m),\alpha}_t$, $I^{(m)}_t := I^{(m,m)}_t$ and $J^{(m)}_t := J^{(m,m)}_t$ for every $m\in \mathbb N\cup\{\infty\}$, $\alpha \in \mathcal U$ and $t\geq 0$. 
	The particle systems $\{(X^{(m),\alpha}_t)_{t\geq 0}: \alpha \in \mathcal U\}$ will be referred to as the $m$-truncated version of the original branching-coalescing Brownian particle system.

	For each $m\in \mathbb N$, from how the $m$-truncated particle system is constructed, one can easily verify that it is dominated by the original particle system in the sense that $X_t^{(m),\alpha }\in \mathbb R$ implies that $X_t^{\alpha} = X_t^{(m),\alpha }$ for every $\alpha \in \mathcal U$ and $t\geq 0$ almost surely. 
	In particular, the set of labels $I_t^{(m)}$ is a subset of $I_t$ for every $t\geq 0$ almost surely. 
	It is also not hard to verify that $J_t^{(m)}$ is a subset of $J_t$ for every $t\geq 0$ almost surely. 
	This relationship is made more precise in the following lemma, which also serves as an alternative way of interpreting the truncation.
\begin{lemma} \label{lem:TC}
	Almost surely, for each $m\in \mathbb N$ and $t\geq 0$, we have
\begin{equation}
	I_t^{(m)} = \{\alpha \in \mathcal U: \|\alpha\|_{\infty}\leq m, \alpha \in I_t\}
\end{equation}
and
\begin{equation}
	J_t^{(m)} = \{\alpha \in \mathcal U: \|\alpha\|_{\infty} \leq m, \alpha \in J_t\}.  
\end{equation}
In particular, for any $t\geq 0$, $|I_t^{(m)}|$ and $|J_t^{(m)}|$ increasingly converges to $|I_t|$ and $|J_t|$, respectively, as $m\uparrow \infty$.
\end{lemma}

The proof of Lemma \ref{lem:TC} is technical, and is given in Appendix \ref{sec:TC}. 

\subsection{The point processes and their compensator}

In this subsection, let us recall some preliminary results on the stochastic integral of the point processes from \cite{MR1011252}. 
Suppose that $E$ is a Polish space, and $(\Omega, \mathscr F, (\mathscr F_t)_{t\geq 0}, \mathbb P)$ is a filtered probability space satisfying the usual condition. 
We say $\mathfrak G$ is a ($E$-valued) point process if it is a random measure on $(0,\infty) \times E$ and, almost surely, there exists a countable $S\subset (0,\infty)$ and a map $g: S\to E$ such that
\[
\mathfrak G((0,t]\times U) = \ABS{\CB{s\in S: s\leq t, g(s) \in U}}, \quad t>0, U \in \mathscr B(E). 
\]
We say a point process $\mathfrak G$ is adapted, if the process $\mathfrak G((0,\cdot ]\times U)$ is adapted for each $U \in \mathscr B(E)$. 
With
$
\Gamma_{\mathfrak G}:= \{U \in \mathscr B(E): \mathbb E[\mathfrak G((0,t]\times U)] < \infty \text{ for all } t>0\},
$
we say a point process $\mathfrak G$ is $\sigma$-finite, if there exists a sequence of $\{ E_n \in \Gamma_{\mathfrak G}: n \in \mathbb N \}$ such that $E_n \uparrow E$. 
We say a random measure $\mathfrak G$ is of the class QL, if it is an adapted, $\sigma$-finite point process, and there exists a non-negative random measure $\hat {\mathfrak G}$ on $(0,\infty)\times \mathbb E$ such that for every  $U \in \Gamma_{\mathfrak G}$, 
\begin{itemize}
	\item  $\hat {\mathfrak G}((0,\cdot] \times U )$ is a continuous adapted process; and
	\item $\mathfrak G((0,\cdot] \times U) - \hat {\mathfrak G}((0,\cdot] \times U )$ is a martingale. 
\end{itemize}
We call $\hat{\mathfrak G}$ the compensator of $\mathfrak G$.
Denote by $\mathscr L$ the space of predictable random fields on $\mathbb R_+ \times E$. 
For any non-negative random measure $\mathfrak G$ on $(0,\infty)\times E$ and $k \in \{1,2\}$, define 
\[
\mathscr L^{k}_{\mathfrak G} := \CB{ f \in \mathscr L: \mathbb E\SB{\iint_0^t |f(s,y)|^k \mathfrak G(\mathrm ds, \mathrm dy)} < \infty, \forall t\geq 0},
\]
and
\[
\mathscr L^{k,\mathrm{loc}}_{\mathfrak G} := \CB{ f \in \mathscr L: \iint_0^t |f(s,y)|^k \mathfrak G(\mathrm ds, \mathrm dy) < \infty, \forall t\geq 0, \text{a.s.}}.
\]
For each $k \in \{1,2\}$, denote by $\mathscr M^k$ the space of martingales $(m_t)_{t\geq 0}$ such that $\mathbb E[|m_t|^k] < \infty$ for every $t\geq 0$; and by $\mathscr M^{k, \mathrm{loc}}$ the space of processes $(m_t)_{t\geq 0}$ such that $(m_{t\wedge \sigma_n})_{t\geq 0}$ belongs to $\mathscr M^k$ for some sequence of stopping times $\sigma_n \uparrow \infty$. 

For a QL point process $\mathfrak G$ with compensator $\hat{\mathfrak G}$, its compensated stochastic integral, denoted by
\[
\mathcal I_{\tilde {\mathfrak G}}: f \mapsto \iint_0^{\cdot} f(s,y) \tilde {\mathfrak G}(\mathrm ds, \mathrm dy),
\]
is constructed, for example, in \cite{MR1011252}.
We collect some basic facts about this integration in following lemma.
\begin{lemma} \label{eq:CompensatorDecom}
	Let $\mathfrak G$ be a QL point process with compensator $\hat{\mathfrak G}$.
	\begin{enumerate}
		\item 	If $f \in 	\mathscr L^{2}_{\hat{\mathfrak G}}$ then $\mathcal I_{\tilde {\mathfrak G}} f \in \mathscr M^{2}$.
		\item 	If $f \in 	\mathscr L^{2,\mathrm{loc}}_{\hat{\mathfrak G}}$ then $\mathcal I_{\tilde {\mathfrak G}} f \in \mathscr M^{2, \mathrm{loc}}$.
		\item 	$\mathscr L^{1}_{\mathfrak G}=\mathscr L^{1}_{\hat{\mathfrak G}}$. Moreover, if $f \in 	\mathscr L^{1}_{\mathfrak G}$ then $\mathcal I_{\tilde {\mathfrak G}} f \in \mathscr M^{1}$.
		\item   If $f\in \mathscr L^{1,\mathrm{loc}}_{\hat{\mathfrak G}}$ then $f\in \mathscr L^{1,\mathrm{loc}}_{\mathfrak G}$, $\mathcal I_{\tilde {\mathfrak G}} f \in \mathscr M^{1, \mathrm{loc}}$, and
		\[
		\mathcal I_{\tilde {\mathfrak G}} f = \iint_0^{\cdot} f(s,y) \mathfrak G(\mathrm ds, \mathrm dy) - \iint_0^{\cdot} f(s,y) \hat{\mathfrak G}(\mathrm ds, \mathrm dy).
		\]
	\end{enumerate}
\end{lemma}
\begin{extra}
\begin{note}
	We give a proof for the above lemma for the sake of completeness.
	\begin{proof}
		(1) When $f \in	\mathscr L^{2}_{\hat{\mathfrak G}}$, the definition of $\mathcal I_{\tilde {\mathfrak G}} f$ is given as the limit of a Cauchy sequence in $\mathscr M^{2}$ \cite[p. 63]{MR1011252}. 
		
		(2) When $f \in	\mathscr L^{2, \mathrm{loc}}_{\hat{\mathfrak G}}$, $\mathcal I_{\tilde {\mathfrak G}} f$ is defined as an element in $\mathscr M^{2, \mathrm{loc}}$ \cite[p. 63]{MR1011252}. 
		
		(3) When $f \in \mathscr L^{1}_{\hat{\mathfrak G}}$, then it is shown in  \cite[p. 62]{MR1011252} that 
		\[
			\mathbb E\SB{\iint_0^t |f(s,y)| \mathfrak G(\mathrm ds,\mathrm dy) } = \mathbb E\SB{\iint_0^t |f(s,y)| \hat{\mathfrak G}(\mathrm ds,\mathrm dy) }< \infty
		\]
		which implies that $f \in \mathscr L^{1}_{\mathfrak G}$. On the contrary, for any $f \in \mathscr L^{1}_{\mathfrak G}$, noticing that
		\[
			f_n(s,y) := \mathbf 1_{E_n}(y) \mathbf 1_{[-n,n]}( f(s,y)) f(s,y) \in  \mathscr L^{1}_{\hat{\mathfrak G}}, \quad n\in \mathbb N
		\] 
		where $(E_n)_{n\in \mathbb N}$ is a sequence of element in $\Gamma_{\mathfrak G}$ such that $E_n \uparrow E$, we have, by the monotone convergence theorem, 
		\begin{align}
			 &\mathbb E\SB{\iint_0^t |f(s,y)| \hat{\mathfrak G}(\mathrm ds,\mathrm dy) }= \lim_{n\uparrow \infty}  \mathbb E\SB{\iint_0^t |f_n(s,y)| \hat{\mathfrak G}(\mathrm ds,\mathrm dy) }
			 \\&= \lim_{n\uparrow \infty}  \mathbb E\SB{\iint_0^t |f_n(s,y)| \mathfrak G(\mathrm ds,\mathrm dy) }= \mathbb E\SB{\iint_0^t |f(s,y)| \mathfrak G(\mathrm ds,\mathrm dy) } < \infty. 
		\end{align}
		Therefore we have $\mathscr L^{1}_{\mathfrak G}= \mathscr L^{1}_{\hat{\mathfrak G}}$. 
		 Moreover, if $f \in 	\mathscr L^{1}_{\mathfrak G}$ then $\mathcal I_{\tilde {\mathfrak G}} f$ is already defined as an element in $\mathscr M^{1}$ \cite[p. 62]{MR1011252}.
	
		(4) When $f\in \mathscr L^{1,\mathrm{loc}}_{\hat{\mathfrak G}}$, we know that
	\[
	\iint_0^t |f(s,y)| \hat{\mathfrak G}(\mathrm ds, \mathrm dy) < \infty, \quad \forall t\geq 0, \text{a.s.}
	\]
		It is also clear that the integral on the left hand side in the above display is continuous in $t$. 
		In particular, we can find a sequence of predictable stopping times $(\tau_n)_{n\in \mathbb N}$ such that, almost surely, $\tau_n\uparrow \infty$ when $n\uparrow \infty$ and that
		\[
	\iint_0^{\tau_n} |f(s,y)| \hat{\mathfrak G}(\mathrm ds, \mathrm dy) = n, \quad \text{a.s.}
	\]
	From the result in (3), it is clear that $g_n \in \mathscr L^{1}_{\hat{\mathfrak G}} = \mathscr L^{1}_{\mathfrak G}$ where $g_n(s,y) := \mathbf 1_{[0,\tau_n]}(s)f(s,y)$ for each $s\geq 0$ and $y\in \mathbb R$. 
	In particular, almost surely for every $n\in \mathbb N$, 
		\[
\iint_0^{\tau_n} |f(s,y)| \mathfrak G(\mathrm ds, \mathrm dy) < \infty, \quad \text{a.s.}
\]	
	This implies that $f\in \mathscr L^{1,\mathrm{loc}}_{\mathfrak G}$.
	
	From $g_n \in \mathscr L^{1}_{\hat{\mathfrak G}}$ and \cite[p. 62]{MR1011252}, we know that $\mathcal I_{\tilde {\mathfrak G}} g_n \in \mathscr M^1$ and is given by 
\begin{equation}
	\mathcal I_{\tilde {\mathfrak G}} g_n = \iint_0^\cdot g_n(s,y) \mathfrak G(\mathrm ds,\mathrm dy) - \iint_0^\cdot g_n(s,y) \hat{\mathfrak G}(\mathrm ds,\mathrm dy)
\end{equation}
	It is also clear that $f \in \mathscr L^{2,\mathrm{loc}}_{\hat{\mathfrak G}}$. 
	Therefore from \cite[p. 63]{MR1011252}, we know that  almost surely for any $t\geq 0$ and $n\in \mathbb N$,
\begin{equation} \label{eq:AM}
	\mathcal I_{\tilde {\mathfrak G}}  f(t\wedge \tau_n) =  	\mathcal I_{\tilde {\mathfrak G}} g_n(t) = \iint_0^t g_n(s,y) \mathfrak G(\mathrm ds,\mathrm dy) - \iint_0^t g_n(s,y) \hat{\mathfrak G}(\mathrm ds,\mathrm dy).
\end{equation}
	This implies that $\mathcal I_{\tilde {\mathfrak G}}  f \in \mathscr M^{1,\mathrm{loc}}$.

	Since $\tau_n\uparrow \infty$, so for every $t\geq 0$, there exists a random $N=N(t)$ such that $t\leq \tau_{N}$ almost surely. 
	Now for every $t\geq 0$, 
\begin{equation}
	\iint_0^{t} |f(s,y)| \hat{\mathfrak G}(\mathrm ds, \mathrm dy) \leq N(t)<\infty, \quad \text{a.s.}
\end{equation}
	Recall we have shown $f\in \mathscr L^{1,\mathrm{loc}}_{\mathfrak G}$, which says that for every $t\geq 0$, almost surely
\begin{equation}
	\iint_0^{t} |f(s,y)| \mathfrak G(\mathrm ds, \mathrm dy) \leq N(t)<\infty, \quad \text{a.s.}
\end{equation}
	Now, taking $n\uparrow \infty$ in \eqref{eq:AM}, by dominated convergence, we have for every $t\geq 0$ almost surely
\begin{equation}
		\mathcal I_{\tilde {\mathfrak G}}  f(t) =  \iint_0^{\cdot} f(s,y) \mathfrak G(\mathrm ds, \mathrm dy) - \iint_0^{\cdot} f(s,y) \hat{\mathfrak G}(\mathrm ds, \mathrm dy).
\end{equation}
	\end{proof}
\end{note}
\end{extra}

Let us now consider an arbitrary branching-coalescing Brownian particle system \[\{(X^\alpha_t)_{t\geq 0}: \alpha \in \mathcal U\},\] constructed in Section \ref{eq:DualParticle}, with arbitrary branching rate $\mu>0$, initial configuration $(x_i)_{i=1}^n$, and offspring distribution $(p_k)_{k\in \bar{\mathbb N}}$. Recall the definitions of ${\mathfrak N}, \hat {\mathfrak N}$ and $ {\mathfrak M}, \hat {\mathfrak M}$ in \eqref{eq:hatN} and \eqref{eq:ComM}, respectively.
Denote by $(\tilde{\mathscr F}_{t})_{t\geq 0}$ the smallest filtration of the probability space $(\tilde \Omega, \tilde {\mathscr F}, \tilde {\mathbb P})$, satisfying the usual hypothesis, such that the following processes are $(\tilde{\mathscr F}_{t})_{t\geq 0}$-adapted:
\begin{itemize}
	\item 
	$B_\cdot^\alpha$ for each $\alpha \in \mathcal U$; 
	\item 
	$\mathfrak{N}((0,\cdot]\times \{\alpha\} \times \{k\})$ for each $\alpha \in \mathcal U$, and $k \in \bar{\mathbb N}$; and
	\item
	$\mathfrak M\RB{(0,\cdot]\times \{(\alpha,\beta)\}}$ for each $(\alpha, \beta) \in \mathcal R$. 
\end{itemize}
With respect to this filtration $(\tilde{\mathscr F}_{t})_{t\geq 0}$, we can verify that 
\begin{itemize}
	\item 
	the processes $X^\alpha_\cdot$ and $\tilde X^\alpha_\cdot$ are adapted for each $\alpha\in \mathcal U$;
	\item 
	the process $L^{\alpha,\beta}_{\cdot, z}$ are adapted for each $\alpha, \beta \in \mathcal R$ and $z\in \mathbb R$; 
	\item 
	the random variables $\xi_\alpha$ and $\zeta_\alpha$ are stopping times for each $\alpha \in \mathcal U$; and
	\item 
	the random variable $Z_\alpha$ is $\tilde {\mathscr F}_{\zeta_{\alpha,\alpha}}$-measurable for each $\alpha \in \mathcal U$. 
\end{itemize}

The main message of this subsection is the following lemma, whose proof is straightforward, and therefore, omitted.
\begin{lemma} \label{lem:ThePointMeasures}
	The random measures $\mathfrak M$ and $\mathfrak N$ are QL point processes with compensators $\hat {\mathfrak M}$ and $\hat {\mathfrak N}$ respectively.
\end{lemma}  

Using the above result, we can verify the following lemma. 

\begin{lemma} \label{lem:Compensator}
	For every $t\geq 0$ it holds that 
	\[
	\tilde{\mathbb E}\SB{\ABS{J_t}} 
	= \mu \tilde{\mathbb E}\SB{\int_0^t |I_s| \mathrm ds}.
	\]
\end{lemma}
\begin{proof}
	Fix an arbitrary $t\geq 0$.
	Let us first assume that the number of the initial particles is finite, and the offspring distribution is bounded. 
	Notice that for each $\alpha \in \mathcal U$, $\alpha \in J_t$ if and only if there exists a (unique) $s\in (0,t]$ such that $X_{s-}^{\alpha} \in \mathbb R$ and $\mathfrak N(\{s\}\times \{\alpha\} \times \bar {\mathbb N}) = 1$. 
	Therefore, almost surely,
	\begin{align} \label{eq:DecOfBranchingNumbers}
		\ABS{J_{t}} =  \int_{\mathcal U} \int_0^t \mathbf 1_{\{X^{\alpha}_{s-} \in \mathbb R \}}  \mathfrak N (\mathrm ds, \mathrm d\alpha, \bar {\mathbb N}).
	\end{align}
	From Lemma \ref{lem:Coupling}, $|J_t|$ is dominated by $|\bar J_t|$, the number of branching events of the coupling branching Brownian particle system.
	Therefore $\tilde{\mathbb E}[|J_t|] \leq \tilde{\mathbb E}[|\bar J_t|] <\infty$ since we assumed that the number of initial particles is finite and the offspring distribution is bounded.
	Now from Lemmas \ref{eq:CompensatorDecom}, \ref{lem:ThePointMeasures}, and Fubini's theorem we have
	\[
	\tilde{\mathbb E}\SB{\ABS{J_t}}  = \mu \tilde{\mathbb E}\SB{ \sum_{\alpha \in \mathcal U}\int_0^t \mathbf 1_{\{X^{\alpha}_{s-} \in \mathbb R \}} \mathrm ds } = \mu \tilde{\mathbb E}\SB{\int_0^t |I_s| \mathrm ds}
	\]
	as desired.
	
	In the case the initial configuration is not finite or the offspring distribution is unbounded, we can first consider the $m$-truncated particle system where $m\in \mathbb N$. 
	From what we have proved, 
\[
	\tilde{\mathbb E}\SB{\ABS{J^{(m)}_t}} 
	= \mu \tilde{\mathbb E}\SB{\int_0^t |I^{(m)}_s| \mathrm ds}, \quad m\in \mathbb N.
\]
	Taking $m\uparrow \infty$, from Lemma \ref{lem:TC} and the monotone convergence theorem, we obtain the desired result. 
\end{proof}

\begin{extra}
	\begin{proof}
Note that $\mathfrak N$ is a Poisson random measure on the space $(0,\infty) \times \mathcal U \times \bar {\mathbb N}$  with intensity $\hat{\mathfrak N}$ given so that
\begin{equation}
	\hat {\mathfrak N} ((0,t] \times \{\alpha\} \times\{k\}) 
	= \mu t p_k, 
	\quad t>0, \alpha \in \mathcal U, k \in \bar{\mathbb N}. 
\end{equation}
Now it is clear from \cite[p. 60]{MR1011252} that $\mathfrak N$ is a QL point process.

For any $\beta \in \mathcal U$, denote by $\mathscr G_{\beta}$ the $\sigma$-field generated by $\{(B^\alpha_t)_{t\geq 0}: \alpha \in \mathcal U\}$, $\mathfrak N$, and the random elements in \eqref{eq:ConditionalConstruction}.
To show $\mathfrak M$ is a QL point process, we first note that for any $(\alpha, \beta ) \in \mathcal R$
\[
  \mathbb E\SB{\mathfrak M((0,t] \times \{(\alpha,\beta)\}) \middle|  \mathscr G_\beta }
  = \hat{\mathfrak M}((0,t] \times \{(\alpha,\beta)\}) = \frac{1}{2} L^{\alpha,\beta}_t,
\]
which says that 
\begin{align}
 &\mathbb E\SB{\mathfrak M((0,t] \times \{(\alpha,\beta)\}) }= \frac{1}{2} \mathbb E[L^{\alpha,\beta}_t]
 <\infty. 
\end{align}

We still need to verify that $\hat{\mathfrak M}$ is the compensator of $\mathfrak M$. 
Let $U\in \Gamma_{\mathfrak M}$ be arbitrary. We need to verify that 
\begin{itemize}
	\item  $\hat {\mathfrak M}((0,\cdot] \times U )$ is a continuous adapted process; and
	\item $\mathfrak M((0,\cdot] \times U) - \hat {\mathfrak M}((0,\cdot] \times U )$ is a martingale. 
\end{itemize}
Say $U = \{(\alpha_k,\beta_k) \in \mathcal R: k\in K\} \in \Gamma_{\mathfrak M}$ for some countable index set $K$. 
Then 
\begin{align}
	&\mathbb E\SB{\hat {\mathfrak M}((0,\cdot] \times U )}= \mathbb E\SB{\frac{1}{2}\sum_{k\in K} L^{\alpha_k,\beta_k}_t}
	 =\sum_{k\in K} \mathbb E[ \mathbb E\SB{\mathfrak M((0,t] \times \{(\alpha_k,\beta_k)\}) \middle|  \mathscr G_{\beta_k} }]
	 \\& = \mathbb E\SB{\mathfrak M((0,t] \times U)} < \infty. 
\end{align}
In particular we know that almost surely for all $t\geq 0$, 
\[
	\hat {\mathfrak M}((0,\cdot] \times U )=\frac{1}{2}\sum_{k\in K} L^{\alpha_k,\beta_k}_t < \infty.
\]
If $K$ is a finite index set, then $\hat {\mathfrak M}((0,\cdot] \times U )$ is obviously a continuous adapted process. 
Otherwise, we can assume without loss of generality that $K = \mathbb N$, and $\hat {\mathfrak M}((0,\cdot] \times U )$ is the increasing limit of the processes
\[
	\RB{\frac{1}{2}\sum_{k=1}^n L^{\alpha_k,\beta_k}_t}_{t\geq 0}
\]
as $n\uparrow \infty$. 
Observe that, for any fixed $T>0$, 
\begin{align}
	&\sup_{t\leq T}\ABS{\hat {\mathfrak M}((0,t] \times U )  - \frac{1}{2}\sum_{k=1}^n L^{\alpha_k,\beta_k}_t}
	= \sup_{t\leq T}\ABS{\frac{1}{2}\sum_{k=n+1}^\infty L^{\alpha_k,\beta_k}_t }
	\\& = \frac{1}{2}\sum_{k=n+1}^\infty L^{\alpha_k,\beta_k}_T \xrightarrow[n\to \infty]{} 0.
\end{align}
Therefore, from the uniform limit theorem, we know 
\begin{align}
	t\mapsto \hat {\mathfrak M}((0,t] \times U )
\end{align}
is indeed continuous. 

Let us further verify that $t\mapsto \mathfrak M((0,t] \times U) - \hat {\mathfrak M}((0,t] \times U )$ is a martingale. 
To do this, we write
\begin{align}
	&\mathbb E\SB{\mathfrak M((0,t] \times U) - \hat {\mathfrak M}((0,t] \times U )\middle| \mathscr F_s} 
	\\&=  \mathfrak M((0,s] \times U) - \hat {\mathfrak M}((0,s] \times U ) 
	\\& \quad + \mathbb E\SB{\mathfrak M((s,t] \times U) \middle| \mathscr F_s}  - \mathbb E\SB{ \hat {\mathfrak M}((s,t] \times U ) \middle| \mathscr F_s} 
	\\&=  \mathfrak M((0,s] \times U) - \hat {\mathfrak M}((0,s] \times U ) 
	\\& \quad + \sum_{k\in K}\mathbb E\SB{\mathfrak M((s,t] \times \{(\alpha_k,\beta_k)\}) \middle| \mathscr F_s}  - \mathbb E\SB{ \hat {\mathfrak M}((s,t] \times U ) \middle| \mathscr F_s} 
		\\&=  \mathfrak M((0,s] \times U) - \hat {\mathfrak M}((0,s] \times U ) 
	\\& \quad + \sum_{k\in K}\mathbb E\SB{\mathbb E\SB{\mathfrak M((s,t] \times \{(\alpha_k,\beta_k)\}) \middle| \mathscr F_s, \mathscr G_{\beta_k}}  \middle| \mathscr F_s}- \mathbb E\SB{ \hat {\mathfrak M}((s,t] \times U ) \middle| \mathscr F_s} 
			\\&=  \mathfrak M((0,s] \times U) - \hat {\mathfrak M}((0,s] \times U ) 
	\\& \quad + \sum_{k\in K}\mathbb E\SB{\hat {\mathfrak M}((s,t] \times \{(\alpha_k,\beta_k)\}) \middle| \mathscr F_s}- \mathbb E\SB{ \hat {\mathfrak M}((s,t] \times U ) \middle| \mathscr F_s} 
	\\& =\mathfrak M((0,s] \times U) - \hat {\mathfrak M}((0,s] \times U ).
\end{align}
\end{proof}
\end{extra}

\subsection{The embedded killing-coalescing Brownian motions} \label{sec:ECBM}
	In this subsection, we introduce a marking procedure for an arbitrarily given branching-coalescing Brownian particle system $\{(X^\alpha_t)_{t\geq 0}: \alpha \in \mathcal U\}$ with a bounded offspring distribution  and finite many initial particles. 
	This marking procedure marks out an embedded killing-coalescing Brownian particle system.
	(Recall that a killing-coalescing Brownian particle system is a branching coalescing Brownian particle system with $p_0 = 1$.)
	This embedded killing-coalescing Brownian particle system helps us to control the original particle system locally. 
	It will be the main ingredient for the proof of Theorem \ref{prop:Key}. 

For a given stopping time $\tau < \infty$ and a finite subset $\mathcal A$ of $\mathcal U$, by a $(\tau,\mathcal A)$-marking procedure, we mean the following:
\begin{itemize}
\item[\eq]
	At time $\tau$, if $\alpha$ is the $k$-th smallest label in the set $I_\tau \cap \mathcal A$ according to the order $\prec$, then we mark the particle $\alpha$ with number $k$; if $\alpha \in I_\tau \cap \mathcal A^{\mathrm c}$, then we mark that particle with number $\infty$.  
\item[\eq]
	After time $\tau$, each particle carries its mark unless a branching or a coalescing event happens.
\item[\eq]
	For each branching event after time $\tau$, the children will be marked by the number $\infty$ no matter of the mark of the parent. 
\item[\eq]
	For each coalescing event after time $\tau$, if the two particles inducing the coalescing event are marked by the numbers $a$ and $b$, then the survivor (i.e. the particle with the smaller Ulam-Harris label) will be marked by the number $\min\{a,b\}$.
\end{itemize}
Since we assumed that the offspring distribution is bounded  and the number of the initial particles is finite,  from Subsection \ref{sec:Truncated}, there are almost surely only finitely many branching/coalescing events up to any finite time; and thus, the above marking procedure is well-defined.  

For any number $k\in \mathbb N$ and time $t\geq 0$, there exists at most one particle alive at time $\tau + t$ that is marked by the number $k$. 
Denote by $\psi(\tau, \mathcal A, t,  k)$ the Ulam-Harris label of the particle carrying the mark $k$ at time $\tau+t$, provided such particle exists; and set $\psi(\tau, \mathcal A, t, k) = \varnothing$ if such particle does not exist. 
Also define a process $X^{\varnothing}_t = \dagger$ for every $t\geq 0$.
We will refer to the family of processes 
\begin{equation} \label{eq:embedded}
	\CB{\RB{X^{\psi(\tau, \mathcal A,t, k )}_{\tau+t}}_{t\geq 0}:k \in \mathbb N}
\end{equation}
the $(\tau, \mathcal A)$-embedded killing-coalescing Brownian particle system.
Using the strong Markov property of the Brownian motions, it is straightforward to verify the following lemma.

\begin{lemma} \label{lem:Embedded}
	Conditioned on $\tilde{\mathscr F_\tau}$, the $(\tau, \mathcal A)$-embedded killing-coalescing Brownian particle system \eqref{eq:embedded}  is a killing-coalescing Brownian particle system whose initial configuration is $(X^{\alpha^{(k)}}_{\tau})_{k=1}^N$. 
	Here $N := |I_\tau \cap \mathcal A|$, $\{\alpha^{(1)}, \alpha^{(2)}, \dots, \alpha^{(N)}\} =  I_\tau \cap \mathcal A$, and $\alpha^{(1)} \prec \alpha^{(2)} \prec \dots \prec \alpha^{(N)}$.
\end{lemma} 

Let us introduce some more notation related to this embedded killing-coalescing Brownian particle system that will be used later. 
For each $k\in \mathbb N$, define 
\begin{equation} 
	\zeta_k^{(\tau, \mathcal A)} 
	:= \tau +  \sup\CB{ t\geq 0: X^{\psi(\tau, \mathcal A,t, k )}_{\tau+t} \in \mathbb R}, 
	\quad k \in \mathbb N,
\end{equation}
to be the death-time of the mark $k$; and if the particle with mark $k$ induces a branching event at the time $\zeta_k^{(\tau, \mathcal A)}$, then we say $\Theta^{(\tau, \mathcal A)}(k) := 1$; otherwise we set $\Theta^{(\tau, \mathcal A)}(k) := 0$. 
For every $t\geq 0$, define
\begin{equation} \label{eq:IT}
I^{(\tau, \mathcal A)}_{\tau+t} := \{k\in \mathbb N: X^{\psi(\tau, \mathcal A,t, k )}_{\tau+t} \in \mathbb R\}
\end{equation}
to be the collection of the marks that are carried by some alive particles at time $\tau+t$;
and 
\begin{equation} \label{eq:JT}
J^{(\tau, \mathcal A)}_{\tau+t} := \{ k\in \mathbb N: \zeta_k^{(\tau, \mathcal A)} \leq \tau + t,  \Theta^{(\tau, \mathcal A)}(k) = 1\}
\end{equation}
to be the collection of the marks who deceased in a branching event up to time $\tau+t$. 

For example, if we assume that $n<\infty$ and $\{(X^\alpha_t)_{t\geq 0}: \alpha \in \mathcal U\}$ is a coalescing Brownian particle system with initial configuration $(x_i)_{i=1}^n$, then its $(0, \CB{0,\dots,n})$-embedded coalescing Brownian particle system is a killing-coalescing Brownian particle system who shares the same initial configuration $(x_i)_{i=1}^n$. 
This implies that the total population of a killing-coalescing Brownian particle system is stochastically dominated by that of a coalescing Brownian particle system. 
In our earlier paper \cite{barnes2022coming}, we established an upper bound for the the expectation of the total population of the coalescing Brownian particle system. 
Now it is clear that this upper bound also holds for the killing-coalescing Brownian particle system.
In particular, from \cite{barnes2022coming}*{Theorem 1.4 \& Proposition 1.5} we have the following result.

\begin{lemma}[\cite{barnes2022coming}*{Theorem 1.4 \& Proposition 1.5}] \label{lem:ComingDown}
		Consider a killing-coalescing Brownian particle system with initial configuration $(x_i)_{i=1}^n$.
		Suppose that $n < \infty$ and define $N_0 := |\{x_i:i=1,\dots,n\}|$.
		Denote by $\mu$ its branching rate and $|\hat I_t|$ the total population at time $t\geq 0$.
		Then there exists a time $\T\label{t:Coalescing} >0$ and a constant $\C\label{c:Coalescing} >0$ such that
		\[
		\tilde{\mathbb E}\SB{\ABS{\hat I_t}}
		  \leq \frac{\Cr{c:Coalescing} N_0}{\sqrt{t}} \wedge n, \quad t\in [0,\Tr{t:Coalescing}]. 
		\]
	Here $\Tr{t:Coalescing}$ and $\Cr{c:Coalescing}$ are independent of $\mu$ and $n$.
\end{lemma}
We will use both Lemmas \ref{lem:Embedded} and \ref{lem:ComingDown} in the proof of Theorem~\ref{prop:Key}.

\subsection{Some upper bounds provided the offspring distribution and the number of initial particles are bounded}
In this subsection, let us consider a branching-coalescing Brownian particle system $\{(X^\alpha_t)_{t\geq 0}: \alpha \in \mathcal U\}$ with a bounded offspring distribution and an initial configuration $(x_i)_{i=1}^n$ such that $n<\infty$.
	Define random variables \[N_t := |\{x \in \mathbb R: \exists \alpha \in I_t \text{~s.t.~} X_t^{\alpha} = x\}| < \infty, \quad t\geq 0.\]
In this subsection, we aim to give upper bounds for the expectation of the random variables
\begin{equation} \label{eq:BigThree}
	\ABS{I_t}, \quad \ABS{J_t}, \quad \text{and} \quad \int_0^t \ABS{I_s} \mathrm ds.
\end{equation}

\begin{lemma} \label{lem:Branching}
	There exists a (deterministic) time $\T\label{t:key}(\mu)\geq 0$ such that for every $t\in [0, \Tr{t:key}(\mu)]$ it holds that $\tilde{\mathbb E}[| J_t|] \leq  N_0$. 
	Here, $\Tr{t:key}(\mu)$ is independent of the initial configuration and the offspring distribution. 
\end{lemma}

\begin{proof}
	If a particle is labeled by an Ullam-Harris notation $\alpha$ with length $|\alpha| = j$, then we say it is in the $j$-th generation.
	For each $ j \in \mathbb N$ and $t\geq 0$, denote by
\[	
	J_t(j)
	:=\{\alpha \in \mathcal U: \zeta_{\alpha,\alpha}=\zeta_\alpha\leq t, |\alpha|=j\}
\]
	the collection of the labels of particles in the $j$-th generation who induced a branching event before time $t$.
	Then almost surely we have the decomposition
\[
	|J_t| 
	= \sum_{j=1}^\infty |J_t(j)|, \quad t\geq 0.
\]
	Let us take a deterministic time $\Tr{t:key}:=\Tr{t:key}(\mu)>0$ small enough so that $\Tr{t:key} \leq \Tr{t:Coalescing}$ and that
	\begin{equation} \label{eq:TheKeyTime}
		\mu \int_0^{\Tr{t:key}} \frac{\Cr{c:Coalescing}}{\sqrt{s}} \mathrm ds 
		= 2 \mu \Cr{c:Coalescing} \sqrt{\Tr{t:key}}
		\leq \frac{1}{2}.
	\end{equation}
	Here $\Cr{c:Coalescing}$ and $\Tr{t:Coalescing}$ are the constants introduced in Lemma \ref{lem:ComingDown}. 
	Notice that the choice of $\Tr{t:key}$ is independent of the initial configuration and the offspring distribution.
	We claim that 
	\begin{equation} \label{eq:FirstClaim}
		\tilde{\mathbb E} \SB{\ABS{J_{\Tr{t:key}}(j)}} \leq \frac{N_0}{2^j}, \quad j\in \mathbb N.
	\end{equation}
	From this claim we have 
	\[
	\tilde{\mathbb E}\SB{ \ABS{J_{\Tr{t:key}}}} 
	=  \sum_{j=1}^\infty \tilde{\mathbb E}\SB{ \ABS{J_{\Tr{t:key}}(j)}} \leq N_0,
	\]
	as desired for this lemma. 
	
	Let us prove the claim \eqref{eq:FirstClaim} for $j=1$ by using the $(\tau, \mathcal A)$-marking procedure, given as in Subsection \ref{sec:ECBM}, with $\tau = 0$ and $\mathcal A = I_0$.
	From Lemma \ref{lem:Embedded}, the $(0, I_0)$-embedded killing-coalescing Brownian particle system 
	\[
	\CB{ \RB{X^{\psi(0,I_0,t, k)}_{t}}_{t\geq 0} :k \in \mathbb N  }
	\] 
	is a killing-coalescing Brownian particle system with killing rate $\mu$ and initial configuration $(x_i)_{i=1}^{n}$. 
	Recall that the sets of labels $I_t^{(0,I_0)}$ and $J_{t}^{(0, I_0)}$ are given by \eqref{eq:IT} and \eqref{eq:JT} for every $t\geq 0$.
	Observe that almost surely 
	\[
	\ABS{J_{\Tr{t:key}}(1)} \leq  \ABS{ J_{\Tr{t:key}}^{(0, I_0)} }
	\] 
	since any branching event induced by an initial particle is also a branching event of the  $(0, I_0)$-embedded killing-coalescing Brownian particle system. 
	Now, from Lemmas \ref{lem:Compensator}--\ref{lem:ComingDown} and \eqref{eq:TheKeyTime}, we have
	\begin{equation}
		\tilde{\mathbb E}\SB{\ABS{J_{\Tr{t:key}}(1)}} 
		\leq  \tilde{\mathbb E}\SB{   \ABS{ J_{\Tr{t:key}}^{(0, I_0)} }}
		= \mu \tilde{\mathbb E}\SB{ \int_{0}^{\Tr{t:key}} \ABS{ I_{s}^{(0, I_0)} }  \mathrm ds}
		\leq \mu  \int_0^{\Tr{t:key}} \frac{N_0 \Cr{c:Coalescing}}{\sqrt{s}} \mathrm ds 
		\leq N_0/2
	\end{equation}
	as desired.
	
		We now prove the claim \eqref{eq:FirstClaim} by induction over $j\in \mathbb N$. 
	For the sake of induction, let us assume that $\tilde{\mathbb E}[|J_{\Tr{t:key}}(j)|] \leq N_0/2^j$ for some $j\in \mathbb N$. 
	For any Ullam-Harris label $\alpha \in \mathcal U$, let us denote by 
\[
	J_{\Tr{t:key}}^{\alpha} 
	:= \CB{\beta \in \mathcal U: \overleftarrow{\beta} = \alpha, \zeta_\beta = \zeta_{\beta,\beta} \leq T_2 }
\]
	the collection of the labels of the children of the particle $\alpha$ that induced a branching event before time $T_2$. 
	Then we have a decomposition
\[
	\ABS{J_{\Tr{t:key}}(j+1)} 
	= \sum_{\alpha \in \mathcal U:|\alpha| = j} \ABS{J_{\Tr{t:key}}^{\alpha}}. 
\]
	We  claim that 
\begin{equation}\label{eq:AnotherClaim}
	\tilde {\mathbb E} \SB{\ABS{ J_{\Tr{t:key}}^{\alpha} }\middle | \tilde{\mathscr F}_{\rho_\alpha}}
	\leq \mathbf 1_{\{\alpha \in J_{\Tr{t:key}}\}} /2
\end{equation}
	for each $\alpha \in \mathcal U$, where the stopping time $\rho_\alpha$ is given by
\[
	\rho_\alpha := 
\begin{cases}
	\zeta_\alpha, & \quad \text{~if~} \alpha \in J_{\Tr{t:key}},
	\\ \Tr{t:key},  &\quad \text{~otherwise.~}
\end{cases}
\]
	Admitting the claim \eqref{eq:AnotherClaim}, we have
\begin{align}
	&\tilde{\mathbb E}\SB{ \ABS{J_{\Tr{t:key}}(j+1)} } 
	= \tilde{\mathbb E}\SB{\sum_{\alpha \in \mathcal U:|\alpha| = j} \mathbb E\SB{ \ABS{J_{\Tr{t:key}}^{\alpha} } \middle | \tilde{\mathscr F}_{\rho_\alpha}}}
	\\& \leq \frac{1}{2} \tilde{\mathbb E} \SB{\sum_{\alpha \in \mathcal U:|\alpha| = j}\mathbf 1_{\{\alpha \in J_{\Tr{t:key}}\}} }
	= \frac{1}{2}\tilde{\mathbb E}\SB{\ABS{J_{\Tr{t:key}}(j)}}
	\leq N_0/2^{j+1}.
\end{align}
	Now the desired \eqref{eq:FirstClaim} follows by induction.
	
	We still needs to verify the claim \eqref{eq:AnotherClaim} for an arbitrarily fixed $\alpha \in \mathcal U$. 
	Since the offspring distribution are bounded, there exists an $m \in \mathbb N$ such that $p_k = 0$ for every $k\in \bar {\mathbb N}$ with $k>m$. 
	Let us consider the $(\rho_\alpha, \mathcal U_\alpha)$-marking procedure given as in Subsection \ref{sec:ECBM}  for the particle system where
	\begin{equation} \label{eq:Ualpha}
		\mathcal U_\alpha
		:= \{(\alpha,k): k \in \mathbb N\}
	\end{equation}
	is the collection of all the possible labels of the children of the particle $\alpha$.  
	From Lemma \ref{lem:Embedded} we know that, conditioned on $\tilde{\mathscr F}_{\rho_\alpha}$, the $(\rho_\alpha,\mathcal U_\alpha)$-embedded killing-coalescing Brownian particle system 
\[
	\CB{ \RB{X^{\psi(\rho_\alpha,\mathcal U_\alpha,t, k)}_{\rho_\alpha+t}}_{t\geq 0} :k \in \mathbb N  }
\] 
	is a killing-coalescing Brownian particle system with killing rate $\mu$ and initial configuration $(X^{(\alpha,k)}_{\rho_\alpha})_{k=1}^{N_\alpha}$. 
	Here, on the event $\{\alpha \in J_{\Tr{t:key}}\}$, $N_\alpha := Z_\alpha \leq m$ is the number of children of the particle $\alpha$; and on the event $\{\alpha \notin J_{\Tr{t:key}}\} = \{\rho_\alpha = T_2\}$, $N_\alpha := 0$, i.e., there is no initial particle for the $(\rho_\alpha,\mathcal U_\alpha)$-embedded killing-coalescing Brownian particle system.
	Recall that
\[
	\ABS{J_{\Tr{t:key}}^{(\rho_\alpha, \mathcal U_\alpha)}} 
	= \ABS{\CB{k \in \mathbb N:  \zeta^{(\rho_\alpha, \mathcal U_\alpha)}_k \leq \Tr{t:key}, \Theta^{(\rho_\alpha, \mathcal U_\alpha)}(k) = 1}}
\]
	is the number of the branching events of the $(\rho_\alpha,\mathcal U_\alpha)$-embedded killing-coalescing Brownian particle system up to time $\Tr{t:key}$.
	Observe that almost surely 
	\[
	\ABS{J_{\Tr{t:key}}^{\alpha}} \leq  \ABS{ J_{\Tr{t:key}}^{(\rho_\alpha, \mathcal U_\alpha)} }
	\] 
	since any branching event induced by a child of the particle $\alpha$ is also a branching event of the  $(\rho_\alpha,\mathcal U_\alpha)$-embedded killing-coalescing Brownian particle system. 
	Now, from Lemmas \ref{lem:Compensator} and \ref{lem:ComingDown}, we have
	\begin{align}
		&\tilde{\mathbb E}\SB{\ABS{J_{\Tr{t:key}}^{\alpha}} \middle | \tilde{\mathscr F}_{\rho_\alpha}} 
		\leq  \tilde{\mathbb E}\SB{  \ABS{J_{\Tr{t:key}}^{(\rho_\alpha, \mathcal U_\alpha)}} \middle | \tilde{\mathscr F}_{\rho_\alpha}}
		= \mu \tilde{\mathbb E}\SB{ \int_{0}^{\Tr{t:key}-\rho_\alpha} \ABS{ I^{(\rho_\alpha, \mathcal U_\alpha)}_{\rho_\alpha+s}} \mathrm ds \middle | \tilde{\mathscr F}_{\rho_\alpha}}
		\\ & \leq\mathbf 1_{\{\alpha \in J_{\Tr{t:key}}\}}  \mu  \int_0^{\Tr{t:key}} \frac{\Cr{c:Coalescing}}{\sqrt{s}} \mathrm ds 
		\leq \mathbf 1_{\{\alpha \in J_{\Tr{t:key}}\}}/2
	\end{align}
	as claimed.
\end{proof}

As a corollary of Lemmas \ref{lem:Compensator} and \ref{lem:Branching}, we have the following.
\begin{corollary} \label{cor:BranchingEvent}
	Let $\Tr{t:key}(\mu)\geq 0$ be given as in Lemma \ref{lem:Branching}, then 
	\begin{equation} \label{eq:Local1}
		\mu \tilde{\mathbb E}\SB{\int_0^t \ABS{I_s} \mathrm ds}
		= \tilde{\mathbb E}[ \ABS{J_t} ] 
		\leq N_0, \quad \forall t\in  [0, \Tr{t:key}(\mu)].
	\end{equation}
\end{corollary}

Similarly, let us give an upper bound for $\tilde{\mathbb E} \SB{\ABS{I_t}}$.

\begin{lemma} \label{prop:ThePopulation}
	Let $\Tr{t:key}(\mu)\geq 0$ be given as in Lemma \ref{lem:Branching}, then for any $t \in [0, \Tr{t:key}(\mu)]$ it holds that  
	\begin{equation} \label{eq:Local2}
		\tilde{\mathbb E} \SB{\ABS{I_t}} \leq  
			\frac{\Cr{c:Coalescing}N_0}{\sqrt{t}} \wedge n  + 2 \mu \Cr{c:Coalescing}^2 \pi N_0		 
	\end{equation}
		where $\Cr{c:Coalescing}$ is the constant given as in Lemma \ref{lem:ComingDown}.
\end{lemma}

\begin{proof}
	Let us fix an arbitrary $T \in (0, \Tr{t:key}(\mu)]$.
	Notice that for every $\beta \in \mathcal U$ with $|\beta| > 1$, $\beta \in I_T$ implies $\overleftarrow{\beta} \in J_T$. 
	Therefore
	\begin{equation} \label{eq:ITT}
	\ABS{I_T} \leq  |I^{\varnothing}_T| + \sum_{\alpha \in J_T}  	\ABS{I_{T}^{\alpha} }
	\end{equation}
	where 
	\[
		I_{T}^{\varnothing} := \CB{\beta \in \mathcal U: |\beta| =1, X_T^{\beta} \in \mathbb R}
	\]
	is the collection of the labels of the initial particles still alive at time $T$,
	and 
	for every $\alpha \in \mathcal U$, 
	\[
	I_{T}^{\alpha} := \CB{\beta \in \mathcal U: \overleftarrow{\beta} = \alpha, X_T^{\beta} \in \mathbb R}
	\]
	is the collection of the labels of the children of particle $\alpha$ who are alive at time $T$.
	
		Consider the $(\tau, \mathcal A)$-marking procedure, given as in Subsection \ref{sec:ECBM}, with $\tau = 0$ and $\mathcal A = I_0$.
		From Lemma \ref{lem:Embedded}, the $(0, I_0)$-embedded killing-coalescing Brownian particle system 
		\[
		\CB{ \RB{X^{\psi(0,I_0,t, k)}_{t}}_{t\geq 0} :k \in \mathbb N  }
		\] 
		is a killing-coalescing Brownian particle system with killing rate $\mu$ and initial configuration $(x_i)_{i=1}^{n}$. 
		Recall that the sets of labels $I_T^{(0,I_0)}$ are given by \eqref{eq:IT} for every $T\geq 0$.
		Note that $I_T^{\varnothing}$ is a subset of $I_T^{(0,I_0)}$, since any initial particles that are alive at time $T$ are also marked by a finite integer in the $(0, I_0)$-marking procedure. 
		Therefore, by Lemma \ref{lem:ComingDown}, we have
		\begin{equation}
			\tilde{\mathbb E}\SB{\ABS{I_T^{\varnothing}} } \leq \tilde{\mathbb E}\SB{\ABS{ I_T^{(0,I_0)}} } \leq \frac{\Cr{c:Coalescing}N_0}{\sqrt{T}} \wedge n.
		\end{equation}
	
	Since the offspring distribution is bounded, there exists an $m \in \mathbb N$ such that $p_k = 0$ for every $k>m$.
	We claim that for every $\alpha \in \mathcal U$,
	\begin{equation} \label{eq:OneClaim}
		\tilde{\mathbb E}\SB{ \ABS{I_{T}^{\alpha}}\middle | \tilde {\mathscr F}_{\rho_\alpha}}  \leq \frac{\Cr{c:Coalescing}}{\sqrt{T - \rho_\alpha}} \wedge m
	\end{equation}
	where the stopping time $\rho_\alpha$ is defined by
	\[
	\rho_\alpha := 
	\begin{cases}
		\zeta_\alpha, & \quad \text{~if~} \alpha \in J_{T},
		\\ T,  &\quad \text{~otherwise.~}
	\end{cases}
	\]
	From this claim and \eqref{eq:ITT}, we know that 
	\begin{align}
		&\tilde {\mathbb E}\SB{\ABS{I_T}}
		\leq  \frac{\Cr{c:Coalescing}N_0}{\sqrt{T}} \wedge  n +  \tilde {\mathbb E} \SB{  \sum_{\alpha \in J_T}  	 \tilde {\mathbb E} \SB{\ABS{I_{T}^{\alpha} } \middle| \tilde {\mathscr F}_{\rho_\alpha}} }
		\\&\leq  \frac{\Cr{c:Coalescing}N_0}{\sqrt{T}} \wedge  n +  \tilde {\mathbb E} \SB{  \sum_{\alpha \in J_T}  	\RB{\frac{\Cr{c:Coalescing}}{\sqrt{T-\zeta_\alpha}} \wedge m} }.
	\end{align}
	Notice that for each $\alpha \in \mathcal U$, $\alpha \in J_T$ if and only if there exists a (unique) $s\in (0,T]$ such that $X_{s-}^{\alpha} \in \mathbb R$ and $\mathfrak N(\CB{s}\times \CB{\alpha}\times \bar{\mathbb N}) = 1$; and in this case, it holds that $\zeta_\alpha = s$. 
	Therefore, almost surely
	\begin{equation}\label{eq:DecOfNumOfParticle}
		\sum_{\alpha \in J_T}  	\RB{ \frac{\Cr{c:Coalescing}}{\sqrt{T-\zeta_\alpha}} \wedge m }  
		=\int_{\mathcal U}\int_0^T	\RB{ \frac{\Cr{c:Coalescing}}{\sqrt{T-s}} \wedge m }\mathbf 1_{\{X^{ \alpha}_{s-} \in \mathbb R \}} \mathfrak N(\mathrm ds, \mathrm d\alpha, \bar {\mathbb N}).
	\end{equation}
	Notice that the left hand side of \eqref{eq:DecOfNumOfParticle} is dominated by $m |\bar J_T|$ (see Lemma \ref{lem:Coupling}), which, under the assumption of the bounded offspring distribution and finite many initial particles, has finite first moment.
	Therefore from Lemmas \ref{eq:CompensatorDecom} and \ref{lem:ThePointMeasures}, we have
	\[
	\tilde{\mathbb E}\SB{\sum_{\alpha \in J_T}  	\RB{ \frac{\Cr{c:Coalescing}}{\sqrt{T-\zeta_\alpha}} \wedge m } } = \mu \tilde{\mathbb E}\SB{ \sum_{\alpha \in \mathcal U} \int_0^T \RB{ \frac{\Cr{c:Coalescing}}{\sqrt{T-s}} \wedge m }\mathbf 1_{\{X^{ \alpha}_{s-} \in \mathbb R \}} \mathrm ds }.
	\] 
	From this, and Fubini's theorem, we know that 
	\begin{equation}
		\tilde {\mathbb E}\SB{\ABS{I_T}}
		\leq  \frac{\Cr{c:Coalescing}N_0}{\sqrt{T}} \wedge  n + \mu \int_0^T \frac{\Cr{c:Coalescing}}{\sqrt{T-s}}  \tilde {\mathbb E} \SB{ \ABS{I_s} } \mathrm ds.
	\end{equation}
	Since $T \in (0, \Tr{t:key}(\mu)]$ is arbitrary, we can iterate the above inequality and get from Fubini's theorem that
		\begin{align}
		& \tilde {\mathbb E}\SB{\ABS{I_T}}
		\leq\frac{\Cr{c:Coalescing}N_0}{\sqrt{T}} \wedge n + \mu \int_0^T \frac{\Cr{c:Coalescing}}{\sqrt{T-s}}  \SB{ \frac{\Cr{c:Coalescing}N_0}{\sqrt{s}} \wedge n+ \mu \int_0^s \frac{\Cr{c:Coalescing}}{\sqrt{s-r}}  \tilde {\mathbb E} \SB{ \ABS{I_r} } \mathrm dr} \mathrm ds
			\\& \leq \frac{\Cr{c:Coalescing}N_0}{\sqrt{T}} \wedge n + \mu N_0 C_1^2 \int_0^T \frac{1}{\sqrt{T-s}} \frac{1}{\sqrt{s}}  \mathrm ds + {}
			\\&\qquad  \mu^2 \Cr{c:Coalescing}^2 \int_0^T  \tilde {\mathbb E} \SB{ \ABS{I_r} }  \RB{\int_r^T \frac{1}{\sqrt{T-s}}  \frac{1}{\sqrt{s-r}}  \mathrm ds } \mathrm dr
			\\& = \frac{\Cr{c:Coalescing}N_0}{\sqrt{T}} \wedge n  + \mu \Cr{c:Coalescing}^2 \pi N_0 + \mu^2 \Cr{c:Coalescing}^2 \pi \int_0^T  \tilde {\mathbb E} \SB{ \ABS{I_r} }  \mathrm dr.
	\end{align}
	Here in the last step, we used the fact that, for every $r\in [0,T)$,
	\begin{equation}
		\int_r^T \frac{1}{\sqrt{T-s}}  \frac{1}{\sqrt{s-r}}  \mathrm ds 
		=  \int_0^{T-r} \frac{1}{\sqrt{T-r-l}}  \frac{1}{\sqrt{l}} \mathrm dl
		= \int_0^{1}  \frac{1}{\sqrt{z(1-z)}} \mathrm d z
		= \pi.
	\end{equation}
	Now, from Corollary \ref{cor:BranchingEvent} and \eqref{eq:TheKeyTime}, we have
	\begin{align}
		& \tilde {\mathbb E}\SB{\ABS{I_T}} \leq   \frac{\Cr{c:Coalescing}N_0}{\sqrt{T}} \wedge n  + 2 \mu \Cr{c:Coalescing}^2 \pi N_0 
	\end{align}
	as desired for this lemma.
	
	We still need to verify  the claim \eqref{eq:OneClaim} for an arbitrarily fixed $\alpha \in \mathcal U$. 
	Consider the $(\rho_\alpha, \mathcal U_\alpha)$-marking procedure for the particle system where $\mathcal U_\alpha $ is given as in \eqref{eq:Ualpha}.
	From Lemma \ref{lem:Embedded} we know that, conditioned on $\tilde{\mathscr F}_{\rho_\alpha}$, the $(\rho_\alpha,\mathcal U_\alpha)$-embedded killing-coalescing Brownian particle system 
\[
	\CB{ \RB{X^{\psi(\rho_\alpha,\mathcal U_\alpha,t, k)}_{\rho_\alpha+t}}_{t\geq 0} :k \in \mathbb N  }
\] 
	is a killing-coalescing Brownian particle system with killing rate $\mu$ and initial configuration $(X^{(\alpha,k)}_{\rho_\alpha})_{k=1}^{N_\alpha}$. 
	Here, on the event $\{\alpha \in J_{T}\}$, $N_\alpha = Z_\alpha \leq m$ is the number of children of the particle $\alpha$; and on the event $\{\alpha \notin J_{T}\} = \{\rho_\alpha = T\}$, we have $N_\alpha = 0$, i.e., there is no initial particle for the $(\rho_\alpha,\mathcal U_\alpha)$-embedded killing-coalescing Brownian particle system.
	Recall that
\[
	\ABS{I_{T}^{(\rho_\alpha, \mathcal U_\alpha)}} 
	= \ABS{\CB{k\in \mathbb N: X^{\psi(\rho_\alpha, \mathcal U_\alpha,T-\rho_\alpha, k )}_{T} \in \mathbb R}}
\]
	is the number of particles of the $(\rho_\alpha,\mathcal U_\alpha)$-embedded killing-coalescing Brownian particle system at time $T$.
	Observe that 
	\[
	\ABS{I_{T}^{\alpha}} \leq  \ABS{I_{T}^{(\rho_\alpha, \mathcal U_\alpha)}}, \quad \text{a.s.}
	\] 
	To see this, note that both sides of the above inequality equals $0$ on the event $\{\alpha \notin J_T\} = \{\rho_\alpha = T\}$; and on the event $\{\alpha \in J_T\}$, any child of the particle $\alpha$ is always marked by a finite number in the $(\rho_\alpha,\mathcal U_\alpha)$-marking procedure.
	Now from Lemma \ref{lem:ComingDown} we have that 
	\begin{equation}
		\tilde {\mathbb E} \SB{\ABS{I_{T}^{\alpha}} \middle | \tilde{\mathscr F}_{\rho_\alpha}}
		\leq \tilde {\mathbb E} \SB{ \ABS{I_{T}^{(\rho_\alpha, \mathcal U_\alpha)}} \middle | \tilde{\mathscr F}_{\rho_\alpha}}
		\leq \frac{\Cr{c:Coalescing}}{\sqrt{T- \rho_\alpha}} \wedge m 
	\end{equation}
	as claimed.
\end{proof}

Corollary \ref{cor:BranchingEvent} and Lemma \ref{prop:ThePopulation} give the upper bounds for the expectations of the random variables listed in \eqref{eq:BigThree} up to the time $\Tr{t:key}(\mu)$. 
Using the Markov property of the measure valued process  $(\mathbb X_t)_{t\geq 0}$ given in \eqref{eq:MeasureValuedProcess}, we can verify that for any $t \in (\Tr{t:key}, 2\Tr{t:key}]$, 
\begin{align}
		&\tilde {\mathbb E}\SB{\ABS{I_t}} 
		\leq \tilde{\mathbb E}\SB{\frac{\Cr{c:Coalescing}N_{T_2}}{\sqrt{t-T_2}} \wedge \ABS{I_{T_2}}  + 2 \mu \Cr{c:Coalescing}^2 \pi N_{T_2}}
	\\&\leq  \RB{1+2 \mu \Cr{c:Coalescing}^2\pi } \tilde {\mathbb E}  \SB{\ABS{I_{\Tr{t:key}}} } 
	\leq \RB{1+2 \mu \Cr{c:Coalescing}^2\pi } \RB{\frac{\Cr{c:Coalescing}}{\sqrt{T_2}}  + 2 \mu \Cr{c:Coalescing}^2 \pi} N_0
\end{align}
and 
	\begin{align}
		&\tilde {\mathbb E}\SB{\ABS{J_t}} 
		= \mu \tilde {\mathbb E}\SB{\int_0^t \ABS{I_s} \mathrm ds} 
		= \mu \tilde {\mathbb E}\SB{\int_0^{\Tr{t:key}} \ABS{I_s} \mathrm ds}  + \mu \tilde {\mathbb E}\SB{\int_{\Tr{t:key}}^t \ABS{I_s} \mathrm ds} 
		\\&	\leq N_0 +  \tilde {\mathbb E}\SB{N_{T_2}} \leq N_0 +  \tilde {\mathbb E}\SB{\ABS{I_{\Tr{t:key}}} } 
		\leq\RB{1+\frac{\Cr{c:Coalescing}}{\sqrt{T_2}}  + 2 \mu \Cr{c:Coalescing}^2 \pi} N_0.
	\end{align}
Repeating this procedure inductively for $t\in (k\Tr{t:key}, (k+1)\Tr{t:key}]$ with $k \in \mathbb N$, one can verify the following result.
\begin{corollary} \label{thm:TheUseful}
	For every $t\geq 0$, there exist $\C\label{c:Universal1}(\mu,t)>0$ and $\C\label{c:Universal2}(\mu,t)>0$ such that
\begin{equation}
	\tilde {\mathbb E}\SB{\ABS{I_s}} 
	 \leq \frac{\Cr{c:Coalescing}N_0}{\sqrt{s}} \wedge n + \Cr{c:Universal1}(\mu,t) N_0, \quad 0\leq s\leq t
\end{equation}
	and
	\[
	\tilde {\mathbb E}\SB{\ABS{J_t}} 
	= \mu \tilde {\mathbb E}\SB{\int_0^t \ABS{I_s} \mathrm ds} 
	\leq \Cr{c:Universal2}(\mu,t)  N_0. 
	\]
	Here, the constants $\Cr{c:Universal1}(\mu,t)$ and $\Cr{c:Universal2}(\mu,t)$ are independent of the initial configuration and the bounded offspring distribution.
\end{corollary}

\subsection{Uniform upper bound for arbitrary offspring distribution and initial configuration} \label{sec:proofkey}
In this subsection, let us consider a branching-coalescing Brownian particle system $\{(X^\alpha_t)_{t\geq 0}: \alpha \in \mathcal U\}$ with an arbitrary offspring distribution $(p_k)_{k\in \bar{\mathbb N}}$ and an arbitrary initial configuration $(x_i)_{i=1}^n$. 
 Suppose that $N_0:=\ABS{\{x_i:i\in \mathbb N, i\leq n\}} < \infty$.  
\begin{proof}[Proof of Theorem~\ref{prop:Key}]
	For every $m\in \mathbb N$, denote by $\{(X^{(m),\alpha}_t)_{t\geq 0}: \alpha \in \mathcal U\}$ the $m$-truncated version of the particle system $\{(X^\alpha_t)_{t\geq 0}: \alpha \in \mathcal U\}$ given as in Subsection \ref{sec:Truncated}.
	Denote $(I_t^{(m)})_{t \geq 0}$ and $(J_t^{(m)})_{t\geq 0}$ as in  \eqref{eq:Imt} and \eqref{eq:Jmt}. 
	It was known from Lemma \ref{lem:Coupling} that for every $t\geq 0$, $|I_t^{(m)}|$ and $|J_t^{(m)}|$ monotonically increase to $\ABS{I_t}$ and $\ABS{J_t}$, respectfully, as $m\uparrow \infty$. 
	
	Since $\{(X^{(m),\alpha}_t)_{t\geq 0}: \alpha \in \mathcal U\}$ is a branching-coalescing Brownian particle system with  its offspring distribution and initial number of particles both bounded by $m$, we can conclude from Corollary \ref{thm:TheUseful} that for every $t\geq 0$,
	\[
	\tilde {\mathbb E}\SB{\ABS{I^{(m)}_t}}  
	\leq 
	 \frac{\Cr{c:Coalescing}N_0}{\sqrt{t}} \wedge m + \Cr{c:Universal1}(\mu,t) N_0 
	\] 
	and
	\[
	\tilde {\mathbb E}\SB{\ABS{J^{(m)}_t}} 
	= \mu \tilde {\mathbb E}\SB{\int_0^t \ABS{I^{(m)}_s} \mathrm ds} 
	\leq \Cr{c:Universal2}(\mu,t)  N_0
	\]
	where the constants $\Cr{c:Universal1}(\mu,t)>0$ and $\Cr{c:Universal2}(\mu,t)>0$ are independent of the initial configuration $(x_i)_{i=1}^n$ and the truncation number $m$. 
	Taking $m\uparrow \infty$, the desired result now follows from the monotone convergence theorem. 
\end{proof}

\subsection{The exponential term} \label{sec:ET}
	In this subsection, let $(b_k)_{k\in \bar{\mathbb N}}$ be a family of real numbers satisfying \eqref{eq:NonDegenerate} and \eqref{eq:CLZsCondition} for some $R\geq 1$.
	Let $\{(X_t^\alpha)_{t\geq 0} : \alpha \in \mathcal U\}$ be a branching-coalescing Brownian particle system with initial configuration $(x_i)_{i=1}^n$  such that $n< \infty$, the branching rate $\mu$ is given as in \eqref{eq:TheBranchingRate}, and the offspring distribution $(p_k)_{k \in \bar {\mathbb N}}$ is given as in \eqref{eq:offspring}.
	Recall from \eqref{eq:Imt} that, for every $m\in \mathbb N$, $I^{(m)}_t$ is the labels of the particles in the $m$-truncated particle system living at time $t\geq 0$.
	Also recall that $(K_t)_{t\geq 0}$ is given as in \eqref{eq:Kt}.
	We will prove a result which is stronger than Proposition \ref{prop:ExponentialTerm}.
	This stronger result will be used later in the proof of Proposition \ref{prop:Duality}.

\begin{lemma}  \label{lem:exponential}
	For every $T\geq 0$, it holds that
\begin{equation} 
	\sup_{0\leq t\leq T}\tilde{\mathbb E}\SB{\RB{1+e^{K_t}}\RB{1+\ABS{I_t} + \ABS{I_t^{(m)}}^2}} 
	< \infty.
\end{equation}
\end{lemma}

	We postpone the proof of Lemma \ref{lem:exponential} to Appendix \ref{sec:exp}. 
	The proof uses a supermartingale argument which is in a similar spirit to the proof of Lemma 3 of \cite{MR1813840}. 

\begin{proof}[Proof of Proposition \ref{prop:ExponentialTerm}]
	The desired result is an immediate corollary of Lemma \ref{lem:exponential}.
\end{proof}

We also need another expectation bound for the truncated particle system. 
It will be used in the proof of Proposition \ref{prop:Duality}. 

\begin{lemma} \label{lem:bug}
	For every $m\in \mathbb N$ with $m\geq 2$ and $T\geq 0$, it holds that
\begin{equation} \label{eq:DR315}
	\tilde{\mathbb E}\SB{\RB{1+e^{K_T^{(m)}}} \sum_{\alpha,\beta \in I_{[0,T]}^{(m)}: \alpha \prec \beta} \RB{1 + \sup_{z\in \mathbb R} L^{\alpha,\beta}_{T,z}} } < \infty
\end{equation}
	where
	\begin{equation} 
		K^{(m)}_t
		:= \RB{\mu+b_1-\frac{1}{m}} \int_0^t |I^{(m)}_s| \mathrm ds, 
		\quad t\geq 0
	\end{equation}
	and
\begin{equation}
	I_{[0,T]}^{(m)} := \bigcup_{s\in [0,T]} I_s^{(m)}
\end{equation}
	is the collection of labels of the particles born up until time $T$ in the $m$-truncated branching-coalescing Brownian particle system.
\end{lemma}

	We postpone the proof of Lemma \ref{lem:bug} to Appendix \ref{sec:exp}. 
	In the proof, roughly speaking, we first analyze the moments of the exponential term and the local time terms separately, and then use H\"older's inequality.
	This is in a similar spirit to an argument in \cite[p.~1725]{MR1813840}.

\section{Proof of Proposition \ref{prop:Duality}}
\label{sec:Duality}
In this section, we assume that the assumptions in Proposition \ref{prop:Duality} hold.
More precisely, let $f\in \mathcal C(\mathbb R, [0,1])$, $n\in \mathbb N$ and $(x_i)_{i=1}^n$ be a finite list of real numbers.
Let the real-valued function $(b(z))_{z\in [0,1]}$ satisfy $b(0)\geq 0 \geq b(1)>-\infty$, $\eqref{eq:drift}$ and \eqref{eq:CLZsCondition} for some $R\geq 1$.
Suppose that the $\mathcal C(\mathbb R, [0,1])$-valued process $(u_t)_{t\geq 0}$, on a filtered probability space $(\Omega, \mathscr F, (\mathscr F_t)_{t\geq 0},\mathbb P_f)$, is a solution to the SPDE \eqref{eq:SPDE} with initial value $u_0 = f$.
Let $\{(X_{t}^\alpha)_{t\geq 0}: \alpha \in \mathcal U\}$ be a branching-coalescing Brownian particle system, on a probability space $(\tilde \Omega, \tilde {\mathscr F}, \tilde{\mathbb P})$, with initial configuration $(x_i)_{i=1}^n$, branching rate $\mu>0$ given by \eqref{eq:Branching}, and offspring distribution $(p_k)_{k\in  \bar{\mathbb N}}$ given by \eqref{eq:offspring}.

Denote by $\mathbf P$ the product probability measure $\mathbb P_f \times \tilde{\mathbb P}$ on the product space $\Omega \times \tilde \Omega$.
To establish the duality (Proposition \ref{prop:Duality}), we consider
\begin{equation} \label{eq:gts}
	\Xi_{t,s}^{\epsilon,m}:= \mathbf E \SB{ (-1)^{\ABS{\tilde J^{(m)}_s}} e^{K^{(m)}_s}\prod_{\alpha \in I^{(m)}_s} \RB{P_\epsilon u_t} \RB{X_s^{\alpha}}}, \quad t,s\geq 0, \epsilon \geq 0, m \in \mathbb N.
\end{equation}
Here, $(P_\epsilon)_{\epsilon\geq 0}$ is the one-dimensional heat semi-group, i.e. the transition semi-group of the one-dimensional Brownian motion; $I_s^{(m)}$ is the collections of the labels of the living particles at time $s\geq 0$ of the $m$-truncated branching-coalescing Brownian particle system given as in \eqref{eq:Imt};
\[
\tilde J_s^{(m)} := \{\alpha \in \mathcal U: \zeta^{(m)}_\alpha = \zeta_{\alpha,\alpha} \leq s, b_{Z_\alpha} < 0\}, \quad s\geq 0, m \in \mathbb N;
\]
and
\[
K_s^{(m)}:= (\mu + b_1 - \frac{1}{m})\int_0^s |I_r^{(m)}| \mathrm dr, \quad s\geq 0, m \in \mathbb N.
\]
Notice that, as $m\uparrow \infty$, the almost sure limit of the alternating  term $(-1)^{|\tilde J^{(m)}_s|}$ in \eqref{eq:gts} is $(-1)^{|\tilde J_s|}$, since by Theorem~\ref{prop:Key} and Lemma \ref{lem:TC}, $|\tilde J^{(m)}_s| \uparrow |\tilde J_s| \leq |J_s| < \infty$, a.s. 
Also, observe that
\begin{equation} \label{eq:Domination}
	\ABS{(-1)^{|\tilde J^{(m)}_s|} e^{K^{(m)}_s}\prod_{\alpha \in I^{(m)}_s}\RB{P_\epsilon u_t}\RB{X_s^{\alpha}}} 
	\leq 1+e^{K_s} \in L^1(\mathbf P), \quad s,t\geq 0, \epsilon \geq 0, m \in \mathbb N,
\end{equation}
by Proposition \ref{prop:ExponentialTerm}. 
Therefore, the right hand side of \eqref{eq:gts} is well-defined and finite.
Moreover, by Lemma \ref{lem:TC} and the dominated convergence theorem, we have
\begin{equation} \label{eq:mToInfinity}
	\lim_{m\to \infty} \Xi_{t,s}^{\epsilon,m}= \Xi_{t,s}^{\epsilon,\infty} := \mathbf E \SB{ (-1)^{|\tilde J_s|} e^{K_s}\prod_{\alpha \in I_s}\RB{P_\epsilon u_t}\RB{X_s^\alpha} }, \quad t,s, \epsilon \geq 0.
\end{equation}
Now, the desired duality formula \eqref{eq:Duality} can be written as $\Xi_{T,0}^{0,\infty}=\Xi_{0,T}^{0,\infty}$.

Our proof of Proposition \ref{prop:Duality} follows closely the strategy of the proof of Theorem 1 of \cite{MR1813840}. 
There are mainly three steps: 
\begin{itemize}
	\item Step 1. By applying  Ito's formula to
  $ (-1)^{\ABS{\tilde J^{(m)}_s}} e^{K^{(m)}_s}\prod_{\alpha \in I^{(m)}_s} \RB{P_\epsilon u_t} \RB{X_s^{\alpha}}$ as a function of $u_t$,
 we obtain a decomposition for  $\Xi_{t,s}^{\epsilon,m} - \Xi_{0,s}^{\epsilon,m}$.
	\item Step 2. 
 By applying  Ito's formula to
  $ (-1)^{\ABS{\tilde J^{(m)}_s}} e^{K^{(m)}_s}\prod_{\alpha \in I^{(m)}_s} \RB{P_\epsilon u_t} \RB{X_s^{\alpha}}$ as a function of  $\{(X_{s}^\alpha)_{s\geq 0}: \alpha \in \mathcal U\}$,
 we obtain a decomposition for 	$\Xi_{t,s}^{\epsilon,m} - \Xi_{t,0}^{\epsilon,m}$.
	\item Step 3. By inserting the two decompositions above in the equality 
\[
	\int_0^T \RB{\Xi_{r,0}^{\epsilon,m} - \Xi_{0,r}^{\epsilon,m}}\mathrm dr
	= \int_0^T \RB{ \Xi_{T-s,s}^{\epsilon,m} - \Xi_{0,s}^{\epsilon,m}} \mathrm ds - \int_0^T \RB{ \Xi_{t, T-t}^{\epsilon,m} - \Xi_{t,0}^{\epsilon,m}} \mathrm dt, 
\]
	and then by taking the iterated limit
as we first let $\epsilon \downarrow 0$ and then let $m\uparrow \infty$, we can verify that
\[
	\int_0^T \RB{\Xi_{r,0}^{0,\infty} - \Xi_{0,r}^{0,\infty}}\mathrm dr = 0
\]
	for every $T$. This, and a continuity result of the maps $r\mapsto \Xi_{0,r}^{0,\infty}$ and $r\mapsto \Xi_{r,0}^{0,\infty}$ 
will finish the proof of  Proposition \ref{prop:Duality}. 
\end{itemize}

	Let us mention a crucial difference between our approach and the one in \cite{MR1813840}. In \cite{MR1813840}, the particle system is stopped at a sequence of stopping times, instead of having truncated offspring at each of its branching events. 
	This is partially due to the fact that the offspring distribution considered in \cite{MR1813840} already 
	have all finite moments, without the need of further truncation. 

	The precise statements of the three main steps above are given in the following three lemmas.  

\begin{lemma}[Step 1] \label{lem:DecXiBySPDE}
	For any $t,s\geq 0$, $\epsilon > 0$ and $m\in \mathbb N$, it holds that
\[
	\Xi_{t,s}^{\epsilon,m} - \Xi_{0,s}^{\epsilon,m} = \Lambda_{t,s}^{\epsilon,m} + \Phi_{t,s}^{\epsilon,m}  + \Psi_{t,s}^{\epsilon,m} 
\]
	where
\begin{equation}
	\Lambda_{t,s}^{\epsilon,m}  
	:= \mathbf E \SB{ (-1)^{|\tilde J^{(m)}_s|} e^{K^{(m)}_s} \int_0^t \RB{\sum_{\alpha \in I_s^{(m)}}  \RB{\frac{\Delta}{2}P_\epsilon u_r}\RB{X^{\alpha}_s} \prod_{\beta\in I^{(m)}_s\setminus\{\alpha\}}  \RB{P_\epsilon u_r}\RB{X^{\beta}_s} }\mathrm dr},
\end{equation}
\begin{equation}\label{eq:Phi1}
	\Phi_{t,s}^{\epsilon,m}  
	:= \mathbf E \SB{ (-1)^{|\tilde J^{(m)}_s|} e^{K^{(m)}_s} \int_0^t \RB{\sum_{\alpha \in I_s^{(m)}}  \RB{ P_\epsilon \RB{b\circ u_r}}\RB{X^{\alpha}_s} \prod_{\beta\in I^{(m)}_s\setminus\{\alpha\}}  \RB{P_\epsilon u_r}\RB{X^{\beta}_s} }\mathrm dr},
\end{equation}
	and
\begin{align}
	&\Psi_{t,s}^{\epsilon,m}  
	:= \mathbf E\left[  (-1)^{|\tilde J^{(m)}_s|} e^{K^{(m)}_s} \times {} \vphantom{\int_0^t \mathrm dr \int \sigma\RB{u_r(y)} \RB{\sum_{i,j \in I_s^{(m)}: i\prec j} p_\epsilon(y-X^i_s)p_\epsilon(y-X^j_s) \prod_{k \in I_s^{(m)} \setminus \{i,j\}}\RB{P_\epsilon u_r}(X^k_s)} \mathrm dy} \right.
	\\ &\left. \vphantom{(-1)^{|\tilde J^{(m)}_s|} e^{K^{(m)}_s}} \int_0^t \mathrm dr \int \sigma\RB{u_r(y)}^2 \RB{\sum_{\alpha,\beta \in I_s^{(m)}: \alpha\prec \beta} p_\epsilon(y-X^{\alpha}_s)p_\epsilon(y-X^{\beta}_s) \prod_{\gamma \in I_s^{(m)} \setminus \{\alpha,\beta\}}\RB{P_\epsilon u_r}(X^{\gamma}_s)} \mathrm dy\right]
\end{align}
	are all well-defined. 
	Furthermore, for any $\epsilon > 0$ and $m\in \mathbb N$,
	\begin{equation}
		\sup_{s,t\in [0, T]} \max\CB{	\ABS{\Xi_{t,s}^{\epsilon,m}}, \ABS{\Lambda_{t,s}^{\epsilon,m}}, \ABS{\Phi_{t,s}^{\epsilon,m}}, \ABS{\Psi_{t,s}^{\epsilon,m}}} < \infty, \quad T\geq 0.
	\end{equation}
\end{lemma}

\begin{lemma}[Step 2]\label{lem:DecXiByCBBM}
	For any $t,s \geq 0$, $\epsilon > 0$ and $m\in \mathbb N$, it holds that
\[
	\Xi_{t,s}^{\epsilon,m} - \Xi_{t,0}^{\epsilon,m} 
	= \tilde \Lambda_{t,s}^{\epsilon,m} + \tilde \Phi_{t,s}^{\epsilon,m}  + \tilde \Psi_{t,s}^{\epsilon,m} 
\]
	where
\begin{equation}
	\tilde \Lambda_{t,s}^{\epsilon,m} 
	:= \mathbf E \SB{ \int_0^s (-1)^{|\tilde J^{(m)}_r|} e^{K^{(m)}_r}  \RB{\sum_{\alpha \in I_r^{(m)}}  \RB{\frac{\Delta}{2}P_\epsilon u_t}\RB{X^{\alpha}_r} \prod_{\beta \in I^{(m)}_r\setminus\{\alpha\}}  \RB{P_\epsilon u_t}\RB{X^{\beta}_r} }\mathrm dr},
\end{equation}
\begin{equation}\label{eq:Phi2} 
	\tilde \Phi_{t,s}^{\epsilon,m} 
	:= \mathbf E \SB{ \int_0^s (-1)^{|\tilde J^{(m)}_r|} e^{K^{(m)}_r} \RB{\sum_{\alpha \in I_r^{(m)}}  \RB{  b^{(m)}\circ \RB{P_\epsilon u_t}}\RB{X^{\alpha}_r} \prod_{\beta\in I^{(m)}_r\setminus\{\alpha\}}  \RB{P_\epsilon u_t}\RB{X^{\beta}_r} }\mathrm dr},
\end{equation}
	and
\begin{align} \label{eq:RL}
	&\tilde \Psi_{t,s}^{\epsilon,m}  
	:= \frac{1}{2} \mathbf E  \Bigg[ \int_0^s (-1)^{|\tilde J^{(m)}_r|} e^{K^{(m)}_r}  \times {}
	\\&\qquad \qquad \sum_{\alpha,\beta\in I^{(m)}_r: \alpha\prec \beta}  \RB{\sigma \circ \RB{P_\epsilon u_t}}(X_{r}^{\alpha})^2 \RB{\prod_{\gamma \in I^{(m)}_r\setminus \{\alpha,\beta\} } \RB{P_\epsilon u_t}\RB{X_r^{\gamma}} } \mathrm dL_r^{\alpha,\beta}\Bigg]
\end{align}
	are all well-defined. 
	Here,
	\begin{equation} \label{eq:DefBmz}
		b^{(m)}(z) :=\sum_{k\in \bar{\mathbb N}} b_k z^{k\wedge m} - \frac{1}{m}z, \quad z\in [0,1], m \in \mathbb N.
	\end{equation}
	Furthermore, for any $\epsilon>0$ and $m\in \mathbb N$, 
	\begin{equation}
		\sup_{s,t\in [0, T]} \max\CB{	\ABS{\tilde \Lambda_{t,s}^{\epsilon,m}}, \ABS{\tilde \Phi_{t,s}^{\epsilon,m}}, \ABS{\tilde \Psi_{t,s}^{\epsilon,m}}} < \infty, \quad T\geq 0.
	\end{equation}
\end{lemma}

The above two lemmas allows us to write down the following decomposition:
For any $T\geq 0$, $\epsilon > 0$ and $m\in \mathbb N$,
\begin{align}\label{eq:DifferenceXi}
	&\int_0^T \RB{\Xi_{r,0}^{\epsilon,m} - \Xi_{0,r}^{\epsilon,m}}\mathrm dr
	= \int_0^T \RB{ \Xi_{T-s,s}^{\epsilon,m} - \Xi_{0,s}^{\epsilon,m}} \mathrm ds - \int_0^T \RB{ \Xi_{t, T-t}^{\epsilon,m} - \Xi_{t,0}^{\epsilon,m}} \mathrm dt
	\\&= \int_0^T  \Lambda_{T-s,s}^{\epsilon,m} \mathrm ds - \int_0^T  \tilde \Lambda_{t,T-t}^{\epsilon,m} \mathrm dt + \int_0^T \Phi_{T-s,s}^{\epsilon,m} \mathrm ds -  \int_0^T \tilde \Phi_{t,T-t}^{\epsilon,m} \mathrm dt + {}
	\\ &\quad \int_0^T \Psi_{T-s,s}^{\epsilon,m} \mathrm ds -  \int_0^T \tilde \Psi_{t,T-t}^{\epsilon,m} \mathrm dt.
\end{align}

\begin{lemma}[Step 3] \label{lem:Cancellation} $~$
	\begin{enumerate}
		\item For every $s,t\geq 0$ and $m \in \mathbb N$, it holds that 
		\[
		\lim_{\epsilon \downarrow 0} \Xi_{t,s}^{\epsilon, m}=  \Xi_{t,s}^{0, m}.
		\]
		\item Both $r\mapsto \Xi_{r,0}^{0,\infty}$ and $r\mapsto \Xi_{0,r}^{0,\infty}$ are continuous functions on $[0,\infty)$.
		\item For every $T\geq 0$, $\epsilon > 0$ and $m\in \mathbb N$, it holds that
		\[\int_0^T \Lambda_{T-s,s}^{\epsilon,m} \mathrm ds =\int_0^T \tilde \Lambda_{t,T-t}^{\epsilon,m} \mathrm dt.\]
		\item For every $T\geq 0$ and $m\in \mathbb N$, it holds that
		\[
		\lim_{\epsilon \downarrow 0}  \RB{ \int_0^T \Psi_{T-s,s}^{\epsilon,m} \mathrm ds - \int_0^T \tilde \Psi_{t,T-t}^{\epsilon,m} \mathrm dt } = 0.
		\]
		\item For every $T\geq 0$, it holds that
		\[
		\lim_{m\uparrow \infty} \lim_{\epsilon \downarrow 0} \RB{ \int_0^T \Phi_{T-s,s}^{\epsilon,m} \mathrm ds - \int_0^T \tilde \Phi_{t,T-t}^{\epsilon,m} \mathrm dt } = 0.
		\]
	\end{enumerate}
\end{lemma}
Let us first explain how Proposition \ref{prop:Duality} follows from  Lemmas \ref{lem:DecXiBySPDE}, \ref{lem:DecXiByCBBM}, and \ref{lem:Cancellation}.
\begin{proof}[Proof of Proposition \ref{prop:Duality}]
	Let $T\geq 0$ be arbitrary.
	By \eqref{eq:Domination} and Proposition \ref{prop:ExponentialTerm}, we know that
	\[
	\sup_{0\leq s\leq T, t\geq 0, \epsilon\geq 0, m \in \mathbb N} |\Xi_{t,s}^{\epsilon, m}| \leq \tilde{\mathbb E}[1+e^{K_T}] < \infty. 
	\]
	Therefore, by Lemma \ref{lem:Cancellation} (1), \eqref{eq:mToInfinity}, and the bounded convergence theorem, we have
	\[
	\lim_{m\to \infty} \lim_{\epsilon\downarrow 0} \int_0^T \RB{\Xi_{r,0}^{\epsilon,m} - \Xi_{0,r}^{\epsilon,m}}\mathrm dr = \int_0^T \RB{\Xi_{r,0}^{0,\infty} - \Xi_{0,r}^{0,\infty}}\mathrm dr. 
	\]
	By taking $\epsilon\downarrow 0$ and then $m \uparrow \infty$ in \eqref{eq:DifferenceXi}, we get from Lemma \ref{lem:Cancellation} (3--5) that
	\begin{align}
		\int_0^T \RB{\Xi_{r,0}^{0,\infty} - \Xi_{0,r}^{0,\infty}}\mathrm dr = 0.
	\end{align}
	Finally, since $T>0$ is arbitrary, from Lemma \ref{lem:Cancellation} (2), we get $\Xi_{T,0}^{0,\infty} =  \Xi_{0,T}^{0,\infty}$ as desired.
\end{proof}

The rest of this section is devoted to  the proofs of Lemmas \ref{lem:DecXiBySPDE}, \ref{lem:DecXiByCBBM}, and \ref{lem:Cancellation}.

\begin{proof}[Proof of Lemma \ref{lem:DecXiBySPDE}]
	Let us fix an arbitrary $\epsilon > 0$.

	\emph{Step 1.}
	Recall that for any $g \in \mathcal C(\mathbb R, [0,1])$ and $x\in \mathbb R$,
\[
	(P_\epsilon g) (x) = \int p_\epsilon (x-y) g(y) \mathrm d y
\]
	where $p_\epsilon(x) = e^{-x^2/(2\epsilon)}/\sqrt{2\pi \epsilon}, (\epsilon,x)\in (0,\infty)\times \mathbb R$ is the heat kernel. 
	It is standard to argue, see \cite[p. 431]{MR1271224} for example, that for any $t\geq 0$ and $x\in \mathbb R$, the following holds almost surely
\begin{align}
	&\RB{P_\epsilon u_t}(x)- (P_\epsilon u_0)(x)
	\\&= \int_0^t \frac{\Delta}{2} (P_\epsilon u_r) (x) \mathrm dr 	+ \int_0^t \big(P_\epsilon (b\circ u_r)\big)(x)\mathrm dr  + \iint_0^t \sigma \RB{u_r(y)} p_\epsilon (x-y) W(\mathrm dr\mathrm dy).
\end{align}
\begin{extra}
	(We will verify this in Note \ref{note:411}.)
\end{extra} 
	Applying It\^o's formula, for any real-valued finite list $(x_i)_{i\in I}$ with $I = \{1,\cdots, n\}$ and $t\geq 0$, we have
\begin{align} \label{eq:TheItoDec}
	&\prod_{i=1}^n \RB{P_\epsilon u_t}(x_i) - \prod_{i = 1}^n \RB{P_\epsilon u_0}(x_i)
	\\&= \int_0^t \RB{\sum_{i\in I} \RB{\frac{\Delta}{2}P_\epsilon u_r}(x_i) \prod_{j\in I \setminus \{i\}} (P_\epsilon u_r)(x_j) } \mathrm dr + {} 
	\\& \quad \int_0^t \RB{\sum_{i\in I} \RB{P_\epsilon \RB{b\circ u_r}}(x_i) \prod_{j\in I \setminus \{i\}} (P_\epsilon u_r)(x_j) } \mathrm dr + {}
	\\& \quad \iint_0^t  \RB{\sum_{i\in I} \sigma(u_r(y)) p_\epsilon(x_i-y) \prod_{j\in I \setminus \{i\}} (P_\epsilon u_r)(x_j)} W(\mathrm dr \mathrm dy) + {}
	\\& \quad \iint_0^t \RB{\sum_{(i,j) \in I^2: i < j} \sigma(u_r(y))^2 p_\epsilon(x_i - y)p_\epsilon(x_j - y)\prod_{k\in I \setminus \{i,j\}} (P_\epsilon u_r)(x_k)} \mathrm dr \mathrm dy.
\end{align}
	
	\begin{extra}
		\begin{note} \label{note:411}
			Define norms
			\[
			\|f\|_{(\lambda)} = \sup_{x\in \mathbb R} |e^{\lambda|x|}f(x)|, \quad f\in \mathcal C(\mathbb R), \lambda \in \mathbb R, 
			\]
			and spaces
			\[
			\mathcal C_{\mathrm{rap}} := \{ f\in \mathcal C(\mathbb R): \|f\|_{(\lambda)} < \infty \text{ for all } \lambda > 0\}
			\]
			and
			\[
			\mathcal C_{\mathrm{rap}}^2 := \{ f\in\mathcal C_{\mathrm{rap}}: f' , f'' \in \mathcal C_{\mathrm{rap}}\}.
			\]
			For an arbitrarily fixed $x\in \mathbb R$, let we define
			\[
			\phi(y):=p_\epsilon(x-y) = e^{-(x-y)^2/(2\epsilon)}/\sqrt{2\pi \epsilon}, \quad y\in \mathbb R.
			\]
			It is clear that $(\phi(y))_{y\in \mathbb R} \in \mathcal C_{\mathrm{rap}}^2$. 
			From \cite[p. 430]{MR1271224}, we know that for any $t\geq 0$ almost surely, 
			\begin{equation} \label{eq:TempEq}
				\langle u_t, \phi\rangle = \langle f, \phi\rangle + \int_0^t \langle u_s, \frac{\Delta}{2}\phi \rangle \mathrm ds + \int_0^t \langle b(u_s), \phi \rangle\mathrm ds + \iint_0^t \sqrt{\sigma(u_s(y))} \phi(y) W(\mathrm ds \mathrm dy).
			\end{equation}
			Notice that, for $y \in \mathbb R$, we have
			\[
			\partial_x p_\epsilon (x-y) = p_\epsilon (x-y) \partial_x (-\frac{(x-y)^2}{2\epsilon} )= \frac{-(x-y)}{\epsilon}  p_\epsilon (x-y)
			\]
			and
			\[
			\partial_x^2 p_\epsilon (x-y) =  p_\epsilon (x-y)	\partial_x (\frac{-(x-y)}{\epsilon}) +  \frac{-(x-y)}{\epsilon} \partial_x p_\epsilon (x-y) = \frac{(x-y)^2-\epsilon}{\epsilon^2}p_\epsilon (x-y).
			\]
			Therefore, for any finite open interval $U\subset \mathbb R$, we can verify that, for every $s\geq 0$,
			\[
			\int \sup_{x\in U}|\partial_x \RB{u_s(y) p_\epsilon (x-y)}| \mathrm dy < \infty; \quad \int \sup_{x\in U}|\partial^2_x \RB{u_s(y) p_\epsilon (x-y)}| \mathrm dy < \infty.
			\]
			Therefore, by the dominated convergence theorem and the mean value theorem, we can verify that for every $s\geq 0$,
			\begin{equation} \label{eq:ChangeDerivative1}
				\int u_s(y) \partial_x p_\epsilon(x-y) \mathrm dy =  \partial_x\int u_s(y) p_\epsilon(x-y) \mathrm dy 
			\end{equation}
			and
			\[
			\int u_s(y) \partial^2_x p_\epsilon(x-y) \mathrm dy =  \partial_x\int u_s(y) \partial_x p_\epsilon(x-y) \mathrm dy
			= \partial_x^2 \int u_{s}(y) p_\epsilon(x-y) \mathrm dy.
			\]
			This implies that for $s\geq 0$, 
			\begin{equation} \label{eq:ChangeDerivative}
				\int u_s(y) \frac{\Delta_y}{2} p_\epsilon(x-y) \mathrm dy=\frac{\Delta}{2} P_\epsilon u_s (x).
			\end{equation}
			Put this back in \eqref{eq:TempEq} we arrived at what we want.
		\end{note}
	\end{extra}
	
	\emph{Step 2.}
It is easy to see that the stochastic integral with respect to white noise  at the  right hand side of~\eqref{eq:TheItoDec} is a maringale. Then, by taking  expectation on both sides of  \eqref{eq:TheItoDec} with respect to the measure $\mathbb P$, we can verify that for each $(x_i)_{i\in I}$ and $t\geq 0$, 
	\begin{align}\label{eq:TakingExpectation}
		\\&\mathbb E\SB{\prod_{i\in I} \RB{P_\epsilon u_t}(x_i)} - \mathbb E\SB{ \prod_{i\in I} \RB{P_\epsilon u_0}(x_i) }
		\\&= \mathbb E\SB{ \int_0^t \RB{\sum_{i\in I} \RB{\frac{\Delta}{2}P_\epsilon u_r}(x_i) \prod_{j\in I \setminus \{i\}} (P_\epsilon u_r)(x_j) } \mathrm dr} + {} 
		\\& \qquad \mathbb E\SB{\int_0^t \RB{\sum_{i\in I} \RB{P_\epsilon \RB{b\circ u_r}}(x_i) \prod_{j\in I \setminus \{i\}} (P_\epsilon u_r)(x_j) } \mathrm dr} + {}
		\\& \qquad \mathbb E\SB{ \iint_0^t \RB{\sum_{(i,j) \in I^2: i < j} \sigma(u_r(y))^2 p_\epsilon(x_i - y)p_\epsilon(x_j - y)\prod_{k\in I \setminus \{i,j\}} (P_\epsilon u_r)(x_k)} \mathrm dr \mathrm dy}.
	\end{align}
	\begin{extra}
		(We will verify this in Note \ref{note:412}.)
	\end{extra}
	Here, it is straightforward to verify that the two expectations on the left hand side of \eqref{eq:TakingExpectation} are bounded by $1$; the first, second, and the third, expectations on the right hand side of  \eqref{eq:TakingExpectation} are bounded by $\epsilon^{-1} |I| t$, $\|b\|_\infty |I| t$, and $\|\sigma\|^2_\infty \|p_\epsilon\|_\infty |I|^2t$, respectively.
	\begin{extra}
		(We will verify these in Note \ref{note:412}.)
	\end{extra}
	
\begin{extra}
\begin{note}\label{note:412}
	Fix $(x_i)_{i\in I}$ and $t\geq 0$. 
	To justify \eqref{eq:TakingExpectation}, we first need to verify that
\begin{equation} \label{eq:TheMartingalePart}
	\mathbb E\SB{ \iint_0^t  \RB{\sum_{i\in I}\sigma(u_r(y)) p_\epsilon(x_i-y) \prod_{j\in I \setminus \{i\}} (P_\epsilon u_r)(x_j)} W(\mathrm dr \mathrm dy)} 
	= 0.
\end{equation}
	Notice that the integrand, with respect to $\mathbb P$, is a continuous local martingale whose quadratic variation at time $t$ is bounded by $\|p_\epsilon\|_\infty |I|^2 \|\sigma\|^2_\infty t$.
	In fact, the quadratic variation is given by 
\begin{align}
	&\iint_0^t  \RB{\sum_{i\in I}\sigma(u_r(y)) p_\epsilon(x_i-y) \prod_{j\in I \setminus \{i\}} (P_\epsilon u_r)(x_j)}^2 \mathrm dr \mathrm dy
	\\&\leq 	\iint_0^t |I| \sum_{i\in I}\RB{\sigma(u_r(y)) p_\epsilon(x_i-y) \prod_{j\in I \setminus \{i\}} (P_\epsilon u_r)(x_j)}^2 \mathrm dr \mathrm dy
	\\&\leq  |I| \|\sigma\|^2_\infty \|p_\epsilon\|_\infty \sum_{i\in I} \iint_0^t p_\epsilon(x_i-y) \mathrm dr \mathrm dy 
	= |I|^2 \|\sigma\|^2_\infty \|p_\epsilon\|_\infty t.
\end{align}
	Therefore, \eqref{eq:TheMartingalePart} holds since this integrand is actually an $L^2$-martingale.
	We still need to verify that each of the expectations in \eqref{eq:TakingExpectation} are well-define. 
	In fact, the left two expectations are well-defined, since the integrands are bounded by $1$.
	For the first term on the right-hand side, we first notice from \eqref{eq:ChangeDerivative} that for every $x\in \mathbb R$, 
\[
	\frac{\Delta}{2} P_\epsilon u_t (x) 
	= \frac{1}{2}\int p''_\epsilon (x-y) u_t(y) \mathrm dy,
\]
	and that
\[
	p'_\epsilon(x)
	= p_\epsilon (x) \cdot (-\frac{x}{\epsilon}); \quad p''_\epsilon(x) = p_\epsilon(x) \SB{-\frac{1}{\epsilon}+\frac{x^2}{\epsilon^2}},
\]
	and therefore
\begin{align} \label{eq:LaplaceHeat}
	&\ABS{\frac{\Delta}{2} P_\epsilon u_t (x)} \leq \frac{1}{2}\int p_\epsilon(x-y) \SB{\frac{1}{\epsilon}+\frac{(x-y)^2}{\epsilon^2}} u_t(y) \mathrm dy
	\\&\leq  \frac{1}{2}\RB{\frac{1}{\epsilon} + \frac{1}{\epsilon^2}\int p_\epsilon(y)y^2 \mathrm dy} = \epsilon^{-1}.
\end{align}
	From this we know that the first term on the right-hand side of \eqref{eq:TakingExpectation} is bounded by $\epsilon^{-1} |I| t$. 
	Let us consider the second term on the right hand side. 
	Recall the drift 
\[
	b(z) 
	= \sum_{k=0}^\infty b_k z^k + b_\infty z^\infty, \quad z\in [0,1]
\]
	where we assumed \eqref{eq:CLZsCondition} for some $R\geq 1$. 
	In particular, $b$ is a bounded function on $[0,1]$.
	Now the integrand, with respect to $\mathbb E$, in the second term of the right hand side of \eqref{eq:TakingExpectation} is bounded by $\|b\|_\infty |I| t$.
	Finally, observing that
\[
	\int \sigma(u_r(y))^2p_\epsilon(x_i - y)p_\epsilon(x_j - y) \mathrm dy \leq \|\sigma\|^2_\infty \|p_\epsilon\|_\infty,
\]
	the integrand, with respect to $\mathbb E$, in the last term on the right hand side of \eqref{eq:TakingExpectation} is bounded by $\|\sigma\|^2_\infty \|p_\epsilon\|_\infty |I|^2t$. 
	Therefore \eqref{eq:TakingExpectation} holds.
\end{note}
\end{extra}
	
\emph{Step 3.}
	Let us replace the deterministic $(x_i)_{i\in I}$ in \eqref{eq:TakingExpectation} by the random $(X_s^{\alpha})_{\alpha \in I_s^{(m)}}$, and take expectations with respect to $\tilde{\mathbb P}$, after multiplied by $(-1)^{\tilde J_s^{(m)}} e^{K_s^{(m)}}$, for each of the terms in \eqref{eq:TakingExpectation}. 
	This leads us to the desired result for this lemma after applying Fubini's theorem.
	To use Fubini's theorem, of course, we need to verify the integrable conditions for each term.

	For instance, for the first term on the right hand side, it is sufficient to show that
\begin{equation} \label{eq:IntegratableCondiotion}
	\mathbf E\SB{(1+e^{K_s}) |I_s| } < \infty, \quad s\geq 0,
\end{equation}
	since for each $s\geq 0$, almost surely,
\begin{align}
	&\int_0^t \ABS{(-1)^{|\tilde J^{(m)}_s|} e^{K^{(m)}_s} \RB{\sum_{\alpha \in I_s^{(m)}}  \RB{\frac{\Delta}{2}P_\epsilon u_r}\RB{X^\alpha_s} \prod_{\beta \in I^{(m)}_s\setminus\{\alpha\}}  \RB{P_\epsilon u_r}\RB{X^\beta_s} }} \mathrm dr
	\\&\leq \epsilon^{-1}\int_0^t e^{K^{(m)}_s} |I_s^{(m)}| \mathrm dr 
	\leq t (1+e^{K_s}) \epsilon^{-1} |I_s|. 
\end{align}
	Note that  \eqref{eq:IntegratableCondiotion} holds by Lemma \ref{lem:exponential}, and furthermore,
\begin{equation}
	\sup_{s,t \in [0,T]} |\Lambda_{t,s}^{\epsilon,m}| 
	\leq \frac{T}{\epsilon} \sup_{0\leq s\leq T} \tilde{\mathbb E} \SB{(1+e^{K_s}) |I_s|} < \infty, \quad T\geq 0.
\end{equation}
	Similar arguments are valid for other terms of \eqref{eq:TakingExpectation} as well, while replacing $(x_i)_{i\in I}$  with $(X_s^\alpha)_{\alpha \in I_s^{(m)}}$, multiplying each terms with $(-1)^{\tilde J_s^{(m)}} e^{K_s^{(m)}}$, and then taking expectations with respect to $\tilde{\mathbb P}$. 
\begin{extra}
	(Let us verify this in Note \ref{note:413}.)
\end{extra}
	We are done. 
\end{proof}	
	\begin{extra}
		\begin{note} \label{note:413}
			Let us also verify for the other terms.
			\begin{itemize}
				\item 
				For the two terms on the left hand side of \eqref{eq:TakingExpectation}, it is sufficient to show that for every $s\geq 0$,
				\begin{equation} \label{eq:eKs}
					\mathbf E[1+e^{K_s}] = \tilde{\mathbb E}[1+e^{K_s}]< \infty
				\end{equation}
				since almost surely, 
				\begin{equation}
					\ABS{(-1)^{J_s^{(m)}} e^{K_s^{(m)}} \prod_{\alpha \in I_s^{(m)}} \RB{P_\epsilon u_t}(X_s^\alpha) } \leq 1+e^{K_s}.
				\end{equation}
				Note from Proposition \ref{prop:ExponentialTerm}, \eqref{eq:eKs} holds as desired, and furthermore, 
				\begin{equation}
					\sup_{s,t \in [0,T]} |\Xi_{t,s}^{\epsilon,m}| \leq \sup_{s \in [0,T]}\tilde{\mathbb E}[1+e^{K_s}]< \infty, \quad T\geq 0.
				\end{equation}
				\item 
				For the second term on the right hand side of \eqref{eq:TakingExpectation}, it is sufficient to show that  \eqref{eq:IntegratableCondiotion} holds 
				since for every $t\geq 0$, 
				\begin{align}
					&\int_0^t \ABS{ (-1)^{J_s^{(m)}} e^{K_s^{(m)}} \RB{\sum_{\alpha \in I_s^{(m)}} \RB{P_\epsilon \RB{b\circ u_r}}(X^\alpha_s) \prod_{\beta \in I_s^{(m)} \setminus \{\alpha\}} (P_\epsilon u_r)(X^\beta_s) } } \mathrm dr 
					\\&\leq t (1+ e^{K_s}) \|b\|_\infty |I_s|.
				\end{align}
				Note from Lemma \ref{lem:exponential}, \eqref{eq:IntegratableCondiotion} holds as desired, and furthermore,
				\begin{equation}
					\sup_{s,t \in [0,T]} |\Phi_{t,s}^{\epsilon,m}| \leq \|b\|_\infty T \sup_{s\in [0,T]}\tilde {\mathbb E}[(1+e^{K_s})  |I_s|] < \infty , \quad T\geq 0.
				\end{equation}
				\item 
				For the last term on the right hand side of \eqref{eq:TakingExpectation}, it is sufficient to show that 
				\begin{align} 
					&\mathbf E\SB{\iint_0^t \sum_{(\alpha,\beta) \in (I_s^{(m)})^2: \alpha\prec \beta} (1+e^{K_s}) p_\epsilon(X^\alpha_s - y)  \mathrm dr \mathrm dy} 
					\\&\leq \mathbf E\SB{\int_0^t (1+e^{K_s}) |I^{(m)}_s|^2 \mathrm dr} = t \mathbf E\SB{(1+e^{K_s}) |I^{(m)}_s|^2 } 
					\\&\label{eq:Is2}< \infty, \quad t\geq 0,
				\end{align}
				since
				\begin{align}
					&\ABS{(-1)^{J_s^{(m)}} e^{K_s^{(m)}} 
						\RB{\sum_{(\alpha,\beta) \in (I_s^{(m)})^2: \alpha \prec \beta} \sigma(u_r(y))^2 p_\epsilon(X^\alpha_s - y)p_\epsilon(X^\beta_s - y)\prod_{\gamma \in I_s^{(m)} \setminus \{\alpha,\beta\}} (P_\epsilon u_r)(X_s^\gamma)} }
					\\& \leq \|p_\epsilon\|_\infty \|\sigma\|^2_\infty (1+e^{K_s}) \sum_{(\alpha,\beta) \in (I_s^{(m)})^2: \alpha \prec \beta} p_\epsilon(X^\alpha_s - y).
				\end{align}
				Note from Lemma \ref{lem:exponential}, \eqref{eq:Is2} holds as desired; and	furthermore,
				\begin{equation}
					\sup_{s,t \in [0,T]} |\Psi_{t,s}^{\epsilon,m}| 
					\leq \|p_\epsilon\|_\infty \|\sigma\|^2_\infty T \sup_{0\leq s\leq T} \tilde{\mathbb E}[(1+e^{K_s})|I_s^{(m)}|^2]
					< \infty, \quad T\geq 0.
				\end{equation}
			\end{itemize}
		\end{note}
	\end{extra}

\begin{proof}[Proof of Lemma \ref{lem:DecXiByCBBM}]
	\emph{Step 1.} 
	Take an arbitrary $h \in \mathcal C_{\mathrm b}^{2}(\mathbb R, [0,1])$, i.e. a $[0,1]$-valued twice continuously differentiable function $h$ on $\mathbb R$ satisfying $\|h'\|_\infty < \infty$ and $\|h''\|_\infty < \infty$. 
	Also define $h(\dagger) := 1$. 
	Recall that $\xi_\alpha$ and $\zeta_\alpha^{(m)}$ are respectively the birth-time and the death-time of the particle $\alpha \in \mathcal U$ in the $m$-truncated branching-coalescing Brownian particle system. 
	Using Ito's formula, it is standard to verify that for any deterministic finite subset $\mathcal I$ of $\mathcal U$ and $s \geq 0$, we have almost surely 
\begin{align} \label{eq:PreIF}
	&\prod_{\alpha \in \mathcal I}h(X^{(m),\alpha}_s) - \prod_{\alpha \in \mathcal I}h(X^{(m),\alpha}_0) -  \sum_{r\leq s} \RB{\mathbf \Delta \prod_{\alpha \in \mathcal I}h(X^{(m),\alpha}_r)} 
	\\&= \int_0^s \sum_{\alpha \in \mathcal I} \RB{\prod_{\beta \in \mathcal I \setminus \{\alpha\}}h(X^{(m),\beta}_{r-})}\mathbf 1_{(\xi_\alpha, \zeta^{(m)}_\alpha]}(r) \frac{1}{2}h''(X^{(m),\alpha}_{r-})\mathrm dr + {}
	\\&\quad \int_0^s \sum_{\alpha \in \mathcal I} \RB{\prod_{\beta \in \mathcal I \setminus \{\alpha\}}h(X^{(m),\beta}_{r-})}\mathbf 1_{(\xi_\alpha, \zeta^{(m)}_\alpha]}(r) h'(X^{(m),\alpha}_{r-}) \mathrm dB^\alpha_r.
\end{align}
	Here, we write $\mathbf {\Delta} A_t := A_{t} - A_{t-}$ for any $t\geq 0$ and c\`adl\`ag process $(A_t)_{t\geq 0}$.
\begin{extra}
	(We will verify the result for this step in Note \ref{note:414})
\end{extra}
\begin{extra}
\begin{note}\label{note:414}
	In order to verify the claimed result in Step 1, we introduce Ito's formula for processes with finite jumps. 
	We say a semimartingale $(A_t)_{t\geq 0}$ has locally finite jumps, if 
\[
	\sum_{s\leq t} \mathbf 1_{\{\mathbf \Delta A_s \neq 0\}} < \infty, \quad t\geq 0, \text{a.s.}
\]
	For a given semimartingale $(A_t)_{t\geq 0}$ with locally finite jumps, define its pure jump part
\[
	\{A_t\}^{\mathrm J}:= \sum_{s\leq t} \mathbf \Delta A_s, \quad t\geq 0,
\]
	and its continuous part
\[
	\{A_t\}^{\mathrm C}:= A_t - \{A_t\}^{\mathrm J}, \quad t\geq 0.
\]
	Ito's formula for semimartingales with locally finite jumps can be formulated as follows, c.f. \cite[Theorem 20.7]{MR4226142}:
	Suppose that $\{(A^i_t)_{t\geq 0}: i = 1, \cdots, n\}$ is a family of semimartingales with locally finite jumps, and that $H \in \mathcal C^{2}(\mathbb R^n, \mathbb R)$, then $(H(A^1_t, \cdots, A^n_t))_{t\geq 0}$ is a semimartingale with locally finite jump, and almost surely,
\begin{align}\label{eq:Ito} 
	\mathrm d\{H(A^1_t, \cdots, A^n_t)\}^{\mathrm C} 
	&= \sum_{i=1}^n (\partial_i H)(A^1_{t-}, \cdots, A^n_{t-}) \mathrm d \{A^i_t\}^{\mathrm C} + {}
	\\& \quad \frac{1}{2}\sum_{i=1}^n \sum_{j=1}^n(\partial_j\partial_i H)(A^1_{t-}, \cdots, A^n_{t-}) \mathrm d\AB{ \{A^i_t\}^{\mathrm C}, \{A^j_t\}^{\mathrm C}}.
\end{align}
	
	Now, from $1- h(X^{(m),\alpha}_r) = \RB{1-h(\tilde X^\alpha_r)} \mathbf 1_{[\xi_\alpha, \zeta^{(m)}_\alpha)}(r) $ and the fact that $(\tilde X^\alpha_r)_{r\geq 0}$ is a Brownian motion, we have 
\begin{align}
	&-\mathrm dh(X^{(m),\alpha}_r)
	= \RB{1-h(\tilde X^\alpha_r)} \mathrm d\mathbf 1_{[\xi_\alpha, \zeta^{(m)}_\alpha)}(r) - \mathbf 1_{(\xi_\alpha, \zeta^{(m)}_\alpha]}(r)\mathrm d h(\tilde X^\alpha_r)
	\\&=\RB{1-h(\tilde X^\alpha_r)} \mathrm d\mathbf 1_{[\xi_\alpha, \zeta^{(m)}_\alpha)}(r) -\mathbf 1_{(\xi_\alpha, \zeta^{(m)}_\alpha]}(r) \RB{h'(\tilde X^\alpha_r) \mathrm d\tilde X^\alpha_r + \frac{1}{2}h''(\tilde X^\alpha_r)\mathrm dr}.
\end{align}
	From the definition of $(\tilde X^\alpha_r)_{r\geq 0}$, we have 
\begin{align} \label{eq:XAR}
	\mathrm d\tilde X^\alpha_r = \mathbf 1_{(0, \xi_\alpha]}(r) \mathrm d\tilde X^{\overleftarrow{\alpha}}_r + \mathbf 1_{( \xi_\alpha,\infty)}(r) \mathrm dB^\alpha_r.
\end{align}
	Now, noticing that $\tilde X^\alpha_r = X^{(m),\alpha}_{r-}$ for $r\in (\xi_\alpha, \zeta^{(m)}_\alpha]$, we have that $h(X^{(m),\alpha}_\cdot)$ is a semimartingale with locally finite jumps, and that 
\begin{align}
	\mathrm d\{h(X^{(m),\alpha}_r)\}^{\mathrm C}= \mathbf 1_{(\xi_\alpha, \zeta^{(m)}_\alpha]}(r) \RB{h'(X^{(m),\alpha}_{r-}) \mathrm dB^\alpha_r+ \frac{1}{2}h''(X^{(m),\alpha}_{r-})\mathrm dr}.
\end{align}
	For $\alpha \neq \beta$, it is easy to see that $\mathrm d\AB{\CB{h(X^{(m),\alpha}_r)}^{\mathrm C},\CB{h(X^{(m),\beta}_r)}^{\mathrm C}}=0$.
	Therefore, for any deterministic finite subset $\mathcal I$ of $\mathcal U$, we have 
\begin{equation}
	\mathrm d\CB{\prod_{i \in \mathcal I}h(X^{m,i}_r)}^{\mathrm C}
	= \sum_{i \in \mathcal I} \RB{\prod_{j \in \mathcal I \setminus \{i\}}h(X^{m,j}_{r-})}\mathbf 1_{(\xi_i, \zeta^{(m)}_i]}(r) \RB{h'(X^{m,i}_{r-}) \mathrm dB^i_r+ \frac{1}{2}h''(X^{m,i}_{r-})\mathrm dr}.
\end{equation}
	Writing this in its integral form, we arrive at the desired result \eqref{eq:PreIF}.
\end{note}
\end{extra}
	
	\emph{Step 2.}
	Let us fix the arbitrary $m \in \mathbb N$, and take a sequence of deterministic finite subset $\mathcal I_n$ of $\mathcal U$ such that $\mathcal I_n \uparrow \mathcal U$ as $n\uparrow \infty$.
	Recall that $I_r^{(m)}$ is the set of labels of the living particles at time $r\geq 0$ for the $m$-truncated branching-coalescing Brownian particle system.
	From Subsection \ref{sec:Truncated}, we have for each $s\geq 0$,
\[
	I_{[0,s]}^{(m)}
	:= \bigcup_{r\in [0, s]} I_r^{(m)}
\]
	is a (random) finite subset of $\mathcal U$.  
	Therefore, almost surely for every $s\geq 0$, there exists a large (random) $\tilde N_s \in \mathbb N$, such that for any $n \geq \tilde N_s$ it holds that $I_{[0,s]}^{(m)} \subset \mathcal I_n$.
	Using this, we can verify that, after replacing $\mathcal I$ by this sequence of $\mathcal I_n$ and then taking $n\uparrow \infty$, the second term on the right hand side of  \eqref{eq:PreIF} is Cauchy sequence in $\mathscr M^2_{\mathrm c}$, the space of continuous $L^2$-martingales; while all the other terms of \eqref{eq:PreIF} converges almost surely. 
\begin{extra}
	(We will verify this in Note \ref{note:415}.)
\end{extra}
	This allows us to verify that almost surely for each $s\geq 0$, 
\begin{align}
	&\prod_{ \alpha \in I_s^{(m)}}h(X^{(m),\alpha}_s) - \prod_{\alpha \in I_0^{(m)}}h(X^{(m),\alpha}_0) - \sum_{r\leq s} \CB{\mathbf \Delta \prod_{\alpha \in I_r^{(m)}}h(X^{(m),\alpha}_r)}
	\\&=  \frac{1}{2}\int_0^s \sum_{\alpha \in I_{r-}^{(m)} } \RB{\prod_{\beta \in I_{r-}^{(m)}\setminus \{\alpha\}}h(X^{(m),\beta}_{r-})}  h''(X^{(m),\alpha}_{r-})\mathrm dr + M^{m,h}_s
\end{align}
	where $M_\cdot^{m,h}$ is a continuous $L^2$-martingale with quadratic variation 
\[
	\langle M^{m,h}_s \rangle = \int_0^s \sum_{\alpha \in I^{(m)}_{r-}}  \RB{\prod_{\beta \in I^{(m)}_{r-} \setminus \{\alpha\}}h(X^{(m),\beta}_{r-})}^2 h'(X^{(m),\beta}_{r-})^2  \mathrm dr, \quad s\geq 0.
\]
\begin{extra}
	(We will verify this in Note \ref{note:416}.) 
\end{extra}
	
	\begin{extra}
		\begin{note} \label{note:415}
			Let us verify that $\{(M_n(s))_{s\geq 0}: n \in \mathbb N\}$ is a Cauchy sequence in $\mathscr M^2_{\mathrm c}$ where
			\begin{equation}
				M_n(s) : = \int_0^s \sum_{\alpha \in \mathcal I_n} \RB{\prod_{\beta \in \mathcal I_n \setminus \{\alpha\}}h(X^{(m),\beta}_{r-})}\mathbf 1_{(\xi_i, \zeta^{(m)}_\alpha]}(r) h'(X^{(m),\alpha}_{r-}) \mathrm dB^\alpha_r.
			\end{equation}
			In fact, to see $\{M_n(\cdot): n\in \mathbb N\}$ is an $\mathscr M^2_{\mathrm c}$-valued sequence, it is suffice to show the following that their quadratic variations is integrable:  For every $n\in \mathbb N$ and $s\geq 0$, 
			\begin{align} \label{eq:LntUB}
				\langle M_n(s)\rangle &= \sum_{\alpha \in \mathcal I_n} \int_0^s \RB{\prod_{\beta \in \mathcal I_n \setminus \{\alpha\}}h(X^{(m),\beta}_{r-})}^2 \mathbf 1_{(\xi_\alpha, \zeta^{(m)}_\alpha]}(r)^2 h'(X^{(m),\alpha}_{r-})^2  \mathrm dr 
				\\& \leq \|h'\|_\infty \int_0^s  |I_{r-}^{(m)}| \mathrm dr 
				\\& \in L^1(\tilde {\mathbb P}), \quad \text{by Theorem~\ref{prop:Key}.}
			\end{align}
			Also, since $(\mathcal I_n)_{n\geq 0}$ is an increasing sequence of sets, we have for $k \leq n$ and $s\geq 0$, 
			\begin{align}
				M_n(s)- M_k(s) &= \int_0^s \sum_{\alpha \in \mathcal I_k} \RB{\prod_{\beta \in \mathcal I_n \setminus \{\alpha\}}h(X^{(m),\beta}_{r-})-\prod_{\beta \in \mathcal I_k \setminus \{\alpha\}}h(X^{(m),\beta}_{r-})}\mathbf 1_{(\xi_\alpha, \zeta^{(m)}_\alpha]}(r) h'(X^{(m),\alpha}_{r-}) \mathrm dB^\alpha_r
				+ {}
				\\& \quad \int_0^s \sum_{\alpha \in \mathcal I_n\setminus \mathcal I_k} \RB{\prod_{\beta \in \mathcal I_n \setminus \{\alpha\}}h(X^{(m),\beta}_{r-}) }\mathbf 1_{(\xi_\alpha, \zeta^{(m)}_\alpha]}(r) h'(X^{(m),\alpha}_{r-}) \mathrm dB^\alpha_r.
			\end{align}
			Calculating the quadratic variation of $M_n(\cdot) - M_k(\cdot)$, we obtain for every $s\geq 0$,
			\begin{align}
				&\langle M_n(s)- M_k(s) \rangle 
				\\&= \sum_{\alpha \in \mathcal I_k} \int_0^s  \RB{\prod_{\beta \in \mathcal I_n \setminus \{\alpha\}}h(X^{(m),\beta}_{r-}) - \prod_{\beta \in \mathcal I_k \setminus \{\alpha\}}h(X^{(m),\beta}_{r-})}^2 \mathbf 1_{(\xi_\alpha, \zeta^{(m)}_\alpha]}(r)^2 h'(X^{(m),\alpha}_{r-})^2 \mathrm dr
				+ {}
				\\& \quad  \sum_{\alpha \in \mathcal I_n \setminus \mathcal I_k} \int_0^s \RB{\prod_{\beta \in \mathcal I_n \setminus \{\alpha\}}h(X^{(m),\beta}_{r-}) }^2 \mathbf 1_{(\xi_\alpha, \zeta^{(m)}_\alpha]}(r)^2 h'(X^{(m),\alpha}_{r-})^2 \mathrm dr.
			\end{align}
			Observe that for $s\geq 0$ and $n,k\geq \tilde N_s$, it holds almost surely that $\langle M_n(s)- M_k(s) \rangle =0$.
			This implies that 
			\[
			\lim_{k\to \infty} \sup_{n\geq k} \AB{M_n(s)- M_k(s)} = 0, \quad s\geq 0.
			\]
			Also notice that, almst surely, we have the domination
			\begin{align}
				\sup_{k,n\in \mathbb N: n\geq  k}\ABS{\AB{M_n(s)- M_k(s)}}  
				&\leq \|h'\|_\infty^2 3 \int_0^s |I_{r-}^{(m)}| \mathrm dr
				\\& \in L^1(\tilde {\mathbb P}), \quad \text{by Theorem~\ref{prop:Key}.}
			\end{align}
			Now, by Ito's isometry and dominated convergence theorem, we have for $s\geq 0$,
			\begin{align}
				&\lim_{k\to \infty} \sup_{n\geq k}\tilde {\mathbb E}\SB{\ABS{M_n(s) - M_k(s)}^2} 
				= \lim_{k\to \infty} \sup_{n\geq k} \tilde{\mathbb E}\SB{\langle M_n(s)- M_k(s) \rangle} 
				\\&\leq \lim_{k\to \infty}  \tilde{\mathbb E}\SB{\sup_{n\geq k} \langle M_n(s)- M_k(s) \rangle}= 0.
			\end{align}
			This is sufficient for the claim that $\{M_n(\cdot): n \in \mathbb N\}$ is a Cauchy sequence in $\mathscr M^2_{\mathrm c}$.
		\end{note}
		
		\begin{note} \label{note:416}
			To see the claimed result, we first verify, using $h(\dagger) = 1$, that for every $s\geq 0$, almost surely
			\begin{align}
				&\lim_{n\uparrow \infty} \prod_{ \alpha \in \mathcal I_n}h(X^{(m),\alpha}_s) = \prod_{\alpha \in I_s^{(m)}}h(X^{(m),\alpha}_s);
				\\ &\lim_{n\uparrow \infty} \prod_{\alpha \in \mathcal I_n}h(X^{(m),\alpha}_0) = \prod_{\alpha \in I_s^{(m)}}h(X^{(m),\alpha}_0);
				\\&\lim_{n\uparrow \infty} \sum_{r\leq s} \RB{\mathbf \Delta \prod_{\alpha \in \mathcal I_n}h(X^{(m),\alpha}_r)}  = \sum_{r\leq s} \RB{\mathbf \Delta \prod_{\alpha \in I^{(m)}_r}h(X^{(m),\alpha}_r)};
			\end{align}
			and 
			\begin{align}
				&\lim_{n\uparrow \infty} \int_0^s \sum_{\alpha \in \mathcal I_n} \RB{\prod_{\beta \in \mathcal I_n \setminus \{\alpha\}}h(X^{(m),\beta}_{r-})}\mathbf 1_{(\xi_\alpha, \zeta^{(m)}_\alpha]}(r) \frac{1}{2}h''(X^{(m),\alpha}_{r-})\mathrm dr
				\\&= \int_0^s \sum_{\alpha \in I^{(m)}_{r-}} \RB{\prod_{\beta \in I^{(m)}_{r-} \setminus \{\alpha\}}h(X^{(m),\beta}_{r-})} \frac{1}{2}h''(X^{(m),\alpha}_{r-})\mathrm dr.
			\end{align}
			Let us denote the $\mathscr M^2_{\mathrm c}$-limit of the sequence $\{M_n(\cdot): n \in \mathbb N\}$ by $M^{m,h}_\cdot$.
			To find the quadratic variation of $M^{m,h}_\cdot$, let us first define a process $\AB{M^{m,h}_\cdot}$ so that
			\begin{equation}\label{eq:Mmh}
				\AB{M_n(s)} \xrightarrow[n\to \infty]{\text{a.s.}} \int_0^s \sum_{i \in I^{(m)}_{r-}}  \RB{\prod_{j \in I^{(m)}_{r-} \setminus \{i\}}h(X^{m,j}_{r-})}^2 h'(X^{m,i}_{r-})^2  \mathrm dr =: \AB{M^{m,h}_s}, \quad s\geq 0.
			\end{equation}
			Observe that $\AB{M^{m,h}_\cdot}$ is a predictable integrable increasing process, in the sense of \cite{MR1011252}*{p.~35}, since
			\[
			\AB{M^{m,h}_s} \leq \|h'\|_\infty^2 \int_0^s |I^{(m)}_{r-}| \mathrm dr \in L^1(\tilde{\mathbb P}), \quad \text{by Theorem~\ref{prop:Key}.}
			\]
			Therefore, to verify that $\AB{M^{m,h}_\cdot}$ is the quadratic variation of the process  $M^{m,h}_\cdot$, we only need to verify that $(M^{m,h}_\cdot)^2 - \AB{M^{m,h}_\cdot}$ is a martingale \cite{MR1011252}*{Definition 2.1}.
			Notice that, we already have
			\[
			\tilde {\mathbb E}[ M_n(t)^2 - \langle M_n(t)\rangle | \tilde{\mathscr F}_s] = M_n(s)^2 - \langle M_n(s)\rangle
			\]
			for every $n\in \mathbb N$ and $0\leq s\leq t$.
			This is equivalent of saying that
			\begin{equation} \label{eq:CEF}
				\tilde {\mathbb E}\SB{\RB{M_n(t)^2 - \langle M_n(t)\rangle} \mathbf 1_H} = \tilde {\mathbb E}\SB{\RB{M_n(s)^2 - \langle M_n(s)\rangle}\mathbf 1_H}
			\end{equation}
			for every $n\in \mathbb N$, $0\leq s\leq t$ and $H \in \mathscr F_s$.
			Fix arbitrary $0\leq s\leq t $ and $H \in \mathscr F_s$ and observe that
			\[
			\tilde {\mathbb E}\SB{ (M_n(t) \mathbf 1_H - M^{m,h}_t \mathbf 1_H)^2 } \leq \tilde {\mathbb E}\SB{ (M_n(t) - M^{m,h}_t)^2 } \xrightarrow[n\to \infty]{} 0,
			\]
			which says that 
			\[M_n(t)\mathbf 1_H \xrightarrow[n\to \infty]{L^2} M^{m,h}_t \mathbf 1_H,\]
			which further implies that
			\[\tilde {\mathbb E}\SB{ M_n(t)^2 \mathbf 1_H} \xrightarrow[n\to \infty]{} \tilde {\mathbb E}\SB{ (M^{m,h}_t)^2 \mathbf 1_H}. \]
			Also from \eqref{eq:LntUB} and the dominated convergence theorem, we have 
			\[\tilde {\mathbb E}\SB{\AB{M_n(t)} \mathbf 1_H} 
			\xrightarrow [n\to \infty]{}
			\tilde {\mathbb E}
			\SB{\AB{M_t^{m,h}} \mathbf 1_H}.
			\]
			Therefore, by taking $n\to \infty$ in \eqref{eq:CEF}, we get
			\[
			\tilde {\mathbb E}\SB{\RB{(M^{m,h}_t)^2 - \langle M^{m,h}_t \rangle} \mathbf 1_H} = \tilde {\mathbb E}\SB{\RB{(M^{m,h}_s)^2- \langle M^{m,h}_s \rangle }\mathbf 1_H}, \quad H \in \mathscr F_s.
			\]
			As we have discussed, this implies that $\langle M^{m,h}_\cdot \rangle$, given by \eqref{eq:Mmh}, is the quadratic variation of $M^{m,h}_\cdot$. 
			Now, after replacing $\mathcal I$ by $\mathcal I_n$ in \eqref{eq:PreIF} and then taking $n\to \infty$, we can verify the desired claim.
		\end{note}
	\end{extra}
	
	\emph{Step 3.}
	Let us define the process
\[
	H^{m,h}_r 
	:= (-1)^{\tilde J^{(m)}_r} e^{K_r^{(m)}} \prod_{\alpha \in I_r^{(m)}}h\RB{X^{(m),\alpha}_r}, \quad r\geq 0.
\]
	Since there are only finitely many jumps for the process $H^{m,h}_\cdot$ up to any finite time, it is straightforward to verify, using Step 2 and Ito's formula, that almost surely for every $s\geq 0$,
\begin{align}\label{eq:Conti}
	&H_s^{m,h} - H_0^{m,h} 
	\\&= \int_0^s H_{r-}^{m,h} \sum_{\alpha \in I_{r-}^{(m)}}   \RB{\mu + b_1 -\frac{1}{m} + \frac{h''(X^{(m),\alpha}_{r-})}{2h(X^{(m),\alpha}_{r-})} } \mathrm dr + {} 
		\\& \quad \int_0^s (-1)^{\tilde J^{(m)}_{r-}} e^{K_{r-}^{(m)}} \mathrm dM^{m,h}_r + {}
	\\& \quad  \int_{(0,s]\times \mathcal U \times \bar{\mathbb N}} \mathbf 1_{\{\alpha \in I_{r-}^{(m)}\}} H_{r-}^{m,h} \RB{(-1)^{\mathbf 1_{\{b_k<0\}}}\frac{h(X^{(m),\alpha}_{r-})^{k\wedge m}}{h(X^{(m),\alpha}_{r-})} -1}
	\mathfrak N(\mathrm dr, \mathrm d\alpha, \mathrm dk) + {}
	\\& \quad \int_{(0,s]\times \mathcal R} \mathbf 1_{\{\alpha,\beta\in I^{(m)}_{r-}\}} H_{r-}^{m,h} (h(X^{(m),\alpha}_{r-})^{-1}-1)\mathfrak M(\mathrm dr, \mathrm d(\alpha,\beta)).
\end{align}
\begin{extra}
	(We will verify this in Note \ref{note:417})
\begin{note} \label{note:417}
	In fact, by using \eqref{eq:Ito} and Step 2, we have 
\begin{align}
	&\mathrm d \CB{H_r^{m,h}}^{\mathrm c}
	=  H_{r-}^{m,h} \mathrm dK_r^{(m)} + (-1)^{\tilde J^{(m)}_{r-}} e^{K_{r-}^{(m)}} \mathrm d\CB{\prod_{\alpha \in I_r^{(m)}}h(X^{(m),\alpha}_r)}^{\mathrm c} 
	\\&	= H_{r-}^{m,h} \RB{\mu + b_1-\frac{1}{m}} |I_{r-}^{(m)}|\mathrm dr + \sum_{\alpha \in I_{r-}^{(m)}} H_{r-}^{m,h}\frac{h''(X^{(m),\alpha}_{r-})}{2h(X^{(m),\alpha}_{r-})}\mathrm dr+ (-1)^{\tilde J^{(m)}_{r-}} e^{K_{r-}^{(m)}} \mathrm dM^{m,h}_r
	\\&	 = H_{r-}^{m,h} \sum_{\alpha \in I_{r-}^{(m)}}   \RB{\mu + b_1 - \frac{1}{m} + \frac{h''(X^{(m),\alpha}_{r-})}{2h(X^{(m),\alpha}_{r-})} } \mathrm dr + (-1)^{\tilde J^{(m)}_{r-}} e^{K_{r-}^{(m)}} \mathrm dM^{m,h}_r.
\end{align}
	As for the jumping part, since there are only finitely many jumps up to any finite time, it is straightforward to verify that 
\begin{align}
	&\sum_{r\leq s} \RB{\mathbf \Delta H_r^{m,h}} 
	\\&\label{eq:Jumping}= \int_{(0,s]\times \mathcal U \times \mathbb Z_+} \mathbf 1_{\{\alpha \in I_{r-}^{(m)}\}} H_{r-}^{m,h}
	\RB{(-1)^{\mathbf 1_{\{b_k < 0\}}} \frac{ h(X^{(m),\alpha}_{r-})^{k\wedge m}}{h(X^{(m),\alpha}_{r-})} -1}
	\mathfrak N(\mathrm dr, \mathrm d\alpha, \mathrm dk) + {}
	\\& \quad \int_{(0,s]\times \mathcal R} \mathbf 1_{\{\alpha,\beta \in I^{(m)}_{r-}\}} H_{r-}^{m,h} (h(X^{(m),\alpha}_{r-})^{-1}-1)\mathfrak M(\mathrm dr, \mathrm d(\alpha,\beta)).
\end{align}
\end{note} 
\end{extra}
	
	\emph{Step 4.}
	We want to take the expectation of \eqref{eq:Conti}. 
	However, it is not clear whether the second term on the right hand side 
	is a (true) martingale.
	Notice that its quadratic variation is given by 
	\begin{align}
		&\AB{\int_0^s (-1)^{\tilde J^{(m)}_{r-}} e^{K_{r-}^{(m)}} \mathrm dM^{m,h}_r}
		= \int_0^s e^{2K_{r-}^{(m)}}  \sum_{\alpha \in I^{(m)}_{r-}}  \RB{\prod_{\beta \in I^{(m)}_{r-} \setminus \{\alpha\}}h(X^{(m),\beta}_{r-})}^2 h'(X^{(m),\alpha}_{r-})^2  \mathrm dr
		\\&\leq \|h'\|_\infty^2 \int_0^s e^{2K_{r-}^{(m)}} |I_{r-}^{(m)}|  \mathrm dr, \quad s\geq 0.
	\end{align}
	Therefore, we can define a sequence of predictable stopping time $\tau_n \uparrow \infty$ by
	\begin{equation}
		\tau_n := \inf\CB{t\geq 0: \|h'\|_\infty^2 \int_0^t e^{2K_{r-}^{(m)}} |I_{r-}^{(m)}|  \mathrm dr = n}, \quad n \in \mathbb N,
	\end{equation}
	which guarantees that, for each $n\in \mathbb N$,
	\begin{equation}
		s\mapsto \int_0^{s\wedge \tau_n} (-1)^{\tilde J^{(m)}_{r-}} e^{K_{r-}^{(m)}} \mathrm dM^{m,h}_r
	\end{equation}
	is an $L^2$-martingale.
	Let us then take the expectation of \eqref{eq:Conti} while replacing $s$ by $s\wedge \tau_n$
	and obtain, for each $n\in \mathbb N$ and $s\geq 0$,
\begin{align} \label{eq:PreSE}
	\\&\tilde{\mathbb E}\SB{H_{s\wedge \tau_n}^{m,h}} -\tilde{\mathbb E}\SB{ H_0^{m,h}} 
	\\& = \tilde {\mathbb E}\SB{ \int_0^{s\wedge \tau_n} H_{r-}^{m,h} \sum_{\alpha \in I_{r-}^{(m)}}   \RB{\mu + b_1 -\frac{1}{m}+ \frac{h''(X^{(m),\alpha}_{r-})}{2h(X^{(m),\alpha}_{r-})} } \mathrm dr} + {}
	\\& ~\quad \tilde {\mathbb E}\SB{\int_{(0,s\wedge \tau_n]\times \mathcal U \times \bar{\mathbb N}} \mathbf 1_{\{\alpha \in I_{r-}^{(m)}\}} H_{r-}^{m,h}
	\RB{(-1)^{\mathbf 1_{\{b_k < 0\}}}\frac{h(X^{(m),\alpha}_{r-})^{k\wedge m}}{h(X^{(m),\alpha}_{r-})}  -1} \hat{\mathfrak N}(\mathrm dr, \mathrm d\alpha, \mathrm dk)} + {}
	\\& ~\quad \tilde{\mathbb E} \SB{\int_{(0,s \wedge \tau_n]\times \mathcal R} \mathbf 1_{\{\alpha,\beta \in I^{(m)}_{r-}\}} H_{r-}^{m,h} (h(X^{(m),\alpha}_{r-})^{-1}-1)\hat{\mathfrak M}(\mathrm dr, \mathrm d(\alpha,\beta))}.
\end{align}
	\begin{extra}
		(We will explain this in Note \ref{note:418}.)
	\end{extra}
	Here, we have replaced $\mathfrak N$, and $\mathfrak M$, by their compensators $\hat{\mathfrak N}$, and $\hat{\mathfrak M}$, respectively. 
	This is allowed, due to Lemma \ref{eq:CompensatorDecom} and the fact that for any $s\geq 0$ and $n\in \mathbb N$, 
\begin{align} \label{eq:Cret}
	\\&\int_{(0,s\wedge \tau_n]\times \mathcal U \times \bar{\mathbb N} } \ABS{ \mathbf 1_{\{\alpha \in I_{r-}^{(m)}\}} H_{r-}^{m,h}
	\RB{(-1)^{\mathbf 1_{\{b_k < 0\}}}\frac{h(X^{(m),\alpha}_{r-})^{k\wedge m}}{h(X^{(m),\alpha}_{r-})} -1} }
	\hat{\mathfrak N}(\mathrm dr, \mathrm d\alpha, \mathrm dk) 
	\\&\leq  \int_0^{s} \sum_{\alpha \in I_{r-}^{(m)}} \sum_{k\in \bar {\mathbb N}} \ABS{ \mathbf 1_{\{\alpha \in I_{r-}^{(m)}\}} \frac{H_{r-}^{m,h}}{h(X^{(m),\alpha}_{r-})}
	\RB{ (-1)^{\mathbf 1_{\{b_k < 0\}}} h(X^{(m),\alpha}_{r-})^{k\wedge m} -h(X^{(m),\alpha}_{r-}) } } \mu p_k \mathrm dr
	\\& \leq 2 \mu \int_0^{s} e^{K_{r-}^{(m)}}|I_{r-}^{(m)}| \mathrm dr \in L^1(\tilde {\mathbb P}),  \quad \text{by Lemma \ref{lem:exponential}}
\end{align}
	and that
\begin{align}\label{eq:MC}
	&2\int_{(0,s \wedge \tau_n]\times \mathcal R} \ABS{\mathbf 1_{\{ \alpha,\beta\in I^{(m)}_{r-}\}} H_{r-}^{m,h} (h(X^{(m),\alpha}_{r-})^{-1}-1)}\hat{\mathfrak M}(\mathrm dr, \mathrm d(\alpha,\beta))
	\\&\leq \sum_{(\alpha,\beta) \in \mathcal R} \int_0^s \ABS{\mathbf 1_{\{\alpha,\beta\in I^{(m)}_{r-}\}} \frac{H_{r-}^{m,h}}{h(X^{(m),\alpha}_{r-})} (1-h(X^{(m),\alpha}_{r-}))} \mathrm d L^{\alpha,\beta}_{r}
	\\&\leq 2\sum_{(\alpha,\beta) \in \mathcal R} \int_0^s \mathbf 1_{\{\alpha,\beta\in I^{(m)}_{r-}\}} e^{K_{r-}^{(m)}} \mathrm d L^{\alpha,\beta}_{r}
	\\& \leq 2 (1+e^{K_{s}^{(m)}} ) \sum_{\alpha,\beta \in I_{[0,s]}^{(m)}: \alpha \prec \beta} L_s^{\alpha,\beta}
	\in L^1(\tilde {\mathbb P}),  \quad \text{by Lemma \ref{lem:bug}.}
\end{align}
	Finally, observing that there are certain cancellations on the right hand side of \eqref{eq:PreSE}, we obtain that for any $s\geq 0$ and $n\in \mathbb N$, 
\begin{align}\label{eq:StopedExpectation} 
	&\tilde{\mathbb E}\SB{H_{s\wedge \tau_n}^{m,h}} -\tilde{\mathbb E}\SB{ H_0^{m,h}} 
	\\& =\tilde {\mathbb E}\SB{\int_0^{s\wedge \tau_n}  H_{r-}^{m,h} \sum_{\alpha \in I_{r-}^{(m)}} \frac{h''(X^\alpha_{r-})}{2h(X^\alpha_{r-})}\mathrm dr} +{} 
	\\& \qquad \tilde {\mathbb E}\SB{\int_0^{s\wedge \tau_n}  H_{r-}^{m,h} \sum_{\alpha \in I_{r-}^{(m)}} \frac{b^{(m)}(h(X^\alpha_{r-}))}{h(X^\alpha_{r-})}  \mathrm dr} + {}
	\\&  \qquad \tilde {\mathbb E}\SB{\int_{(0,s\wedge \tau_n]\times \mathcal R} \mathbf 1_{\{\alpha,\beta\in I^{(m)}_{r-}\}} H_{r-}^{m,h} (h(X^\alpha_{r-})^{-1}-1)\hat{\mathfrak M}(\mathrm dr, \mathrm d(\alpha,\beta))}.
\end{align}
	Here, recall that $b^{(m)}$ is defined in \eqref{eq:DefBmz}. 
\begin{extra}
	(This cancellation will also be made clear in Note \ref{note:418}.)
\end{extra}
\begin{extra}
\begin{note} \label{note:418}
	Let us give more details while taking expectations over \eqref{eq:Conti}.
	Essentially, we observe the following:
\begin{itemize}
\item 
	The first term on the left hand side of \eqref{eq:Conti}, after replacing $s$ by $s\wedge \tau_n$, is dominated by $e^{K_s^{(m)}}\in L^1(\tilde {\mathbb P})$;
\item 
	The second term on the left hand side of \eqref{eq:Conti}, after replacing $s$ by $s\wedge \tau_n$, is dominated by $1$;
\item 
	The first term on the right hand side of \eqref{eq:Conti}, after replacing $s$ by $s\wedge \tau_n$, is dominated by 
\begin{align}
	&\int_0^{s} \ABS{ \sum_{\alpha \in I_{r-}^{(m)}} \frac{H_{r-}^{m,h}}{h(X^\alpha_{r-})}  \RB{(\mu + b_1-\frac{1}{m})h(X^\alpha_{r-}) + \frac{h''(X^\alpha_{r-})}{2} } } \mathrm dr 
	\\&\leq (\mu + |b_1| + \frac{1}{m}+ \|h''\|_\infty/2) \int_0^s e^{K_{r-}^{(m)}} |I_{r-}^{(m)}| \mathrm dr 
	\in L^1(\tilde {\mathbb P})
\end{align}
	thanks to Lemma \ref{lem:exponential};
\item 
	The second term on the right hand side of \eqref{eq:Conti} while replacing $s$ by $s\wedge \tau_n$ is a martingale, thanks to how $\tau_n$ is defined;
\item 
	Thanks to Lemma \ref{eq:CompensatorDecom} and \eqref{eq:Cret}, the third term on the right hand side of \eqref{eq:Conti}, after replacing $s$ by $s\wedge \tau_n$, has expectation
\begin{align}
	&\tilde {\mathbb E}\SB{\int_{(0,s\wedge \tau_n]\times \mathcal U \times {\bar{\mathbb N}}} \mathbf 1_{\{\alpha \in I_{r-}^{(m)}\}} H_{r-}^{m,h}\RB{(-1)^{\mathbf 1_{\{b_k < 0\}}}\frac{h(X^{(m),\alpha}_{r-})^{m\wedge k}}{h(X^{(m),\alpha}_{r-})}  -1} \hat{\mathfrak N}(\mathrm dr, \mathrm d\alpha, \mathrm dk)}
	\\& = \mu \tilde {\mathbb E}\SB{\int_0^{s\wedge \tau_n}  \sum_{\alpha \in I_{r-}^{(m)}} \sum_{k\in \bar{\mathbb N}}  p_k H_{r-}^{m,h}\RB{(-1)^{\mathbf 1_{\{b_k < 0\}}} \frac{h(X^{(m),\alpha}_{r-})^{m\wedge k}}{h(X^{(m),\alpha}_{r-})}  -1} \mathrm dr}
	\\&=	\tilde {\mathbb E}\SB{\int_0^{s\wedge \tau_n}  \sum_{\alpha \in I_{r-}^{(m)}} H_{r-}^{m,h} \RB{\sum_{k\in \bar{\mathbb N}: k\neq 1} b_k \frac{  h(X^{(m),\alpha}_{r-})^{k\wedge m}  }{h(X^{(m),\alpha}_{r-})} - \mu} \mathrm dr}
	\\&=	\tilde {\mathbb E}\SB{\int_0^{s\wedge \tau_n}  \sum_{\alpha \in I_{r-}^{(m)}} H_{r-}^{m,h} \RB{ \frac{b^{(m)}(h(X^\alpha_{r-}))}{h(X^\alpha_{r-})}-\RB{\mu+b_1 - \frac{1}{m}}} \mathrm dr}.
\end{align}
	It is also clear now how the cancellation happens between the first and the third term.
\item 
	Since all the other terms of \eqref{eq:Conti}, after replacing $s$ by $s\wedge \tau_n$, are integrable against $\tilde {\mathbb P}$, so is the last term, i.e. the fourth term on the right hand side. 
	Furthermore, \eqref{eq:MC} and Lemma \ref{eq:CompensatorDecom} allow us to replace $\mathfrak M$ by its compensator.
\end{itemize}
\end{note}
\end{extra}

\emph{Step 5.}
Let us fix and arbitrary $s\geq 0$.
Observe that, for each $n\in \mathbb N$, the integrands in the first and the second terms on the left hand side of  \eqref{eq:StopedExpectation} are dominated by $1+e^{K_s^{(m)}}$ and $1$ repectively; while the integrands in the first and the second terms on the right hand side of \eqref{eq:StopedExpectation} are dominated by 
\[
 \frac{1}{2}\|h''\|_\infty \int_0^s e^{K_{r-}^{(m)}} |I_{r-}^{(m)}| \mathrm dr \quad  \text{and} \quad  \|b^{(m)}\|_\infty \int_0^s e^{K_{r-}^{(m)}} |I_{r-}^{(m)}| \mathrm dr
\]
respectively.
Also observe from \eqref{eq:MC} that the integrand in the third term on the right hand side of \eqref{eq:StopedExpectation} is dominated by 
\begin{align}\label{eq:MM}
	2 (1+e^{K_{s}^{(m)}} ) \sum_{\alpha,\beta \in I_{[0,s]}^{(m)}: \alpha \prec \beta} L_s^{\alpha,\beta}.
\end{align}
Now, by using Lemmas \ref{lem:exponential}, \ref{lem:bug} and the dominated convergence theorem, after taking $n\to \infty$ in \eqref{eq:StopedExpectation}, we can verify that  \eqref{eq:StopedExpectation} still holds after replacing $s\wedge \tau_n$ by $s$. That is
\begin{align} \label{eq:StopedExpectation3}
	\\&\tilde{\mathbb E}\SB{H_{s}^{m,h}} -\tilde{\mathbb E}\SB{ H_0^{m,h}} 
	\\& = \tilde {\mathbb E}\SB{\int_0^{s}  H_{r-}^{m,h} \sum_{\alpha \in I_{r-}^{(m)}} \frac{h''(X^\alpha_{r-})}{2h(X^\alpha_{r-})} \mathrm dr} + \tilde {\mathbb E}\SB{\int_0^{s}  H_{r-}^{m,h} \sum_{ \alpha \in I_{r-}^{(m)}} \frac{b^{(m)}(h(X^\alpha_{r-}))}{h(X^\alpha_{r-})}  \mathrm dr} + {}
	\\& ~\quad \tilde {\mathbb E}\SB{\int_{(0,s]\times \mathcal R} \mathbf 1_{\{\alpha,\beta\in I^{(m)}_{r-}\}} H_{r-}^{m,h} (h(X^\alpha_{r-})^{-1}-1)\hat{\mathfrak M}(\mathrm dr, \mathrm d(\alpha,\beta))}.
\end{align}

\emph{Step 6.}
\begin{extra}
	(We will verify in Note \ref{note:419} that $P_\epsilon u_t \in C^2_{\mathrm b}(\mathbb R)$ for any $t\geq 0$ and $\epsilon > 0$.)
\end{extra}
	Fix an arbitrary $t\geq 0$. 
	After replacing the arbitrarily chosen $h \in C^2_{\mathrm b}(\mathbb R,[0,1])$ by $P_\epsilon u_t$ and then taking expectation with respect to $\mathbb P$ on both sides of \eqref{eq:StopedExpectation3}, we can verify the desired result for this proposition while applying Fubini's theorem.
	Of course, to use Fubini's theorem, we need to verify the integrability of \eqref{eq:StopedExpectation3} for each term.
	For instance, for the third term on the right hand side, we can verify, with a similar argument as in \eqref{eq:MC}, that 
 \begin{align}\label{eq:UP3}
	&\mathbf E \SB{\int_{(0,s]\times \mathcal R} \ABS{\mathbf 1_{\{\alpha,\beta\in I^{(m)}_{r-}\}} H_{r-}^{m,P_\epsilon u_t} (\RB{P_\epsilon u_t}(X^\alpha_{r-})^{-1}-1)}\hat{\mathfrak M}(\mathrm dr, \mathrm d(\alpha,\beta))}
	\\& \leq 2 \mathbf E\SB{(1+e^{K_{s}^{(m)}} ) \sum_{\alpha,\beta \in I_{[0,s]}^{(m)}: \alpha \prec \beta} L_s^{\alpha,\beta}}
	< \infty.
\end{align}
	Therefore, by Fubini's theorem and Lemma \ref{lem:ThePointMeasures}, we have
\begin{align}
	&\int \left. \tilde {\mathbb E}\SB{\int_{(0,s]\times \mathcal R} \mathbf 1_{\{\alpha,\beta\in I^{(m)}_{r-}\}} H_{r-}^{m,h} (h(X^\beta_{r-})^{-1}-1)\hat{\mathfrak M}(\mathrm dr, \mathrm d(\alpha,\beta))} \right|_{h = P_\epsilon u_t} \mathrm d\mathbb P
	\\& = \mathbf E \SB{\int_{(0,s]\times \mathcal R} \mathbf 1_{\{\alpha,\beta\in I^{(m)}_{r-}\}} H_{r-}^{m,P_\epsilon u_t} (\RB{P_\epsilon u_t}(X^\beta_{r-})^{-1}-1)\hat{\mathfrak M}(\mathrm dr, \mathrm d(\alpha,\beta))}
	\\& = \frac{1}{2}\mathbf E \left[ \int_0^s \sum_{\alpha,\beta \in I^{(m)}_{r-}:\alpha\prec \beta} (-1)^{J_{r-}^m}e^{K_{r-}^{(m)}} \RB{ \prod_{\gamma \in I_{r-}^{(m)}\setminus \{\alpha,\beta\}} \RB{P_\epsilon u_t}(X^\gamma_{r-})} \times \vphantom{\RB{\RB{P_\epsilon u_t}(X^\alpha_{r-})-\RB{P_\epsilon u_t}(X^\alpha_{r-})^2} \mathrm dL_r^{\alpha,\beta}} \right.
	\\&\qquad \qquad \qquad \qquad  \left. \vphantom{\int_0^s \sum_{i,j\in I^{(m)}_{r-}:i\prec j} (-1)^{J_{r-}^m}e^{K_{r-}^{(m)}} \RB{ \prod_{k\in I_{r-}^{(m)}\setminus \{i,j\}} \RB{P_\epsilon u_t}(X^k_{r-})} }\RB{\RB{P_\epsilon u_t}(X^\alpha_{r-})-\RB{P_\epsilon u_t}(X^\alpha_{r-})\cdot \RB{P_\epsilon u_t}(X^\beta_{r-})} \mathrm dL_r^{\alpha,\beta} \right]
	 =	\tilde \Psi_{t,s}^{\epsilon,m}.
\end{align}
	Here, we used the fact that $\mathbf 1_{\{\alpha,\beta \in I_{r-}^{(m)}\}}\mathrm dL_r^{\alpha,\beta} = \mathbf 1_{\{\alpha,\beta \in I_{r-}^{(m)}, X^{\alpha}_{r-} = X^{\beta}_{r-}\}} \mathrm d L_r^{\alpha,\beta}$ which is standard for the Brownian local times.
	Note that \eqref{eq:UP3} also implies that
\[
	\sup_{0\leq s,t\leq T} \ABS{\tilde \Psi_{t,s}^{\epsilon,m}} 
	\leq 2 \mathbf E\SB{(1+e^{K_{T}^{(m)}} ) \sum_{\alpha,\beta \in I_{[0,T]}^{(m)}: \alpha \prec \beta} L_T^{\alpha,\beta}}
	< \infty.
\]
Similar arguments are valid for all the other terms on the right hand side of  \eqref{eq:StopedExpectation3}, and we are done.
\begin{extra}
	(We will verify this in Note \ref{note:420}.)
\end{extra}
\begin{extra}
\begin{note}\label{note:419}
	Let us verify that $P_\epsilon u_t \in C^2_{\mathrm b}(\mathbb R,[0,1])$ as long as $\epsilon > 0$.
	It is obvious that $P_\epsilon u_t \in C^2(\mathbb R, [0,1])$. So we only need to calculate $\| \RB{P_\epsilon u_t}'\|_\infty$ and $\|\RB{P_\epsilon u_t}''\|_\infty$.
	Using \eqref{eq:ChangeDerivative1}, we have 
	\begin{align}
		&\ABS{\partial_x\int u_s(y) p_\epsilon(x-y) \mathrm dy} 
		= \ABS{\int u_s(y) \partial_x p_\epsilon(x-y) \mathrm dy}
		= \ABS{\int u_s(y) p_\epsilon (x-y) \cdot (-\frac{x-y}{\epsilon}) \mathrm dy}
		\\&\leq \int p_\epsilon (x-y) \cdot \frac{|x-y|}{\epsilon} \mathrm dy < \infty.
	\end{align}
	This implies that $\| \RB{P_\epsilon u_t}'\|_\infty < \infty$.
	And it is already verified in \eqref{eq:LaplaceHeat} that $\| \RB{P_\epsilon u_t}''\|_\infty < \epsilon^{-1}$.
\end{note}
\begin{note} \label{note:420}
	Let us verify the integrable conditions for the Fubini's theorems that have been used for this step. 
	\begin{itemize}
		\item 
		Using Fubini's theorem and Proposition \ref{prop:ExponentialTerm} that
		\begin{align}
			\ABS{H_r^{m,P_\epsilon u_t }}
			\leq  e^{K_r^{(m)}} \in L^1(\mathbf P), \quad r\geq 0
		\end{align}
		we know that the first and the second term on the left hand side of \eqref{eq:StopedExpectation3}, after replacing the arbitrary $h \in C^2_{\mathrm b}(\mathbb R)$ by $P_\epsilon u_t$ and then taking expectation against $\mathbb P$, become $\mathbf E\SB{H_s^{m,P_\epsilon u_t} }=\Xi_{t,s}^{\epsilon,m} $ and $\mathbf E\SB{H_0^{m,P_\epsilon u_t} }= \Xi_{t,0}^{\epsilon,m} $ respectively.
		\item 
		Using Fubini's theorem, \eqref{eq:LaplaceHeat}, and Lemma \ref{lem:exponential} that
		\begin{align}
			&\ABS{\int_0^{s}  H_{r-}^{m,P_\epsilon u_t} \sum_{\alpha \in I_{r-}^{(m)}} \frac{(\Delta P_\epsilon u_t)(X^\alpha_{r-})}{2\RB{P_\epsilon u_t}(X^\alpha_{r-})} \mathrm dr}
			\\&\leq  \epsilon^{-1} \int_0^s e^{K_{r-}^{(m)}} |I_{r-}^{(m)}| \mathrm dr \in L^1(\mathbf P)
		\end{align}
		we know that the first term on the right hand side of \eqref{eq:StopedExpectation3}, after replacing the arbitrary $h \in C^2_{\mathrm b}(\mathbb R)$ by $P_\epsilon u_t$ and then taking expectation against $\mathbb P$, becomes 
		\begin{align}
			\mathbf E\SB{\int_0^{s}  H_{r-}^{m,P_\epsilon u_t} \sum_{\alpha \in I_{r-}^{(m)}} \frac{(\Delta P_\epsilon u_t)(X^\alpha_{r-})}{2\RB{P_\epsilon u_t}(X^\alpha_{r-})} \mathrm dr} = \tilde \Lambda_{t,s}^{\epsilon,m}.
		\end{align}
		\item 
		Using Fubini's theorem and Lemma \ref{lem:exponential} that
		\begin{align}
			&\ABS{\int_0^{s}  H_{r-}^{m,P_\epsilon u_t} \sum_{\alpha \in I_{r-}^{(m)}} \frac{b^{(m)}(\RB{P_\epsilon u_t}(X^\alpha_{r-}))}{\RB{P_\epsilon u_t}(X^\alpha_{r-})}  \mathrm dr}
			\\&\leq \|b^{(m)}\|_\infty  \int_0^s e^{K_{r-}^{(m)}} |I_{r-}^{(m)}| \mathrm dr \in L^1(\mathbf P)
		\end{align}
		we know that the second term on the right hand side of \eqref{eq:StopedExpectation3}, after replacing the arbitrary $h \in C^2_{\mathrm b}(\mathbb R)$ by $P_\epsilon u_t$ and then taking expectation against $\mathbb P$, becomes 
		\begin{align}
			\mathbf E\SB{\int_0^{s}  H_{r-}^{m,P_\epsilon u_t} \sum_{\alpha \in I_{r-}^{(m)}} \frac{b^{(m)}(\RB{P_\epsilon u_t}(X^\alpha_{r-}))}{\RB{P_\epsilon u_t}(X^\alpha_{r-})}  \mathrm dr} = \tilde \Phi_{t,s}^{\epsilon,m}.
		\end{align}
	\end{itemize}
\end{note}
\end{extra}
\end{proof}

\begin{proof}[Proof of Lemma \ref{lem:Cancellation} (1)]
		\begin{extra} 
		Note that for any bounded continuous function $f$, 
		\[
		P_\epsilon f(x) = \mathbb E_x[f(B_\epsilon)] \xrightarrow[\epsilon\downarrow 0]{} f(x), \quad \text{by bounded convergence theorem.}
		\]
	\end{extra}
	The desired result follows from the continuity of $u$ and the dominated convergence theorem, immediately after noticing that the integrand in  \eqref{eq:gts} is dominated by $e^{K_s^{(m)}}$, which is integrable with respect to $\mathbf P$ by Proposition \ref{prop:ExponentialTerm}.
\end{proof}

\begin{proof}[Proof of Lemma \ref{lem:Cancellation} (2)]
	From \eqref{eq:mToInfinity}, we know that for any $r\geq 0$,
	\begin{equation} 
		\Xi_{r,0}^{0,\infty} = \mathbb E \SB{ \prod_{i =1}^n u_r(x_i)} \quad \text{and} \quad \Xi_{0,r}^{0,\infty} = \tilde{\mathbb E} \SB{ (-1)^{|\tilde J_r|} e^{K_r}\prod_{\alpha \in I_r} f(X_r^\alpha)}.
	\end{equation}
	From the bounded convergence theorem, and the fact that $r\mapsto u_r(x)$ is a continuous function bounded by $1$ for each $x\in \mathbb R$, we have that $r\mapsto \Xi_{r,0}^{0,\infty}$ is continuous. 
	
		To show that $r \mapsto \Xi_{0,r}^{0,\infty}$ is continuous, we fix an arbitrary (deterministic) $r\geq 0$ and define the event $\tilde \Omega_r \subset \tilde \Omega$ such that
		\begin{align}
			\tilde{\Omega}_{r}
			&=\CB{ \mathfrak N(\{r\}\times\mathcal U \times {\bar{\mathbb N}}) = 0} \cap \CB{\mathfrak M(\{r\}\times\mathcal R) = 0} \cap {} 
			\\& \quad \CB{\int_0^{t} |I_{s}| \mathrm ds < \infty, \forall t\geq 0} \cap \CB{J_t < \infty, \forall t\geq 0}.
		\end{align}
		From the property of the Poisson random measure, we have
\begin{equation}
	 \tilde{\mathbb E}[\mathfrak N(\{r\}\times\mathcal U \times \mathbb Z_+)]= \sum_{\alpha \in \mathcal U} \sum_{k\in \bar {\mathbb N}}\hat{\mathfrak N}(\{r\}\times\{\alpha\}\times \{k\}) = 0,
\end{equation}
and
\begin{equation}
	\tilde{\mathbb E}[\mathfrak M(\{r\}\times\mathcal R)]=\sum_{(\alpha,\beta) \in \mathcal R} \tilde{\mathbb E}[\mathfrak M(\{r\}\times\{(\alpha,\beta)\})] = \sum_{(\alpha,\beta) \in \mathcal R} \tilde{\mathbb E}[\hat{\mathfrak M}(\{r\}\times\{(\alpha,\beta)\})] = 0. 
\end{equation}
		From this and Theorem~\ref{prop:Key}, it is clear that $\tilde {\mathbb P}(\tilde{\Omega}_r) = 1$. 
		
		Firstly note that, on the event $\tilde {\Omega}_r$, since there are only finitely many branching events up to any finite time and there is no branching occurring at time $r$, there exists a (random) $\varepsilon:=\varepsilon(r) > 0$ such that there is no branching event occurring in the time interval $(r-2\varepsilon, r + 2\varepsilon)$.
		Secondly note that, on the event $\tilde {\Omega}_r$, since $t\mapsto |I_t|$ is non-increasing on $(r-2\varepsilon, r+2\varepsilon)$ and
		\[
		\int_{r-2\varepsilon}^{r+2\varepsilon} |I_t| \mathrm dt 
		< \infty,
		\]
		it must hold that $|I_{t}| < \infty$ for every $t\in (r-2\varepsilon, r+2\varepsilon)$; in particular, $|I_{r-\varepsilon}| < \infty$.
		Thirdly note that, on the event $\tilde {\Omega}_r$, since there are only finitely many coalescing events occurring in the time interval $(r-\varepsilon, r+\varepsilon)$ and non of them occurs at the time $r$, there exists a random $\tilde \varepsilon:=\tilde \varepsilon(r) > 0$ such that there is no change of the total number of particles in the time interval $(r-\tilde \varepsilon , r+\tilde \varepsilon)$. 
		Let us now take an arbitrary (deterministic) sequence $(r_n)_{n\in \mathbb N} \subset (0,r+1)$ such that $r_n \to r$ as $n\to \infty$.
		Then, it can be verified that 
		\[
		H_n
		:=\mathbf 1_{\tilde \Omega_r}(-1)^{|\tilde J_{r_n}|} e^{K_{r_n}}\prod_{\alpha \in I_{r_n}} f(X_{r_n}^\alpha) \xrightarrow[n\to\infty]{\text{a.s.}} H:=\mathbf 1_{\tilde \Omega_r}(-1)^{|\tilde J_r|} e^{K_r}\prod_{\alpha \in I_r} f(X_r^\alpha).
		\]
		From Proposition \ref{prop:ExponentialTerm}, we know that each element of $(H_n)_{n\in \mathbb N}$ 
		is dominated by $1+e^{K_{r+1}} \in L_1(\tilde {\mathbb P})$.
		Therefore, by the dominated convergence theorem, we have
		\[
		\Xi_{0,{r_n}}^{0,\infty} 
		= \tilde {\mathbb E}[H_n] \xrightarrow[n\to \infty]{}  \tilde {\mathbb E}[H] =  \Xi_{0,{r}}^{0,\infty}.
		\]
		Finally, since $r\geq 0$ and $(r_n)_{n\in \mathbb N}$ are arbitrary, we obtain the continuity of $r\mapsto \Xi_{0,{r}}^{0,\infty}$.
\end{proof}

\begin{proof}[Proof of Lemma \ref{lem:Cancellation} (3)]
	From Lemma \ref{lem:exponential} and the fact that $\NORM{\frac{\Delta}{2}P_\epsilon u_t}_{\infty} \leq \epsilon^{-1}$, we have 
\begin{align}
	&\mathbf E \SB{ \iint_{s,t\geq 0, s+t\leq T} \ABS{(-1)^{|\tilde J^{(m)}_s|} e^{K^{(m)}_s} \RB{\sum_{\alpha \in I_s^{(m)}}  \RB{\frac{\Delta}{2}P_\epsilon u_t}\RB{X^\alpha_s} \prod_{\beta\in I^{(m)}_s\setminus\{\alpha\}}  \RB{P_\epsilon u_t}\RB{X^\beta_s} }}\mathrm dt \mathrm ds} 
	\\& \leq \epsilon^{-1} \mathbf E \SB{ \iint_{s,t\geq 0, s+t\leq T} e^{K^{(m)}_s} \ABS{I_s^{(m)}} \mathrm dt \mathrm ds} 
	 \leq \epsilon^{-1} T \tilde {\mathbb E}\SB{\int_{0}^T e^{K^{(m)}_s} \ABS{I_s^{(m)}} \mathrm ds}< \infty.
\end{align}
	Now by Fubini's theorem we know that both 
	\[\int_0^T \Lambda_{T-s,s}^{\epsilon,m} \mathrm ds \quad \text{and} \quad \int_0^T \tilde \Lambda_{t,T-t}^{\epsilon,m} \mathrm dt \] 
	are equal to 
\begin{align}
	\mathbf E \SB{ \iint_{s,t\geq 0, s+t\leq T} (-1)^{|\tilde J^{(m)}_s|} e^{K^{(m)}_s} \RB{\sum_{\alpha \in I_s^{(m)}}  \RB{\frac{\Delta}{2}P_\epsilon u_t}\RB{X^\alpha_s} \prod_{\beta\in I^{(m)}_s\setminus\{\alpha\}}  \RB{P_\epsilon u_t}\RB{X^\beta_s} }\mathrm dt \mathrm ds}.
\end{align}
	The desired result now follows.
\end{proof}

\begin{proof}[Proof of Lemma \ref{lem:Cancellation} (4)]
\emph{Step 1.} For any $m\in \mathbb N$ and $s,t\geq 0$, define $\tilde \Psi^{0,m}_{t,s}$ by replacing $\epsilon$ by $0$ in \eqref{eq:RL}.
	Also, define random variables
	\begin{equation} \label{eq:Ktr}
		\mathcal K_{t,r}^{\epsilon, m} (\alpha,\beta)
		:= \mathbf 1_{\{\alpha,\beta \in I_r^{(m)}\}} (-1)^{|\tilde J^{(m)}_r|} e^{K^{(m)}_r} \RB{\sigma \circ (P_\epsilon u_{t})}(X_{r}^\alpha)^2 \RB{\prod_{\gamma \in I^{(m)}_r\setminus \{\alpha,\beta\} } \RB{P_\epsilon u_{t}}\RB{X_r^\gamma} }
	\end{equation}
	which clearly satisfies
	\begin{equation}
		\ABS{\mathcal K_{t,r}^{\epsilon, m} (\alpha,\beta)} \leq \|\sigma\|^2_\infty \mathbf 1_{\{\alpha,\beta \in I_r^{(m)}\}} e^{K^{(m)}_r}
	\end{equation}
	for every $\epsilon \geq 0$, $m\in \mathbb N$, $t,r\geq 0$ and $(\alpha,\beta) \in \mathcal R$.
	Therefore, we can verify from \eqref{lem:bug} that
	\begin{align} \label{eq:SS}
		&\int_0^T \mathbf E \SB{ \sum_{(\alpha,\beta)\in \mathcal R} \int_0^{T-t} \ABS{\mathcal K_{t,r}^{\epsilon, m} (\alpha,\beta)} \mathrm dL^{\alpha,\beta}_r }  \mathrm dt
		\\&\leq \int_0^T \mathbf E \SB{\sum_{(\alpha,\beta)\in \mathcal R} \|\sigma\|^2_\infty \int_0^{T-t} \mathbf 1_{\{\alpha,\beta \in I_r^{(m)}\}} e^{K^{(m)}_r} \mathrm dL_r^{\alpha,\beta} } \mathrm dt
		\\&\leq  \|\sigma\|^2_\infty T \mathbf E \SB{ (1+ e^{K^{(m)}_T} ) \sum_{\alpha,\beta\in I_{[0,T]}^{(m)}: \alpha \prec \beta} L_T^{\alpha,\beta} } < \infty.
	\end{align}
	Therefore, by the Fubini's theorem and dominated convergence theorem, we can verify
\begin{align}
	&\int_0^T \tilde \Psi^{\epsilon,m}_{t,T-t} \mathrm dt = \frac{1}{2}	\int_0^T \mathbf E \SB{ \sum_{\alpha,\beta\in \mathcal R} \int_0^{T-t} \mathcal K_{t,r}^{\epsilon, m} (\alpha,\beta) \mathrm dL^{\alpha,\beta}_r  } \mathrm dt
	\\& \xrightarrow[\epsilon \to 0]{} \frac{1}{2}\int_0^T \mathbf E \SB{ \sum_{\alpha,\beta\in \mathcal R} \int_0^{T-r} \mathcal K_{t,r}^{0, m} (\alpha,\beta) \mathrm dL^{\alpha,\beta}_r  } \mathrm dt = \int_0^T \tilde \Psi^{0,m}_{t,T-t} \mathrm dt.
\end{align}

\emph{Step 2.}
	It can be verified from Fubini's theorem that for any $s,t\geq 0$, $\epsilon > 0$, and $m \in \mathbb N$,
\begin{align}
	\Psi_{t,s}^{\epsilon,m}  = \mathbf E\SB{ \sum_{(\alpha,\beta) \in \mathcal R}  \int p_\epsilon(y-X^\alpha_s)p_\epsilon(y-X^\beta_s) Y_{t,s}^{\epsilon, m}(y;\alpha,\beta) \mathrm dy}
\end{align}
	where
\begin{equation} \label{eq:Yts}
	Y_{t,s}^{\epsilon, m}(y;\alpha,\beta)
	:=(-1)^{|\tilde J^{(m)}_{s}|} e^{K^{(m)}_{s}} \mathbf 1_{\{\alpha,\beta \in I_{s}^{(m)}\}} \int_0^{t}  \sigma\RB{u_r(y)}^2 \RB{   \prod_{\gamma \in I_{s}^{(m)} \setminus \{\alpha,\beta\}}\RB{P_\epsilon u_r}(X^\gamma_{s})} \mathrm dr.
\end{equation}
\begin{extra}
	(For the integrable condition used here, see Note \ref{note:430}.)
\end{extra}
	By Fubini's theorem again, and by substituting $y$ with $y+X^\alpha_s$, we have 
\begin{align}
	&\int_0^T  \Psi_{T-s,s}^{\epsilon,m} \mathrm ds
	= \mathbf E\SB{\sum_{(\alpha,\beta) \in \mathcal R} \int \mathrm dy \int_0^T p_\epsilon(y)p_\epsilon(y+X^\alpha_s-X^\beta_s) Y_{T-s,s}^{\epsilon, m}(y+X^\alpha_s;\alpha,\beta)  \mathrm ds }.
\end{align}
\begin{extra}
	(For the integrable condition used here, see Note \ref{note:431}.)
\end{extra}

\begin{extra}
	\begin{note}\label{note:430}
		Notice that 
		\begin{align}
			&\mathbf E\Bigg\{ \sum_{(i,j) \in \mathcal R}  
				\int_0^t\int 
			\\&  \ABS{(-1)^{|\tilde J^{(m)}_s|} e^{K^{(m)}_s}\sigma\RB{u_r(y)}^2 \RB{\mathbf 1_{\{i,j \in I_s^{(m)}: i\prec j\}} p_\epsilon(y-X^i_s)p_\epsilon(y-X^j_s) \prod_{k \in I_s^{(m)} \setminus \{i,j\}}\RB{P_\epsilon u_r}(X^k_s)} } 
			\\& \qquad \qquad \qquad \qquad \qquad \qquad \qquad \qquad \qquad \qquad \qquad \qquad \qquad \qquad \qquad \qquad \qquad  \mathrm dy \mathrm dr\Bigg\}
			\\& \leq \|\sigma\|_\infty \|p_\epsilon\|_\infty \mathbf E\SB{
				\int_0^t  e^{K^{(m)}_s} \ABS{I_s^{(m)}}^2 \mathrm dr } < \infty
		\end{align}
		by Lemma \ref{lem:exponential}. 
	\end{note}
\end{extra}

\begin{extra}\begin{note}
		\label{note:431}
		Notice that 
		\begin{align}
			&\int_0^T \mathbf E\SB{ \sum_{(i,j) \in \mathcal R}  \int \ABS{p_\epsilon(y-X^i_s)p_\epsilon(y-X^j_s) Y_{T-s,s}^{\epsilon, m}(y;i,j)} \mathrm dy} \mathrm ds
			\\ &\leq T \|\sigma\|^2_\infty \|p_\epsilon\|_\infty \int_0^T \mathbf E
			\SB{
				e^{K^{(m)}_s} |I_s^{(m)}|^2
			} \mathrm ds < \infty
		\end{align}
				by Lemma \ref{lem:exponential}. 
	\end{note}
\end{extra}

\emph{Step 3.}
	Recall from \eqref{eq:LocalTime} that $L^{\alpha,\beta}_{\cdot,z}$ is the local time of the process $\tilde X^\alpha_\cdot - \tilde X^\beta_\cdot$ at the level $z\in \mathbb R$.
	By the theorem  of the occupation density, c.f. \cite[Theorem 29.5]{MR4226142} and \cite[Lemma 2]{MR1813840}, we can verify that for each $(\alpha,\beta) \in \mathcal R$, $\epsilon > 0$ and $y\in \mathbb R$, almost surely 
\begin{align} \label{eq:NeedExplain}
	&2\int_0^T p_\epsilon(y)p_\epsilon(y+X^\alpha_s-X^\beta_s) Y_{T-s,s}^{\epsilon, m}(y+X^\alpha_s;\alpha,\beta)  \mathrm ds 
	\\&= \int_0^T p_\epsilon(y)p_\epsilon(y+\tilde X^\alpha_s-\tilde X^\beta_s) Y_{T-s,s}^{\epsilon, m}(y+X^\alpha_s;\alpha,\beta)  \mathrm d \AB{ \tilde X^\alpha_s-\tilde X^\beta_s}
	\\& =  \int \mathrm dz \int_0^T p_\epsilon(y)p_\epsilon(y+z) Y_{T-s,s}^{\epsilon, m}(y+X^\alpha_s;\alpha,\beta)  \mathrm d L_{s,z}^{\alpha,\beta}.
\end{align}
\begin{extra}
	(We will explain \eqref{eq:NeedExplain} with details in Note \ref{note:433}.)
\end{extra}
Using the dominated convergence theorem, we can verify that the expression in \eqref{eq:NeedExplain} is almost surely continuous in $y \in \mathbb R$.
\begin{extra}
	(We will explain this continuity with details in Note \ref{note:434}.)
\end{extra}
Therefore, \eqref{eq:NeedExplain}  actually holds for every $y \in \mathbb R$ and $(\alpha,\beta) \in \mathcal R$, almost surely, for every $\epsilon > 0$.
Therefore, we have
\begin{align}
	&2\int_0^T  \Psi_{T-s,s}^{\epsilon,m} \mathrm ds
	= \mathbf E\SB{\sum_{(\alpha,\beta) \in \mathcal R} \int \mathrm dy \int \mathrm dz \int_0^T p_\epsilon(y)p_\epsilon(y+z) Y_{T-s,s}^{\epsilon, m}(y+X^\alpha_s;\alpha,\beta)  \mathrm d L_{s,z}^{\alpha,\beta}}.
\end{align}

\begin{extra}
	\begin{note} \label{note:433}
		Let us give some explanation for \eqref{eq:NeedExplain} while fixing $(i,j) \in \mathcal R$, $\epsilon > 0$ and $y\in \mathbb R$.
		For $(i,j) \in \mathcal R$ with 
		$i = (i_1, \cdots, i_n)$ and $j = (j_1, \cdots, j_m)$, firstly define $i_{n'} =0 $ and $j_{m'} =0$ for $n' > n$ and $m' > m$, and then, under the condition that $i_1=j_1$ but $i \neq j$, define their common ancestor $i\wedge j := (i_1, \cdots, i_k)$ with $k$ being the unique number satisfying $(i_1,\cdots, i_k) = (j_1,\cdots, j_k)$ but $i_{k+1}\neq j_{k+1}$.
		If $i_1 \neq j_1$, then we write $i\wedge j =\varnothing$ and write $\zeta_{\varnothing,\varnothing} =0$.
		By \eqref{eq:XAR}, we can verify that, for $i, j\in \mathcal R$,  
		\[\mathrm d \AB{ \tilde X^i_s, \tilde X^j_s }= \mathbf 1_{(0, \zeta_{i\wedge j, i\wedge j}]} \mathrm d s\]
		From this, we can verify that, for $(i, j)\in \mathcal R$,  
		\begin{align}
			\mathrm d\AB{\tilde X^i_s-\tilde X^j_s} = \mathrm d\AB{\tilde X^i_s} - 2\mathrm d\AB{\tilde X^i_s, \tilde X^j_s} + \mathrm d\AB{\tilde X^j_s}
			= 2 \mathbf 1_{(\zeta_{i\wedge j, i\wedge j},\infty)}\mathrm ds.
		\end{align}
		Now, it is clear that  for $(i, j)\in \mathcal R$, 
		\begin{align}
			\mathbf 1_{\{i,j \in I_{s-}^{(m)}\}}\mathrm d\AB{\tilde X^i_s-\tilde X^j_s}
			= 2 \mathbf 1_{\{i,j \in I_{s-}^{(m)}\}} \mathrm ds.
		\end{align}
		This explains the second line of \eqref{eq:NeedExplain}. 
		For the third line, we work with the filtration $\{\tilde{\mathscr G}_s = \tilde {\mathscr F}_s \otimes \mathscr F: s\geq 0\}$. 
		Notice that $s\mapsto Y_{T-s,s}^{\epsilon, m}(y+X^i_s;i,j)$ is a $(\tilde {\mathscr G}_s)$-predictable process. 
		Now the third line follows from the fact that
		\begin{align}
			&\int_0^T \ABS{p_\epsilon(y)p_\epsilon(y+\tilde X^i_s-\tilde X^j_s) Y_{T-s,s}^{\epsilon, m}(y+X^i_s;i,j)}  \mathrm d \AB{ \tilde X^i_s-\tilde X^j_s}
			\\&\leq 2 T^2 \|\sigma\|_\infty \|p_\epsilon\|_\infty^2 e^{K^m_T} < \infty, \quad \text{a.s.}
		\end{align}
		and the theorem of the occupation density, which can be summarized as follows.
	\end{note}
	
	\begin{lemma} \label{lem:OD}
		Let $X$ be a continuous semi-martingale with regularized local time $L$ (in the sense of \cite[Theorem 29.4]{MR4226142}).
		For any bounded non-negative measurable function $f: \mathbb R \to \mathbb R_+$ and any predictable process $(Y_s)_{s\geq 0}$ such that $Y$ either is non-negative or satisfies
		\begin{equation} \label{eq:IOTI}
			\int_0^t f(X_s) |Y_s| \mathrm d\AB{X_s} < \infty, \quad t\geq 0, \text{a.s.}
		\end{equation}
		it holds that
		\begin{equation} \label{eq:OD}
			\int_0^t f(X_s) Y_s \mathrm d\AB{X_s} = \int \mathrm dz \int_0^t f(z)Y_s \mathrm dL_{s,z}, \quad t\geq 0, \text{a.s.}
		\end{equation}
	\end{lemma} 
	\begin{proof}
		Fix the function $f$.
		Denote by $\mathscr H$ the collection of bounded measurable process $Y$ such that \eqref{eq:OD} holds.
		Firstly note that $\mathscr H$ contains the constant function \cite[Theorem 29.5]{MR4226142}.
		Secondly, it is easy to verify that $\mathscr H$ is a vector space over $\mathbb R$.
		Thirdly, note by the monotone convergence theorem that
		\[
		(Y^n)_{n\in \mathbb N} \subset \mathscr H,  0 \leq Y^1 \leq \cdots \leq Y^n \uparrow Y, Y \text{~bounded~} \implies Y \in \mathscr H. 
		\]
		Therefore $\mathscr H$ is an MVS in the sense of \cite[Appendix A0.]{Sharpe1988General}.
		Now suppose that $Y \in \mathscr K$, the space of bounded predictable step processes i.e. 
		\[
		Y_s = K_0 \mathbf 1_{\{0\}}(s) + \sum_{k=1}^n K_{a_k} 	\mathbf 1_{(a_k, b_k]}(s), \quad s\geq 0,
		\]
		for some (deterministic) $a_0:=0\leq a_1 < b_1 \leq a_2 < b_2 \leq \cdots \leq a_n < b_n < \infty$ and, for each $k\in \{0,\dots n\}$, $\tilde{\mathscr F}_{a_k}$-measurable bounded $\mathbb R$-valued random variable $K_{a_k}$.
		We can verify that almost surely for every $t\geq 0$, 
		\begin{align}
			&\int_0^t f(X_s) Y_s \mathrm d\AB{X_s} = 
			\sum_{k=1}^n K_{a_k}  \int_{0}^{t} f(X_s) \mathbf 1_{(a_k, b_k]}(s) \mathrm d\AB{X_s} 
			\\&= \sum_{k=1}^n K_{a_k} \int \mathrm dz \int_{0}^{t} f(z) \mathbf 1_{(a_k, b_k]}(s) \mathrm d L_{s,z}
			=   \int \mathrm dz \int_{0}^{t} f(z) Y_s  \mathrm d L_{s,z}.
		\end{align}
		Therefore, by the monotone class theorem (\cite[Theorem A0.6]{Sharpe1988General}), we have $\mathrm b \sigma(\mathscr K) \subset \mathscr H$.
		Also from the fact that for any left continuous process $Y(\omega,s)$, there exists a sequence of $(Y^n)_{n\in \mathbb N} \subset \mathscr K$ such that $Y^n(w,s) \xrightarrow[]{} Y(w,s)$ for every $s\geq 0$ and $\omega \in \tilde \Omega$, we can verify that $\sigma (\mathscr K)$ is the predictable $\sigma$-field over $\Omega \times \mathbb R_+$.
		
		Now we have shown that \eqref{eq:OD} holds for every bounded predictable process $Y$. 
		Therefore \eqref{eq:OD} also holds for non-negative predictable process $Y$, by using the the approximation $Y\wedge n$ and the monotone convergence theorem.
		Finally, for any predictable process $Y$ satisfying the integrable condition \eqref{eq:IOTI}, by using the decomposition $Y =Y^+ - Y^-$, we can verify the desired result. 
	\end{proof}
\end{extra}

\begin{extra}
	\begin{note}
		\label{note:434}
		To see the desired continuity, it is sufficient to verify that 
		\begin{align}
			&\int_0^T \sup_{y\in \mathbb R} \ABS{p_\epsilon(y)p_\epsilon(y+X^i_s-X^j_s) Y_{T-s,s}^{\epsilon, m}(y+X^i_s;i,j) } \mathrm ds 
			\\&\leq T^2 \|\sigma\|_\infty \|p_\epsilon\|_\infty^2 e^{K^m_T} < \infty, \quad \text{a.s.}
		\end{align}
		and
		\begin{align}
			& \int \mathrm dz \int_0^T \sup_{y\in \mathbb R} \ABS{p_\epsilon(y)p_\epsilon(y+z) Y_{T-s,s}^{\epsilon, m}(y+X^i_s;i,j)}  \mathrm d L_{s,z}^{i,j}
			\\& \leq T \|\sigma\|_\infty \|p_\epsilon\|_\infty^2 e^{K^m_T} \int \mathrm dz \int_0^T   \mathrm d L_{s,z}^{i,j}
			= T \|\sigma\|_\infty \|p_\epsilon\|_\infty^2 e^{K^m_T} \AB{\tilde X^i_T - \tilde X^j_T} < \infty, \quad \text{a.s.}
		\end{align}
	\end{note}
\end{extra}

\emph{Step 4.}
	By an argument similar to \cite[p.~1725]{MR1813840}, we can verify that almost surely
\begin{align}
	& \lim_{\epsilon \downarrow 0} \sum_{(\alpha,\beta)\in \mathcal R}\int \mathrm dy \int \mathrm dz \int_0^T p_\epsilon(y)p_\epsilon(y+z) Y_{T-s,s}^{\epsilon, m}(y+X^\alpha_s;\alpha,\beta)  \mathrm d L_{s,z}^{\alpha,\beta} 
	\\& \label{eq:DR}= \sum_{(\alpha,\beta)\in \mathcal R}\int_0^T Y_{T-s,s}^{0, m}(X^\alpha_s;\alpha,\beta)  \mathrm d L_{s}^{\alpha,\beta}.
\end{align}
\begin{extra} (We will provide more details for \eqref{eq:DR} in Note \ref{note:444}.) \end{extra}
	Notice also that almost surely
\begin{align}
	\int_0^T \ABS{Y_{T-s,s}^{\epsilon, m}(y+X^\alpha_s;\alpha,\beta)} \mathrm d L_{s,z}^{\alpha,\beta} 
	\leq (1+ e^{K_T^{(m)}}) T \|\sigma\|^2_\infty \mathbf 1_{\{\alpha,\beta \in I_{[0,T]}^{(m)}\}} \sup_{z_0\in \mathbb R} L^{\alpha,\beta}_{T,z_0},
\end{align}
	and therefore
\begin{align}
	&\sum_{(\alpha,\beta)\in \mathcal R}\ABS{ \int \mathrm dy \int \mathrm dz \int_0^T p_\epsilon(y)p_\epsilon(y+z) Y_{T-s,s}^{\epsilon, m}(y+X^\alpha_s;\alpha,\beta)  \mathrm d L_{s,z}^{\alpha,\beta} }
	\\ &\leq (1+ e^{K_T^{(m)}}) T \|\sigma\|^2_\infty \sum_{\alpha,\beta \in I_{[0,T]}^{(m)}: \alpha \prec \beta } \sup_{z_0\in \mathbb R} L^{\alpha,\beta}_{T,z_0}
\end{align}
which is integrable, thanks to Lemma \ref{lem:bug}.
Now, by Steps 2 and 3, the dominated convergence theorem, and Fubini's theorem, we can verify that
\begin{align}\label{eq:last}
	&2\int_0^T  \Psi_{T-s,s}^{\epsilon,m} \mathrm ds 
	\xrightarrow[\epsilon \downarrow 0]{}\mathbf E\SB{\sum_{(\alpha,\beta) \in \mathcal R}  \int_0^T Y_{T-s,s}^{0, m}(X^\alpha_s;\alpha,\beta)  \mathrm d L_{s}^{\alpha,\beta}}
	= 2\int_0^T \tilde \Psi_{t,T-t}^{0, m} \mathrm dt.
\end{align}
\begin{extra} (We verify the equality using Fubini's theorem in Note \ref{note:446}.) \end{extra}
Combining this with Step 1, we are done.
\end{proof}
\begin{extra}
\begin{note} \label{note:444}
	Notice that, the summation in \eqref{eq:DR} is in fact a finite sum, over indices $(i,j)\in (I_{[0,T]}^{(m)})^2\cap \mathcal R$. 
	Therefore, we only have to show that almost surely, for every $(i,j) \in \mathcal R$,
\begin{align}
	& \lim_{\epsilon \downarrow 0} \int \mathrm dy \int \mathrm dz \int_0^T p_\epsilon(y)p_\epsilon(y+z) Y_{T-s,s}^{\epsilon, m}(y+X^i_s;i,j)  \mathrm d L_{s,z}^{i,j} 
	\\&\label{eq:RDD}= \int_0^T Y_{T-s,s}^{0, m}(X^i_s;i,j)  \mathrm d L_{s}^{i,j}.
\end{align}
	Fixing $(i,j)\in \mathcal R$, $T>0$ and $m\in \mathbb N$, let us define random variables
	\begin{equation} \label{eq:G}
	G(\epsilon, y,z) := \int_0^T Y_{T-s,s}^{\epsilon, m}(y+X^i_s;i,j)  \mathrm d L_{s,z}^{i,j}, \quad \epsilon \geq 0, y, z\in \mathbb R
	\end{equation}
	where, recall from \eqref{eq:Yts} that
	\[
	Y_{t,s}^{\epsilon, m}(y;i,j)
	=(-1)^{|\tilde J^{(m)}_{s-}|} e^{K^{(m)}_{s-}} \mathbf 1_{\{i,j \in I_{s-}^{(m)}\}} \int_0^{t}  \sigma\RB{u_r(y)}^2 \RB{   \prod_{k \in I_{s-}^{(m)} \setminus \{i,j\}}\RB{P_\epsilon u_r}(X^k_{s-})} \mathrm dr.
	\]
	Using the fact that the truncated particle system won't explode, and the continuity of the Brownian paths, the random field $u$, and the function $\sigma$, we can define random variables
	\[
	\mathbf r := \sup_{s\in [0,T], i\in I_{[0,T]}^{(m)}} |\tilde X^i_s| < \infty,
	\]
	\[
	\chi_{y} := \sup_{r\in [0,T]}\sup_{x\in[-\mathbf r, \mathbf r] }|\sigma(u_r(y+x))-\sigma(u_r(x)) |< \infty, \quad y \in \mathbb R,
	\]
	and
	\[
	\kappa_{\epsilon} =  \sup_{r\in [0,T]}\sup_{x\in[-\mathbf r, \mathbf r] }\ABS{(P_\epsilon u_r)(x)-u_r(x)}< \infty, \quad \epsilon \geq 0.
	\]
	One key observation is that almost surely $\chi_y \to \chi_0 = 0$ when $y\to 0$.
	In fact, suppose for the sake of contradiction that with positive probability, there exists a sequence $y_n \to 0$ such that $\limsup_{n\to \infty}\chi_{y_n} > 0$;
	then on this event, there exists a sequence $y'_n\to 0$ such that $\inf_{n\in \mathbb N}\chi_{y'_n} > 0$; 
	and therefore on this event, there exists sequences $y'_n\to 0$, $(r_n) \in [0,T]$ and $(x_n) \in [-\mathbf r, \mathbf r]$ such that $\inf_{n\in \mathbb N}|\sigma(u_{r_n}(y'_n+x_n))-\sigma(u_{r_n}(x_n)) | > 0$;
	and therefore on this event, there exists sequences $y''_n\to 0$, $r'_n \to r$ in $[0,T]$ and $x'_n \to x$ in $[-\mathbf r, \mathbf r]$ such that 
	\begin{align}
		&0 < \inf_{n\in \mathbb N} |\sigma(u_{r'_n}(y''_n+x'_n))-\sigma(u_{r'_n}(x'_n))| 
		\\&\leq \lim_{n\to \infty}  |\sigma(u_{r'_n}(y''_n+x'_n))-\sigma(u_{r'_n}(x'_n))| = 0,
	\end{align}
	which is impossible.
	Using a exactly similar argument, we have that almost surely $\kappa_\epsilon \to \kappa_0 = 0$ when $\epsilon \downarrow 0$.
	Now we verify that almost surely on the event $\{i,j\in I^{(m)}_{[0,T]}\}$, for $\epsilon \geq 0$, $y\in \mathbb R$ and $z\in \mathbb R$,
	\begin{align}
		&\ABS{G(\epsilon, y, z) - G(\epsilon, 0 , z)} \leq \int_0^T \ABS{Y_{T-s,s}^{\epsilon, m}(y+X^i_s;i,j)  - Y_{T-s,s}^{\epsilon, m}(X^i_s;i,j) }\mathrm d L_{s,z}^{i,j}
		\\& \leq \int_0^T e^{K^{(m)}_{T-}} \int_0^{T}   \ABS{   \sigma\RB{u_r(y+\tilde X^i_s)} - \sigma\RB{u_r(\tilde X^i_s)}  }\mathrm dr \mathrm d L_{s,z}^{i,j}
		\\& \leq e^{K^{(m)}_{T-}} \chi_{y}T \sup_{z_0\in \mathbb R}L_{T,z_0}^{i,j}.
	\end{align}
	We can also verify that almost surely on the event $\{i,j\in I^{(m)}_{[0,T]}\}$, for every $\epsilon \geq 0$ and $z\in \mathbb R$, 
	\begin{align}
		&\ABS{G(\epsilon, 0, z) - G(0, 0 , z)} 
		\leq \int_0^T \ABS{Y_{T-s,s}^{\epsilon, m}(X^i_s;i,j) - Y_{T-s,s}^{0, m}(X^i_s;i,j)}  \mathrm dL_{s,z}^{i,j} 
		\\& 	\leq e^{K^{(m)}_{T-}} \|\sigma\|_\infty \int_0^T \int_0^{T} \mathbf 1_{\{i,j \in I_{s-}^{(m)}\}}  \ABS{   \prod_{k \in I_{s-}^{(m)} \setminus \{i,j\}}\RB{P_\epsilon u_r}(X^k_{s-})-\prod_{k \in I_{s-}^{(m)} \setminus \{i,j\}}u_r(X^k_{s-})} \mathrm dr \mathrm dL_{s,z}^{i,j} 
		\\& 	\leq e^{K^{(m)}_{T-}} \|\sigma\|_\infty \int_0^T \int_0^{T} \mathbf 1_{\{i,j \in I_{s-}^{(m)}\}}  \sum_{k \in I_{s-}^{(m)} \setminus \{i,j\}} \ABS{   \RB{P_\epsilon u_r}(X^k_{s-}) - u_r(X^k_{s-})} \mathrm dr \mathrm dL_{s,z}^{i,j} 
		\\& 	\leq T\kappa_\epsilon  e^{K^{(m)}_{T-}} \|\sigma\|_\infty  \RB{\sup_{s\in [0,T]}|I_s^{(m)}|} \RB{ \sup_{z_0\in \mathbb R} L^{i,j}_{T,z_0}}.
	\end{align}
	Here we used the inequality that for any $(a_i)_{i=1}^n, (b_i)_{i=1}^n \subset [0,1]$ we have
	\[
	\ABS{\prod_{i=1}^n a_i - \prod_{i=1}^n b_i} \leq \sum_{i=1}^n \ABS{a_i - b_i},
	\]
	which can be verified by induction that
	\begin{align}
		&\ABS{\prod_{i=1}^{n+1} a_i - \prod_{i=1}^{n+1} b_i} 
		\leq \ABS{a_{n+1}\prod_{i=1}^{n} a_i - b_{n+1}\prod_{i=1}^{n} a_i} + \ABS{b_{n+1}\prod_{i=1}^{n} a_i- b_{n+1} \prod_{i=1}^{n} b_i} 
		\\&\leq \ABS{a_{n+1} - b_{n+1}}\prod_{i=1}^{n} a_i + b_{n+1}\ABS{\prod_{i=1}^{n} a_i- \prod_{i=1}^{n} b_i} 
		\leq \sum_{i=1}^{n+1}\ABS{a_i - b_i}.
	\end{align}
	Next, let us notice that
	\[
	Y_s := Y_{T-s,s}^{0, m}(X^i_s;i,j), \quad s\in [0,T]
	\]
	is LCRL process with only finitely many jumps up to time $T$. 
	If we denote by $\tau_0:= 0< \tau_1 < \cdots < \tau_{l-1} < T=: \tau_l$ its jumping times, then for each $k \in \{1, \cdots, l\}$, $s\mapsto Y_s$ is continuous, and therefore uniformly continuous by the Heine–Cantor theorem, on $(\tau_{k-1}, \tau_k]$.
	Therefore, we can verify that almost surely, for any $\iota \in \mathbb N$, there exists a (random) partition $a^{(\iota)}_0:= 0< a^{(\iota)}_1 < \cdots < a^{(\iota)}_{N_\iota-1} < T=: a^{(\iota)}_{N_\iota}$, such that 
	\begin{align}
		\sup_{s\in [0, T]} \ABS{Y_s - Y_s^{(\iota)}} \leq \frac{1}{\iota}
	\end{align}
	where $Y^{(\iota)}$ are elementary functions given by 
	\[
	Y^{(\iota)}_s := Y_0 \mathbf 1_{\{0\}}(s)+ \sum_{k=1}^{N_\iota} Y_{a^{(\iota)}_{k-1}} \mathbf 1_{(a^{(\iota)}_{k-1}, a^{(\iota)}_k]}(s), \quad s \in [0, T].
	\]
	Now, we can verify that almost surely for any $z\in \mathbb R$ and $\iota\in \mathbb N$, 
	\begin{align}
		&\ABS{G(0,0,z) - G(0,0,0)} = \ABS{\int_0^T Y_s \mathrm dL_{s,z}^{i,j} - \int_0^T Y_s \mathrm dL_{s,0}^{i,j}}
		\\&\leq  \ABS{\int_0^T Y_s \mathrm dL_{s,z}^{i,j} - \int_0^T Y^{(\iota)}_s \mathrm dL_{s,z}^{i,j}} + \ABS{\int_0^T Y^{(\iota)}_s \mathrm dL_{s,z}^{i,j} - \int_0^T Y^{(\iota)}_s \mathrm dL_{s,0}^{i,j}} 
		\\& \quad + \ABS{\int_0^T Y^{(\iota)}_s \mathrm dL_{s,0}^{i,j} - \int_0^T Y_s \mathrm dL_{s,0}^{i,j}} 
		\\& \leq \frac{2}{\iota} \RB{\sup_{z_0\in \mathbb R} L^{i,j}_{T,z_0}} + \ABS{\sum_{k=1}^{N_\iota} Y_{a^{(\iota)}_{k-1}} \RB{L^{i,j}_{a^{(\iota)}_k,z} - L^{i,j}_{a^{(\iota)}_{k-1},z}} - \sum_{k=1}^{N_\iota} Y_{a^{(\iota)}_{k-1}} \RB{L^{i,j}_{a^{(\iota)}_k,0} - L^{i,j}_{a^{(\iota)}_{k-1},0}}}.
	\end{align}
	Therefore, by taking $z\to 0$ first and then $\iota \to \infty$, we can verify that almost surely
	\begin{equation}
		\lim_{z\to 0} \ABS{G(0,0,z) - G(0,0,0)} = 0. 
	\end{equation}
	To summarize, we have that almost surely on the event $\{i,j\in I^{(m)}_{[0,T]}\}$ when $(\epsilon, y, z) \to (0,0,0)$,
	\begin{align}
		&\ABS{G(\epsilon, y, z) - G(0, 0, 0) }
		\\&\leq \ABS{G(\epsilon, y, z) - G(\epsilon, 0, z)}+\ABS{G(\epsilon, 0, z) - G(0, 0, z)}+ \ABS{G(0, 0, z) - G(0, 0, 0)}
		\\& \leq e^{K^{(m)}_{T-}} \chi_{y}T \RB{\sup_{z_0\in \mathbb R}L_{T,z_0}^{i,j}} + T\kappa_{\epsilon}  e^{K^{(m)}_{T-}} \|\sigma\|_\infty  \RB{\sup_{s\in [0,T]}|I_s^{(m)}|} \RB{ \sup_{z_0\in \mathbb R} L^{i,j}_{T,z_0}}+ {}
		\\& \qquad \ABS{G(0, 0, z) - G(0, 0, 0)} 
		\\& \xrightarrow{} 0.
	\end{align}
	Notice that on the compliment of the event $\{i,j\in I^{(m)}_{[0,T]}\}$, we always have $G(\epsilon, y, z) = G(0, 0, 0) = 0$.
	In other word, we have shown that almost surely $(\epsilon, y, z)\mapsto G(\epsilon, y, z)$ is continuous at $(\epsilon, y, z) = (0,0,0)$.
	Also note that, almost surely, for any $\epsilon \geq 0, y \in \mathbb R$ and $z\in \mathbb R$,
	\begin{align}
		&\ABS{G(\epsilon, y, z) } 
		\leq \int_0^T \ABS{Y_{T-s,s}^{\epsilon, m}(y+X^i_s;i,j) } \mathrm d L_{s,z}^{i,j}
		\\&\leq \int_0^T \ABS{(-1)^{|\tilde J^{(m)}_{s-}|} e^{K^{(m)}_{s-}} \mathbf 1_{\{i,j \in I_{s-}^{(m)}\}} \int_0^{T-s}  \sigma\RB{u_r(y+X^i_s)} \RB{   \prod_{k \in I_{s-}^{(m)} \setminus \{i,j\}}\RB{P_\epsilon u_r}(X^k_{s-})} \mathrm dr} \mathrm d L_{s,z}^{i,j}
		\\& \leq (1+e^{K^{(m)}_{T}}) T \|\sigma\|_\infty \RB{\sup_{z_0\in \mathbb R} L^{i,j}_{T,z_0}}  < \infty.
	\end{align}
	From these, we can verify the desired result \eqref{eq:RDD} using the following lemma.
\end{note}
	\begin{lemma}
		Suppose that $G: (\epsilon, y, z) \mapsto G(\epsilon, y, z)$ is a bounded function on $\mathbb R_+\times \mathbb R^2$, and is continuous at $(\epsilon, y, z) = (0,0,0)$.
		Then
		\begin{align}
			\int \mathrm dy \int p_\epsilon(y)p_\epsilon(y+z) G(\epsilon, y, z)\mathrm dz \xrightarrow[\epsilon \downarrow 0]{} G(0,0,0).
		\end{align}
\end{lemma}
\begin{proof}
	Suppose that $b^1$ and $b^2$ are two independent standard Brownian motions initiated at position $0$ under the probability $\Pi$. 
	Then it holds that
	\begin{align}
			&\int \mathrm dy \int p_\epsilon(y)p_\epsilon(y+z) G(\epsilon, y, z)\mathrm dz  = 	\int \mathrm dy \int p_\epsilon(y)p_\epsilon(z) G(\epsilon, y, z-y)\mathrm dz
			\\&=\Pi\SB{G(\epsilon, b^1_\epsilon, b^2_\epsilon-b^1_\epsilon)} \xrightarrow[\epsilon \downarrow 0]{} G(0,0,0)
	\end{align}
	by the bounded convergence theorem.
\end{proof}
\begin{note}
	\label{note:446}
	From \eqref{eq:SS} and Fubini's theorem
	\begin{align}
		&2\int_0^T \tilde \Psi^{0,m}_{t,T-t} \mathrm dt 
		= \int_0^T \mathbf E \SB{ \sum_{(i,j)\in \mathcal R} \int_0^{T-t} \mathcal K_{t,r}^{0, m} (i,j) \mathrm dL^{i,j}_r  } \mathrm dt
		\\&=  \mathbf E \SB{ \sum_{(i,j)\in \mathcal R}  \int_0^T \int_0^T \mathbf 1_{\{r+t\leq T\}}  \mathcal K_{t,r}^{0, m} (i,j)\mathrm dt \mathrm dL^{i,j}_r    } 
		\\&=  \mathbf E \SB{ \sum_{(i,j)\in \mathcal R}  \int_0^T \int_0^{T-s}  \mathcal K_{t,s}^{0, m} (i,j)\mathrm dt \mathrm dL^{i,j}_s    }.
	\end{align}
	Notice from their definition \eqref{eq:Ktr} and  \eqref{eq:Yts} that
	\begin{align}
		\int_0^{T-s}  \mathcal K_{t,s}^{0, m} (i,j)\mathrm dt = Y_{T-s,s}^{0, m}(X^i_s;i,j).
	\end{align}
\end{note}
\end{extra}

\begin{proof}[Proof of Lemma \ref{lem:Cancellation} (5)]
\emph{Step 1.}
	For every $t,s\geq 0$ and $m\in \mathbb N$, define $\Phi_{t,s}^{0,m}$ and $\tilde \Phi_{t,s}^{0,m}$ by taking $\epsilon = 0$ in \eqref{eq:Phi1} and \eqref{eq:Phi2} respectively.
	For every $t,s\geq 0$, define $\Phi_{t,s}^{0,\infty}$ and $\tilde \Phi_{t,s}^{0,\infty}$ by 
	\begin{align}
		\Phi_{t,s}^{0,\infty}  := \mathbf E \SB{ (-1)^{|\tilde J_s|} e^{K_s} \int_0^t \RB{\sum_{\alpha \in I_s}   \RB{b\circ u_r} \RB{X^\alpha_s} \prod_{\beta\in I_s\setminus\{\alpha\}}  u_r\RB{X^\beta_s} }\mathrm dr}
	\end{align}
	and 
	\begin{align}
	 \tilde \Phi_{t,s}^{0,\infty}  := \mathbf E \SB{ \int_0^s (-1)^{|\tilde J_r|} e^{K_r} \RB{\sum_{\alpha \in I_r}  (b \circ u_t) \RB{X^\alpha_r} \prod_{\beta \in I_r\setminus\{\alpha\}}   u_t\RB{X^\beta_r} }\mathrm dr}.
	\end{align}
	
\emph{Step 2.}
	We will show that, for a fixed arbitrary $T>0$ and $m\in \mathbb N$ with $m\geq 2$,
	\begin{align}
	\lim_{\epsilon \downarrow 0} \int_0^T \tilde \Phi_{t,T-t}^{\epsilon,m} \mathrm dt
		= \int_0^T \tilde \Phi_{t,T-t}^{0,m} \mathrm dt.
	\end{align}
	Firstly, we note that almost surely with respect to $\mathbf P$, for any $t,s\geq 0$ and $\epsilon \geq 0$, the random variable
	\begin{align}
		\ABS{\mathbf 1_{\{s+t \leq T\}} (-1)^{|\tilde J^{(m)}_s|} e^{K^{(m)}_s} \mathbf 1_{\{\alpha \in I_s^{(m)}\}}  \RB{  b^{(m)}\circ \RB{P_\epsilon u_t}}\RB{X^\alpha_s} \prod_{\beta\in I^{(m)}_s\setminus\{\alpha\}}  \RB{P_\epsilon u_t}\RB{X^\beta_s}}
	\end{align}
	is bounded by $\C\label{const:b}\mathbf 1_{\{\alpha \in I_s\}}  (1+e^{K_s})$ where $
		\Cr{const:b} := \sum_{k\in \bar {\mathbb N}} |b_k|+1 < \infty$.
	\begin{extra}
	(We will explain this in Note \ref{note:077}.)
	\end{extra}
	Secondly notice that, by Lemma \ref{lem:exponential}, 
	\begin{align}
		&\int_0^T \mathbf E \SB{ \sum_{\alpha \in \mathcal U}  \int_0^T \mathbf 1_{\{\alpha \in I_s\}}  \RB{1+e^{K_s}} \mathrm dt} \mathrm ds 
		= T \int_0^T \mathbf E \SB{ |I_s|  \RB{1+e^{K_s}} } \mathrm ds < \infty.
	\end{align}
	Thirdly notice that, almost surely with respect to $\mathbf P$, for any $t,s\geq 0$ and $\alpha \in \mathcal U$, the random variables 
\[
	\mathbf 1_{\{\alpha \in I_s^{(m)}\}}\RB{ P_\epsilon u_t}(X_s^\alpha) 
	\text{~and~} \mathbf 1_{\{\alpha \in I_s^{(m)}\}}\RB{ b^{(m)} \circ \RB{P_\epsilon u_t}}(X_s^\alpha)
\] 
	converge, as $\epsilon \downarrow 0$, to  
\[
	\mathbf 1_{\{\alpha \in I_s^{(m)}\}}u_t(X_s^\alpha)
	\text{~and~} \mathbf 1_{\{\alpha \in I_s^{(m)}\}}\RB{ b^{(m)} \circ u_t}(X_s^\alpha),
\] 
	respectively.
	Here, we used the fact that $z\mapsto b^{(m)}(z)$ is a bounded continuous map, and almost surely, $x\mapsto u_t(x)$ is a bounded continuous map. 
	From those, we can verify using Fubini's theorem and dominated convergence theorem that the desired result for this step holds.
\begin{extra}
	(We will verify this with more details in Note \ref{note:499}.)
\end{extra}

\emph{Step 3.}
We will show that, for fixed arbitrary $T>0$ and $m\in \mathbb N$ with $m\geq 2$,
\begin{align}
	\lim_{\epsilon \downarrow 0} \int_0^T \Phi_{T-s,s}^{\epsilon,m} \mathrm ds
	= \int_0^T \Phi_{T-s,s}^{0,m} \mathrm ds.
\end{align}
Recall that
\begin{equation}
	b(z)= \sum_{k=0}^\infty b_k z^k + b_\infty z^\infty, \quad z\in [0,1],
\end{equation}
where $z^\infty := \mathbf 1_{\{z=1\}}$ for $z\in [0,1]$. 
If $b_\infty = 0$ then the map $z \mapsto b(z)$ is continuous on $[0,1]$, and the desired result for this step follows from an argument similar to Step 2. 
However, if $b_\infty > 0$, then $b(z)$ is not continuous at $z = 1$.  
So we will use a different argument here which depends on a technical result: Proposition \ref{lem:Last} in Appendix \ref{sec:Las}.

Notice that for any $\epsilon \geq 0$, 
\begin{equation} \label{eq:KK3}
	\int_0^T \Phi_{T-s,s}^{\epsilon,m} \mathrm ds = \int_0^T \mathbf E \SB{  \int_0^{T-s} \mathfrak{B}^{\epsilon, m}_{r,s} \mathrm dr} \mathrm ds,
\end{equation}
where
\begin{equation}
	\mathfrak{B}^{\epsilon, m}_{r,s} := (-1)^{|\tilde J^{(m)}_s|} e^{K^{(m)}_s} \RB{\sum_{\alpha \in I_s^{(m)}}  \RB{ P_\epsilon \RB{b\circ u_r}}\RB{X^{\alpha}_s} \prod_{\beta \in I^{(m)}_s\setminus\{\alpha\}}  \RB{P_\epsilon u_r}\RB{X^{\beta}_s} }.
\end{equation}
Also note that, for any $\epsilon\geq 0$ and $r,s \in [0,T]$, 
\begin{equation} \label{eq:UPs}
	|\mathfrak{B}^{\epsilon, m}_{r,s}| \leq e^{K^{(m)}_s} \|b\|_\infty |I_s^{(m)}|
\end{equation}
and that by Lemma \ref{lem:exponential},
\begin{equation}
	\int_0^T \mathbf E \SB{  \int_0^{T-s} e^{K^{(m)}_s} \|b\|_\infty |I_s^{(m)}| \mathrm dr} \mathrm ds = (T-s) \mathbf E \SB{\int_0^T  e^{K^{(m)}_s} \|b\|_\infty |I_s^{(m)}|  \mathrm ds} < \infty. 
\end{equation}
Therefore, by Fubini's theorem, 
\begin{itemize}
	\item[\eq\label{eq:T1}] the orders of the integration and the expectations on the right hand side of \eqref{eq:KK3} are interchangeable.
\end{itemize}
On the other hand, we can verify from Lemma \ref{lem:Coupling} that
\begin{itemize}
	\item[\eq\label{eq:T2}] $\tilde\tau_n \uparrow \infty$ as $n\uparrow \infty$ where  $0=\tilde \tau_0< \tilde \tau_1 < \tilde \tau_2 < \cdots$ are the occurring times of the branching/coalescing events for the $m$-truncated branching-coalescing particle system $\{(X^{(m),\alpha}_t)_{t\geq 0}: \alpha \in \mathcal U\}$. 
\end{itemize}
Therefore, from \eqref{eq:KK3}, \eqref{eq:T1} and \eqref{eq:T2}, for any $\epsilon \geq 0$, 
\begin{equation}
	\int_0^T \Phi_{T-s,s}^{\epsilon,m} \mathrm ds = \int_0^T \mathbf E \SB{   \int_0^{T-r}\mathfrak{B}^{\epsilon, m}_{r,s} \mathrm ds } \mathrm dr
	= \int_0^T \mathbf E \SB{   \sum_{k=1}^\infty  \int_{0}^{T-r} \mathbf 1_{\{s\in [\tilde \tau_{k-1},\tilde \tau_k)\}}\mathfrak{B}^{\epsilon, m}_{r,s} \mathrm ds } \mathrm dr.
\end{equation}
Using Fubini's theorem again, we have for every $\epsilon \geq 0$, 
\begin{align}
	&\int_0^T \Phi_{T-s,s}^{\epsilon,m} \mathrm ds = \int_0^T \mathbb E_f \SB{   \sum_{k=1}^\infty \int_{0}^{T-r} \tilde{\mathbb E} \SB{ \mathbf 1_{\{s\in [\tilde \tau_{k-1},\tilde \tau_k)\}} \mathfrak{B}^{\epsilon, m}_{r,s}} \mathrm ds  } \mathrm dr
	\\& \label{eq:PL34} = \int_0^T \mathbb E_f \SB{   \sum_{k=1}^\infty  \int_{0}^{T-r}  \tilde{\mathbb E} \SB{ \tilde{\mathbb E} \SB{ \mathbf 1_{\{s\in [\tilde \tau_{k-1},\tilde \tau_k)\}} \mathfrak{B}^{\epsilon, m}_{r,s} \middle| \tilde{\mathscr F}_{\tilde \tau_{k-1}}} }  \mathrm ds } \mathrm dr.
\end{align}
Recall here that $\mathbb E_f$ is the expectation corresponding to the random field $u$, and $\tilde {\mathbb E}$ is the expectation corresponding to the $m$-truncated branching-coalescing Brownian particles system. 

Fixing arbitrary $r,s\in [0,T]$ and $k \in \mathbb N$, we can write for every $\epsilon \geq 0$,
\begin{align} \label{eq:LS9}
	&\tilde{\mathbb E} \SB{ \mathbf 1_{\{s\in [\tilde \tau_{k-1},\tilde \tau_k)\}} \mathfrak{B}^{\epsilon, m}_{r,s} \middle| \tilde{\mathscr F}_{\tilde \tau_{k-1}}}
	= \mathbf 1_{\{s\geq\tilde \tau_{k-1}\}} 	\tilde{\mathbb E} \SB{ \mathbf 1_{\{s<\tilde \tau_k\}} \mathfrak{B}^{\epsilon, m}_{r,s} \middle| \tilde{\mathscr F}_{\tilde \tau_{k-1}}}
	\\& = 	\mathbf 1_{\{s\geq \tilde \tau_{k-1}\}} (-1)^{\ABS{\tilde J^{(m)}_{\tilde \tau_{k-1}}}} e^{K^{(m)}_{\tilde \tau_{k-1}}} \exp\CB{(\mu+b_1+\frac{1}{m})(s-\tilde \tau_{k-1})\ABS{I^{(m)}_{\tilde \tau_{k-1}}}}	\times {}
	\\& \qquad \tilde{\mathbb E} \SB{ \mathbf 1_{\{s<\tilde \tau_k\}} \RB{\sum_{\alpha \in I_s^{(m)}}  \RB{ P_\epsilon \RB{b\circ u_r}}\RB{X^{\alpha}_s} \prod_{\beta \in I^{(m)}_s\setminus\{\alpha\}}  \RB{P_\epsilon u_r}\RB{X^{\beta}_s} } \middle| \tilde{\mathscr F}_{\tilde \tau_{k-1}}}.
\end{align}
Notice that, from the strong Markov property of Brownian motions, after the time $\tilde \tau_{k-1}$, the particles in the $m$-truncated branching-coalescing Brownian particle system will evolve as independent Brownian motions until the next occurring time of its branching/coalescing event.
Therefore, we can further write for every $\epsilon\geq 0$ that 
\begin{align} 
	&\tilde{\mathbb E} \SB{ \mathbf 1_{\{s\in [\tilde \tau_{k-1},\tilde \tau_k)\}} \mathfrak{B}^{\epsilon, m}_{r,s} \middle| \tilde{\mathscr F}_{\tilde \tau_{k-1}}}
	\\& \label{eq:Rep} = 	\mathbf 1_{\{s\geq \tilde \tau_{k-1}\}} (-1)^{\ABS{\tilde J^{(m)}_{\tilde \tau_{k-1}}}} e^{K^{(m)}_{\tilde \tau_{k-1}}} e^{(\mu+b_1+\frac{1}{m})(s-\tilde \tau_{k-1})\ABS{I^{(m)}_{\tau_{k-1}}}} (\mathfrak P^\epsilon_{s-\tilde \tau_{k-1}} F)(\tilde X).
\end{align}
Here,
\begin{equation}
	\tilde X := (X^{\alpha_1}_{\tilde \tau_{k-1}},  \cdots, X^{\alpha_N}_{\tilde \tau_{k-1}}) \in \mathbb R^N
\end{equation} 
is a (random) finite list of real numbers with $N \in \mathbb N$ and $(\alpha_k)_{k=1}^N$ given so that 
\begin{equation}
	\alpha_1 \prec \cdots \prec \alpha_N\quad \text{and} \quad \{\alpha_1, \cdots, \alpha_N\}=I^{(m)}_{\tilde \tau_{k-1}};
\end{equation}
$F$ is a (random) bounded measurable function on $\mathbb R^N$ given so that
\begin{equation}
	F(x_1, \cdots, x_N) = \sum_{i=1}^N (b\circ u_r)(x_i)\prod_{j\in \{1,\cdots, N\}\setminus\{i\}} u_r(x_j), \quad (x_1,\dots, x_n)\in \mathbb R^N;
\end{equation}
and $(\mathfrak P^\epsilon_t)_{t\geq 0,\epsilon \geq 0}$ are the operators given as in \eqref{eq:Last} with $n$ replaced by the random $N$. 

Now, from \eqref{eq:Rep}, Proposition \ref{lem:Last} and the fact that $\tilde{\mathbb P}(\tilde \tau_{k-1} = s) = 0$, we have almost surely
\begin{equation}
	\tilde{\mathbb E} \SB{ \mathbf 1_{\{s\in [\tilde\tau_{k-1},\tilde\tau_k)\}} \mathfrak{B}^{\epsilon, m}_{r,s} \middle| \tilde{\mathscr F}_{\tilde\tau_{k-1}}} \xrightarrow[\epsilon \to 0]{} \tilde{\mathbb E} \SB{ \mathbf 1_{\{s\in [\tilde\tau_{k-1},\tilde\tau_k)\}} \mathfrak{B}^{0, m}_{r,s} \middle| \tilde{\mathscr F}_{\tilde\tau_{k-1}}}.
\end{equation}
Also observe from \eqref{eq:UPs} that
\begin{align}
	\ABS{\tilde{\mathbb E} \SB{ \mathbf 1_{s\in [\tilde\tau_{k-1},\tilde\tau_k)} \mathfrak{B}^{\epsilon, m}_{r,s} \middle| \tilde{\mathscr F}_{\tilde\tau_{k-1}}}}
	\leq \tilde{\mathbb E} \SB{ \mathbf 1_{s\in [\tilde\tau_{k-1},\tilde\tau_k)} e^{K^{(m)}_s} \|b\|_\infty |I_s^{(m)}| \middle| \tilde{\mathscr F}_{\tilde\tau_{k-1}}}
\end{align}
and that 
\begin{align}
	&\int_0^T \mathbb E_f \SB{   \sum_{k=1}^\infty  \int_{0}^{T-r}  \tilde{\mathbb E} \SB{  \tilde{\mathbb E} \SB{ \mathbf 1_{s\in [\tilde \tau_{k-1},\tilde \tau_k)} e^{K^{(m)}_s} \|b\|_\infty |I_s^{(m)}| \middle| \tilde{\mathscr F}_{\tilde \tau_{k-1}}} }  \mathrm ds } \mathrm dr
	\\&= \int_0^T \mathbf E \SB{   \sum_{k=1}^\infty  \int_{0}^{T-r}  \mathbf 1_{s\in [\tilde\tau_{k-1},\tilde\tau_k)} e^{K^{(m)}_s} \|b\|_\infty |I_s^{(m)}| \mathrm ds } \mathrm dr
	\\& = \int_0^T \mathbf E \SB{   \int_{0}^{T-r}   e^{K^{(m)}_s} \|b\|_\infty |I_s^{(m)}|    \mathrm ds } \mathrm dr < \infty. 
\end{align}
So by applying the dominated convergence theorem on the right hand side of \eqref{eq:PL34}, we get the desired result for this step. 
	
\emph{Step 4.}
	Notice from Lemma \ref{lem:TC} that almost surely with respect to $\mathbf P$, for any $t,s\geq 0$ and $\alpha \in \mathcal U$, the random variables 
\begin{equation}
	(-1)^{|\tilde J^{(m)}_s|}, \quad K^{(m)}_s, \quad \mathbf 1_{\{\alpha \in I_s^{(m)}\}}, \quad \prod_{\beta\in I^{(m)}_s\setminus\{\alpha\}}  u_t\RB{X^\beta_s}, \quad\text{and}\quad \mathbf 1_{\{\alpha \in I_s^{(m)}\}}b^{(m)}\circ u_t(X_s^\alpha),
\end{equation}
converge, as $m\uparrow \infty$, to 
\begin{equation}
	(-1)^{|\tilde J_s|}, \quad K_s, \quad \mathbf 1_{\{\alpha \in I_s\}}, \quad \prod_{\beta \in I_s\setminus\{\alpha\}}  u_t\RB{X^\beta_s }, \quad\text{and}\quad \mathbf 1_{\{\alpha \in I_s\}}b\circ u_t(X_s^\alpha),
\end{equation}
respectively. 
\begin{extra} (We will explain this in Note \ref{note:489}.) \end{extra}
	From this, we can verify, using Fubini's theorem and dominated convergence theorem, that
\begin{equation} \label{eq:E1}
	\lim_{m \uparrow \infty}  \int_0^T \Phi_{T-s,s}^{0,m} \mathrm ds  =  \int_0^T \Phi_{T-s,s}^{0,\infty} \mathrm ds,
\end{equation}
	and that
\begin{equation} \label{eq:E2}
	\lim_{m \uparrow \infty} \int_0^T \tilde \Phi_{t,T-t}^{0,m} \mathrm dt 
	= \int_0^T \tilde \Phi_{t,T-t}^{0,\infty} \mathrm dt.
\end{equation}
		\begin{extra}
		(We will verify those in Note \ref{note:500}.)
	\end{extra}
	Using Fubini's theorem again, we have
	\begin{equation} \label{eq:E3}
		 \int_0^T \Phi_{T-s,s}^{0,\infty} \mathrm ds = \int_0^T \tilde \Phi_{t,T-t}^{0,\infty} \mathrm dt.
	\end{equation}
				\begin{extra}
		(This will also be verified in Note \ref{note:500}.)
	\end{extra}
	Finally, from the results in Steps 1 and 2, \eqref{eq:E1}, \eqref{eq:E2} and \eqref{eq:E3}, we have
\begin{equation}
	\lim_{m\uparrow \infty}\lim_{\epsilon \downarrow 0}\int_0^T \Phi_{T-s,s}^{\epsilon,m} \mathrm ds = \lim_{m\uparrow \infty}\lim_{\epsilon \downarrow 0}\int_0^T \tilde \Phi_{t,T-t}^{\epsilon,m} \mathrm dt,
\end{equation}
as desired.
\end{proof}

	\begin{extra}
	\begin{note}\label{note:077}
		The only thing we need to verify is that 
	\[
		\ABS{b(z)}  \leq \sum_{k\in \bar {\mathbb N}} \ABS{b_k} z^k \leq \sum_{k\in \bar {\mathbb N}} \ABS{b_k} +1 = \Cr{const:b}
	\]
	and that
	\begin{align} \label{eq:refl}
		&\ABS{b^{(m)}(z)} \leq \sum_{k\in \bar{\mathbb N}} \ABS{b_k}  + \frac{1}{m} \leq  \Cr{const:b},
	\end{align}
	when $z\in [0,1]$.
	\end{note}
		\begin{note}\label{note:499}
		By Fubini's theorem and the dominated convergence theorem, we have
		\begin{align}
			&\int_0^T \tilde \Phi_{t,T-t}^{\epsilon,m} \mathrm dt
			\\& = \int_0^T \mathbf E \SB{ \int_0^{T-t} (-1)^{|\tilde J^{(m)}_r|} e^{K^{(m)}_r} \RB{\sum_{i\in I_r^{(m)}}  \RB{  b^{(m)}\circ \RB{P_\epsilon u_t}}\RB{X^i_r} \prod_{j\in I^m_r\setminus\{i\}}  \RB{P_\epsilon u_t}\RB{X^j_r} }\mathrm dr} \mathrm dt
			\\& = \int_0^T \mathbf E \SB{ \sum_{i\in \mathcal U}  \int_0^T \mathbf 1_{\{t+s\leq T\}} (-1)^{|\tilde J^{(m)}_s|} e^{K^{(m)}_s}  \mathbf 1_{\{\alpha \in I_s^{(m)}\}}  \RB{  b^{(m)}\circ \RB{P_\epsilon u_t}}\RB{X^i_s} \prod_{j\in I^{(m)}_s\setminus\{i\}}  \RB{P_\epsilon u_t}\RB{X^j_s} \mathrm ds} \mathrm dt
			\\& \xrightarrow[\epsilon \downarrow 0]{}\int_0^T \mathbf E \SB{ \sum_{i\in \mathcal U}  \int_0^T \mathbf 1_{\{t+s\leq T\}} (-1)^{|\tilde J^{(m)}_s|} e^{K^{(m)}_s}  \mathbf 1_{\{\alpha \in I_s^{(m)}\}}  \RB{  b^{(m)}\circ u_t}\RB{X^i_s} \prod_{j\in I^{(m)}_s\setminus\{i\}}  u_t\RB{X^j_s} \mathrm ds} \mathrm dt
			\\&= \int_0^T \tilde \Phi_{t,T-t}^{0,m} \mathrm dt.
		\end{align}
		Here we used the fact that the polynomial $z\mapsto b^{(m)}(z)$ is a continuous function on $[0,1]$.
	\end{note}
	\begin{note}
		\label{note:489}
		Let us note that
		\[
		b^{(m)}(z) = \sum_{k \in \bar{\mathbb N}} b_kz^{k\wedge m}- \frac{1}{m}z, \quad z\in [0,1].
		\]
		Fix arbitrary $z\in [0,1]$.
		Observe that for any $k \in \bar{\mathbb N}$, $b_kz^{k\wedge m} \to b_kz^k$ as $m\to \infty$. Also observe that $|b_kz^k| \leq |b_k|$ and $\sum_k |b_k|< \infty$.
		So by the dominated convergence theorem, we have $b^{(m)}$ converges to $b$ as $m\to \infty$. 
	\end{note}
	\begin{note} \label{note:500}
		By using Fubini's theorem and the dominated convergence theorem again, we have 
					\begin{align}
			&\int_0^T \Phi_{T-s,s}^{0,m} \mathrm ds 
			\\& = \int_0^T \mathbf E \SB{ \sum_{i\in \mathcal U}  \int_0^T \mathbf 1_{\{s+t\leq T\}} (-1)^{|\tilde J^{(m)}_s|} e^{K^{(m)}_s}  \mathbf 1_{\{\alpha \in I_s^{(m)}\}}  \RB{b\circ u_t}\RB{X^i_s} \prod_{j\in I^{(m)}_s\setminus\{i\}}  u_t\RB{X^j_s }\mathrm dt} \mathrm ds
			\\& \xrightarrow[m \uparrow \infty]{}\int_0^T \mathbf E \SB{ \sum_{i\in \mathcal U}  \int_0^T \mathbf 1_{\{s+t\leq T\}} (-1)^{|\tilde J_s|} e^{K_s}  \mathbf 1_{\{\alpha \in I_s\}}  \RB{b\circ u_t}\RB{X^i_s} \prod_{j\in I_s\setminus\{i\}}  u_t\RB{X^j_s }\mathrm dt} \mathrm ds
			\\&= \int_0^T \Phi_{T-s,s}^{0,\infty} \mathrm ds.
		\end{align}
By using Fubini's theorem and the dominated convergence theorem again, we have 
\begin{align}
	&\int_0^T \tilde \Phi_{t,T-t}^{0,m} \mathrm dt
	\\& = \int_0^T \mathbf E \SB{ \sum_{i\in \mathcal U}  \int_0^T \mathbf 1_{\{t+s\leq T\}} (-1)^{|\tilde J^{(m)}_s|} e^{K^{(m)}_s}  \mathbf 1_{\{\alpha \in I_s^{(m)}\}}  \RB{  b^{(m)}\circ u_t}\RB{X^i_s} \prod_{j\in I^{(m)}_s\setminus\{i\}}  u_t\RB{X^j_s} \mathrm ds} \mathrm dt
	\\& \xrightarrow[m \uparrow \infty]{}\int_0^T \mathbf E \SB{ \sum_{i\in \mathcal U}  \int_0^T \mathbf 1_{\{t+s\leq T\}} (-1)^{|\tilde J_s|} e^{K_s}  \mathbf 1_{\{\alpha \in I_s\}}  \RB{  b\circ u_t}\RB{X^i_s} \prod_{j\in I_s\setminus\{i\}}  u_t\RB{X^j_s} \mathrm ds} \mathrm dt
	\\&= \int_0^T \tilde \Phi_{t,T-t}^{0,\infty} \mathrm dt.
\end{align}
Finally, by use Fubini's theorem, it is easy to see that
	\begin{equation}
	\int_0^T \Phi_{T-s,s}^{0,\infty} \mathrm ds = \int_0^T \tilde \Phi_{t,T-t}^{0,\infty} \mathrm dt.
\end{equation}
		\end{note}
	\end{extra}
	
	\section{Proof of the weak existence part of Theorem \ref{thm:Main}} \label{sec:Last}

	As have been noted in Subsection \ref{sec:Main}, the weak existence of SPDE \eqref{eq:SPDE} is standard for $b_\infty = 0$. 
		So, we have to verify existence only for the case of $b_\infty \neq 0$. 
		For simplicity, let us also assume that $b_k = 0$ for every $k\in \mathbb N\setminus\{1\}$, as the argument for the more general cases  is similar.
		Then, with these parameters, the SPDE \eqref{eq:SPDE} is given by
		\begin{equation} \label{eq:Ex}
		\begin{cases}
		\partial_t u_t= \frac{1}{2}\partial_x^2 u_t + b_1 u_t + b_\infty \1_{\{1\}}(u_t) + \sqrt{u_t(1-u_t)} \dot W,
		\\ u_0 = f.
		\end{cases}
		\end{equation}
		Due to the condition \eqref{eq:CLZsCondition}, we have  $b_1 \leq -|b_\infty|$.

The idea of the existence proof is to construct an approximating sequence of $\mathcal C(\mathbb R, [0,1])$-valued processes $\{(u^{(m)}_t(x))_{t\geq 0, x\in \mathbb R}: m\geq 1\}$, to show that this sequence is tight, and it has a   limit point that solves~\eqref{eq:Ex}. 

	This sequence will be constructed to solve the following SPDEs, 
	\begin{equation}
		\label{eq:8_1}
		\begin{cases}
			\partial_t u^{(m)}_t= \frac{1}{2}\partial_x^2 u^{(m)}_t + b_1 u^{(k)}_t + b_\infty \RB{u^{(m)}_t}^m + \sqrt{u^{(m)}_t(1-u^{(m)}_t)} \dot W^{(m)},
			\\ u^{(m)}_0 = f,
		\end{cases}
	\end{equation}
	where $(\dot W^{(m)})_{n\in \mathbb N}$ is a sequence of space-time  white noises.

The main difficulty in the proof is to show that a sub-sequential weak limit point of $\{u^{(m)}: m\geq 1\}$ indeed solves~\eqref{eq:Ex}. 
It is non-trivial since it is not clear why convergence of any subsequence  
$u^{(m_k)}$  to $u$, would imply convergence of $\RB{u^{(m_k)}_t(x)}^{m_k}$ to $\1_{\{1\}}(u_t(x))$. 
The challenge comes from the discontinuity of the function $\1_{\{1\}}(\cdot)$. 
 
 	This difficulty will be resolved via a duality argument, which is based on the convergence of the 
 	dual particle system of $u^{(m)}$ to  the dual particle system of $u$. 
 We will give  the details below, but  first
let us introduce the settings for the rest of this section and state some useful lemmas. 
	Let $x\in \mathbb R$ be arbitrary. 
	Let $\{(X_t^\alpha)_{t\geq 0}: \alpha \in \mathcal U\}$ be a coalescing-branching Brownian particle system with branching rate $\mu := |b_\infty|$, offspring distribution $p_\infty = 1$, and initial configuration $(x_i)_{i=1}^\infty$ such that $x_i = x$ for every $i \in \mathbb N$.
	This particle system is defined on some probability space $(\tilde \Omega, \tilde {\mathcal F}, \tilde{\mathbb P})$.  If we want to emphasize that all the particles start at $x$, we write 
$ \tilde{\mathbb P}^x$ for the probability measure, and  $\tilde{\mathbb E}^x$ for the corresponding expectation.
	For each $l, m\in \mathbb N \cup \{\infty\}$, recall  that the $(l,m)$-truncated version of this particle system $\{(X^{(l,m),\alpha}_t)_{t\geq 0}: \alpha \in \mathcal U\}$, given as in Subsection \ref{sec:Truncated}, is a coalescing-branching Brownian particle system with branching rate $\mu$, offspring distribution $p_m = 1$, and initial configuration $(x_i)_{i=1}^l$. 
	Also recall the sets of labels $I_t^{(l,m)}$ and $J_t^{(l,m)}$ for $t\geq 0$ are given as in \eqref{eq:Imt} and \eqref{eq:Jmt} respectively.
	The following lemma is a variant of Lemma \ref{lem:TC}. It allows us to approximate $|I^{(l,\infty)}_t|$ and $|J^{(l,\infty)}_t|$ from below. 
	
	\begin{lemma} \label{lem:DT}
		Almost surely, for each $l,m\in \mathbb N$ with $l\leq m$ and $t\geq 0$, we have
		\begin{equation}
			I_t^{(l,m)} = \{\alpha \in \mathcal U: \|\alpha\|_{\infty}\leq m, \alpha \in I^{(l,\infty)}_t\}
		\end{equation}
		and
		\begin{equation}
			J_t^{(l,m)} = \{\alpha \in \mathcal U: \|\alpha\|_{\infty} \leq m, \alpha \in J^{(l,\infty)}_t\}.  
		\end{equation}
	\end{lemma}

		\begin{extra}
			(We give the proof of Lemma \ref{lem:DT} in Note \ref{note:DTT} below.)
\begin{note} \label{note:DTT}
	\begin{proof}[Proof of Lemma \ref{lem:DT}]
		We claim that, 
		\begin{itemize}
			\item[\eq\label{eq:CR}] for each $\alpha \in \mathcal U$ and $l,m\in \mathbb N$ with $l\leq m$, $\zeta^{(l,m)}_\alpha = \mathbf 1_{\{m< \|\alpha\|_\infty \}}\xi_\alpha +  \mathbf 1_{\{m\geq \|\alpha\|_\infty \}}\zeta^{(l,\infty)}_\alpha$ almost surely.
		\end{itemize}
		Fixing $l,m \in \mathbb N$ with $l\leq m$, we will prove this claim by induction over $\alpha \in \mathcal U$.
		If $\alpha = 1$, then it is easy to see that $\zeta^{(l,m)}_\alpha= \zeta_{\alpha,\alpha} = \zeta^{(l,\infty)}_\alpha$. 
		Let us fix an arbitrary $\beta \in \mathcal U$, and for the sake of induction, assume that the desired claim \eqref{eq:C} holds for every $\alpha \prec \beta$. 
		
		Firstly, we will show that almost surely $\zeta^{(l,m)}_\beta = \xi_\beta$ provided $m < \|\beta \|_\infty$. 
		To do this, we discuss in two different cases.
		\begin{itemize}
			\item[(i)] $m < \|\beta \|_\infty$ and $|\beta|=1$. In this case, we have $l\leq m < \beta$. So by (iii) of \eqref{eq:TT}, we have $\zeta^{(l,m)}_\beta = \xi_\beta$ as desired. 
			\item[(ii)] $m < \|\beta \|_\infty$ and $|\beta |>1$. 
			In this case, we can show that the event 
			\[ \CB{ \zeta^{(l,m)}_{\overleftarrow{\beta}} = \zeta_{\overleftarrow{\beta},\overleftarrow{\beta}}} \cap \CB{\beta_{|\beta|} \leq Z_{\overleftarrow{\beta}} \wedge m} \]
			has $0$ probability. 
			In fact, if the above event happens, we have $\beta_{|\beta|} \leq m$. 
			This, and the condition $m < \|\beta\|_\infty$, implies that $m < \|\overleftarrow{\beta}\|_\infty$.
			From what we assumed for the sake of induction, we must have $\zeta^{(l,m)}_{\overleftarrow{\beta}} = \xi_{\overleftarrow{\beta}}$. 
			This further implies that $ \xi_{\overleftarrow{\beta}} = \zeta_{\overleftarrow{\beta},\overleftarrow{\beta}}$ which has $0$ probability.
			Now, by (iii) of \eqref{eq:TT}, we have $\zeta^{(l,m)}_\beta = \xi_\beta$ as desired. 
		\end{itemize}
		
		Secondly, we will show that almost surely $\zeta^{(l,m)}_\beta = \zeta^{(l,\infty)}_\beta$ provided $m \geq \|\beta\|_\infty$. 
		To do this, we discuss in four different cases.
		\begin{itemize}
			\item[(i)] 
			$m \geq \|\beta\|_\infty$, $|\beta|=1$ and $\beta > l$. 
			In this case, (iii) of \eqref{eq:TT} we have $\zeta^{(l,m)}_\beta = \xi_\beta = \zeta^{(l,\infty)}_\beta$ as desired. 
			\item[(ii)] 
			$m \geq \|\beta\|_\infty$, $|\beta|=1$ and $\beta \leq l$. 
			In this case, by (i) of \eqref{eq:TT}  we have
			\begin{equation}
				\zeta^{(l,m)}_\beta 
				= \inf \RB{ \CB{\zeta_{\beta,\beta}} \cup \CB{\zeta_{\alpha,\beta}: \alpha \in \mathcal U, \alpha \prec \beta, \zeta_{\alpha,\beta} \leq \zeta^{(l,m)}_\alpha}}
			\end{equation}
			and 
			\begin{equation}
				\zeta^{(l,\infty)}_\beta 
				= \inf \RB{ \CB{\zeta_{\beta,\beta}} \cup \CB{\zeta_{\alpha,\beta}: \alpha \in \mathcal U, \alpha \prec \beta, \zeta_{\alpha,\beta} \leq \zeta^{(l,\infty)}_\alpha}}.
			\end{equation}
			Note that, $\alpha \prec \beta$ actually implies that $\|\alpha\|_\infty \leq m$. 
			So by what we assumed for the sake of induction, we have $\zeta^{(l,m)}_\alpha = \zeta^{(l,\infty)}_\alpha$ for every $\alpha \prec \beta$. 
			Now the above two equations implies that $\zeta^{(l,m)}_\beta = \zeta^{(l,\infty)}_\beta$ as desired. 
			\item[(iii)] 	
			$m \geq \|\beta\|_\infty$, $|\beta| > 1$, $\zeta^{(l,m)}_{\overleftarrow{\beta}} = \zeta_{\overleftarrow{\beta},\overleftarrow{\beta}}$ and $\beta_{|\beta|} \leq Z_{\overleftarrow{\beta}} ^{(m)} = Z_{\overleftarrow{\beta}} \wedge m$. 
			In this case, by (ii) of \eqref{eq:TT}, we have 
			\begin{equation}
				\zeta^{(l,m)}_\beta 
				= \inf \RB{ \CB{\zeta_{\beta,\beta}} \cup \CB{\zeta_{\alpha,\beta}: \alpha \in \mathcal U, \alpha \prec \beta, \zeta_{\alpha,\beta} \leq \zeta^{(l,m)}_\alpha}}.
			\end{equation}
			Also note that $\|\overleftarrow{\beta}\|_\infty \leq \|\beta\|_\infty \leq m$. 
			So from what we assumed for the sake of induction, we have $\zeta^{(l,\infty)}_{\overleftarrow{\beta}} =\zeta^{(l,m)}_{\overleftarrow{\beta}} = \zeta_{\overleftarrow{\beta}}$. 
			By (ii) of \eqref{eq:TT} again, we have 
			\begin{equation}
				\zeta^{(l,\infty)}_\beta 
				= \inf \RB{ \CB{\zeta_{\beta,\beta}} \cup \CB{\zeta_{\alpha,\beta}: \alpha \in \mathcal U, \alpha \prec \beta, \zeta_{\alpha,\beta} \leq \zeta^{(l,\infty)}_\alpha}}.
			\end{equation}
			Similar to the previous case, we have $\zeta^{(l,m)}_\beta  = 	\zeta^{(l,\infty)}_\beta $ as desired.
			\item[(iv)] 
			$m \geq \|\beta\|_\infty$, $|\beta| > 1$, but the condition in (iii) does not hold. 
			In this case, by (ii) of \eqref{eq:TT}, we have $\zeta^{( l,m)}_\beta  = \xi_\beta$. 
			We can verify by contradiction that the event 
			\[ 
			\CB{\zeta^{(l,\infty)}_{\overleftarrow{\beta}} = \zeta_{\overleftarrow{\beta},\overleftarrow{\beta}}} \cap \CB{\beta_{|\beta|} \leq Z_{\overleftarrow{\beta}}} 
			\]
			happens with $0$ probability.  
			In fact, if it does happen, from $\|\overleftarrow{\beta}\|_\infty \leq \|\beta\|_\infty \leq m$ and what we have assumed for the sake of induction, $\zeta^{(l,m)}_{\overleftarrow{\beta}} = \zeta_{\overleftarrow{\beta},\overleftarrow{\beta}}$. 
			This implies that $\xi_\beta = \zeta_{\overleftarrow{\beta},\overleftarrow{\beta}}$ which happens with $0$ probability.
			Now, by (iii) of  \eqref{eq:TT}, we have $\zeta^{(l,\infty)}_\beta  = 	\xi_\beta = \zeta^{(l,m)} $ as desired. 
		\end{itemize} 
		
		To sum up, we have proved claim \eqref{eq:CR}. 
		The desired result for this lemma follows immediately.
	\end{proof}
	\end{note}
	\end{extra}

		We omit the proof of the above lemma, because it is very similar to the proof of Lemma \ref{lem:TC} in Appendix \ref{sec:TC}. 
Recall that the sets of indices $(I_t)_{t\geq 0},(J_t)_{t\geq 0}$ for the non-truncated system were defined in~\eqref{eq:It}, \eqref{eq:Jt}.
	Let us consider the event $\tilde \Omega' := \cup_{n=1}^\infty\{|J_{1/n}| = 0\}$. 
	From Theorem \ref{prop:Key}, we have for every $\epsilon > 0$, 
	\begin{align}
		\tilde {\mathbb P}(|J_\epsilon | = 0) = 1 - \tilde {\mathbb P}(|J_\epsilon | \geq 1) \geq 1 - \tilde {\mathbb E}[|J_\epsilon |] = 1 -  \mu  \int_0^\epsilon \tilde {\mathbb E}\SB{|I_s|} \mathrm ds 
	\end{align}
	which implies that 
	\begin{equation} \label{eq:FF}
	\tilde {\mathbb P}(\tilde \Omega') = \lim_{n\to \infty} \tilde {\mathbb P}(|J_{1/n} | = 0) = 1.
	\end{equation}
	Define the random integer
	$L_\epsilon = \sup\{\alpha_1: \alpha \in I_\epsilon\} \vee 0$ for each $\epsilon > 0$.
	Since for $\epsilon > 0$, almost surely $I_\epsilon$ is a finite set (Theorem \ref{prop:Key}), we have almost surely $L_\epsilon < \infty$. 
	The following lemma allows us to approximate $|I_t|$ and $|J_t|$ by $|I_t^{(l,\infty)}|$ and $|J_t^{(l,\infty)}|$ when $l\uparrow \infty$.
	\begin{lemma} \label{lem:KS}
For any $\epsilon > 0$, almost surely on the event $\{|J_\epsilon| = 0\}$, for every finite integer $l\geq L_\epsilon$ and $t\geq 0$, we have
\begin{equation}
	I_{t}^{(l,\infty)} = \{\alpha \in \mathcal U: \alpha_1 \leq l, \alpha \in I_t\}
\end{equation}
and
\begin{equation}
	J_{t}^{(l,\infty)} = \{\alpha \in \mathcal U: \alpha_1 \leq l, \alpha \in J_t\}.
\end{equation}
	\end{lemma} 
We omit the proof of the above lemma, because it is also similar to the proof of Lemma \ref{lem:TC} in Appendix \ref{sec:TC}. 
\begin{lemma} \label{lem:CP} 
	Suppose that $(Z_k)_{k\in \mathbb N}$ is a sequence of $[0,1]$-valued random variables converging almost surely to a $[0,1]$-valued random variable $Z$. 
	Assume that $\mathbb E[Z_k^k]$ converges to $\mathbb E[\1_{\{1\}}(Z)]$ when $k\uparrow \infty$.  
	Then $Z_k^k$ converges to $\1_{\{1\}}(Z)$ in $L^p$ for every $p\geq 1$ when $k\uparrow \infty$.
\end{lemma}
\begin{proof}
	Let us first prove the $L^1$ convergence.
	Note that for every $k\in \mathbb N$,
\begin{align}
	&\mathbb E\SB{\ABS{Z_k^k - \1_{\{1\}}(Z)}} =\mathbb E\SB{\ABS{Z_k^k - \1_{\{1\}}(Z)} \mathbf 1_{\{Z = 1\}}} +  \mathbb E\SB{\ABS{Z_k^k - 
\1_{\{1\}}(Z)} \mathbf 1_{\{Z < 1\}}} 
	 \\&= \mathbb E\SB{\RB{1-Z_k^k} \mathbf 1_{\{Z = 1\}}} +  \mathbb E\SB{ Z_k^k  \mathbf 1_{\{Z < 1\}}} 
	 \\& \label{eq:AA}= \mathbb E\SB{\1_{\{1\}}(Z)} - \mathbb E \SB{Z_k^k} +  2\mathbb E\SB{ Z_k^k  \mathbf 1_{\{Z < 1\}}}.
\end{align}
	It is clear that $Z_k^k$ converges to $0$ on the event $\{Z<1\}$, and therefore the third term on the right hand side of \eqref{eq:AA} converges to $0$, when $k\uparrow \infty$. 
	Now \eqref{eq:AA} implies that $Z_k^k$ converges to $\1_{\{1\}}(Z)$ in $L^1$ when $k\uparrow \infty$. 
	
	It is then clear that $Z_k^k$ converges in probability to $\1_{\{1\}}(Z)$.
	Also, it can be verified form bounded convergence theorem that for any $p\geq 1$, 
	\[
		\lim_{k\to \infty} \mathbb E\SB{\RB{Z_k^k}^p} = \mathbb E\SB{\RB{\1_{\{1\}}(Z)}^p}.
	\]
	Now from \cite[Theorem 5.12]{MR4226142} we also obtain the $L^p$ convergence for every $p\geq 1$ as desired. 
\end{proof}
\begin{extra}
	(We give the proof of Lemma \ref{lem:KS} in Note \ref{note:KS} below.)
	\begin{note} \label{note:KS}
		\begin{proof}[Proof of Lemma \ref{lem:KS}]
			Fix the arbitrary $\epsilon > 0$ and $l\in \mathbb N$.
			We first claim that almost surely on the event $\{|J_\epsilon| = 0, l \geq L_\epsilon\}$, 
			\begin{itemize}
				\item[\eq\label{eq:RR}] for each $\alpha \in \mathcal U$ with $\alpha_1 > l$, if $|\alpha| = 1$ then $\zeta_\alpha \leq \epsilon$, and if $|\alpha|>1$ then $\xi_\alpha = \zeta_\alpha$. 
			\end{itemize}
			Let us prove the claim \eqref{eq:RR} by induction over $\alpha \in \mathcal U$. Obviously, it holds for $\alpha = 1$ since nothing needs to be proved. 
			Let us fix an arbitrary $\beta \in \mathcal U$, and for the sake of induction, assume that the desired claim \eqref{eq:RR} holds for every $\alpha \prec \beta$. 
			We discuss in six different cases under the event $\{|J_\epsilon| = 0, l \geq L_\epsilon\}$:
			\begin{itemize}
			\item[(i)] $|\beta| = 1$ and $\beta_1\leq l$. In this case, nothing needs to be proved.
			\item[(ii)] $|\beta| = 1$ and $\beta_1> l$. For the same of contradiction, assume that $\zeta_\beta > \epsilon$. Then $\beta \in I_\epsilon$, and therefore $L_\epsilon \geq \beta$. 
			This contradicts to the condition $ l \geq L_\epsilon$. So the desired \eqref{eq:RR} holds in this case.  
			\item[(iii)] $|\beta|>1$ and $\beta_1 \leq l$. In this case, there is nothing to be proved.
			\item[(iv)] $|\beta| > 1$, $\beta_1 > l$, $|\overleftarrow{\beta}|>1$, $\zeta_{\overleftarrow{\beta}} = \zeta_{\overleftarrow{\beta},\overleftarrow{\beta}}$ and $\beta_{|\beta|} \leq Z_{\overleftarrow{\beta}}$.
			In this case, since $\overleftarrow{\beta}_1 = \beta_1 > l$, from what we have assumed for the sake of induction, we have $\xi_{\overleftarrow{\beta}} = \zeta_{\overleftarrow{\beta}}$. 
			Now, we have $\xi_{\overleftarrow{\beta}}  = \zeta_{\overleftarrow{\beta},\overleftarrow{\beta}}$ which happens with $0$ probability. So this case won't really happen.
			\item[(v)] $|\beta| > 1$, $\beta_1 > l$, $|\overleftarrow{\beta}|=1$, $\zeta_{\overleftarrow{\beta}} = \zeta_{\overleftarrow{\beta},\overleftarrow{\beta}}$ and $\beta_{|\beta|} \leq Z_{\overleftarrow{\beta}}$.
			In this case, since $\overleftarrow{\beta} = \beta_1 > l$, from what we have assumed for the sake of induction, we have $\zeta_{\overleftarrow{\beta}} \leq \epsilon$. 
			Therefore $\zeta_{\overleftarrow{\beta}} = \zeta_{\overleftarrow{\beta},\overleftarrow{\beta}} \leq \epsilon$, and from the definition of $J_\epsilon$, we have $\overleftarrow{\beta} \in J_\epsilon$.
			This contradicts to the condition that $|J_\epsilon| = 0$. Therefore, this case won't really happen.
			\item[(vi)] Both $|\beta| > 1$ and $\beta_1 > l$ hold, but one of  $\zeta_{\overleftarrow{\beta}} = \zeta_{\overleftarrow{\beta},\overleftarrow{\beta}}$ and $\beta_{|\beta|} \leq Z_{\overleftarrow{\beta}}$ doesn't hold. In this case, by (iii) of \eqref{eq:Lifetime}, we have $\zeta_\beta = \xi_\beta$ as desired. 
		\end{itemize}
			
			We then claim that almost surely on the event $\{|J_\epsilon| = 0, l \geq L_\epsilon\}$, 
			\begin{itemize}
				\item[\eq\label{eq:RI}] for each $\alpha \in \mathcal U$, $\zeta^{(l,\infty)}_\alpha = \mathbf 1_{\{\alpha_1> l\}}\xi_\alpha +  \mathbf 1_{\{\alpha_1 \leq l \}}\zeta_\alpha$.
			\end{itemize}
			We will prove this claim again by induction over $\alpha \in \mathcal U$.
			If $\alpha = 1$, then it is easy to see that $\zeta^{(l,\infty)}_\alpha= \zeta_{\alpha,\alpha} = \zeta_\alpha$. 
			Let us fix an arbitrary $\beta \in \mathcal U$, and for the sake of induction, assume that the desired claim \eqref{eq:RI} holds for every $\alpha \prec \beta$. 
			We discuss in six different cases under the event $\{|J_\epsilon| = 0, l \geq L_\epsilon\}$:
			\begin{itemize}
				\item[(i)] $\beta_1> l$ and $|\beta|=1$. In this case, by (iii) of \eqref{eq:TT}, we have $\zeta^{(l,\infty)}_\beta = \xi_\beta$ as desired. 
				\item[(ii)] $\beta_1> l$ and $|\beta |>1$. 
				In this case, we can show that the event 
				\[ \CB{ \zeta^{(l,\infty)}_{\overleftarrow{\beta}} = \zeta_{\overleftarrow{\beta},\overleftarrow{\beta}}} \cap \CB{\beta_{|\beta|} \leq Z_{\overleftarrow{\beta}}} \]
				won't happen. 
				In fact, from $\beta_1> l$ we have $\overleftarrow{\beta}_1> l$. 
				So by what we assumed for the sake of induction, we must have $\zeta^{(l,\infty)}_{\overleftarrow{\beta}} = \xi_{\overleftarrow{\beta}}$. 
				Now, we have $ \xi_{\overleftarrow{\beta}} = \zeta_{\overleftarrow{\beta},\overleftarrow{\beta}}$ which has $0$ probability.
				Therefore, by (iii) of \eqref{eq:TT}, we have $\zeta^{(l,\infty)}_\beta = \xi_\beta$ as desired. 
				\item[(iii)] 
				$\beta_1\leq l$ and $|\beta|=1$. 
				In this case, by (i) of \eqref{eq:TT}  we have
				\begin{equation}
					\zeta^{(l,\infty)}_\beta 
					= \inf \RB{ \CB{\zeta_{\beta,\beta}} \cup \CB{\zeta_{\alpha,\beta}: \alpha \in \mathcal U, \alpha \prec \beta, \zeta_{\alpha,\beta} \leq \zeta^{(l,\infty)}_\alpha}}
				\end{equation}
				and by (i) of \eqref{eq:Lifetime}, we have
				\begin{equation}
					\zeta_\beta 
					= \inf \RB{ \CB{\zeta_{\beta,\beta}} \cup \CB{\zeta_{\alpha,\beta}: \alpha \in \mathcal U, \alpha \prec \beta, \zeta_{\alpha,\beta} \leq \zeta_\alpha}}.
				\end{equation}
				Note that, $\alpha \prec \beta$ actually implies that $\alpha_1\leq \|\alpha\|_\infty \leq \| \beta\|_\infty = \beta_1 \leq l$. 
				So by what we assumed for the sake of induction, we have $\zeta^{(l,\infty)}_\alpha = \zeta_\alpha$ for every $\alpha \prec \beta$. 
				Now the above two equations implies that $\zeta^{(l,\infty)}_\beta = \zeta_\beta$ as desired. 
				
				\item[(iv)] 	
				$\beta_1\leq l$, $|\beta| > 1$, $\zeta^{(l,\infty)}_{\overleftarrow{\beta}} = \zeta_{\overleftarrow{\beta},\overleftarrow{\beta}}$ and $\beta_{|\beta|} \leq Z_{\overleftarrow{\beta}}$. 
				Firstly, note that in this case $\overleftarrow{\beta}_1 = \beta_1 \leq l$. 
				So from what we assumed for the sake of induction, we have $\zeta^{(l,\infty)}_{\overleftarrow{\beta}} = \zeta_{\overleftarrow{\beta}}$.
				Therefore $\zeta_{\overleftarrow{\beta}} = \zeta_{\overleftarrow{\beta},\overleftarrow{\beta}}$.
				We also know that $\zeta_{\overleftarrow{\beta},\overleftarrow{\beta}}> \epsilon$. 
				In fact, if otherwise, then from the definition of $J_\epsilon$, we have $\overleftarrow{\beta} \in J_\epsilon$ which contradicts to the condition $|J_\epsilon| = 0$. 
				
				Secondly, we claim that in this case 
				$\zeta_{\alpha,\beta} > \zeta^{(l,\infty)}_\alpha \vee \zeta_\alpha$ 
				for every $\alpha \prec \beta$ with $\alpha_1 > l$. 
				We discuss this claim in two different cases:
				\begin{itemize}
					\item[(a)] $|\alpha| = 1$. In this case, from \eqref{eq:RR}, we have $\zeta_\alpha < \epsilon$; and from what we have assumed for the sake of induction, $\zeta^{(l,\infty)}_\alpha = \xi_\alpha$.
					Now from \eqref{eq:xiab}, the definition of $\zeta_{\alpha,\beta}$, we have
									$
					\zeta_{\alpha,\beta} > \xi_\beta = \zeta_{\overleftarrow{\beta},\overleftarrow{\beta}} > \epsilon > \zeta_\alpha
					$
					and $\zeta_{\alpha,\beta} > \xi_\alpha = \zeta^{(l,\infty)}_\alpha$ as desired.
					\item[(b)] $|\alpha|>1$.  In this case, from \eqref{eq:RR}, we have $\zeta_\alpha = \xi_\alpha$; and from what we have assumed for the sake of induction, $\zeta^{(l,\infty)}_\alpha = \xi_\alpha$.
					Now from \eqref{eq:xiab}, the definition of $\zeta_{\alpha,\beta}$, we have
					$
					\zeta_{\alpha,\beta}  > \xi_\alpha = \zeta_\alpha = \zeta^{(l,\infty)}_\alpha$ as desired.
				\end{itemize}
				
				Now from this claim and (ii) of \eqref{eq:TT}, we have 
				\begin{align}
					\zeta^{(l,\infty)}_\beta 
					&= \inf \RB{ \CB{\zeta_{\beta,\beta}} \cup \CB{\zeta_{\alpha,\beta}: \alpha \in \mathcal U, \alpha \prec \beta, \zeta_{\alpha,\beta} \leq \zeta^{(l,\infty)}_\alpha}}
					\\&= \inf \RB{ \CB{\zeta_{\beta,\beta}} \cup \CB{\zeta_{\alpha,\beta}: \alpha \in \mathcal U, \alpha \prec \beta, \alpha_1\leq l, \zeta_{\alpha,\beta} \leq \zeta^{(l,\infty)}_\alpha}}.
				\end{align}
				Also, from the previous claim and (ii) of \eqref{eq:Lifetime}, we have 
				\begin{align}
					\zeta_\beta 
					&= \inf \RB{ \CB{\zeta_{\beta,\beta}} \cup \CB{\zeta_{\alpha,\beta}: \alpha \in \mathcal U, \alpha \prec \beta, \zeta_{\alpha,\beta} \leq \zeta_\alpha}}
					\\&= \inf \RB{ \CB{\zeta_{\beta,\beta}} \cup \CB{\zeta_{\alpha,\beta}: \alpha \in \mathcal U, \alpha \prec \beta, \alpha_1\leq l, \zeta_{\alpha,\beta} \leq \zeta_\alpha}}.
				\end{align}
				Note that for every $\alpha \prec \beta$ with $\alpha_1 \leq l$, from what we assumed for the sake of induction, $\zeta^{(l,\infty)}_\alpha = \zeta_\alpha$.
				Now the above two equations implies that $\zeta^{(l,\infty)}_\beta  = \zeta_\beta$ as desired. 
				
				\item[(v)] 
				$\beta_1\leq l$, $|\beta| > 1$, but the condition in (iv) does not hold. 
				In this case, by (iii) of \eqref{eq:TT}, we have $\zeta^{( l,\infty)}_\beta  = \xi_\beta$. 
				We can verify by contradiction that the event 
				\[ 
				\CB{\zeta_{\overleftarrow{\beta}} = \zeta_{\overleftarrow{\beta},\overleftarrow{\beta}}} \cap \CB{\beta_{|\beta|} \leq Z_{\overleftarrow{\beta}}} 
				\]
				won't happen.  
				In fact, if otherwise, from $\overleftarrow{\beta}_1 = \beta_1 \leq l$ and what we have assumed for the sake of induction, $\zeta^{(l,\infty)}_{\overleftarrow{\beta}} = \zeta_{\overleftarrow{\beta}}$. 
				This implies that $\zeta^{(l,\infty)}_{\overleftarrow{\beta}} = \zeta_{\overleftarrow{\beta},\overleftarrow{\beta}}$ which contradicts to the assumption that the condition in (iv) does not hold.
				Now, by (iii) of  \eqref{eq:Lifetime}, we have $ \zeta_\beta = 	\xi_\beta =  \zeta^{(l,\infty)}_\beta $ as desired. 
			\end{itemize} 
			
			To sum up, we have proved claim \eqref{eq:RI}. 
			The desired result for this lemma follows immediately.
		\end{proof}
	\end{note}
\end{extra}
	
	Now we are ready to present the main steps of the existence proof. 
	The construction of  the approximating sequence $u^{(m)}$ and one of its weak limit points will be carried out in the Steps~1, 2. 
	The non-trivial part, as we have mentioned above, is to show that the limit point indeed solves the SPDE~\eqref{eq:Ex}; this will be done in the Steps 3--7.  
	\begin{proof}[Proof of the weak existence part of Theorem \ref{thm:Main}]
		\emph{Step 1.}
		Let $\mathcal C(\mathbb R)$ be the space of real-valued continuous functions on $\mathbb R$. 
		Define $\|f\|_{(p)} := \sup_{x\in \mathbb R} |e^{p|x|}f(x)|$ for every $p\in \mathbb R$ and $f\in \mathcal C(\mathbb R)$. 
		Define the complete separable metric space $\mathcal C_{\mathrm{tem}}(\mathbb R):=\{f\in \mathcal C(\mathbb R): \forall p > 0,\|f\|_{(-p)}<\infty\}$ equipped with the metric   \[ d_{\mathrm{tem}}(f,g) := \sum_{k=1}^\infty 2^{-k}  (\|f-g\|_{(-k^{-1})}\wedge 1), \quad f,g\in \mathcal C_{\mathrm {tem}}(\mathbb R).\]
		Denote by $\mathcal C(\mathbb R_+, \mathcal C_{\mathrm{tem}}(\mathbb R))$ the space of $\mathcal C_{\mathrm{tem}}(\mathbb R)$-valued continuous paths on $\mathbb R_+$, equipped with the topology of uniform convergence on compact sets.
		For each $m \in \mathbb N$, we have the existence of a $\mathcal C(\mathbb R_+, \mathcal C_{\mathrm{tem}}(\mathbb R))$-valued random element $u^{(m)}=(u_t^{(m)})_{t\geq 0} =(u^{(m)}_t(x))_{t\geq 0, x\in \mathbb R}$ satisfying the SPDE~\eqref{eq:8_1} on  some probability space. 
As we have mentioned, existence of solution to~\eqref{eq:8_1} is standard (see e.g. 
\cite[Theorem 2.6]{MR1271224} and \cite[Section 2.1]{MR4259374}).
		Moreover, $u^{(m)}_t(x)$ takes values in $[0,1]$ for each $m\in \mathbb N$, $t\geq 0$ and $x\in \mathbb R$. 
		Note in particular, the random elements $u^{(m)}$, for different $m \in \mathbb N$, are not necessarily driven by the same noise, nor necessarily defined in the same probability space.  
		
		\emph{Step 2.}
		One can also verify, by the standard theory (c.g. \cite{MR1271224}), that the sequence of $\mathcal C(\mathbb R_+, \mathcal C_{\text{tem}}(\mathbb R))$-valued random elements $\{u^{(m)}: m\in \mathbb N\}$ is tight. 
		In particular, by using Prohorov's theorem and Skorohod's representation theorem, there exists a (deterministic) strictly increasing $\mathbb N$-valued sequence $(m_k)_{k \in \mathbb N}$, and a sequence of $\mathcal C(\mathbb R_+, \mathcal C_{\text{tem}}(\mathbb R))$-valued random elements $(\tilde u^{(m_k)})_{k\in \mathbb N}$ defined in a common probability space, such that 
		\begin{itemize}
			\item[\eq\label{eq:FE}] for each $k \in \mathbb N$, the law of $\tilde u^{(m_k)}$ equals to the law of $u^{(m_k)}$;
			\item[\eq\label{eq:EF}] the limit $\tilde u = \lim_{k \to \infty} \tilde u^{(m_k)}$ exists almost surely with respect to the topology of $\mathcal C(\mathbb R_+, \mathcal C_{\text{tem}}(\mathbb R))$. 
		\end{itemize}
		It is also clear that $(\tilde u_t(x))_{t\geq 0, x\in \mathbb R}$ is a $[0,1]$-valued continuous random field. 
		
		In the rest of the proof, we will show that $\tilde u$ solves the martingale problem corresponding to the SPDE \eqref{eq:Ex}, that is, for any compactly supported smooth (testing) function $\phi$ on $\mathbb R$, almost surely for every $t\geq 0$,
		\begin{align}
			&\int \phi(x)\tilde u_t(x) \mathrm dx - \int \phi(x) f(x) \mathrm dx 
			\\& \label{eq:AL}= \frac{1}{2}\iint_0^t \phi''(x) \tilde u_s(x) \mathrm ds\mathrm  dx + b_1 \iint_0^t  \phi(x) \tilde u_s(x) \mathrm ds\mathrm  dx 
			+ {}
			\\& \qquad b_\infty \iint_0^t  \phi(x)  \1_{\{1\}}(\tilde u_s(x)) \mathrm ds\mathrm  dx + M_t(\phi)
		\end{align}
		where $(M_t(\phi))_{t\geq 0}$ is an $L^2$-martingale with quadratic variation  
				\begin{equation}
			\AB{M_\cdot(\phi)}_t  = \iint_0^t \phi(x)^2 \tilde u_s(x) (1- \tilde u_s(x)) \mathrm ds \mathrm dx, \quad t\geq 0. 
		\end{equation}
		
		\emph{Step 3.}
		Fix an arbitrary $x\in \mathbb R$. 
		From Proposition \ref{prop:Duality}, we have that 
		\begin{align}
			&\mathbb E\SB{ \RB{u_t^{(m)}(x)}^l} 
			\label{eq:um}\\&= \tilde {\mathbb E}^{x}\SB{(-1)^{|J^{(l,m)}_t| \mathbf 1_{\{b_\infty < 0\}}} \exp\CB{(b_1+|b_\infty|)\int_0^t \ABS{I^{(l,m)}_s}\mathrm ds}\prod_{\alpha \in I^{(l,m)}_t} f\RB{X^{\alpha}_t}}
		\end{align}
		holds for each finite $l,m\in \mathbb N$, and $t\geq 0$.
		While fixing $l$, replacing $m$ by $m_k$, and then taking $k\uparrow \infty$ in \eqref{eq:um}, we obtain 
		\begin{align}
			&\mathbb E\SB{ \tilde u_t(x)^l} 
			\\\label{eq:mm}&= \tilde {\mathbb E}^{x}\SB{(-1)^{|J^{(l,\infty)}_t| \mathbf 1_{\{b_\infty < 0\}}} \exp\CB{(b_1+|b_\infty|)\int_0^t \ABS{I^{(l,\infty)}_s}\mathrm ds}\prod_{\alpha \in I^{(l,\infty)}_t} f\RB{X^{\alpha}_t}}
		\end{align}
		for every $l\in \mathbb N$ and $t\geq 0$. 
		Here, we used the bounded convergence theorem (recall that $b_1+|b_\infty|\leq 0$), Step 2, Lemma \ref{lem:DT}, and the fact (from Theorem \ref{prop:Key}) that 
 $|J^{(l,\infty)}_t| < \infty$ almost surely. 
		
		\emph{Step 4.} In this step, we will show that for every $t\geq 0$, 
				\begin{align}
\label{eq:9_2}
			&\mathbb E\SB{  \1_{\{1\}}(\tilde u_t(x)) } 
			= \tilde {\mathbb E}^{x}\SB{(-1)^{|J_t| \mathbf 1_{\{b_\infty < 0\}}} \exp\CB{(b_1+|b_\infty|)\int_0^t \ABS{I^{}_s}\mathrm ds}\prod_{\alpha \in I^{}_t} f\RB{X^{\alpha}_t}}.
		\end{align}
		Note that, from Step 3, for arbitrary $t\geq 0$ and $\epsilon > 0$, 
		\begin{align}
			&\ABS{\mathbb E\SB{ \tilde u_t(x)^l} - \tilde {\mathbb E}^{x}\SB{\mathbf 1_{\{|J_\epsilon| = 0\}}(-1)^{|J^{(l,\infty)}_t| \mathbf 1_{\{b_\infty < 0\}}} e^{(b_1+|b_\infty|)\int_0^t |I^{(l,\infty)}_s|\mathrm ds}\prod_{\alpha \in I^{(l,\infty)}_t} f\RB{X^{\alpha}_t}}}
			\\ & \leq \tilde {\mathbb P}^{x}\RB{|J_\epsilon| > 0};
		\end{align}
		which, by taking $l\uparrow \infty$ and then $\epsilon\downarrow 0$, implies the desired result for this step.
		Here, we used the bounded convergence theorem, Lemma \ref{lem:KS}, \eqref{eq:FF} and the fact (from Theorem \ref{prop:Key}) that almost surely $|J_t| < \infty$.
		
		\emph{Step 5.} While taking $l = m = m_k$ and then $k \uparrow \infty$ in \eqref{eq:um}, we obtain 
		\begin{equation}
			\lim_{k\to \infty} \mathbb E\SB{\RB{\tilde u^{(m_k)}_t(x)}^{m_k}}  = \tilde {\mathbb E}^{x}\SB{(-1)^{|J_t| \mathbf 1_{\{b_\infty < 0\}}} \exp\CB{(b_1+|b_\infty|)\int_0^t \ABS{I^{}_s}\mathrm ds}\prod_{\alpha \in I^{}_t} f\RB{X^{\alpha}_t}}.
		\end{equation}
		Here, we used the bounded convergence theorem, Lemma \ref{lem:TC}, and again the fact that almost surely $|J_t| < \infty$.
		Combine this with Step 4, we obtain that
			\begin{equation}
			\lim_{k\to \infty} \mathbb E\SB{\RB{\tilde u^{(m_k)}_t(x)}^{m_k}} = \mathbb E\SB{  \1_{\{1\}}(\tilde u_t(x))}. 
		\end{equation}
		Combine this further with \eqref{eq:EF} and Lemma \ref{lem:CP}, we obtain that
		\begin{equation}
			 \RB{\tilde u^{(m_k)}_t(x)}^{m_k} \to \1_{\{1\}}(\tilde u_t(x)) 
		\end{equation}
		in $L^p$ when $k\to \infty$ for every $p\geq 1$.
		
		\emph{Step 6.}
		Fix arbitrary $\phi\in \mathcal C_c^\infty(\mathbb R)$, where $ \mathcal C_c^\infty(\mathbb R)$ is the space of  compactly supported  infinitely differentiable functions on $\mathbb R$. 
		In this step, we want to show that, for every $t\geq 0$, when $k\uparrow \infty$,
		\[
			\iint_0^t \phi(x) \RB{\tilde u^{(m_k)}_s(x)}^{m_k}\mathrm ds \mathrm dx \to \iint_0^t \phi(x) \1_{\{1\}}(\tilde u_s(x)) \mathrm ds \mathrm dx
		\]
		in $L^2$. 
		In fact, we can calculate by H\"older's inequality that, for each $t\geq 0$ and $k\in \mathbb N$, 
		\begin{align}
			&\mathbb E\SB{\RB{\iint_0^t \phi(x) \RB{\tilde u^{(m_k)}_s(x)}^{m_k}\mathrm ds \mathrm dx - \iint_0^t \phi(x) \1_{\{1\}}(\tilde u_s(x)) \mathrm ds \mathrm dx}^2}
			\\ &\leq \mathbb E\SB{ \RB{ \iint_0^t \phi(x)^2 \RB{\RB{\tilde u^{(m_k)}_s(x)}^{m_k} - \1_{\{1\}}(\tilde u_s(x))  }^2 \mathrm ds \mathrm dx} \cdot \RB{ \iint_0^t \mathbf 1_{\{\phi(x)\neq 0\}} \mathrm ds \mathrm dx}}
			\\& \label{eq:ek} \leq t \Cr{const:phi} (\phi) \iint_0^t\phi(x)^2 \mathbb E\SB{\RB{\RB{\tilde u^{(m_k)}_s(x)}^{m_k}-\1_{\{1\}}(\tilde u_s(x)) }^2} \mathrm ds
		\end{align}
		where $0<\C\label{const:phi}(\phi)<\infty$ is a constant only depending on the support of $\phi$. 
		From Step 5, while taking $k\uparrow \infty$, we have for every $s\geq 0$ and $x\in \mathbb R$, 
				\begin{equation}
			\mathbb E\SB{\RB{\RB{\tilde u^{(m_k)}_s(x)}^{m_k}-\1_{\{1\}}(\tilde u_s(x)) }^2 } \to 0.
		\end{equation}
		From this and \eqref{eq:ek}, we obtain the desired result for this step using the bounded convergence theorem. 
		
		\emph{Step 7.} 	Fix arbitrary $\phi \in \mathcal C_c^\infty(\mathbb R)$.
		From \eqref{eq:FE}, we know that almost surely for every $t\geq 0$, 
		\begin{align} 
		&\int \phi(x)\tilde u^{(m_k)}_t(x) \mathrm dx - \int \phi(x) f(x) \mathrm dx 
		\\& \label{eq:LA}= \frac{1}{2}\iint_0^t \phi''(x) \tilde u^{(m_k)}_s(x) \mathrm ds\mathrm  dx + b_1 \iint_0^t  \phi(x) \tilde u^{(m_k)}_s(x) \mathrm ds\mathrm  dx 
		+ {}
		\\& \qquad b_\infty \iint_0^t  \phi(x) \RB{\tilde u^{(m_k)}_s(x)}^{m_k} \mathrm ds\mathrm  dx + M^{(m_k)}_t(\phi)
		\end{align}
		where $(M^{(m_k)}_t(\phi))_{t\geq 0}$ is an $L^2$-martingale with quadratic variation
		\begin{equation}
			\AB{M^{(m_k)}_\cdot(\phi)}_t = \iint_0^t \phi(x)^2 \tilde u^{(m_k)}_s(x) (1- \tilde u^{(m_k)}_s(x)) \mathrm ds \mathrm dx. 
		\end{equation}
		From \eqref{eq:EF} and the bounded convergence, we can verify that, for any $t\geq 0$, the left hand side and the first two terms on the right hand side of \eqref{eq:LA} all converge in $L^2$ while $k\uparrow \infty$. 
		Combine these with the result in Step 6, we know that for any $t\geq 0$, $M^{(m_k)}_t(\phi)$ converges in $L^2$ to a random variable $M_t(\phi)$ while $k\uparrow \infty$. 
		Now, by the standard theory for continuous $L^2$-martingales, see \cite[Proposition 1.3 \& Theorem 4.6]{MR1102676} for example, the limit $(M_t(\phi))_{t\geq 0}$ is an $L^2$-martingale with quadratic variation 
		\begin{equation}
			\AB{M_\cdot(\phi)}_t  = \iint_0^t \phi(x)^2 \tilde u_s(x) (1- \tilde u_s(x)) \mathrm ds \mathrm dx. 
		\end{equation}
		Now, by taking $k\uparrow \infty$ in \eqref{eq:LA}, we can verify the desired result \eqref{eq:AL}. 
		
		\emph{Final Step.} 
		By extending the probability space if necessary, it is standard to show (c.f. \cite[Proof of Lemma 2.4]{MR0958288}) that the martingale problem solution $(\tilde u_t(x))_{t\geq 0, x\in \mathbb R}$ is also a (mild) solution to the SPDE \eqref{eq:Ex} with respect to some space-time white noise in some probability space. Therefore, we are done. 
	\end{proof}
	
	\section{Proofs of Lemmas \ref{lem:discont1} and \ref{lem:discont2}} \label{sec:Last1}
	\begin{proof}[Proof of Lemma~\ref{lem:discont1}]
		We will prove the lemma by contradiction. Fix arbitrary $b^{(1)}_\infty, b^{(2)}_\infty\in[-1,1]$ with $b^{(1)}_\infty\not=b^{(2)}_\infty$. 
		For $i=1,2$, let $w^{(i)}$ be the unique in law solution to \eqref{SPDE_3_1} with $b_\infty=b^{(i)}_\infty$, $w^{(i)}_0= f \in\mathcal C(\mathbb R, [0,1])$. 
		Assume that $w^{(1)}$ and $w^{(2)}$ have the same laws. Fix arbitrary  non-negative and not identically zero  $\phi\in \mathcal C_c^\infty(\mathbb R)$. 
		Then for $i=1,2,$ we have 
		\begin{align}
			&\int \phi(x)w^{(i)}_t(x) \mathrm dx - \int \phi(x) f(x) \mathrm dx 
			\\& \label{eq:discont_8_1}= \frac{1}{2}\iint_0^t \phi''(x) w^{(i)}_s(x) \mathrm ds\mathrm  dx +\iint_0^t  \phi(x) (1-w^{(i)}_s(x)) \mathrm ds\mathrm  dx 
			{}
			\\& \qquad -b^{(i)}_\infty \iint_0^t  \phi(x) \1_{\{0\}}(w^{(i)}_s(x)) \mathrm ds\mathrm  dx + M^{(i)}_t(\phi)
		\end{align}
		where $(M^{(i)}_t(\phi))_{t\geq 0}$ is an $L^2$-martingale with quadratic variation  
		\begin{equation}
			\AB{M^{(i)}_\cdot(\phi)}_t  = \iint_0^t \phi(x)^2 w^{(i)}_s(x) (1- w^{(i)}_s(x)) \mathrm ds \mathrm dx, \quad t\geq 0. 
		\end{equation}
		By taking expectation on both sides of~\eqref{eq:discont_8_1}, and recalling that  $w^{(1)}$ and $w^{(2)}$ have the same law, we immediately get that 
		\begin{equation}
			\label{eq:9_1}
			b^{(1)}_\infty \mathbb E\SB{ \iint_0^t  \phi(x) \1_{\{0\}}(w^{(1)}_s(x)) \mathrm ds\mathrm  dx}  = b^{(2)}_\infty\mathbb E\SB{ \iint_0^t  \phi(x) \1_{\{0\}}(w^{(2)}_s(x)) \mathrm ds\mathrm  dx}, \quad t\geq 0. 
		\end{equation}
		
		Let us check that the expectations in~\eqref{eq:9_1} are not zero. To  this end, it is enough to check that 
		\begin{eqnarray}
			&&\mathbb E\SB{\1_{\{0\}}(w^{(i)}_t(x))}=\mathbb P\left(w^{(i)}_t(x)=0\right)>0,\;\; \forall t>0, x\in \mathbb R, i=1,2.
		\end{eqnarray}
		We will derive this by comparison argument. Let $w$ be any solution to 
		\begin{equation}
			\label{eq:9_3}
			\begin{cases}
				\partial_t w_t = \frac{1}{2}\Delta w_t + 2(1 -w_t)   + \sqrt{w_t(1-w_t)} \dot W, 
				\\ w_0 = f. 
			\end{cases}
		\end{equation}
		Clearly $u_t=1-w_t, t\geq 0,$ satisfies~\eqref{eq:SPDE} with $b(u)=-2u$ and initial conditions $u_0=1-f$. 
		Thus, such $u$ is a unique in law solution of~\eqref{eq:SPDE}. 
		Thefere, $w$  is a unique in law solution  to~\eqref{eq:9_3}. By our assumptions on $b^{(i)}_\infty$, we immediately have 
		$2(1 -z)\geq (1-z)-b^{(i)}_\infty \1_{\{0\}}(z),$  for $z\in[0,1]$, and $i=1,2.$ 
		Therefore, the drift in~\eqref{eq:9_3} dominates from above the drifts in equations for $w^{(i)}, i=1,2.$ This, by weak uniqueness, implies that $w$ scholastically dominates $w^{(1)}$ and $w^{(2)}$ from above. Therefore, 
		\begin{equation}
			\label{eq:9_5}
			\mathbb E\SB{\1_{\{0\}}(w^{(i)}_t(x))}\geq  \mathbb E\SB{\1_{\{0\}}(w_t(x))},\;\;\forall t\geq 0, x\in \mathbb R, i=1,2.
		\end{equation} 
		However, by duality formula~\eqref{eq:9_2} we obtain
		\begin{align}
			\label{eq:9_4}
			&\mathbb E\SB{\1_{\{0\}}(w_t(x))} = \mathbb E\SB{  \1_{\{1\}}(u_t(x)) } 
			= \tilde {\mathbb E}^{x}\SB{\exp\CB{-2\int_0^t \ABS{I^{}_s}\mathrm ds}\prod_{\alpha \in I^{}_t} (1-f)\RB{X^{\alpha}_t}}
		\end{align}
		where $\{(X^\alpha_t)_{t\geq 0}: \alpha \in \mathcal U\}$ is a coalescing Brownian particle system with initial configuration $(x_i)_{i=1}^\infty$ such that $x_i = x$ for every $i\in \mathbb N$. 
		Since $f\not\equiv1$, and $\int_0^t \ABS{I^{}_s}\mathrm ds$, $\ABS{I^{}_t}$  are almost surely finite for $t>0$ by Theprem~\ref{prop:Key},   we immediately  get from the properties of coalescent Brownian motions that 
		$$ \mathbb E\SB{\1_{\{0\}}(w_t(x))}>0, \quad t> 0, x\in \mathbb R.$$ 
		Then, by~\eqref{eq:9_5}
		we have 
		\begin{equation}
			\label{eq:9_6}
			\mathbb E\SB{\1_{\{0\}}(w^{(i)}_t(x))}>0,\;\;\forall t> 0, x\in \mathbb R. 
		\end{equation} 
		This and our assumptions of $\phi$  imply that the expectations in~\eqref{eq:9_1} do not equal to zero. 
		
		Now, from
		our assumption (by contradiction) that  $w^{(1)}$ and $w^{(2)}$ have the same laws, we obtain that $b^{(1)}_\infty=b^{(2)}_\infty$ which contradicts  the condition of this lemma.  This implies that $w^{(1)}$ and $w^{(2)}$  should  have different laws. 
	\end{proof}
	
	\begin{proof}[Proof of Lemma~\ref{lem:discont2} (i)]
		Let us show the non-uniqueness result when the initial value $X_0 = 0$. 
		It is easy to see that one solution to~\eqref{SDE_4_1}  with $b_\infty=1$ is $X_t\equiv 0, t\geq 0,$ while the other one can be the solution to~\eqref{SDE_4_2}. 
		Let us show that solution  to~\eqref{SDE_4_2} indeed also 
		solves~\eqref{SDE_4_1} (with $b_\infty=1$).  First let us  check that if $X$ solves~\eqref{SDE_4_2}, then 
		\begin{equation}
			\label{eq:occup_time1}
			\int_0^t \mathbf 1_{\{0\}}(X_s)\mathrm ds=0,\;\forall t\geq 0.
		\end{equation}
		In fact, one can easily get~\eqref{eq:occup_time1}, by following  the steps in the proof of Proposition~XI.1.5 in~\cite{MR1725357}. First, by using Theorem~VI.1.7 in~\cite{MR1725357} one shows that 
		\begin{equation}
			\label{eq:loc_time0}
			L_t^0(X)=2\int_0^t \mathbf 1_{\{0\}}(X_s)\RB{1 -X_s}\mathrm ds= 2 \int_0^t \mathbf 1_{\{0\}}(X_s) \mathrm ds,
		\end{equation}
		where $L_t^0(X), t\geq 0,$ is  the local time of $X$ at zero.  Then, again 
		following the proof of 
		of Proposition~XI.1.5 in~\cite{MR1725357}
		one derives that $L_t^0(X)=0, \forall t\geq 0$, and then \eqref{eq:occup_time1} follows by~\eqref{eq:loc_time0}. 
		With~\eqref{eq:occup_time1} at hand the result is immediate. 
	\end{proof}
	
	\begin{proof}[Proof of Lemma~\ref{lem:discont2} (ii)]
		The pathwise (and thus week) uniqueness for~\eqref{SDE_4_2}
		follows from Theorem~IX.3.5 in~\cite{MR1725357}. Fix arbitrary $b_\infty\in [-1, 1)$. To prove the claim we need to show that any solution $X$ to~\eqref{SDE_4_1} also solves 
		\eqref{SDE_4_2}. This will follow if we show that for any $X$ solving~\eqref{SDE_4_1}, the following holds:
		\begin{equation}
			\label{eq:occup_time2}
			b_\infty \int_0^t \mathbf 1_{\{0\}}(X_s)\mathrm ds=0,\;\forall t\geq 0.
		\end{equation}
		To obtain~\eqref{eq:occup_time2}, we follow the same strategy as in the proof of (i) of this lemma. 
		By Theorem~ VI.1.7 in~\cite{MR1725357}   we get that 
		\begin{equation}
			\label{eq:loc_time1}
			L_t^0(X)= 2 \int_0^t \mathbf 1_{\{0\}}(X_s) \mathrm dV_s,
		\end{equation}
		where 
		$V_t=\int_0^t \RB{\mathbf 1_{(0,1]}(X_s)(1 -X_s) \mathrm  + (1-b_\infty)\mathbf 1_{\{0\}}(X_s)}\, \mathrm ds, t\geq 0,$ is the drift in~\eqref{SDE_4_1}.  Substituting the definition of $V$ into
		\eqref{eq:loc_time1} we get 
		\begin{equation}
			\label{eq:loc_time2}
			L_t^0(X)= 2(1-b_\infty)\int_0^t \mathbf 1_{\{0\}}(X_s) \mathrm ds,
		\end{equation}
		Then, again 
		following the proof of 
		of Proposition~XI.1.5 in~\cite{MR1725357}
		we  derive that $L_t^0(X)=0, \forall t\geq 0$, and hence from~\eqref{eq:loc_time2} (recall that $b_\infty<1$)  we get that 
		\begin{equation}
			\label{eq:loc_time3}
			\int_0^t \mathbf 1_{\{0\}}(X_s) \mathrm ds = 0.
		\end{equation}
		Thus, \eqref{eq:occup_time2} follows and we are done.   
	\end{proof}

\appendix \section{}

\subsection{Proof of Lemma \ref{lem:TC}}\label{sec:TC}
\begin{proof}[Proof of Lemma \ref{lem:TC}]
	We claim that, 
	\begin{itemize}
		\item[\eq\label{eq:C}] for each $\alpha \in \mathcal U$ and $m\in \mathbb N$, $\zeta^{(m)}_\alpha = \mathbf 1_{\{m< \|\alpha\|_\infty \}}\xi_\alpha +  \mathbf 1_{\{m\geq \|\alpha\|_\infty \}}\zeta_\alpha$ almost surely.
	\end{itemize}
	Fixing $m \in \mathbb N$, we will prove this claim by induction over $\alpha \in \mathcal U$.
	If $\alpha = 1$, then it is easy to see that $\zeta^{(m)}_\alpha= \zeta_{\alpha,\alpha} = \zeta_\alpha$. 
	Let us fix an arbitrary $\beta \in \mathcal U$, and for the sake of induction, assume that the desired claim \eqref{eq:C} holds for every $\alpha \prec \beta$. 
	
	Firstly, we show that almost surely $\zeta^{(m)}_\beta = \xi_\beta$ provided $m < \|\beta \|_\infty$. 
	To do this, we discuss in two different cases.
	\begin{itemize}
		\item[(i)] $m < \|\beta \|_\infty$ and $|\beta|=1$. In this case, we have $m < \beta$. So by (iii) of \eqref{eq:TT}, we have $\zeta^{(m)}_\beta = \xi_\beta$ as desired. 
		\item[(ii)] $m < \|\beta \|_\infty$ and $|\beta |>1$. 
		In this case, we can show that the event 
		\[ \CB{ \zeta^{(m)}_{\overleftarrow{\beta}} = \zeta_{\overleftarrow{\beta},\overleftarrow{\beta}}} \cap \CB{\beta_{|\beta|} \leq Z_{\overleftarrow{\beta}} \wedge m} \]
		almost surely won't happen. 
		In fact, if the above event happens, we have $\beta_{|\beta|} \leq m$. 
		This, and the condition $m < \|\beta\|_\infty$, implies that $m < \|\overleftarrow{\beta}\|_\infty$.
		From what we assumed for the sake of induction, we must have $\zeta^{(m)}_{\overleftarrow{\beta}} = \xi_{\overleftarrow{\beta}}$. 
		This further implies that $\xi_{\overleftarrow{\beta}} = \zeta_{\overleftarrow{\beta},\overleftarrow{\beta}}$ which has $0$ probability.
		Now, by (iii) of \eqref{eq:TT}, we have $\zeta^{(m)}_\beta = \xi_\beta$ as desired. 
	\end{itemize}
	
	Secondly, we show that almost surely $\zeta^{(m)}_\beta = \zeta_\beta$ provided $m \geq \|\beta\|_\infty$. 
	To do this, we discuss in four different cases.
	\begin{itemize}
		\item[(i)] 
		$m \geq \|\beta\|_\infty$, $|\beta|=1$ and $\beta > n$. 
		In this case, by (iii) of \eqref{eq:Lifetime} and (iii) of \eqref{eq:TT} we have $\zeta^{(m)}_\beta = \xi_\beta = \zeta_\beta$ as desired. 
		\item[(ii)] 
		$m \geq \|\beta\|_\infty$, $|\beta|=1$ and $\beta \leq n$. 
		In this case, since $\beta \leq m$, by (i) of \eqref{eq:TT}  we have
		\begin{equation}
			\zeta^{(m)}_\beta 
			= \inf \RB{ \CB{\zeta_{\beta,\beta}} \cup \CB{\zeta_{\alpha,\beta}: \alpha \in \mathcal U, \alpha \prec \beta, \zeta_{\alpha,\beta} \leq \zeta^{(m)}_\alpha}}.
		\end{equation}
		Note that, $\alpha \prec \beta$ actually implies that $\|\alpha\|_\infty \leq m$. 
		So by what we assumed for the sake of induction, we have $\zeta^{(m)}_\alpha = \zeta_\alpha$ for every $\alpha \prec \beta$. 
		Now the above equation can be rewritten as
		\begin{equation}
			\zeta^{(m)}_\beta 
			= \inf \RB{ \CB{\zeta_{\beta,\beta}} \cup \CB{\zeta_{\alpha,\beta}: \alpha \in \mathcal U, \alpha \prec \beta, \zeta_{\alpha,\beta} \leq \zeta_\alpha}}.
		\end{equation}
		This and (i) of \eqref{eq:Lifetime} imply that $\zeta^{(m)}_\beta = \zeta_\beta$ as desired. 
		\item[(iii)] 	
		$m \geq \|\beta\|_\infty$, $|\beta| > 1$, $\zeta^{(m)}_{\overleftarrow{\beta}} = \zeta_{\overleftarrow{\beta},\overleftarrow{\beta}}$ and $\beta_{|\beta|} \leq Z_{\overleftarrow{\beta}} ^{(m)} = Z_{\overleftarrow{\beta}} \wedge m$. 
		In this case, by (ii) of \eqref{eq:TT}, we have 
		\begin{equation}
			\zeta^{(m)}_\beta 
			= \inf \RB{ \CB{\zeta_{\beta,\beta}} \cup \CB{\zeta_{\alpha,\beta}: \alpha \in \mathcal U, \alpha \prec \beta, \zeta_{\alpha,\beta} \leq \zeta^{(m)}_\alpha}}.
		\end{equation}
		Similar to the previous case, we can rewrite the above as
		\begin{equation}
			\zeta^{(m)}_\beta 
			= \inf \RB{ \CB{\zeta_{\beta,\beta}} \cup \CB{\zeta_{\alpha,\beta}: \alpha \in \mathcal U, \alpha \prec \beta, \zeta_{\alpha,\beta} \leq \zeta_\alpha}}.
		\end{equation}
		Also note that $\|\overleftarrow{\beta}\|_\infty \leq \|\beta\|_\infty \leq m$. 
		So from what we assumed for the sake of induction, we have $\zeta^{(m)}_{\overleftarrow{\beta}} = \zeta_{\overleftarrow{\beta}}$. 
		This implies that $\zeta_{\overleftarrow{\beta}} = \zeta_{\overleftarrow{\beta},\overleftarrow{\beta}}$ and $\beta_{|\beta|} \leq Z_{\overleftarrow{\beta}}$.
		Now from (ii) of \eqref{eq:Lifetime}, we have $\zeta^{(m)}_\beta  = 	\zeta_\beta $ as desired.
		\item[(iv)] 
		$m \geq \|\beta\|_\infty$, $|\beta| > 1$, and the condition in (iii) does not hold. 
		In this case, by (ii) of \eqref{eq:TT}, we have $\zeta^{(m)}_\beta  = \xi_\beta$. 
		We can verify by contradiction that the event 
		\[ 
		\CB{\zeta_{\overleftarrow{\beta}} = \zeta_{\overleftarrow{\beta},\overleftarrow{\beta}}} \cap \CB{\beta_{|\beta|} \leq Z_{\overleftarrow{\beta}}} 
		\]
		almost surely won't happen.  
		In fact, if otherwise, then from the condition $\|\beta\|_\infty \leq m$, we have $\beta_{|\beta|} \leq Z_{\overleftarrow{\beta}}\wedge m$. 
		Note that we also have $\|\overleftarrow{\beta}\|_\infty \leq \|\beta\|_\infty \leq m$. 
		So from what we have assumed for the sake of induction, $\zeta^{(m)}_{\overleftarrow{\beta}} = \zeta_{\overleftarrow{\beta}}$. 
		Then we arrived at a contradiction that the condition in (iii) holds. 
		Now, by (iii) of \eqref{eq:Lifetime}, we have $\zeta^{(m)}_\beta  = 	\zeta_\beta $ as desired. 
	\end{itemize} 
	
	To sum up, we proved claim \eqref{eq:C}. 
	The desired result for this lemma follows immediately.
\end{proof}

\subsection{Proofs of Lemmas \ref{lem:exponential} and \ref{lem:bug}} \label{sec:exp}

\begin{proof}[Proof of Lemma \ref{lem:exponential}]
	Note that by Proposition \ref{thm:TheUseful},
	\begin{equation}
		\sup_{0\leq t\leq T}\tilde{\mathbb E}\SB{ |I_t| } 
		< \infty.
	\end{equation}
	We can also verify that
	\begin{equation}
		\sup_{0\leq t\leq T} \tilde{\mathbb E} \SB{ \ABS{I^{(m)}_t}^2 } 
		< \infty.  
	\end{equation}
	In fact, by  Lemma \ref{lem:Coupling}, $|I^{(m)}_t|$ is dominated by the total population at time $t$ of a continuous-time Galton-Watson process with a bounded offspring distribution; and therefore, have all finite moments, thank to the standard theory of branching processes \cite[p. 103]{MR0163361}.
	
	It is then clear that the desired result for this lemma is trivial if $\mu + b_1 \leq 0$, because in this case, the term $1+ e^{K_t}$ is almost surely bounded by $2$.
	In particular, the result is trivial if $R = 1$, because $R = 1$ actually implies that $\mu + b_1 \leq 0$ by \eqref{eq:CLZsCondition}.
	Also note that the result is trivial if $b_\infty \neq 0$, since in order that the condition \eqref{eq:CLZsCondition} to hold, we must have $R = 1$ in this case. 
	So for the rest of this proof, we only have to show
	\begin{equation} \label{eq:DDR}
		\sup_{0\leq t\leq T}\tilde{\mathbb E}\SB{e^{K_t}\RB{1+\ABS{I_t} + \ABS{I_t^{(m)}}^2}} 
		< \infty
	\end{equation}
	under the assumption that $\mu + b_1 > 0$,  $b_\infty = 0$  and $R>1$. 
	
	From Theorem~\ref{prop:Key} and that $p_\infty = \mu^{-1}|b_\infty| =0$, we can verify that  the process $(|I_t|)_{t\geq 0}$ has finite jumps up to any finite time.
	This allows us to write down the decomposition
	\begin{align}
		R^{|I_t|} - R^{|I_0|} 
		= &\int_{\mathcal U \times \mathbb Z_+}\int_0^t (R^{|I_{s-}| + k - 1}-R^{|I_{s-}|})\mathbf 1_{\{X_{s-}^\alpha\in \mathbb R\}} \mathfrak N(\mathrm ds, \mathrm d\alpha,\mathrm dk) 
		+ {}
		\\& \int_{\mathcal R}\int_{0}^t (R^{|I_{s-}| - 1}-R^{|I_{s-}|}) \mathbf 1_{\{X_{s-}^\alpha, X_{s-}^\beta \in \mathbb R\}} \mathfrak M(\mathrm ds, \mathrm d(\alpha,\beta))
	\end{align}
	for every $t\geq 0$ where the integrals are simply finite sums.
	Consider the process
	\begin{align}
		Z_t 
		:= e^{K_t} R^{|I_t|}, \quad t\geq 0.
	\end{align}
	From the integration by parts formula, see \cite[p.~444]{MR4226142} for example, we have for $t \geq 0$,
	\begin{align}
		&Z_t - Z_0= \int_0^t R^{|I_{s-}|} \mathrm d e^{K_s} + \int_0^t e^{K_s} \mathrm d R^{|I_s|} 
		\\&= (\mu+b_1)\int_0^t e^{K_s}  R^{|I_{s-}|} |I_{s-}|\mathrm d  s + {}
		\\& \label{eq:VerifyIntegrand1} \quad \int_{\mathcal U \times \mathbb Z_+}\int_0^t e^{K_s} (R^{|I_{s-}| + k - 1}-R^{|I_{s-}|})\mathbf 1_{\{X_{s-}^\alpha\in \mathbb R\}} \mathfrak N(\mathrm ds, \mathrm d\alpha,\mathrm dk) + {}
		\\& \label{eq:VerifyIntegrand2} \quad \int_{\mathcal R}\int_{0}^t e^{K_s} (R^{|I_{s-}| - 1}-R^{|I_{s-}|}) \mathbf 1_{\{X_{s-}^\alpha, X_{s-}^\beta \in \mathbb R\}} \mathfrak M(\mathrm ds, \mathrm d(\alpha,\beta)). 
	\end{align}
	From Lemma \ref{lem:ThePointMeasures}, we know that $\mathfrak N$ and $\mathfrak M$ are QL point processes with compensators $\hat{\mathfrak N}$ and $\hat{\mathfrak M}$ respectively.
	(Recall that $\hat{\mathfrak N}$ and $\hat{\mathfrak M}$ are given in \eqref{eq:hatN} and \eqref{eq:ComM} respectively.)
	Now from Lemma \ref{eq:CompensatorDecom}, there exists a local martingale $(m_t)_{t\geq 0}$ such that  
	\begin{align}
		&Z_t - Z_0 
		= m_t + (\mu+b_1)\int_0^t e^{K_s} R^{|I_{s-}|} |I_{s-}|\mathrm d  s + {}
		\\& \qquad \qquad \quad \int_{\mathcal U \times \mathbb Z_+}\int_0^t e^{K_s} (R^{|I_{s-}| + k - 1}-R^{|I_{s-}|})\mathbf 1_{\{X_{s-}^\alpha\in \mathbb R\}} \hat{\mathfrak N}(\mathrm ds, \mathrm d\alpha,\mathrm dk) + {}
		\\&  \qquad \qquad \quad \int_{\mathcal R}\int_{0}^t e^{K_s} (R^{|I_{s-}| - 1}-R^{|I_{s-}|}) \mathbf 1_{\{X_{s-}^\alpha, X_{s-}^\beta \in \mathbb R\}} \hat{\mathfrak M}(\mathrm ds, \mathrm d(\alpha,\beta))
		\\&= m_t + \int_0^t e^{K_s} R^{|I_{s-}|}  |I_{s-}| \RB{ b_1 + \sum_{k\in \mathbb Z_+ \setminus \{1\}} R^{k - 1}  |b_k| }\mathrm ds + {}
		\\& \quad \int_{\mathcal R}\int_{0}^t e^{K_s} (R^{|I_{s-}| - 1}-R^{|I_{s-}|}) \mathbf 1_{\{X_{s-}^\alpha, X_{s-}^\beta \in \mathbb R\}} \hat{\mathfrak M}(\mathrm ds, \mathrm d(\alpha,\beta)).
	\end{align}
	\begin{extra}
		(Let us verify this in Note \ref{note:sub}.)
	\end{extra}
	By \eqref{eq:CLZsCondition}, we know that $(Z_t)_{t\geq 0}$ is a local supermartingale. 
	From the fact that $(Z_t)_{t\geq 0}$ is non-negative, we can verify that it has finite mean for any $t\geq 0$.
	In fact, since there exists a sequence of stopping time $\tau_k$ such that $\tau_k \uparrow \infty$ almost surely as $k \uparrow \infty$ and $(Z_{t\wedge \tau_k})_{t\geq 0}$ is a supermartingale for each $k \in \mathbb N$, we have by Fatou's lemma
	\[
	\tilde{\mathbb E}[Z_t] 
	= \tilde{\mathbb E}\SB{\lim_{k \to \infty} Z_{t\wedge \tau_k}} 
	\leq \liminf_{k \to\infty} \tilde{\mathbb E} [Z_{t\wedge \tau_k}] 
	\leq  \tilde{\mathbb E}[Z_{0}] = R^n < \infty.
	\]
	It is also clear from Lemma \ref{lem:TC} that $|I_t^{(m)}|\leq |I_t|$ almost surely for every $t\geq 0$.
	The desired result \eqref{eq:DDR} now follows since
	\begin{align}
		&\tilde {\mathbb E}\SB{e^{K_t}(1+|I_t| + |I_t^{(m)}|^2)} 
		\leq \tilde {\mathbb E}\SB{e^{K_t}(1+|I_t| + |I_t|^2)}
		\\&\leq \Cr{const:RI}(R) \tilde {\mathbb E}\SB{e^{K_t} R^{|I_t|}} \leq \Cr{const:RI}(R) R^n
	\end{align}
	where $\C\label{const:RI}(R) = \sup_{x\in \mathbb R} (1+|x| +|x|^2)/ R^{|x|} \in (0,\infty)$.
	\begin{extra}
		\begin{note} \label{note:sub}
			We need to show that the integrand in \eqref{eq:VerifyIntegrand1} belongs to $\mathscr L^{1,\mathrm{loc}}_{\hat{\mathfrak N}}$ and the integrand in \eqref{eq:VerifyIntegrand2} belongs to $\mathscr L^{1,\mathrm{loc}}_{\hat{\mathfrak M}}$. Those are actually easy to verify: Almost surely
			\begin{align}
				&\int_{\mathcal U \times \mathbb Z_+}\int_0^t \ABS{ e^{K_s} (R^{|I_{s-}| + k - 1}-R^{|I_{s-}|})\mathbf 1_{\{X_{s-}^\alpha\in \mathbb R\}} } \hat{\mathfrak N}(\mathrm ds, \mathrm d\alpha,\mathrm dk)
				\\&= \sum_{\alpha \in \mathcal U} \sum_{k\in \mathbb Z_+} p_k \int_0^t e^{K_s}  \ABS{ R^{k - 1}-1} R^{|I_{s-}|} \mathbf 1_{\{X_{s-}^\alpha\in \mathbb R\}}  \mu\mathrm ds
				\\&= \RB{\sum_{k\in \mathbb Z_+} p_k\ABS{ R^{k - 1}-1} }  \int_0^t e^{K_s}  R^{|I_{s-}|} |I_{s-}| \mu\mathrm ds < \infty,
			\end{align}
			and
			\begin{align}
				&\int_{\mathcal R}\int_{0}^t \ABS{e^{K_s} (R^{|I_{s-}| - 1}-R^{|I_{s-}|}) \mathbf 1_{\{X_{s-}^\alpha, X_{s-}^\beta \in \mathbb R\}} } \hat{\mathfrak M}(\mathrm ds, \mathrm d(\alpha,\beta))
				\\&=\frac{1}{2}\sum_{(\alpha,\beta)\in \mathcal R} \int_{0}^t \ABS{e^{K_s} (R^{|I_{s-}| - 1}-R^{|I_{s-}|}) } \mathbf 1_{\{X_{s-}^\alpha, X_{s-}^\beta \in \mathbb R\}}  \mathrm dL^{\alpha,\beta}_s
				\\& = \frac{1}{2} \RB{ \sup_{s\in [0,t]}\ABS{e^{K_s} (R^{|I_{s-}| - 1}-R^{|I_{s-}|}) } }  \int_{0}^t \sum_{\alpha,\beta\in I_{s-}: \alpha \prec \beta}  \mathrm dL^{\alpha,\beta}_s < \infty.
			\end{align}
			The very last inequality holds because it is essentially finite sums.
		\end{note}
	\end{extra}
\end{proof}

\begin{proof}[Proof of Lemma \ref{lem:bug}]
	
	\emph{Step 1.} 
	Let us first mention a result about the all-level supremum of the local time of the Brownian motion.
	Suppose that $L_{t,z}$ is the local time  of a standard 1-dimensional Brownian motion at level $z\in \mathbb R$ up to time $t\geq 0$. 
	Without loss of generality, we can assume that $L_{t,z}$ is jointly continuous in $t\geq 0$ and $z\in \mathbb R$ \cite[Corollary 1.8 Chapter VI]{MR1725357}. 
	It is known that for any finite time $t\geq 0$, the all-level supremum of this local time 
	\begin{equation} \label{eq:LS}
		\sup_{z\in \mathbb R} L_{t,z}
	\end{equation}
	up to time $t$ has all finite moments.
	Indeed, this result has already been used in \cite[p. 1725]{MR1813840}; and two different characterizations of \eqref{eq:LS} appeared in \cite[Excises 1.22]{MR1725357} and \cite{MR0788182}, respectively. 
	\begin{extra}
		For the sake of completeness, we give an explanation for this in Note \ref{note:5}.
	\end{extra}
	
	\emph{Step 2.}
	Fix $m\in \mathbb N$ with $m\geq n$ and $T\geq 0$. 
	Let us construct yet another particle system, denoted by $\{(\bar X^{(m),\alpha}_t)_{t\geq 0}: \alpha \in \mathcal U\}$, in the probability space where both the original branching-coalescing Brownian particle system $\{(X^{\alpha}_t)_{t\geq 0}: \alpha \in \mathcal U\}$ and its $m$-truncated version $\{(X^{(m),\alpha}_t)_{t\geq 0}: \alpha \in \mathcal U\}$ are constructed.
	This new particle system $\{(\bar X^{(m),\alpha}_t)_{t\geq 0}: \alpha \in \mathcal U\}$ will be constructed as a branching Brownian particle system which produces exactly $m$-many children at each of its branching event, and does not induce coalescing event. 
	And it will be sharing the same initial configuration $(x_i)_{i=1}^n$.  
	
	More precisely, we construct $\{(\bar X^{(m),\alpha}_t)_{t\geq 0}: \alpha \in \mathcal U\}$ through \eqref{eq:MBBM1} and \eqref{eq:MBBM2} below. 
	(Recall that $\{\zeta_{\alpha,\alpha}:\alpha \in \mathcal U\}$ and $\{(\tilde X^\alpha_t)_{t\geq 0}: \alpha \in \mathcal U\}$ are already constructed in Section \ref{eq:DualParticle} along with $\{(X^{\alpha}_t)_{t\geq 0}: \alpha \in \mathcal U\}$.) 
	\begin{itemize}
		\item[\eq\label{eq:MBBM1}]
		For each $\beta \in \mathcal U$, define $\mathbb R_+$-valued random variable $\bar \zeta^{(m)}_\beta$ inductively so that
		\begin{itemize}
			\item[(i)] 
			if $|\beta|=1$ and $\beta \leq n$, then $\bar \zeta^{(m)}_\beta := \zeta_{\beta,\beta}$. 
			\item[(ii)] 
			if $|\beta| > 1$, $\bar \zeta^{(m)}_{\overleftarrow{\beta}} = \zeta_{\overleftarrow{\beta},\overleftarrow{\beta}}$ and $\beta_{|\beta|} \leq m$, then $\bar \zeta^{(m)}_\beta := \zeta_{\beta,\beta};$
			\item[(iii)] 
			if neither of the conditions in (i) nor (ii) holds, then $\bar \zeta^{(m)}_\beta := \xi_\beta$.
		\end{itemize}
		\item[\eq\label{eq:MBBM2}]
		For each $\beta \in \mathcal U$, define $\mathbb R\cup \{\dagger\}$-valued process 
		\begin{equation}
			\bar X^{(m),\beta}_t:= 
			\begin{cases}
				\dagger, &\quad t\in [0, \xi_\beta),
				\\ \tilde X^\beta_t, &\quad t\in [\xi_\beta, \bar \zeta^{(m)}_\beta),
				\\ \dagger, & \quad t\in [\bar \zeta^{(m)}_\beta, \infty).
			\end{cases}
		\end{equation}
	\end{itemize}
	It is also clear, c.f. Lemma \ref{lem:Coupling}, that the $m$-truncated branching-coalescing Brownian particle system is dominated by $\{(\bar X^{(m),\alpha}_t)_{t\geq 0}: \alpha \in \mathcal U\}$  in the sense that almost surely,
	\begin{equation}
		I^{(m)}_t \subset \bar I^{(m)}_t 
		:= \{\alpha \in \mathcal U: \bar X^{(m),\alpha}_t \in \mathbb R\},
	\end{equation}
	and
	\[
	J^{(m)}_t \subset \bar J^{(m)}_t
	:= \{\alpha \in \mathcal U: \zeta_{\alpha,\alpha}=\bar \zeta^{(m)}_\alpha \leq t\}. 
	\]
	
	\emph{Step 3.}
	Fix arbitrary $(\alpha,\beta) \in \mathcal R$. 
	From the construction of the Brownian motions $(\tilde X^\alpha_t)_{t\geq 0}$ and $(\tilde X^\beta_t)_{t\geq 0}$, and the strong Markov property of the Brownian motions, we know that there exists a stopping time $\tau_{\alpha,\beta}$ satisfying $\tilde X^\alpha_t = \tilde X^\beta_t$ on $[0, \tau_{\alpha,\beta}]$; and that
	\begin{equation}
		\hat X^{\alpha}_t:=\tilde X^\alpha_{\tau_{\alpha,\beta}+ t} - \tilde X^\alpha_{\tau_{\alpha,\beta}}, \quad t\geq 0
	\end{equation}
	and 
	\begin{equation}
		\hat X^{\beta}_t:=\tilde X^\beta_{\tau_{\alpha,\beta}+ t} - \tilde X^\beta_{\tau_{\alpha,\beta}}, \quad t\geq 0
	\end{equation}
	are two independent Brownian motions with zero initial values.
	Denote by $\hat L^{\alpha,\beta}_{t,z}$ the local time of the process $\hat X^\alpha_\cdot - \hat X^\beta_\cdot$ up to time $t\geq 0$ at level $z\in \mathbb R$. 
	We can assume without loss of generality that $\hat L^{\alpha,\beta}_{t,z}$ is continuous in both $t\geq 0$ at $z\in \mathbb R$, c.f. \cite[Theorem 29.4]{MR4226142}.
	It is clear that
	\[
	L^{\alpha,\beta}_{T,z} = \mathbf 1_{\{\tau_{\alpha,\beta} \leq T\}} \hat L^{\alpha,\beta}_{T -\tau_{\alpha,\beta}  ,z} \leq \hat L^{\alpha,\beta}_{T,z} , \quad T\geq 0, z\in \mathbb R, \text{a.s.}
	\]
	
	\emph{Step 4.}
	Denote by $\mathscr G$ the minimal $\sigma$-field containing all the information about the genealogical structure of the branching Brownian motions $\{(\bar X^{(m),\gamma}_t)_{t\geq 0}: \gamma \in \mathcal U\}$, i.e. the $\sigma$-field generated by the death-times $\{\bar \zeta_\gamma^{(m)}: \gamma \in \mathcal U\}$.
	It is clear that the Brownian motions $\hat X^\alpha_\cdot$ and $\hat X^\beta_\cdot$ are independent of the $\sigma$-field  $\mathscr G$. 
	Therefore, using the result in Step 1, it can be shown that, for any $T\geq 0$ and $\varrho \geq 0$, there exists a constant $\C\label{const:LTM}(T,\varrho)<\infty$, which is independent of the choice of the arbitrary $(\alpha,\beta)\in \mathcal R$, such that
	\begin{align}
		&\tilde {\mathbb E}\SB{\RB{1+\sup_{z\in \mathbb R}L^{\alpha,\beta}_{T,z}}^{1+\varrho} \middle| \mathscr G}
		\leq \tilde {\mathbb E}\SB{\RB{1+\sup_{z\in \mathbb R}\hat L^{\alpha,\beta}_{T,z}}^{1+\varrho} \middle| \mathscr G}
		\\&=\tilde {\mathbb E}\SB{\RB{1+\sup_{z\in \mathbb R}\hat L^{\alpha,\beta}_{T,z}}^{1+\varrho}} = \Cr{const:LTM}(T,\varrho).
	\end{align}
	
	\emph{Step 5.}
	Using Jensen's inequality, we can verify that for any $\varrho \geq 0$, 
	\begin{align}
		&\tilde{\mathbb E}\SB{\RB{\sum_{\alpha,\beta \in I_{[0,T]}^{(m)}: \alpha \prec \beta} \RB{1 + \sup_{z\in \mathbb R} L^{\alpha,\beta}_{T,z}} }^{1+\varrho}} 
		\leq \tilde{\mathbb E}\SB{\RB{\sum_{\alpha,\beta \in \bar I_{[0,T]}^{(m)}: \alpha \prec \beta} \RB{1 + \sup_{z\in \mathbb R} L^{\alpha,\beta}_{T,z}} }^{1+\varrho}} 
		\\&\leq \tilde{\mathbb E}\SB{\ABS{\mathcal R \cap \RB{\bar I_{[0,T]}^{(m)} \times \bar I_{[0,T]}^{(m)}}}^{\varrho}\sum_{\alpha,\beta \in \bar I_{[0,T]}^{(m)}: \alpha\prec \beta} \RB{1 + \sup_{z\in \mathbb R} L^{\alpha,\beta}_{T,z}}^{1+\varrho} } 
		\\& = \tilde{\mathbb E}\SB{\ABS{\mathcal R \cap \RB{\bar I_{[0,T]}^{(m)} \times \bar I_{[0,T]}^{(m)}}}^{\varrho}\sum_{\alpha,\beta \in \bar I_{[0,T]}^{(m)}: \alpha\prec \beta} \tilde{\mathbb E} \SB{\RB{1 + \sup_{z\in \mathbb R} L^{\alpha,\beta}_{T,z}}^{1+\varrho} \middle| \mathscr G} } 
		\\&\leq \Cr{const:LTM}(T,\varrho) \tilde{\mathbb E}\SB{\ABS{\bar I_{[0,T]}^{(m)} }^{2+2\varrho} } 
		< \infty.
	\end{align}
	The last inequality is due to the standard theory of branching processes \cite[p. 103]{MR0163361}, and our assumption that the number of offspring at each branching event of the branching Brownian particle system $\{(\bar X_t^\alpha)_{t\geq 0}: \alpha \in \mathcal U\}$ is exactly $m$. 
	\begin{extra}
		(This will be explained with more details in Note \ref{note:06}.)
	\end{extra}
	
	\begin{extra}
		\begin{note} \label{note:06}
			For the branching Brownian motion $\{(\bar X^{(m),\alpha}_t)_{t\geq 0}: \alpha \in \mathcal U\}$, 
			from the fact that each branching event produces exactly $m$ children, we have the equality
			\begin{equation}
				\frac{\bar I_{[0,T]}^{(m)} - n}{m} = \# \text{branching events up to time $T$} = \frac{\bar I^{(m)}_T - n}{m-1}.
			\end{equation}
			Therefore
			\begin{align}
				\tilde{\mathbb E}\SB{\ABS{\bar I_{[0,T]}^{(m)} }^{2+2\varrho} } 
				= \tilde{\mathbb E}\SB{\ABS{ m \frac{\bar I^{(m)}_T - n}{m-1} +n }^{2+2\varrho} } < \infty.
			\end{align}
		\end{note}
	\end{extra}
	
	\emph{Step 6.}
	Thanks to Step 5, it is clear that the desired result for this lemma is trivial if $\mu + b_1 - \frac{1}{m} \leq 0$, because in this case, the term $1+ e^{K^{(m)}_t}$ is almost surely bounded by $2$.
	In particular, the result is trivial if $R = 1$, because $R = 1$ actually implies that $\mu + b_1 -\frac{1}{m}\leq \mu + b_1 \leq 0$ by \eqref{eq:CLZsCondition}.
	So, for the rest of this proof, we assume that $\mu + b_1 -\frac{1}{m}> 0$ and $R>1$ holds.
	And we only have to prove
	\begin{equation} \label{eq:DR31}
		\tilde{\mathbb E}\SB{e^{K_T^{(m)}} \sum_{\alpha,\beta \in I_{[0,T]}^{(m)}: \alpha \prec \beta} \RB{1 + \sup_{z\in \mathbb R} L^{\alpha,\beta}_{T,z}} } < \infty.
	\end{equation}
	
	\emph{Step 7.}
	Let us show that there exists a deterministic $\vartheta_m>0$ such that 
	\[\tilde {\mathbb E} \Big[e^{(1+\vartheta_m)K^{(m)}_T}\Big] < \infty.
	\]
	In fact, by Lemmas \ref{lem:TC} and \ref{lem:exponential}, we have
	\[
	\tilde {\mathbb E}\SB{\exp\CB{(\mu + b_1)\int_0^t \ABS{I_s^{(m)}} \mathrm ds}} < \infty. 
	\]
	Therefore, by taking a deterministic $\vartheta_m>0$ such that $(1+\vartheta_m)(\mu + b_1-\frac{1}{m}) = \mu + b_1$, we have the desired result for this step.
	
	\emph{Step 8.}
	Let $\vartheta_m>0$ be given as in Step 7. 
	Define $\varrho_m>0$ so that $(1+\vartheta_m)^{-1} + (1+\varrho_m)^{-1} = 1$. 
	Now by H\"{o}lder's inequality, we have 
	\begin{align}
		&\tilde{\mathbb E}\SB{e^{K_T^{(m)}} \sum_{ \alpha,\beta \in I_{[0,T]}^{(m)}: \alpha \prec \beta} \RB{1 + \sup_{z\in \mathbb R} L^{\alpha,\beta}_{T,z}} }
		\\&\leq \tilde{\mathbb E}\SB{e^{(1+\vartheta_m)K_T^{(m)}} }^{\frac{1}{1+\vartheta_m}} \tilde{\mathbb E}\SB{ \RB{\sum_{\alpha,\beta \in I_{[0,T]}^{(m)}: \alpha \prec \beta} \RB{1 + \sup_{z\in \mathbb R} L^{\alpha,\beta}_{T,z}} }^{1+\varrho_m}}^{\frac{1}{1+\varrho_m}},
	\end{align}
	which is finite, thanks to Steps 5 and 7.
	We are done.
\end{proof}

\begin{extra}
	\begin{note} \label{note:5}
		Let $L_{t,z}$ be the local time of a standard $1d$ Brownian motion, initiated at the origin, at level $z\in \mathbb R$ up to time $t$. 
		To bound the moments of the supremum of the random field $L_{t,z}$, we consider the classical dyadic argument.
		For the argument to work, we need two lemmas characterizing how the random field $L_{t,z}$ fluctuates in terms of its moments.
		\begin{lemma} \label{lem:TF}
			For every $p\in \mathbb N$, there exists a constant $\C\label{const:Time}(p)>0$ such that for any $0\leq s\leq t\leq 1$ and $x\geq 0$ we have
			\begin{equation}
				\mathbb E\SB{ \ABS{ L_{t,x} - L_{s,x}}^p } \leq  \Cr{const:Time}(p)  e^{-\frac{x^2}{2}} |t-s|^{p/2}.
			\end{equation}
		\end{lemma}
		\begin{lemma} \label{lem:SF}
			For every $p\in \mathbb N$ even, there exists a constant $\C\label{const:Spatial}(p)>0$ such that for any $t\in [0,1]$ and $0\leq x\leq y\leq x+1$ we have
			\begin{equation}
				\mathbb E\SB{ \ABS{ L_{t,x} - L_{t,y}}^p } \leq  \Cr{const:Spatial}(p)  e^{-\frac{x^2}{2}} |y-x|^{p/2}.
			\end{equation}
		\end{lemma}
		We postpone the proofs of the above two Lemmas.
		A simple corollary is the following.
		\begin{corollary}
			For every $p\in \mathbb N$ even, there exists a constant $\C\label{const:COR}(p)>0$ such that for any $0\leq s\leq t\leq 1$ and $x,y\geq 0$ with $|x-y|\leq 1$, we have
			\begin{equation}
				\mathbb E\SB{ \ABS{ L_{t,x} - L_{s,y}}^p } \leq  \Cr{const:COR}(p)  e^{-\frac{x^2}{2}} ( |y-x|^{p/2} + |t-s|^{p/2}).
			\end{equation}
		\end{corollary}
		
		For every $(s,y), (s',y') \in \mathbb Z^2$, we say they are cellmate in $\mathbb Z^2$ if $|s-s'|\leq 1$ and $|y-y'|\leq 1$. 
		For every $n\in \mathbb N$ and $(s,y), (s',y') \in (2^{-n}\mathbb Z)^2$, we say $(s,y)$ and $(s',y')$ are cellmate in $(2^{-n}\mathbb Z)^2$ if $(2^ns,2^ny)$ and $(2^ns',2^ny')$ are cellmate in $\mathbb Z^2$. 
		For every $(t,x)\in \mathbb R^2$, define its dyadic rational approximation $((t_n,x_n))_{n\in \mathbb N}$ by $t_n = 2^{-n} \lfloor 2^nt \rfloor$ and $x_n = 2^{-n}\lfloor 2^n x \rfloor$. 
		Obviously, $(t_n, x_n)$ and $(t_{n-1},x_{n-1})$ are cellmate in $(2^{-n}\mathbb Z)^2$ for every $n\in \mathbb N$.
		For every $n\in \mathbb N$, $\epsilon > 0$ and $\gamma > 0$, denote by $A_n(\epsilon,\gamma)$ the event that 
		$
		|L_{t,x} - L_{t',x'}| \leq \epsilon 2^{-\gamma n}
		$
		holds for every pair of cellmates $(t,x)$ and $(t',x')$ in $(2^{-n}\mathbb Z)^2\cap ([0,1]\times \mathbb R_+)$; then it is clear that
		\begin{equation}\label{eq:Ane} 
			\mathbb P(A_n(\epsilon,\gamma)^{\mathrm c}) 
			\leq \sum_{\substack{  (t,x),(s,y)\in (2^{-n}\mathbb Z)^2\cap ([0,1]\times \mathbb R_+): \\ \text{$(t,x)$ and $(s,y)$ are cellmate in }(2^{-n}\mathbb Z)^2}} \mathbb P(|L_{s,y}-L_{t,x}|> \epsilon 2^{-\gamma n}).
		\end{equation}
		From Chebyshev's inequality and Lemma \ref{lem:TF}, fixing an arbitrary $p\in 2\mathbb N$ and $\epsilon>0$, we have 
		\begin{align}
			&\mathbb P(|L_{t,x}-L_{s,y}|> \epsilon 2^{-\gamma n}) \leq \frac{\mathbb E[|L_{t,x}-L_{s,y}|^{p}] }{ \epsilon^p 2^{-p\gamma n}}
			\\&\leq \frac{\Cr{const:COR}(p)}{\epsilon^p 2^{-p\gamma n}}  e^{-\frac{x^2}{2}} (|t-s|^{p/2} + |y-x|^{p/2})
			\leq \frac{2\Cr{const:COR}(p)}{\epsilon^p} e^{-\frac{x^2}{2}} 2^{-(1/2-\gamma)pn}, \quad n\in \mathbb N 
		\end{align}
		provided $(t,x)$ and $(s,y)$ are cellmate in $(2^{-n}\mathbb Z)^2\cap ([0,1]\times \mathbb R_+)$. 
		Putting this back in \eqref{eq:Ane} and using the fact that each element in $(2^{-n}\mathbb Z)^2$ has exactly $9$ cellmates, we have 
		\begin{align} 
			&\mathbb P(A_n(\epsilon,\gamma)^{\mathrm c}) 
			\leq 9 \sum_{ (t,x) \in (2^{-n}\mathbb Z)^2\cap ([0,1]\times \mathbb R_+)} \frac{2\Cr{const:COR}(p)}{\epsilon^p} e^{-\frac{x^2}{2}} 2^{-(1/2-\gamma)pn}
			\\&\label{eq:UPan}= 9 (2^n+1) \sum_{j = 0}^\infty \frac{2\Cr{const:COR}(p)}{\epsilon^p} e^{-\frac{(j2^{-n})^2}{2}} 2^{-(1/2-\gamma)pn}, \quad n\in \mathbb N.
		\end{align}
		Observe that there exists a constant $\C\label{const:uni}>0$ such that for every $n\in \mathbb N$,  
		\begin{align}
			&\sum_{j=0}^\infty e^{-\frac{{(j2^{-n})}^2}{2}}
			= 1+ \sum_{j=1}^\infty \int_{j-1}^{j} e^{-\frac{{(j2^{-n})}^2}{2}} \mathrm dz
			\leq 1+ \sum_{j=1}^\infty \int_{j-1}^{j} e^{-\frac{{(z2^{-n})}^2}{2}} \mathrm dz
			\\&=1+\int_{0}^\infty e^{-\frac{{(z2^{-n})}^2}{2}} \mathrm dz
			= 1+ 2^n \int_{0}^\infty e^{-\frac{{z}^2}{2}} \mathrm dz
			\leq \Cr{const:uni} 2^n.
		\end{align}
		Put this back in \eqref{eq:UPan}, we arrived at
		\begin{align}
			&\mathbb P(A_n(\epsilon,\gamma)^{\mathrm c}) 
			\leq    \frac{\Cr{const:ANE}(p) }{\epsilon^p} 2^{(2-(1/2-\gamma)p)n}, \quad n \in \mathbb N
		\end{align}
		where $\C\label{const:ANE}(p)>0$ is a constant only depending on $p$.  
		Recall that $\gamma >0$ and $p>0$ are both arbitrary. 
		For the rest of the argument, let us fix an arbitrary $p \in 2\mathbb N$ such that $p\geq 6$, and a small enough $\gamma_p > 0$, depending only on $p$, such that $2-(1/2-\gamma_p)p<0$. 
		Let us then define event
		\[
		A(\epsilon, p):= \bigcap_{n=0}^\infty A_n(\epsilon, \gamma_p)
		\]
		on which we have almost surely that
		\[
		L_{t,x}
		=\ABS{ L_{t,x} - L_{t_0,x_0} } \leq \sum_{n=1}^\infty |L_{t_n, x_n} - L_{t_{n-1}, x_{n-1}} | \leq \sum_{n=1}^\infty  \epsilon 2^{-\gamma_p n} 
		= \epsilon \Cr{const:gammaP}(\gamma_p)
		\]
		for every $t \in [0,1)$ and $x\geq 0$
		where $ \C\label{const:gammaP}(\gamma_p)\in (0,\infty)$ is a constant depending only on the value of $\gamma_p$.
		Therefore
		\begin{align}
			&\mathbb P\RB{\sup_{x\geq 0,t\in [0,1]} L_{t,x} >\epsilon \Cr{const:gammaP}(\gamma_p)} 
			\leq \mathbb P(A(\epsilon,p)^{\mathrm c})\leq \sum_{n=0}^\infty \mathbb P(A_n(\epsilon,\gamma_p)^{\mathrm c})
			\\&\leq  \sum_{n=0}^\infty   \frac{\Cr{const:ANE}(p) }{\epsilon^p} 2^{(2-(1/2-\gamma_p)p)n} =: \frac{\Cr{const:pGp}(p,\gamma_p)}{\epsilon^p}
		\end{align}
		where $\C\label{const:pGp}(p,\gamma_p)>0$ is a constant depending only on the value of $p$ and $\gamma_p$. 
		Now, since $\epsilon > 0$ is arbitrary  in the above argument, for every $\underline{p} \in [1,p)$, we have
		\begin{align}
			&\mathbb E\SB{\ABS{\sup_{x\geq 0,t\in [0,1]} L_{t,x}}^{\underline{p}}}
			= \int_0^\infty \mathbb P\RB{\ABS{\sup_{x\geq 0,t\in [0,1]} L_{t,x}}^{\underline{p}} > z} \mathrm dz
			\\& \leq \int_0^\infty \mathbb P\RB{\ABS{\sup_{x\geq 0,t\in [0,1]} L_{t,x}}^{\underline{p}} > \RB{\epsilon \Cr{const:gammaP}(\gamma_p)}^{\underline{p}} } \mathrm d\RB{\epsilon \Cr{const:gammaP}(\gamma_p)}^{\underline{p}} 
			\\& \leq \Cr{const:gammaP}(\gamma_p)^{\underline{p}} \int_0^\infty \RB{ 1\wedge \frac{\Cr{const:pGp}(p,\gamma_p)}{\epsilon^p}} \underline{p} \epsilon^{\underline{p}-1} \mathrm d\epsilon < \infty.
		\end{align}
		Finally, since $p$ can be arbitrarily large, we obtained the finiteness of all the moments of $\sup_{x\geq 0,t\in [0,1]} L_{t,x}$ as desired.
		The rescaling property of the Brownian motion allows us to generalize this result by replacing $[0,1]$ with any compact time interval.
		
		We still need to give the proofs of Lemmas \ref{lem:TF} and \ref{lem:SF}. 
		
		\begin{proof}[Proof of Lemma \ref{lem:TF}]
			Fix $(s,x)$ and $(t,x)$ in $[0,1]\times \mathbb R$ with $s\leq t$. 
			Define stopping time $\tau(s,x):=\inf\{r \geq s: B_r = x\}$. 
			Notice that the process $r\mapsto L_{r,x}$ is flat on the interval $[s,\tau(s,x)]$. 
			Therefore 
			\[
			L_{t,x} - L_{s,x} = 
			\begin{cases}
				0, &\quad \tau(s,x) \geq t,
				\\ L_{t,x} - L_{\tau(s,x),x}, & \quad \tau(s,x) < t.
			\end{cases}
			\]
			Now, by using the strong Markov property of the Brownian motion, Levy's result on the distribution of the local time, BDG's inequality, and the reflection principle, we can verify that
			\begin{align}
				&\mathbb E\SB{ \ABS{ L_{t,x} - L_{s,x}}^p } = 
				\mathbb E\SB{ \ABS{ L_{t,x} - L_{\tau(s,x),x}}^p \mathbf 1_{\{\tau(s,x) < t\}} } 
				\\&= \mathbb E\SB{ \mathbf 1_{\{\tau(s,x) < t\}} \mathbb E\SB{ \ABS{ L_{t,x} - L_{\tau(s,x),x}}^p \middle| \mathscr F_{\tau(s,x) } } }
				= \mathbb E\SB{ \mathbf 1_{\{\tau(s,x) < t\}} \mathbb E_0\SB{ \ABS{ L_{t-r,0}}^p }|_{r=\tau(s,x)}  }
				\\& = \mathbb E\SB{\mathbf 1_{\{\tau(s,x) < t\}} \mathbb E_0 \left.\SB{\ABS{\sup_{s\in [0,t-r]} B_s}^p} \right|_{r=\tau(s,x)}} 
				\\& \leq \Cr{const:BDG}(p)  \mathbb E\SB{\mathbf 1_{\{\tau(s,x) < t\}}  \left. \ABS{t-r}^{p/2} \right|_{r=\tau(s,x)}} 
				\leq \Cr{const:BDG}(p) \ABS{t-s}^{p/2} \mathbb P\RB{\tau(0,x) < t} 
				\\&= 2\Cr{const:BDG}(p) \ABS{t-s}^{p/2} \mathbb P\RB{B_t>x}
			\end{align}
			where $\C\label{const:BDG}(p)$ is a constant depending only on $p$. 
			Notice that, for any $x\geq 0$ and $t \in [0,1]$, 
			\begin{align}
				&\mathbb P\RB{B_t>x}
				= \int_x^\infty \frac{1}{\sqrt{2\pi t}}e^{-\frac{z^2}{2t}}\mathrm dz
				\leq  1 \wedge \int_x^\infty \frac{1}{\sqrt{2\pi t}} \frac{z}{x}e^{-\frac{z^2}{2t}}\mathrm dz
				\\& = 1 \wedge  \int_x^\infty \frac{\sqrt{t}}{\sqrt{2\pi}} \frac{1}{x}e^{-\frac{z^2}{2t}}\mathrm d\frac{z^2}{2t}
				= 1 \wedge\frac{1}{x} \frac{\sqrt{t}}{\sqrt{2\pi}} e^{-\frac{x^2}{2t}}
				\\& \leq  1 \wedge\frac{1}{x} \frac{1}{\sqrt{2\pi}} e^{-\frac{x^2}{2}}
				= 	\RB{ e^{\frac{x^2}{2}} \wedge\frac{1}{x} \frac{1}{\sqrt{2\pi}} }e^{-\frac{x^2}{2}} 
				\\& \label{eq:GT} \leq \RB{ \mathbf 1_{\{x\leq 1\}}e^{\frac{x^2}{2}} + \mathbf 1_{\{x> 1\}}\frac{1}{x} \frac{1}{\sqrt{2\pi}} }e^{-\frac{x^2}{2}} 
				\leq \RB{ e^{\frac{1}{2}} + \frac{1}{\sqrt{2\pi}} }e^{-\frac{x^2}{2}}.
			\end{align}
			Therefore
			\begin{align}
				&\mathbb E\SB{ \ABS{ L_{t,x} - L_{s,x}}^p } 
				\leq 2\Cr{const:BDG}(p) \ABS{t-s}^{p/2} \RB{ e^{\frac{1}{2}} + \frac{1}{\sqrt{2\pi}} }e^{-\frac{x^2}{2}} =: \Cr{const:Time}(p) e^{-\frac{x^2}{2}}  \ABS{t-s}^{p/2}
			\end{align}
			as desired.
		\end{proof}
		\begin{proof}[Proof of Lemma \ref{lem:SF}]
			Consider Tanaka's formula for the local times:
			\begin{align}
				L_{t,x} = \ABS{B_t - x} - |x| - \int_0^t \text{sgn}(B_s- x -) \mathrm dB_s, \quad t\geq 0, x\in \mathbb R, \text{a.s.}
			\end{align} 
			Let us define 
			\begin{align}
				\tilde L_{t,x} := \ABS{B_t - x} - |x|, \quad \hat L_{t,x} := \tilde L_{t,x} - L_{t,x}, \quad t\geq 0, x\in \mathbb R. 
			\end{align}
			Fix $p\in 2\mathbb N$ and $(t,x),(t,y)$ in $[0,1]\times \mathbb R$ with $x\leq y$. 
			It is sufficient to show that 
			\begin{equation} \label{eq:TL}
				\mathbb E\SB{ \ABS{ \tilde L_{t,x} - \tilde L_{t,y}}^p } \leq  \Cr{const:Spatial1}(p)  e^{-\frac{x^2}{2}} |y-x|^{p/2}
			\end{equation}
			and
			\begin{equation} \label{eq:HL}
				\mathbb E\SB{ \ABS{ \hat L_{t,x} - \hat L_{t,y}}^p } \leq  \Cr{const:Spatial2}(p)  e^{-\frac{x^2}{2}} |y-x|^{p/2}
			\end{equation}
			for some constants $\C\label{const:Spatial1}(p)>0$ and $\C\label{const:Spatial2}(p) >0$ which are independent of the choice of $(t,x)$ and $(t,y)$.
			
			To show \eqref{eq:TL}, notice that 
			\begin{align}
				\tilde L_{t,x}- \tilde L_{t,y} &= 
				\begin{cases}
					0,& B_t\leq x,
					\\ 2B_t - 2x,&x<B_t\leq y,
					\\ 2y- 2x,&B_t \geq y.
				\end{cases}
				\\& \leq 2 \mathbf 1_{\{B_t>x\}} |y-x|.
			\end{align}
			Therefore, by \eqref{eq:GT} and the fact that $|y-x|\leq 1$ and $t\leq 1$, we have
			\begin{align}
				\tilde {\mathbb E}\SB{\ABS{\tilde L_{t,x}- \tilde L_{t,y} }^p}
				\leq 2^p |y-x|^{p} \mathbb P(B_t > x) 
				\leq 2^p |y-x|^{p/2} \RB{ e^{\frac{1}{2}} + \frac{1}{\sqrt{2\pi}} }e^{-\frac{x^2}{2}}
			\end{align}
			which is exactly the desired \eqref{eq:TL}.
			
			To show \eqref{eq:HL}, notice that almost surely
			\begin{equation}
				\hat L_{t,x} - \hat L_{t,y} = \int_0^t \RB{\text{sgn}(B_s - x -) -\text{sgn}(B_s - y -)} \mathrm dB_s= 2\int_0^t \mathbf 1_{(x,y]}(B_s) \mathrm dB_s.
			\end{equation}
			By the BDG inequality, taking $k=p/2 \in \mathbb N$, there exists a constant $\Cr{const:BDG}(p)$, depending only on $p$,  such that
			\begin{align}
				&\mathbb E\SB{ \ABS{\hat L_{t,x} - \hat L_{t,y}}^p } 
				\leq \Cr{const:BDG}(p) \mathbb E\SB{ \RB{ \int_0^t 4\mathbf 1_{(x,y]}(B_s) \mathrm ds }^{\frac{p}{2}} }
				\\&\leq 2^p\Cr{const:BDG}(p) \mathbb E\SB{ \RB{ \int_0^1 \mathbf 1_{(x,y]}(B_s) \mathrm ds }^{k} }
				\\& = 2^p\Cr{const:BDG}(p)  \int_{[0,1]^k} \mathbb P(B_{s_i}\in (x,y], \forall i=1,\cdots,k)\mathrm ds_1  \cdots \mathrm ds_k
				\\& = k! 2^p\Cr{const:BDG}(p)  \int_{[0,1]^k} \mathbf 1_{\{0<s_1< s_2< \cdots < s_k < 1\}}\mathbb P(B_{s_i}\in (x,y], \forall i=1,\cdots,k)\mathrm ds_1  \cdots \mathrm ds_k.
			\end{align}
			By induction over $k$, it can be verified, using the Markov property of the Brownian motions, that for any $0< s_1 < s_2 < \cdots < s_k < 1$,
			\[
			\mathbb E\SB{ \prod_{i=1}^{k} \mathbf 1_{\{B_{s_i}\in (x,y]\}}} 
			\leq e^{-\frac{x^2}{2}}\prod_{i=1}^k \frac{|y-x|}{\sqrt{2\pi |s_{i}-s_{i-1}|}}
			\]
			with a convention that $s_0 = 0$. 
			In fact, note that it holds when $k=1$, because
			\[
			\mathbb E\SB{\mathbf 1_{\{B_{s_1}\in (x,y]\}}} 
			= \int_x^y \frac{1}{\sqrt{2\pi s_1}} e^{-\frac{z^2}{2s_1}} \mathrm dz
			\leq e^{-\frac{x^2}{2}}\frac{|y-x|}{\sqrt{2\pi |s_{1}-s_{0}|}};
			\]
			and if it holds for a given $k\in \mathbb N$, then for any $s_{k+1} > s_k$,
			\begin{align}
				&\mathbb E\SB{ \prod_{i=1}^{k+1} \mathbf 1_{\{B_{s_i}\in (x,y]\}}}
				= 	\mathrm E\SB{  \mathrm E\SB{ \mathbf 1_{\{B_{s_{k+1}}\in (x,y]\}} \middle|  \mathscr F_{s_k}} \prod_{i=1}^{k} \mathbf 1_{\{B_{s_i}\in (x,y]\}}}
				\\&= \mathrm E\SB{  \int_x^y p_{s_{k+1}-s_k}(z-B_{s_k}) \mathrm dz  \cdot \prod_{i=1}^{k} \mathbf 1_{\{B_{s_i}\in (x,y]\}} }
				\leq \mathrm E\SB{ \frac{|y-x|}{\sqrt{2\pi |s_{k+1}-s_k|}}   \prod_{i=1}^{k} \mathbf 1_{\{B_{s_i}\in (x,y]\}} }
				\\& \leq e^{-\frac{x^2}{2}}\prod_{i=1}^{k+1} \frac{|y-x|}{\sqrt{2\pi |s_{i}-s_{i-1}|}}.
			\end{align}
			It can also be argued by induction that for every $k\in \mathbb N$,  
			\begin{align}
				\int_{[0,1]^k} \mathbf 1_{\{0<s_1< s_2< \cdots < s_k < 1\}} \RB{ \prod_{i=1}^{k} \frac{1}{\sqrt{s_i-s_{i-1}}} } \mathrm ds_1  \cdots \mathrm ds_k 
				\leq 2^k.
			\end{align}
			In fact, note that it holds when $k=1$, because
			\[
			\int_0^1 \frac{1}{\sqrt{s_1}} \mathrm ds_1 = 2\int_0^1 \mathrm d\sqrt{s_1} =2;
			\]
			and if it holds for a given $k\in \mathbb N$, then
			\begin{align}
				&\int_{[0,1]^{k+1}} \mathbf 1_{\{0<s_1< s_2< \cdots < s_k < s_{k+1}< 1\}} \RB{ \prod_{i=1}^{k+1} \frac{1}{\sqrt{s_i-s_{i-1}}} } \mathrm ds_1  \cdots \mathrm ds_{k+1}
				\\& =\int_{[0,1]^{k}} \mathbf 1_{\{0<s_1< s_2< \cdots < s_k < 1\}} \RB{ \prod_{i=1}^{k} \frac{1}{\sqrt{s_i-s_{i-1}}} } \RB{\int_{s_k}^1 \frac{1}{\sqrt{s_{k+1}-s_k}}\mathrm ds_{k+1}} \mathrm ds_1  \cdots \mathrm ds_k
				\\& =\int_{[0,1]^{k}} \mathbf 1_{\{0<s_1< s_2< \cdots < s_k < 1\}} \RB{ \prod_{i=1}^{k} \frac{1}{\sqrt{s_i-s_{i-1}}} } \RB{\int_{0}^{1-s_k} \frac{1}{\sqrt{s}}\mathrm ds} \mathrm ds_1  \cdots \mathrm ds_k
				\\& \leq 2\int_{[0,1]^{k}} \mathbf 1_{\{0<s_1< s_2< \cdots < s_k < 1\}} \RB{ \prod_{i=1}^{k} \frac{1}{\sqrt{s_i-s_{i-1}}} } \mathrm ds_1  \cdots \mathrm ds_k \leq 2^{k+1}. 
			\end{align}
			Now, from the above results, we have
			\begin{align}
				&\mathbb E\SB{ \ABS{\hat L_{t,x} - \hat L_{t,y}}^p } 
				\\& \leq k! 2^p\Cr{const:BDG}(p)  \int_{[0,1]^k} \mathbf 1_{\{0<s_1< s_2< \cdots < s_k < 1\}}\mathbb P(B_{s_i}\in (x,y], \forall i=1,\cdots,k)\mathrm ds_1  \cdots \mathrm ds_k
				\\& \leq k! 2^p\Cr{const:BDG}(p) e^{-\frac{x^2}{2}} \RB{\frac{|y-x|}{\sqrt{2\pi}}}^k \int_{[0,1]^k} \mathbf 1_{\{0<s_1< s_2< \cdots < s_k < 1\}}\prod_{i=1}^k \frac{1}{\sqrt{|s_{i}-s_{i-1}|}}\mathrm ds_1  \cdots \mathrm ds_k
				\\& =: \Cr{const:Spatial}(p)  e^{-\frac{x^2}{2}} |y-x|^{p/2}
			\end{align}
			which is exactly the desired \eqref{eq:HL}. 
		\end{proof}
	\end{note}
\end{extra}

\subsection{A technical result} \label{sec:Las}
In this subsection, let us fix an arbitrary $n\in \mathbb N$, and let $\{X_t = (X^1_t, \dots, X^n_t): t\geq 0\}$ be an $n$-dimensional Brownian motion, with initial values denoted by $(x_1, \dots, x_n)$, living  in a probability space with its probability measure denoted by $\Pi_{(x_1,\dots, x_n)}$. 
Define stopping time
\begin{equation}
	\tau = \inf \CB{t\geq 0: \frac{1}{4} \sum_{i,j=1}^n L^{i,j}_t  + \mu nt \geq \mathbf e}
\end{equation}
where $\mathbf e$ is an standard exponential random variable, independent of the Brownian motion $(X_t)_{t\geq 0}$, and $L^{i,j}_t$ is the local time of $(X^j_t - X^i_t)_{t\geq 0}$ up to time $t\geq 0$ at the level $0$. 
Define a family of operators $(\mathfrak P^\epsilon_t)_{t\geq 0, \epsilon \geq 0}$ on $\mathrm b\mathcal B(\mathbb R^n)$, the space of bounded measurable functions on $\mathbb R^n$, such that for any $F\in \mathrm b\mathcal B(\mathbb R^n)$, $(x_i)_{i = 1}^n \in \mathbb R^n$, and $\epsilon > 0$, 
\begin{align} 
	\label{eq:Last} 
	\\&(\mathfrak P^\epsilon_t F)(x_1, \dots, x_n) 
	\\&= \Pi_{(x_1, \dots, x_n)} \SB{ \mathbf 1_{\{t\leq \tau\}}\int_{\mathbb R^n} p_\epsilon(X_t^1 - y_1) \cdots p_\epsilon(X_t^n - y_n) F(y_1,\dots, y_n) \mathrm d(y_1,\cdots,y_n)}
\end{align}
and 
\begin{equation} 
	(\mathfrak P^0_t F)(x_1, \dots, x_n) 
	= \Pi_{(x_1, \dots, x_n)} \SB{ F(X_t^1,\dots, X_t^n )}
\end{equation}
where $(p_\epsilon)_{\epsilon > 0}$ are the one-dimensional heat kernels given as in \eqref{eq:HK}. 

The main result of this subsection is the following proposition which will be later used in the proof of Proposition \ref{lem:Cancellation} (5). 
\begin{proposition} \label{lem:Last}
	For any $t > 0$, $F\in \mathrm b\mathcal B(\mathbb R^n)$ and $(x_1, \dots, x_n) \in \mathbb R^n$, it holds that
	\begin{equation}
		(\mathfrak P^\epsilon_t F)(x_1, \dots, x_n) \xrightarrow[\epsilon \downarrow 0]{} 	(\mathfrak P^0_t F)(x_1, \dots, x_n).
	\end{equation}
\end{proposition}
\begin{remark}
	If $F$ is continuous on $\mathbb R^n$, then the result of Lemma \ref{lem:Last} follows from the bounded convergence theorem immediately. 
	So the point here is that $F$ can be discontinuous. 
	Also, it is crucial that $t$ is strictly larger than $0$. 
\end{remark}

Before we give the proof of Proposition \ref{lem:Last}, we mention an analytical fact.
Its proof is elementary, and therefore omitted.
\begin{lemma} \label{lem:CCs}
	Suppose that $h$ is a bounded measurable function on $\mathbb R^n$ and $q$ is a non-negative continuous function on $\mathbb R^n$ such that $\bar q$ is integrable (w.r.t. the Lebesgue measure on $\mathbb R^n$) where for any $(y_1,\dots, y_n)$ in $\mathbb R^n$,
	\begin{equation}\label{eq:DD3}
		\bar q(y_1, \dots, y_n):= \sup_{|z_i|\leq 1, i=1,\dots,n} q(y_1+z_1, \dots, y_n+z_n). 
	\end{equation}
	Then
	\begin{equation}
		\int_{\mathbb R^n} h(y_1,\dots, y_n) q^\epsilon (y_1, \dots, y_n) \mathrm d(y_1, \dots, y_n)   \xrightarrow[\epsilon \downarrow 0]{}	\int_{\mathbb R^n} h(y_1,\dots, y_n) q (y_1, \dots, y_n) \mathrm d(y_1, \dots, y_n) 
	\end{equation}
	where
	\begin{equation}
		q^\epsilon (y_1, \dots, y_n) :=  \int_{\mathbb R^n} p_\epsilon(z_1 - y_1) \dots p_\epsilon(z_n - y_n) q (z_1,\dots, z_n)\mathrm d(z_1,\cdots, z_n).
	\end{equation}
\end{lemma}
\begin{extra}
	We include the proof of Lemma \ref{lem:CCs} here for the sake of completeness.
	\begin{proof}[Proof of Lemma \ref{lem:CCs}]
		Note that 
		\begin{align}
			&\ABS{ \int_{\mathbb R^n} h(y_1,\dots, y_n) q^\epsilon (y_1, \dots, y_n) \mathrm d(y_1, \dots, y_n)  -\int_{\mathbb R^n} h(y_1,\dots, y_n) q (y_1, \dots, y_n) \mathrm d(y_1, \dots, y_n) }
			\\& \leq  \|h\|_\infty  \int_{\mathbb R^n} \ABS{q^\epsilon (y_1, \dots, y_n) - q (y_1, \dots, y_n) }\mathrm d(y_1, \dots, y_n) 
			\\&  =\|h\|_\infty \int  \ABS{ \Pi \SB{q(B_\epsilon^1+y_1, \dots, B_\epsilon^n+y_n) - q(y_1, \dots, y_n) } } \mathrm d(y_1, \dots, y_n) 
			\\&\leq \|h\|_\infty    \Pi \SB{ A(B_\epsilon^1, \dots, B_\epsilon^n) } 
		\end{align}
		where $\{(B^1_\epsilon,\dots, B^n_\epsilon): \epsilon \geq 0\}$ is a $n$-dimensional standard Brownian motion under the probability $\Pi$ and for any $(z_1,\dots, z_n)\in \mathbb R^n$,
		\begin{equation}
			A(z_1,\dots, z_n):= \int \ABS{ q(z_1+y_1,\dots, z_n+y_n) - q(y_1,\dots, y_n) } \mathrm d(y_1,\dots, y_n).
		\end{equation}
		By the dominated convergence theorem and the conditions that $q$ is continuous and that $\bar q$ is integrable, we have
		\begin{equation}
			\lim_{(z_1,\dots,z_n)\to (0,\dots, 0)}A(z_1,\dots, z_n) = A(0) = 0
		\end{equation}
		and that 
		\begin{equation}
			\sup_{(z_1,\dots, z_n)\in \mathbb R^n} |A(z_1,\dots, z_n) | \leq  2 \|q\|_{L_1(\mathbb R^n)} < \infty. 
		\end{equation}
		By those, and bounded convergence theorem, we have 
		\begin{equation}
			\Pi \SB{ A(B_\epsilon^1, \dots, B_\epsilon^n) } \xrightarrow[\epsilon \to 0]{} 0 
		\end{equation}
		which implies the desired result.
	\end{proof}
\end{extra}

\begin{proof}[Proof of Lemma \ref{lem:Last}]
	Fix arbitrary $t > 0$, $F\in \mathrm b\mathcal B(\mathbb R^n)$ and $(x_1, \dots, x_n) \in \mathbb R^n$.
	We note that
	\begin{align}
		&\Pi_{(x_1,\dots,x_n)} \SB{\mathbf 1_{\{t\leq \tau\}}  F(X^1_t, \dots, X^n_t)}
		\\&= \Pi_{(x_1, \dots, x_n)} \SB{ \exp\CB{- \frac{1}{4} \sum_{i,j=1}^n L^{i,j}_t  - \mu n t} F(X^1_t, \dots, X^n_t)}.
	\end{align}
	By using Tanaka's formula and Gilsanov transformation, we can get that 
	\begin{align}\label{eq:GS}
		&\Pi_{(x_1,\dots,x_n)} \SB{ \mathbf 1_{\{t\leq \tau\}}  F(X^1_t,\dots,X^n_t)} 
		\\&= e^{n(n^2-1) \frac{t}{24} - \mu n t}\tilde \Pi_{(x_1,\dots,x_n)}  \SB{ \frac{e^{- \frac{1}{4}\sum_{i,j = 1}^n |X^j_t-X^i_t|}}{e^{-\frac{1}{4}\sum_{i,j=1}^n |x_j - x_i|}}F(X^1_t, \dots, X^n_t)}
	\end{align}
	where $\tilde \Pi_{(x_1,\dots,x_n)}$ is a new probability measure under which $\{(X^i_s)_{s\geq 0}:i=1,\dots, n\}$ is a family of stochastic processes satisfying the SDEs
	\begin{align}\label{eq:GSDE}
		\begin{cases}
			\mathrm dX^i_s = \frac{1}{2}\sum_{j = 1}^n \text{sgn}(X^i_s - X^j_s) \mathrm dt + \mathrm d B^i_s, &\quad i = 1,\dots, n;
			\\ X^i_0 = x_i, &\quad i = 1,\dots, n.
		\end{cases}
	\end{align}
	Here, $\text{sgn}(x) := x \cdot |x|^{-1}$ for $x\in \mathbb R\setminus \{0\}$ and  $\text{sgn}(0):=0$; $\{(B^i_s)_{s\geq 0}: i = 1, \dots, n\}$ is a family of standard independent Brownian motions. \begin{extra}
		(We will verify this Gilsanov transformation result in Note \ref{note:233}.)
	\end{extra}
	
	\begin{extra}
		\begin{note}\label{note:233}
			Recall Tanaka's formula
			\begin{equation}
				L^{i,j}_t = |X^j_t - X^i_t| - |x_j - x_i| - \int_0^t \text{sgn}(X^j_s-X^i_s) \mathrm d(X^j_s- X^i_s), \quad t\geq 0.
			\end{equation}
			Therefore
			\begin{equation}
				e^{-\frac{1}{4}\sum_{i,j = 1}^n L^{i,j}_t} = \frac{e^{- \frac{1}{4}\sum_{i,j = 1}^n |X^j_t-X^i_t|}}{e^{-\frac{1}{4}\sum_{i,j=1}^n |x_j - x_i|}} e^{M_t}
			\end{equation}
			where $(M_t)_{t\geq 0}$ is a local martingale defined by
			\begin{equation}
				M_t := \frac{1}{4} \sum_{i,j=1}^n \int_0^t \text{sgn}(X^j_s-X^i_s) \mathrm d(X^j_s- X^i_s), \quad t\geq 0.
			\end{equation}
			Observe that
			\begin{align}
				&\AB{M_t} = \frac{1}{16} \sum_{i,j=1}^n \sum_{i',j'=1}^n \int_0^t \text{sgn}(X^j_s-X^i_s) \text{sgn}(X^{j'}_s-X^{i'}_s) \mathrm d\AB{X^j_s- X^i_s, X^{j'}_s- X^{i'}_s}
				\\&= \frac{1}{16} \sum_{i,j=1}^n \sum_{i',j'=1}^n \int_0^t  \text{sgn}(X^j_s-X^i_s) \text{sgn}(X^{j'}_s-X^{i'}_s) \RB{\mathbf 1_{j=j'}- \mathbf 1_{j=i'}-\mathbf 1_{i=j'}+\mathbf 1_{i=i'}}\mathrm ds
				\\& = \frac{1}{4} \sum_{\alpha,\beta,\gamma = 1}^n \int_0^t \text{sgn}(X^\alpha_s-X^\beta_s) \text{sgn}(X^\alpha_s-X^\gamma_s) \mathrm ds.
			\end{align}
			For almost every $s\geq 0$, w.r.t. Lebesgue measure, we know that $\{X^i_s: i =1,\dots n\}$ are different from one another; so, we can write $\hat X^1_s < \dots < \hat X^n_s$ as an ordering rearrangement of $\{X^i_s: i =1,\dots n\}$.
			Then,
			\begin{align}
				&\AB{M_t}
				= \frac{1}{4} \int_0^t \sum_{\alpha,\beta,\gamma = 1}^n \text{sgn}(\hat X^\alpha_s-\hat X^\beta_s) \text{sgn}(\hat X^\alpha_s-\hat X^\gamma_s) \mathrm ds
				\\& = \frac{1}{4} \int_0^t \sum_{\alpha = 1}^n \RB{\sum_{\beta=1}^n \text{sgn}(\hat X^\alpha_s-\hat X^\beta_s)}^2 \mathrm ds
				\\& = \frac{1}{4} \int_0^t \sum_{\alpha = 1}^n \RB{\alpha - 1 - (n- \alpha)}^2 \mathrm ds=\frac{1}{12}n(n^2-1)t.
			\end{align}
			It is known that $(e^{M_t- \frac{1}{2}\AB{M_t}})_{t\geq 0}$ is a local martingale \cite[Lemma 19.22]{MR4226142}. 
			From \cite[Theorem 19.24]{MR4226142}, we know that $(e^{M_t- \frac{1}{2}\AB{M_t}})_{t\geq 0}$ is a true martingale. 
			From \cite[Lemma 19.19]{MR4226142}, we can assume, without loss of generality that, there exists a new probability measure such that
			\begin{equation}
				\frac{\mathrm d\tilde \Pi_{(x_1,\dots, x_n)}}{\mathrm d\Pi_{(x_1,\dots, x_n)}} \Big|_{\mathscr F_t} = e^{M_t - \frac{1}{2}\AB{M_t}}, \quad t\geq 0. 
			\end{equation}
			Now, we can verify that \eqref{eq:GS} does hold. 
			We still need to verify that $\{(X^i_t)_{t\geq 0}: i = 1,\dots, n\}$ satisfy SDEs  \eqref{eq:GSDE} under this new probability measure.
			To do this, let us define continuous semi-martingales $\{(B^i_t)_{t\geq 0}:i=1,\dots, n\}$, with initial values $0$, according to the SDEs \eqref{eq:GSDE}. 
			From \cite[Proposition 19.21]{MR4226142}, we know that $\AB{B^i_t, B^j_t} = \mathbf 1_{i=j}t$ for every $i,j \in \{1,\dots, n\}$ and $t\geq 0$.
			So we only have to verify that, under the probability $\tilde \Pi_{(x_1,\dots, x_n)}$, $\{(B^i_t)_{t\geq 0}:i=1,\dots, n\}$ are local martingales \cite[Proposition 19.3]{MR4226142}. 
			In fact, we observe that 
			\begin{equation}
				\mathrm d  e^{M_t-\frac{1}{2}\AB{M_t}} = e^{M_t-\frac{1}{2}\AB{M_t}}\mathrm d\RB{M_t-\frac{1}{2}\AB{M_t}}+\frac{1}{2}e^{M_t-\frac{1}{2}\AB{M_t}}\mathrm d \AB{M_t}
				= e^{M_t-\frac{1}{2}\AB{M_t}}\mathrm dM_t
			\end{equation}
			and that
			\begin{align}
				&\mathrm d\AB{X^\alpha_t,M_t} = \frac{1}{4} \sum_{i,j=1}^n \text{sgn}(X^j_t - X^i_t) \mathrm d\AB{X^\alpha_t,X^j_t - X^i_t}
				\\&= \frac{1}{4} \sum_{i,j=1}^n \text{sgn}(X^j_t - X^i_t) \RB{\mathbf 1_{\alpha = j} - \mathbf 1_{\alpha = i}} \mathrm dt
				=  \frac{1}{2}\sum_{i=1}^n \text{sgn}(X^\alpha_t - X^i_t) \mathrm dt.
			\end{align}
			Therefore, 
			\begin{align}
				&\mathrm dX^\alpha_t -  \frac{1}{e^{M_t-\frac{1}{2}\AB{M_t}} }\mathrm d\AB{X^\alpha_t, e^{M_t-\frac{1}{2}\AB{M_t}} }
				= \mathrm dX^\alpha_t -\mathrm d\AB{X^\alpha_t, M_t }
				\\& =\mathrm dX^\alpha_t  - \frac{1}{2} \sum_{j=1}^n \text{sgn}(X^\alpha_t - X^j_t) \mathrm dt = \mathrm dB^i_t.
			\end{align}
			Now, from \cite[Theorem 19.20]{MR4226142}, we know that, under the probability $\tilde \Pi_{(x_1,\dots, x_n)}$, $\{(B^i_t)_{t\geq 0}:i=1,\dots, n\}$ are local martingales, as desired.
		\end{note}
	\end{extra}
	
	It is known from \cite[Theorem 1.2]{MR4161389} that, under the probability $\tilde \Pi_{(x_1,\dots, x_n)}$, the random vector $(X_t^1,\dots, X_t^n)$ has a continuous density, denoted by $\tilde q_{t}$, with respect to the Lebesgue measure on $\mathbb R^n$.
	Therefore, we have 
	\begin{align}
		&\Pi_{(x_1,\dots, x_n)} \SB{ \mathbf 1_{\{t\leq \tau\}}  F(X^1_t, \dots, X^n_t)}
		\\&= e^{n(n^2-1) \frac{t}{24} - \mu nt + \frac{1}{4}\sum_{i,j = 1}^n |x_i - x_j|} \times {}
		\\& \qquad \int_{\mathbb R^n} e^{-\frac{1}{4}\sum_{i,j=1}^n |z_i-z_j|} \tilde q_{t} (z_1,\dots, z_n) F(z_1,\dots, z_n) \mathrm d(z_1,\dots, z_n) 
		\\ & \label{eq:df}=  \int_{\mathbb R^n} q_{t} (z_1,\dots, z_n) F(z_1,\dots, z_n) \mathrm d(z_1,\dots, z_n)
	\end{align}
	where
	\begin{align}
		q_t(z_1,\dots, z_n) := e^{n(n^2-1) \frac{t}{24} - \mu nt + \frac{1}{4}\sum_{i,j = 1}^n |x_i - x_j| -\frac{1}{4}\sum_{i,j=1}^n |z_i-z_j|}\tilde q_t(z_1,\dots, z_n). 
	\end{align}
	It is also known from  \cite[Theorem 1.2]{MR4161389} that $\tilde q_t$ is dominated by the $n$-dimensional heat kernel, up to certain centering and scaling.
	In particular, we can verify that $\bar q_t$ is integrable w.r.t. the Lebesgue measure on $\mathbb R^n$ where for any $(y_1,\dots, y_n)\in \mathbb R^n$, 
	\begin{equation}
		\bar q_t(y_1,\dots, y_n) := \sup_{|z_i|\leq 1, i=1,\dots,n} q_t(y_1+z_1, \dots, y_n+z_n). 
	\end{equation}
	
	Since $F \in \mathrm b\mathcal B(\mathbb R^n)$ in \eqref{eq:df} is arbitrary, by Fubini's theorem, we have
	\begin{equation}
		(\mathfrak P_t^\epsilon F)(x_1, \dots, x_n)= \int_{\mathbb R^n} q_t^\epsilon (y_1, \dots, y_n)  F(y_1,\dots, y_n) \mathrm d(y_1,\cdots,y_n)
	\end{equation}
	where
	\begin{align}
		q_t^\epsilon (y_1, \dots, y_n) =  \int_{\mathbb R^n} p_\epsilon(z_1 - y_1) \dots p_\epsilon(z_n - y_n) q_{t} (z_1,\dots, z_n)\mathrm d(z_1,\cdots, z_n).
	\end{align}
	Now the desired result follows from Lemma \ref{lem:CCs}.
\end{proof}

\end{document}